\documentclass{amsart}
\usepackage{amsmath, amssymb, amsthm, color, graphicx, enumitem, epsfig, imakeidx}
\usepackage[hidelinks]{hyperref}
\usepackage{chngcntr}
\counterwithin{equation}{section}

\newtheorem{Theorem}{Theorem}[section]
\newtheorem{Lemma}[Theorem]{Lemma}
\newtheorem{Corollary}[Theorem]{Corollary}
\newtheorem{Conjecture}[Theorem]{Conjecture}
\newtheorem{Proposition}[Theorem]{Proposition}
\theoremstyle{definition}
\newtheorem{Example}[Theorem]{Example}
\newtheorem{Definition}[Theorem]{Definition}

\newtheorem{Remark}[Theorem]{Remark}

\def\Index#1{{\it #1}\index{#1}}

\newcommand{\phgq}[4]{
{}_{2}\mathbb P_{1} \left[
\begin{matrix}
#1 & #2\smallskip \\
   & #3
\end{matrix}
\;;\; #4
\right]
}

\newcommand{\fqhat}
{\widehat{\mathbb{F}_{q}^{\times}}}
\newcommand{\fq}
{\mathbb{F}_q}

\newcommand*\HYPERskip{&}
\catcode`,\active
\newcommand*\pFq{
\begingroup
\catcode`\,\active
\def ,{\HYPERskip}%
\doHyper
}
\catcode`\,12
\def\doHyper#1#2#3#4#5{%
\, _{#1}F_{#2}\left[\begin{matrix}#3 \smallskip \\  #4\end{matrix} \; ; \; #5\right]%
\endgroup
}

\newcommand*\HYPER{&}
\catcode`,\active
\newcommand*\pFFq{
\begingroup
\catcode`\,\active
\def ,{\HYPER}%
\doHyperF
}
\catcode`\,12
\def\doHyperF#1#2#3#4#5{%
\, _{#1}{\mathbb F}_{#2}\left[\begin{matrix}#3 \smallskip \\  #4\end{matrix} \; ; \; #5\right]%
\endgroup
}

\newcommand*\HYPERp{&}
\catcode`,\active
\newcommand*\pPq{
\begingroup
\catcode`\,\active
\def ,{\HYPERp}%
\doHyperFp
}
\catcode`\,12
\def\doHyperFp#1#2#3#4#5{%
\, _{#1}{P}_{#2}\left[\begin{matrix}#3 \smallskip \\  #4\end{matrix} \; ; \; #5\right]%
\endgroup
}

\newcommand*\HYPERpp{&}
\catcode`,\active
\newcommand*\pPPq{
\begingroup
\catcode`\,\active
\def ,{\HYPERpp}%
\doHyperFpp
}
\catcode`\,12
\def\doHyperFpp#1#2#3#4#5{%
\, _{#1}{\mathbb P}_{#2}\left[\begin{matrix}#3 \smallskip \\  #4\end{matrix} \; ; \; #5\right]%
\endgroup
}

\def\[{\left[}
\def\]{\right]}
\def\[{\left[}
\def\]{\right]}
\def\({\left(}
\def\){\right)}
\def\ol{\overline}
\def \fp{\frak p}
\def\CC#1#2{\binom {#1}{#2}}
\def\C{\mathbb{C}}
\def\R{\mathbb{R}}
\def\Z{\mathbb{Z}}
\def\Q{\mathbb{Q}}
\def\F{\mathbb{F}}
\def\O{\mathcal{O}}
\def\H{\mathfrak H}
\def\M#1#2#3#4{\begin{pmatrix}#1&#2\\#3&#4\end{pmatrix}}

\def \l {\lambda}
\def\la{\left\langle}
\def\ra{\right\rangle}

\def \eps{\varepsilon}
\def\G{\Gamma}

\def\OqN{ {\mathcal O}_{\mathbb Q(\zeta_N)}}

\def \z{\zeta}
\def\chiup{\protect\raisebox{1pt}{$\chi$}}

\newcommand  \bk{\color{black}}

\makeindex

\begin{document}

%\frontmatter

%%%TITLE
\title{Hypergeometric functions over finite fields}

%%%AUTHORS
\author{Jenny Fuselier, Ling Long, Ravi Ramakrishna, Holly Swisher, Fang-Ting Tu}

%\author{Jenny Fuselier}
\address{High Point University, High Point, NC 27268, USA}
\email{jfuselie@highpoint.edu}

%\author{Ling Long}
\address{Louisiana State University, Baton Rouge, LA 70803, USA}
\email{llong@lsu.edu}

%\author{Ravi Ramakrishna}
\address{Cornell University, Ithaca, NY 14853, USA}
\email{ravi@math.cornell.edu}

%\author{Holly Swisher}
\address{Oregon State University, Corvallis, OR 97331, USA}
\email{swisherh@oregonstate.edu}

%\author{Fang-Ting Tu}
\address{ Louisiana State University, Baton Rouge, LA 70803, USA}
\email{ftu@lsu.edu}

%%%DATE(required, date received by editor)
%\date{XXX, 2016}

%%%SUBJECT CLASS

%%%KEYWORDS

%\dedicatory{Dedication text (use \\[2pt] for line break if necessary)}

\maketitle

\tableofcontents

\begin{abstract}
Building on the developments of many people including Evans, Greene, Katz, McCarthy,  Ono,  Roberts, and Rodriguez-Villegas, we consider period functions for hypergeometric type algebraic varieties over finite fields and consequently study hypergeometric functions over finite fields in a manner that is parallel to that of the classical hypergeometric functions. Using a comparison between the classical gamma function and its finite field analogue the Gauss sum, we give a systematic way to obtain certain types of  hypergeometric  transformation and evaluation formulas over finite fields and interpret them geometrically using a Galois representation perspective. As an application, we obtain a few finite field analogues of algebraic hypergeometric identities, quadratic and higher transformation formulas, and evaluation formulas. We further apply these finite field formulas to compute the number of rational points of certain  hypergeometric varieties.
\end{abstract}

%%%UNNUMBERED CHAPTERS(preface, acknowledgments, etc.)
%\include{Preface}

\section*{Acknowledgments}

Both Long and Tu were supported by NSF DMS \#1303292 and Long was further supported by NSF DMS \#1602047.  Ramakrishna was supported by Simons Foundation Grant \#524863. We would like to thank Microsoft Research and the Office of Research \& Economic Development at Louisiana State University (LSU) for support  which allowed Fuselier and Swisher to visit LSU in 2015. Tu was supported in part by the
National Center for Theoretical Sciences (NCTS) in Taiwan for her visit to LSU in 2015. We are  very grateful to George Andrews,  Frits Beukers,  Irene Bouw, Luca Candelori,  Henri Cohen, Henri Darmon, Ron Evans, Sharon Frechette, Jerome  Hoffman,  Wen-Ching Winnie Li, Matt Papanikolas, Dennis Stanton, John Voight, Daqing Wan, Tonghai Yang and Wadim Zudilin for their interest, enlightening discussions/suggestions  and/or  helpful references.  Special thanks go to Dennis Stanton who gave us two interesting problems and  to Yifan Yang for sharing his results which are stated in Example \ref{eg:(2,3,3)}. Special thanks also go to Rodriguez-Villegas for his lecture on Hypergeometric Motives given at NCTS in 2014.  Our statements are verified by {Magma} or {Sage} programs. The collaboration was carried out on  SageMathCloud.  Long, Swisher and Tu further thank the Institute for Computational and Experimental Research in Mathematics (ICERM) at Brown University for  its  hospitality as part of this work was done while they were ICERM research fellows  during  Fall 2015.  Finally, we thank the referees for the extremely careful reading of our work and the many detailed and helpful suggestions.

%\mainmatter

%%%MAIN CHAPTERS

\section{Introduction}\label{intro}

\subsection{Overview}\label{overview}
Starting with seminal works of Evans \cite{Evans81,Evans86,Evans91}, Greene \cite{Greene, Greene87,Greene93}, Katz \cite{Katz}, Koblitz \cite{Koblitz}, McCarthy \cite{McCarthy, McCarthy-p-adic},  Ono \cite{Ono}, Roberts and Rodriguez-Villegas \cite{RRW}, Greene and  Stanton \cite{Greene-Stanton} et al., the finite field analogues of special functions, in particular generalized hypergeometric functions, have been developed theoretically, and are known to be related to various arithmetic objects  \cite[et al.]{BCM, FOP, Fuselier, Lin-Tu, McCarthy-Papanikolas}. Computations on the corresponding hypergeometric motives have been implemented in computational packages like \texttt{Pari} and  \texttt{Magma},  in particular, see  \cite{Watkins} for the \texttt{Magma} documentation by Watkins.

We focus on the finite field analogues of classical transformation and evaluation formulas based on the papers mentioned above, especially \cite{Greene} by Greene  and \cite{McCarthy} by McCarthy. To achieve our goals, we modify their notation slightly so that the analysis more closely parallels the classical case, and  the  proofs in the complex case can be extended naturally to the finite field setting. We emphasize that the classical $_{n+1}F_n$ functions with rational parameters and suitable  normalizing factors can be computed from the
$_{n+1}P_n$  period  functions
(see {{\eqref{inductiveintegral}}}). These are related to the periods of the hypergeometric varieties of the form
\begin{equation}\label{eq:hyp-var}y^N=x_1^{i_1}\cdots x_n^{i_n}(1-x_{1} )^{j_1}\cdots (1-x_n)^{j_n}(1-\l x_1x_2\cdots x_n)^k,\end{equation}
where $N,i_s,j_t,k$ are  positive  integers, $\l$ is a fixed parameter, and $y$, $x_m$ are variables. It is helpful to distinguish the $_{n+1}F_n$ functions from  the period functions  $_{n+1}P_n$. In particular we consider their finite field analogues, which we denote by $_{n+1}\mathbb F_n$ and $_{n+1}\mathbb P_n$, respectively (see definitions \eqref{n+1Pn} and \eqref{general_HGF}).

In the  finite field setting, we use two main methods throughout. The first method is calculus-style and is based on the conversion between the classical and finite field settings that is well-known to experts. This is summarized in a table in \S \ref{dictionary}.  In the classical setting, the gamma function satisfies reflection and multiplication formulas (see Theorems \ref{thm:reflect} and \ref{thm:multiplication}) which yield many interesting hypergeometric identities.  Roughly speaking, Gauss sums are finite field analogues of the gamma function \cite{Evans91} and they satisfy very similar reflection and multiplication formulas (see \eqref{Gauss1} and Theorem \ref{Hasse Davenport}).  The technicality in converting identities from the classical setting lies in analyzing  the `error terms' associated with Gauss sum computations (which are delta terms here, following Greene's work  \cite{Greene}). This approach  is particularly handy for translating  to the finite field setting a classical transformation formula that satisfies the following condition:
\begin{quote}\index{($\ast$) condition}
($\ast$) it can be proved using only  the binomial theorem,  the reflection and multiplication formulas,  or their corollaries (such as the Pfaff-Saalsch\"utz formula {\color{black} \eqref{eq:pf-s}}).
\end{quote}

For the second method,  we rely on the built-in symmetries for the $_{n+1}\mathbb{F}_n$  hypergeometric functions over finite fields, such as the Kummer relations for $_2\F_1$ (\S \ref{ss:Kummer}), \bk which are very helpful for many purposes including getting around the error term analysis.

We emphasize alignment with geometry by putting the finite field analogues in the explicit context of hypergeometric varieties and their corresponding Galois representations. This gives us important guidelines as the finite field analogues, in particular evaluation formulas such as the Pfaff-Saalsch\"utz identity, do not look exactly like the classical cases. From a different perspective, hypergeometric functions over finite fields are twisted exponential sums in the sense of \cite{AS1, AS2} and are related to hypergeometric motives \cite{BCM, Katz,   RRW}, abstracted from cohomology groups of the algebraic varieties. This direction is pursued by Katz \cite{Katz} and further developed and implemented by many others including  Roberts and Rodriguez-Villegas \cite{RRW},   and Beukers,  Cohen, Mellit \cite{BCM} with a focus on motives defined over $\Q$. It offers a profound  approach to hypergeometric motives and their L-functions by combining both classical  and $p$-adic analysis, arithmetic, combinatorics, computation, and geometry. Our approach to the Galois perspective is derived from  our attempt to understand hypergeometric functions over finite fields parallel to the classical setting and Weil's approach to consider Jacobi sums as Gr\"ossencharacters  (or Gr\"ossencharakteres)  \cite{Weil}. This is reflected in our notation in   {{Sections}} \ref{Gal-background} and \ref{Gal} \bk on the Galois perspective.\\

\subsection{Organization and  the main results}

This memoir is organized as follows. Secrtion \ref{prelim} contains preliminaries on gamma and beta functions and their finite field counterparts, Gauss and Jacobi sums.  In  \S \ref{dictionary} we give a correspondence between objects in the classical and finite field settings.  We recall some well-known facts for classical hypergeometric functions in Section \ref{setting}. We assume in this memoir that the parameters for all the hypergeometric functions are rational numbers. In Section \ref{finitefield}, we introduce the period functions $_{n+1} \mathbb P_n$  in Definition \eqref{n+1Pn},  and the normalized period functions $_{n+1} \mathbb F_n$  in Definition \eqref{general_HGF}.

In Section \ref{Gal-background}, we {{provide}} some related background on Galois representations. In Section \ref{Gal}, we interpret the $_{n+1} \mathbb P_n$ and $_{n+1} \mathbb F_n$ functions, in particular the $n=1$ case, as traces of Galois representations at Frobenius elements via explicit hypergeometric algebraic varieties.  Let  $\z_N:=e^{2\pi i/N}$ be a primitive $N$th root of $1$ and $K=\Q(\zeta_N)$.   We use $\O_K$ for its ring of integers, $\ol K$ for its algebraic closure, and set $G_K:=\text{Gal}(\ol K/K)$.   Recall that a prime ideal $\fp$ of $\mathcal O_K$ is unramified if it is coprime to the discriminant of $K$.  Given a  fixed  rational number of the form $\frac{i}{N}$,  for any  nonzero  prime ideal $\fp$ of $K$ coprime to $N$, one can assign a multiplicative character,  denoted by   $\iota_{\fp}(\frac iN)$, on the residue field  $\O_K/\fp$  of $\fp$ with size $q(\fp):=\#(\O_K/\fp)$,  see \eqref{msymbol}. This assignment,  based on the $N$th power residue symbol, is \emph{compatible} with the Galois perspective when $\fp$ varies, by which we mean one can build continuous   Galois representations using the residue symbol.   Then our setting for the finite field period functions and the dictionary between the two settings give us one way to convert the classical hypergeometric functions to the hypergeometric functions over the residue fields  in a compatible way  in the same sense mentioned above.  We illustrate the map by the following diagram
\begin{multline*}
 \pPq{n+1}{n}{a_1&a_2&\cdots&a_{n+1}}{&b_1&\cdots& b_n}{\l} \\
 \mapsto \pPPq{n+1}{n}{\iota_\fp(a_1)&\iota_\fp(a_2)&\cdots&\iota_\fp(a_{n+1})}{&\iota_\fp(b_1)&\cdots& \iota_\fp(b_n)}{\l; q(\fp)},
 \end{multline*}
where
%$K=\Q(\z_N,\l)$,
$N$ is the least positive common denominator of all $a_i$'s and $b_j$'s, and $\fp$ is any prime ideal of $\O_K$ unramified in $K/\Q$. There is a similar correspondence between the hypergeometric functions $_{n+1}F_n$ and $_{n+1}\F_n$.

Thus when we speak of the finite field analogues of classical period (or hypergeometric) functions in this paper we mean the  converted functions over the finite residue fields, unless we specify otherwise. We show the following.

\begin{Theorem}\label{thm:WIN3a}
Let $a,b,c\in \Q$ with least common denominator $N$ such that $a$, $b$, $a-c$, $b-c\notin \Z$  and $\l\in   \Q\setminus \{0,1\}$. Set $K=\Q(\z_N)$  and denote its ring of integers $\mathcal O_K$. Let $\ell$ be any prime.
Then there exists a   representation $$\sigma_{\l,\ell}:G_{K}:=\text{Gal}(\overline K/K) \to GL_2(\Q_\ell(\zeta_N)),$$ depending on $a,b$ and $c$, that is unramified at all prime ideals $\fp$ of $\mathcal O_K$
that are relatively prime to $\ell \cdot Disc(K/\Q)$ and satisfy $ord_{\fp}(\l)=0=ord_{\fp}(1-\l)$.
Furthermore,
the trace of Frobenius at $\fp$ in the image of $\sigma_{\l,\ell}$
is the well-defined algebraic integer 
%and equal to
%$\sigma_{\l,\ell}$ evaluated at the  Frobenius conjugacy class $\text{Frob}_{\fp}$ at $\fp$ is an algebraic integer (independent of the choice of $\ell$), satisfying
\begin{equation}
%\text{Tr} \,\sigma_{\l,\ell} (\text{Frob}_{\fp})=
- \phgq{\iota_\fp(a)}{\iota_\fp(b)}{\iota_\fp(c)}{\l; q(\fp)}.
\end{equation}
\end{Theorem}

Our proof {{(see \S \ref{ss:6.3})}} makes use of  generalized Legendre curves (recalled in   \S \ref{GLC})  which are $1$-dimensional hypergeometric varieties. Viewing the finite field period functions as the traces of  finite dimensional Galois representations at Frobenius elements is helpful when we translate classical results to the finite field setting. Using Weil's result (\cite{Weil}) realizing Jacobi sums as Gr\"ossencharacters, one can give a similar interpretation for the normalized period functions $_{n+1}\F_n$. In Section \ref{Ramanujan} we  use the Galois perspective to discuss a finite field analogue of the Clausen formula by Evans and Greene, as well as its relations to analogues of some Ramanujan type formulas for $1/\pi$.   For readers only interested in  direct finite field analogues of classical formulas,  Sections \ref{Gal-background}, \ref{Gal} and \ref{Ramanujan} can be skipped as our later proofs (except in Section \ref{ss:application}) are mainly about character sums over finite fields.

In Section \ref{translations} we first recall and reinterpret the finite field analogues of Kummer's 24  relations considered by Greene \cite{Greene}. These relations give us many built-in symmetries for the $_2\F_1$ functions,  which are closely related to the classical Kummer relations for $_2F_1$  (see \S \ref{ss:Kummer}).   Then we use the dictionary between the classical and finite field settings to obtain a few finite field analogues of algebraic hypergeometric identities.   For example, we recall the following formula of Slater (see \cite[(1.5.20)]{Slater}):
\begin{equation}\label{eq:Slater}\pFq{2}{1}{a&a-\frac 12}{&2a}z=\( \frac{1+\sqrt{1-z}}2\)^{1-2a}.
\end{equation} Here, we call such a formula, which expresses a hypergeometric function (formally) as an algebraic function of the argument $z$,  an {\it algebraic hypergeometric identity}.   When $a\in \Q$,  the corresponding monodromy group (\S \ref{monodromy}) is a   finite  group.   The identity \eqref{eq:Slater}  also satisfies the ($\ast$) condition of \S \ref{overview}. To see its finite field analogue, it is tempting to translate the right hand side into a corresponding character evaluated at $\frac{1+\sqrt{1-z}}2$ using the correspondence in  \S \ref{dictionary}. However, Theorem \ref{thm:WIN3a} implies that {this type of direct translation will not yield a correct formula, } as the corresponding Galois representations should be $2$-dimensional. In fact, we will show (see Theorem \ref{thm:FF-Dihedral}) that for $\F_q$ of odd characteristic, $\phi$ the quadratic character, $A$ any multiplicative character on $\mathbb{F}_q^{\times}$ having order at least 3, and $z\in \F_q$,
\begin{equation*}%\label{eq:Slater-FF}
\pFFq{2}1{A&A\phi}{&A^2}{z}=\( \frac{1+\phi(1-z)}2\) \(\overline A^2\(\frac{1+\sqrt{1-z}}2\)+\overline A^2\(\frac{1-\sqrt{1-z}}2\)\).\end{equation*}Note the right hand side is well-defined as its first factor takes value zero if  $1-z$ is not a square in $\F_q$ and 1 otherwise. 
Here we are using a bar to denote complex conjugation.  This result was inspired by a recent result of Tu and Yang in \cite{Tu-Yang2}.  It implies the following  formula.  Writing $\fqhat$ for the set of all multiplicative characters on $\F_q^\times$, if $A$, $B$, $AB$, $A\ol B\in\widehat{\F_q^\times}$ each have order larger than 2, then for $z\neq 1$,
\begin{align*}
  \pFFq{2}{1}{A&{A}\phi}{& A^2}{z}& \pFFq{2}{1}{B&B\phi}{&B^2}{z}\\
  =&\pFFq{2}{1}{AB&  AB\phi}{&(AB)^2}{z} +\ol B^2\(\frac z4\)\pFFq{2}{1}{A\ol B &  A\ol B \phi}{&(A\ol B)^2}{z}.
\end{align*}Note that the last term is necessarily symmetric in $A$ and $B$, although this is not obvious from its appearance.

In Section \ref{higher}, we  apply our main technique, which is to obtain analogues of classical formulas satisfying the ($\ast$) condition using the dictionary between the two settings, to prove a quadratic formula over finite fields (Theorem \ref{thm:quad-2F1}). This is equivalent to a quadratic formula of Greene in \cite{Greene}. The  proof appears technical on the surface, but is parallel to the classical proof.   In comparison, the approaches of Evans and Greene to higher order transformation formulas (such as \cite{Evans-Greene3, Greene}) often involve subtle and clever changes of variables.

Using a similar approach, we prove analogues of the Bailey cubic $_3F_2$ formulas (see Theorems \ref{Thm:Bailey-FF} and \ref{Thm: Bailey FF2}) and consequently an analogue of a formula by Andrews and Stanton, see Theorem \ref{thm:A-S-FF}.

Next, we use a different approach to obtain a finite field analogue of  the following cubic formula (\cite[(5.18)]{Gessel-Stanton})  by Gessel and Stanton. For $a,x\in\C$,
\begin{equation*}
\pFq{2}{1}{a&-a}{&\frac 12}{\frac{27x(1-x)^2}4}=\, \pFq{2}{1}{3a&-3a}{&\frac 12}{\frac{3x}4},
\end{equation*} when both sides converge. %changed , to &
This formula  satisfies the ($\ast$) condition.   We obtain  the following explicit finite field analogue.
\begin{Theorem}\label{thm:cubic1}Let $\F_q$ be the finite field of $q$ elements and assume its residue characteristic is at least $3$.
 If $A\in \widehat{\F_q^\times}$, then for all $x\in \F_q$,
\begin{multline*}%\label{eq:63}
\pFFq{2}{1}{A&\ol{A}}{&\phi}{\frac{27x(1-x)^2}4}=\\ \pFFq{2}{1}{A^3&\ol{A^3}}{&\phi}{\frac{3x}4}-\phi(-3)\delta(x-1)-\phi(-3)A(-1)\delta(x-{4}/3),
\end{multline*}where  $\phi$ is the quadratic character,  $\delta(0)=1$ and $\delta(x)=0$ otherwise.
 \end{Theorem}As a consequence of this cubic formula, we obtain the following  \emph{evaluation formula},  by which we mean it is a formula which gives the value of a hypergeometric series at a fixed argument.
In the theorem statement below $J(\, ,\,)$ and $B(\, , \,)$ denote the Jacobi sum and beta function. We will recall their definitions in  \S \ref{gamma and beta} and \S \ref{ss:Gauss&Jacobi}.  As we shall see later, the Jacobi sums are the finite field analogues of the beta functions.

\begin{Theorem}\label{thm4}  Assume that $q$ is a prime power with $q\equiv 1 \pmod 6$, and let $A$, $\chiup,\eta_3\in \widehat{\F_q^\times}$ be such that $\eta_3$ has order 3, and none of $A^6, \chiup^6,(A\chiup)^3,(\ol{A}\chiup)^3$  is  the trivial character.  Then,
$$
\pFFq{3}{2}{A^3&\ol{A}^3& \ol{\chiup}}{&\phi& \ol{\chiup}^3}{\frac 34}=  \sum_{\substack{B\in \widehat{\F_q^\times}\\ B^3=A^3}}\frac{J(B\chiup, \eta_3)J(\ol{B}\chiup,\ol{\eta_3})}{J(B,\chiup \eta_3)J(\ol{B},\chiup\ol{\eta_3})}= A(-1)\sum_{B^3=A^3}\frac{J(B\chiup, \ol B\chiup)}{J( \chiup\eta_3,\chiup \ol\eta_3)}.
$$
\end{Theorem}
This is an analogue of the following result of Gessel and Stanton \cite{Gessel-Stanton}  for  $n\in \mathbb N$, $a\in \C$,
\begin{equation}\label{GSIntro}
\pFq{3}{2}{1+3a&1-3a&-n}{&\frac 32& -1-3n}{\frac 34}=\frac{B(1+a+n,\frac 23)B(1-a+n,\frac 43)}{B(1+a,n+\frac 23)B(1-a,n+\frac 43)}.
\end{equation}

It is worth mentioning that Gessel and Stanton's proof of \eqref{GSIntro} uses a derivative, {{and}} thus cannot be translated directly to the finite field setting  using our first technique.  We include this case here to make the following point. The Galois perspective { can provide helpful guidance in general, beyond using } the ($\ast$) condition.

In  Section  \ref{ss:application},   we give an explicit application of  finite field formulas in point counting on
 hypergeometric varieties. To be more explicit, we apply Theorem \ref{thm:quad-2F1}, a finite field quadratic transformation,  to obtain the decomposition of a generically 4-dimensional abelian variety arising naturally from  the generalized Legendre curve $y^{12}=x^9(1-x)^5(1-\l x)$, (see Theorem \ref{thm:isogeny}).

In Section \ref{ss:formulas}, we conclude with a summary and some open questions based on numerical evidence.

In summary, our setup and main techniques allow us to make translations  from the classical setting to the finite field setting   in a straightforward manner, and  the Galois perspective gives us additional  helpful guidelines.  Furthermore,  these formulas over finite fields can be used to better understand and further compute global arithmetic invariants  of the corresponding hypergeometric abelian varieties. We have to be careful: if a result requires additional structures in its proof such as derivative structure or the Lagrange inversion formula, for example in the case of the Pascal triangle identity $\binom{n}{m}+\binom{n}{m-1}=\binom{n+1}{m+1}$, then we cannot expect a direct translation.\footnote{In \cite{Evans86} by Evans, there is an analogue for the derivative operator over finite fields (\cite[(2.8)]{Evans86}) which satisfies an analogue of the Leibniz rule (\cite[Theorem 2.2]{Evans86}). However it is not a derivation. In \cite{Greene87}, Greene gives the Lagrange inversion formula over finite fields, see Theorem \ref{FFlagrange} below, and points out its drawbacks.}    Rather,  our approaches are more suitable for translating  results with geometric/motivic interpretation.   In this sense, we will give a few finite field analogues of Ramanujan type formulas for $1/\pi$, see  \S \ref{ss:Ram}.\\

\begin{Remark}
We note that our definition of the finite field hypergeometric function is inspired by the advantages afforded by both the definitions of Greene \cite{Greene} and McCarthy \cite{McCarthy}. We make a small adjustment with the purpose of keeping the benefits of both of these while also closely aligning with the classical context and remaining compatible with the Galois representation perspective.  We note, in particular, that our choice of notation implies that results can be stated for all values of $x$ in $\F_q$ and we can interchange parameters to streamline alignment with classical proofs as much as possible.  For a comparison of our definition with the others, see \S \ref{S:comparison}.
\end{Remark}

This monograph  is an extension of our previous  
%attempts along this line of a  
hybrid approach using both finite field and Galois methods.  In \cite{WIN3a},  Deines et al. study Jacobian varieties arising from generalized Legendre curves. In \cite{WIN3b},  the authors use finite field hypergeometric evaluation formulas to compute the local zeta functions of some higher dimensional hypergeometric varieties. In \cite{FST} the authors consider the Appell-Lauricella hypergeometric functions over finite fields while in the more recent paper \cite{LLT2} the authors obtain a Whipple $_7F_6(1)$ formula over finite fields and use it to construct decomposable Galois representations. This information allows us to explain the geometry behind several  supercongruences satisfied by the corresponding truncated hypergeometric functions. These supercongruences were proved using a method in \cite{Long-Ramakrishna}  by Long and Ramakrishna  which is based on the local analytic properties of the $p$-adic Gamma functions. The approaches to supercongruences involve tools such as hypergeometric functions over finite fields as in  \cite{AO} by Ahlgren and Ono and  \cite{Osburn-S} by Osburn and Schneider, \cite{LTYZ} by Long,  Tu, Yui and Zudilin,  \cite{Long18}  by Long,  or the Wiff-Zeilberger method as in \cite{Zudilin}  by Zudilin and \cite{OZ} by Osburn and Zudilin.

\section{Preliminaries for the Complex and Finite Field Settings}\label{prelim}

In this section, we describe  the foundation for the finite field analogue of the hypergeometric function. To do so, we first recall preliminaries from the classical setting. For more details on the classical setting see \cite{AAR, Bailey, Slater}, and see \cite{Berndt-Evans-Williams,Evans86,Evans91,Greene} for the finite field setup.

\subsection{Gamma and beta functions}\label{gamma and beta}
Most of the details in this subsection can be found in Chapter 1 of  Andrews, Askey and Roy, \cite{AAR}.  We let $\mathbb{N}=\mathbb{Z}_{>0}$ throughout.

Recall the usual binomial coefficient $$\binom{n}{k}=\frac{n!}{k!(n-k)!}.$$   For any  $a\in \C$ and  $n\in \mathbb N$, define the rising factorial by

\begin{equation}\label{risingfactorial}
(a)_n:=a(a+1)\cdots(a+n-1),
\end{equation}

\noindent and define $(a)_0=1$.  Note then, that
\begin{equation}\label{classicbinom}
\frac{(-1)^k(-n)_k}{k!}=\binom{n}{k}.
\end{equation}

The \Index{gamma function} $\G(x)$ and \Index{beta function} $B(x,y)$ are defined as follows.

\begin{Definition}
For $\text{Re}(x)>0$,
$$
\G(x):=\int_0^{\infty} t^{x-1}e^{-t}\, \text{d}t.
$$
\end{Definition}

While  $\G(x)$ can be extended to a meromorphic function with poles at nonpositive integers, its reciprocal $1/\G(x)$ is an entire function and, like $\sin(x)$, has a product representation (see Theorem 1.1.2 of \cite{AAR}).  The gamma function satisfies the functional equation
\begin{equation}\label{Gfunctional}
\G(x+1) = x\G(x),
\end{equation}
which can easily be derived using integration by parts. By \eqref{Gfunctional}, one has that for $n\in \Z_{\ge 0}$,
\begin{equation}
(a)_n=\frac{\G(a+n)}{\G(a)}.
\end{equation}

\begin{Definition}\label{def:beta}
For $\text{Re}(x)>0, \text{Re}(y)>0$
$$
B(x,y):=\int_0^{1} t^{x-1}(1-t)^{y-1}\, \text{d}t.
$$
\end{Definition}
The assumptions on $x$ and $y$  can be relaxed by integrating along the Pochhammer contour path around $0$ and $1$ defined as follows.

  \begin{Definition}\label{def:poch}Let $a$, $b$ be two points in $\mathbb \C P^1$. Each Pochhammer contour $\gamma_{ab}$  is a closed curve corresponding to a commutator of the form $ABA^{-1}B^{-1}$ in the fundamental group of $\pi_1(\C P^1\setminus\{a,b,\infty\})$, where  $A$, $B  \in \pi_1(\C P^1\setminus\{a,b,\infty\})$ are loops around both of the points $a$, $b$ and the superscript $-1$ denotes a path taken in the opposite direction.
\end{Definition}

For example, when  $a=0$, $b=1$,  a contour $\gamma_{01}$ is a closed curve starting from a fixed point $T\in (0,1)$, going
around $0$  and $1$ counterclockwise (in that order) and returning to $T$.
Then one loops around $0$ and $1$ clockwise returning again  to $T$, as indicated by the picture below.
$$
\includegraphics[scale=.85]{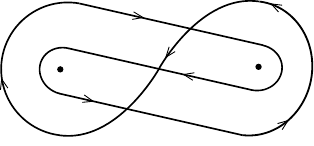}
$$
Integrating over the double contour loop $\gamma_{01}$, the integral
$$
  B(x,y)=\frac1{\(1-e^{2\pi ix}\)\(1-e^{2\pi iy}\)}\int_{\gamma_{01}} t^{x-1}(1-t)^{y-1}\text{d}t
$$
converges for all values of $x$ and $y$.
For details, see \cite{Klein, Whittaker-Watson,Yoshida-love}.

\noindent The beta function satisfies the functional equation
\begin{equation}\label{Bfunctional}
B(x,y)=\frac{x+y}{y}B(x,y+1),
\end{equation}
which can also be obtained using integration by parts (see (1.1.14) in \cite{AAR}).

Using \eqref{Bfunctional}, one can prove (see Theorem 1.1.4 in \cite{AAR}) that the gamma and beta functions are related via

\begin{equation}\label{beta_gamma}
B(x,y)=\frac{\G(x)\G(y)}{\G(x+y)}.
\end{equation}

For $a\in \C$ and a variable $z$, the binomial theorem can be restated in the notation of the rising factorial as
\begin{equation}\label{binomial}
(1-z)^{a} = \sum_{k=0}^\infty \binom{a}{k}(-z)^k =\sum_{k=0}^\infty \frac{(-a)_k}{k!}z^k.
\end{equation}

Binomial coefficients, rising factorials, gamma and beta functions are all used frequently in the  classical theory of hypergeometric functions. In fact,  gamma and beta functions play more primary roles while binomial coefficients and rising factorials make the notation more compact. Also of fundamental importance are specific formulas involving the gamma function, which we now state. First, we give Euler's reflection formula.

\begin{Theorem}[Euler's Reflection Formula\index{Euler reflection formula},  Theorem 1.2.1 of \cite{AAR}]\label{thm:reflect}
For $a\in\mathbb{C}$,
$$ \G(a)\G(1-a)=\frac{\pi}{\sin(\pi a)}.$$
\end{Theorem}

We note that Theorem \ref{thm:reflect} implies $\G(\frac 12)=\sqrt{\pi}$ and, for example, that $\G(\frac 13)\G(\frac 23)=\frac{2\sqrt{3}\pi}{3}$.

The gamma function also satisfies multiplication formulas.
\begin{Theorem}[Legendre's Duplication Formula, Theorem 1.5.1 of \cite{AAR}]\label{thm:double}
For $a\in\mathbb{C}$,
$$\G(2a)\(2 \pi\)^{1/2}  =2^{2a-\frac 12} \G(a) \G\left(a+\frac{1}{2} \right).$$
\noindent Stated in terms of rising factorials, Theorem \ref{thm:double} gives that for all $n\in \mathbb N$,
\begin{equation}\label{eq:double-rising}
(a)_{2n}=2^{2n} \left( \frac{a}{2} \right)_n \left( \frac{a+1}{2} \right)_n.
\end{equation}
\end{Theorem}
One way to prove Theorem \ref{thm:double} is the following.  Using the functional equation \eqref{Gfunctional} for $\G(x)$, one can see that $$ h(x) :=2^{2x-1}\frac{\G(x)\G(x+\frac{1}{2})}{\G(\frac{1}{2})\G(2x)}$$ satisfies the functional equation $$ h(x+1)=h(x).$$  The result can then be derived using Stirling's formula, which describes the asymptotic behavior of $\G(x)$ (see Remark 1.5.1 of \cite{AAR}).

The more general case is Gauss' multiplication formula, given below.

\begin{Theorem}[Gauss' Multiplication Formula\index{Gauss multiplication formula},  Theorem 1.5.2 of \cite{AAR}]\label{thm:multiplication}
For $m\in\mathbb{N}$ and $a\in\C$,
$$\G(ma) (2\pi)^{(m-1)/2}=m^{ma-\frac{1}{2}}\G(a) \G\left(a+\frac{1}{m} \right)\cdots \G \left( a+\frac{m-1}{m} \right).$$ Stated in terms of rising factorials, for any $n\in \mathbb N$,
$$(ma)_{mn}=m^{mn} \prod_{i=0}^{m-1} \( a+\frac{i}m \)_n.$$
\end{Theorem}

We end this subsection with  the Lagrange inversion theorem for formal power series. If $f(z)$ and $g(z)$ are formal power series where $g(0)=0$ and  $g'(0)\neq 0$, then Lagrange's inversion theorem gives a way to write $f$ as a power series in $g(z)$. In particular, one can write

\begin{equation}\label{ClassicalLagrange}
f(z)=f(0)+\sum_{k=1}^{\infty}c_kg(z)^k,
\end{equation}
where
%$$
%  c_k=\stackunder{\text{Res}}{z}\, \frac{f'(z)}{kg(z)^{k}}
%$$
$$
  c_k=\text{Res}_{z}\, \frac{f'(z)}{kg(z)^{k}},
$$
and $\text{Res}_{z}(f)$ denotes the coefficient of $1/z$ in the power series expansion of $f$. Lagrange's original result can be found in \cite{Lagrange}, and \cite{Good} by Good provides a simple proof, as well as generalizations to several variables. We discuss finite field analogues of this result in \S \ref{2.3}.

\subsection{Gauss and Jacobi sums}\label{ss:Gauss&Jacobi}

We now lay the necessary corresponding groundwork for the finite field setting based on \cite{AAR, Berndt-Evans-Williams,Evans81, Evans83, Evans91,Greene}. We  recall the Gauss and Jacobi sums, which are the finite field analogues of the gamma and beta functions, and consequently the binomial coefficient and the rising factorial. We also state analogues of the formulas for $\G(x)$ given in Theorems \ref{thm:reflect}-\ref{thm:multiplication}.  For a comprehensive treatment of Gauss sums and Jacobi sums, see the book by Berndt, Evans and Williams \cite{Berndt-Evans-Williams}.

To begin, fix an odd prime $p$ and let $\F_q$ be a finite field of size $q$, where $q=p^e$.  Recall that a \Index{multiplicative character} $\chiup$ on $\F_q^{\times}$ is a group homomorphism $$\chiup:\F_q^{\times} \rightarrow \C^{\times},$$ and the set $\fqhat$ of all multiplicative characters on $\F_q^{\times}$ forms a cyclic group of order $q-1$ under  multiplication. Throughout, we fix the following notation
\begin{align*}
\eps &= \mbox{ the trivial character in }\fqhat \\
\phi &= \mbox{ the quadratic character in }\fqhat ,
\end{align*}
so that $\eps(a)=1$ for all $a\neq 0$, and $\phi$ is nontrivial such that $\phi^2=\eps$.   We extend the definition of each character $\chiup\in\fqhat$ to all of $\F_q$ by setting $\chiup(0)=0$, including $\eps(0)=0$.\footnote{A different choice would be $\eps(0)=1$, see page 9 of \cite{Berndt-Evans-Williams}. However, we would prefer to work with functions that are multiplicative  for all elements in $\F_q$. Thus $\eps(0)=0$ is preferred.} \index{characters!trivial character, $\eps$}\index{characters!quadratic character, $\phi$}\index{$\eps$: trivial character}\index{$\phi$: quadratic character}

Further, following Greene in \cite{Greene}, it is useful to define a function $\delta$ for $\chiup \in \fqhat$, or $x\in\fq$, respectively,  by \index{$\delta$-function}%\index{functions!$\delta$-function}\index{$\delta$-function}
\begin{align*}
  \delta(\chiup):=\delta_\eps(\chiup) &:=
  \begin{cases}
    1\, &\mbox{ if } \chiup=\eps,\\
    0 \, &\mbox{ if } \chiup\neq\eps;
  \end{cases} \\
   \delta(x):=\delta_0(x) &:=
  \begin{cases}
    1\, &\mbox{ if } x=0, \\
    0\, &\mbox{ if } x\neq 0.
  \end{cases}
\end{align*}
 This definition will allow us to describe formulas which hold for all characters, without having to separate cases involving trivial characters.

Let $\z_N:=e^{2\pi i/N}$. We fix a primitive $p^{th}$ root of unity $\zeta_p$ and $A,B\in \widehat{\F_q^\times}$. We define the \Index{Gauss sum} and \Index{Jacobi sum}, respectively as follows.
\begin{align*}
  g(A)&:=\sum_{x\in \F_q^\times}A(x)\zeta_p^{\text{Tr}^{\F_q}_{\F_p}(x)}, \\
  J(A,B)&:=\sum_{x\in \F_q} A(x)B(1-x),
\end{align*}
where $\text{Tr}^{\F_q}_{\F_p}(x):=x+x^p+x^{p^2}+\cdots +x^{p^{e-1}}$  is the trace of $x$  viewed as a surjective linear map  from $\F_q$ to $\F_p$.
It is not hard to see that  $$g(\eps)=-1,$$ a fact  we will make use of frequently.  While  in general the Gauss sums depend on the choice of primitive root $\zeta_p$,
%or the choice of the additive character
%$\zeta_p^{\text{Tr}^{\F_q}_{\F_p}(x)}$,
the Jacobi sums do not. By definition,  $J(A,B)$ is in the ring of integers of the cyclotomic field $\Q(\z_{q-1})$, while $g(A)$ is in $\Q(\z_{p(q-1)})$.%, where $\z_N:=e^{2\pi i/N}$.

For a more detailed introduction to characters, see \cite{Berndt-Evans-Williams} or Chapter 8 of \cite{IR}.  We note that although \cite{Berndt-Evans-Williams} and \cite{IR} contain proofs of most results we give throughout the rest of this subsection, care must be taken for cases involving the trivial character $\eps$, which is defined to be $1$ at $0$ in both references.

The Gauss sum is the finite field analogue of the gamma function, while the Jacobi sum is the finite field analogue of the beta function \cite{Evans91}. With this in mind, we observe that the following are finite field analogues of Euler's reflection formula from Theorem \ref{thm:reflect}.

%\begin{align}
%  &g(A)g(\overline A)=qA(-1)-(q-1)\delta(A) \label{Gauss1}\\
%  &\frac 1{g(A)g(\overline A)}=\frac{A(-1)}q+\frac{q-1}q\delta(A) %\label{Gauss2}.
%\end{align}

\begin{equation}\label{Gauss1}
 \begin{split}
  &g(A)g(\overline A)=qA(-1)-(q-1)\delta(A) \\
  &\frac 1{g(A)g(\overline A)}=\frac{A(-1)}q+\frac{q-1}q\delta(A).%\label{Gauss2}.
\end{split}
\end{equation}

Note that the second equation of \eqref{Gauss1}
can be derived from the first by multiplying both right hand sides together to obtain $1$.  Furthermore, since $g(\eps)=-1$, we can write

%\begin{align}
%  &g(\overline A)=\frac{qA(-1)}{g(A)}+(q-1)\delta(A) \label{Gauss1a}\\
%  &\frac 1{g(\overline A)}=\frac{A(-1)g(A)}q - \frac{q-1}q\delta(A) %\label{Gauss2a}.
%\end{align}

\begin{equation}\label{Gauss1a}
\begin{split}
 &g(\overline A)=\frac{qA(-1)}{g(A)}+(q-1)\delta(A) \\
  &\frac 1{g(\overline A)}=\frac{A(-1)g(A)}q - \frac{q-1}q\delta(A).
\end{split}
\end{equation}

\noindent The finite field analogues of the multiplication formulas for $\G(x)$ given in Theorems \ref{thm:double} and \ref{thm:multiplication} are known as the Hasse-Davenport Relations. The general version given below is a finite field analogue of Theorem \ref{thm:multiplication}.

\begin{Theorem}[Hasse-Davenport Relation\index{Hasse-Davenport relation},  see Theorem 11.3.5 of \cite{Berndt-Evans-Williams}]\label{Hasse Davenport}
Let $m\in \mathbb N$  and $q=p^e$ be a prime power with $q\equiv1 \pmod{m}.$  For any multiplicative character $\psi\in \widehat{\mathbb{F}_{q}^{\times}}$, we have
$$\prod_{\substack{\chiup \in \widehat{\F_q^\times}\\ \chiup^m=\eps}}g(\chiup\psi)=-g(\psi^m)\psi(m^{-m})\prod_{\substack{\chiup \in \widehat{\F_q^\times} \\ \chiup^m=\eps}} g(\chiup).$$
\end{Theorem}

\noindent The proof in \cite{Berndt-Evans-Williams} for the Hasse-Davenport relation is based on Stickelberger's congruence for Gauss sums stated in Section 11.2 of \cite{Berndt-Evans-Williams} or Section 3.7 of \cite{Cohen1}.\footnote{H. Cohen pointed out to us that by equation \eqref{Gauss1}, the product $-\prod_{\chi \in \widehat{\F_q^\times}, \chi^m=\eps} g(\chiup)$
is necessarily a quadratic character when $m$ is odd; otherwise it is a quadratic character times $g(\phi)$, the quadratic Gauss sum.}

We can obtain the analogue of Legendre's duplication formula in Theorem \ref{thm:double} by taking the case $m=2$ of Theorem \ref{Hasse Davenport}.  This gives that for any multiplicative character $A$
\begin{equation}\label{Hasse-Dav-special}
g(A)g(\phi A)=g(A^2)g(\phi)\overline{A}(4),
\end{equation}where $\phi$ is the quadratic character as before.

Just as there is the relationship between the gamma and beta functions, \eqref{beta_gamma}, one can relate Gauss sums and Jacobi sums by

\begin{equation}\label{JacobiGaussrelation}
J(A,B)=\frac{g(A)g(B)}{g(AB)}+(q-1)B(-1)\delta(AB).
\end{equation}
{ The left hand side of equation \eqref{JacobiGaussrelation} is symmetric in $A$ and $B$ while there is an apparent asymmetry on the right hand side.
Note that $\delta(AB) \neq 0$ means $B=\overline{A}$ and thus
$B(-1)= \overline{A}(-1)=A(-1)$ so the right hand side is in fact symmetric.} 

Below, we note some special cases for Jacobi sums. We have,
\begin{align}%\label{eq:J(A,Abar)}
 J(A,\overline A)&=- A(-1)+(q-1)\delta(A) \label{eq:J(A,Abar)}
\end{align}
and
\begin{equation}\label{JAeps}
%\begin{align}
J(A,\eps)=-1+(q-1)\delta(A).
%\label{eq:JAeps}
%\end{align}
\end{equation}
\bk

In  \cite{Berndt-Evans-Williams} (see Theorem 2.1.4)  there is an elementary proof of \eqref{Hasse-Dav-special} written in terms of Jacobi sums. It gives, for $A\neq\eps$,
\begin{equation}\label{eq:double-Jac}
J(A,A)=\ol A(4)J(A,\phi).
\end{equation}

We now are ready to define finite field analogues for the rising factorial and binomial coefficient\footnote{Note that our binomial coefficient differs from Greene's by a factor of $-q$.} in \eqref{risingfactorial} and \eqref{classicbinom}.\index{finite field analogues!rising factorial}\index{finite field analogues!binomial coefficient} We let

\begin{align}
  (A)_\chi&:=\frac{g(A\chiup)}{g(A)}, \\
  \CC A\chiup&:=-\chiup(-1)J(A,\overline\chiup).\label{eq:FFbim}
\end{align}

\noindent Then we see directly that for  $A,\chiup_1,\chiup_2\in \widehat{\F_q^\times}$,
\begin{align}\label{eq:22}
  (A)_{\chi_1\chi_2}&=(A)_{\chi_1}(A\chiup_1)_{\chi_2},
\end{align}
which is the analogue of the classical identity $(a)_{n+m}=(a)_n(a+n)_m$, for integers $n,m\ge 0$.  The multiplication formula in Theorem \ref{Hasse Davenport} can be written in this notation as
\begin{equation}\label{eq:mul-Jacobi}
 (A^m)_{\psi^m}=\psi(m^m)\prod_{\substack{\chiup\in \widehat{\F_q^\times}\\ \chiup^m=\eps}}(A\chiup)_\psi,
\end{equation}
which compares with the second statement in Theorem \ref{thm:multiplication}.

Now we give a proof of a useful fact about Jacobi sums (see (2.6) of Greene \cite{Greene}), which is the analogue of the classical binomial identity $$\binom{n}{m}=\binom{n}{n-m}.$$
\begin{Lemma}\label{lem:1}
For any characters $A,B$ on $\F_q^\times$, $J(A,\overline B)=A(-1)J(A,B\overline {A}).$  Stated using binomial notation, $$\binom{A}{ B}=\binom{A}{A\ol B}.$$
\end{Lemma}
\begin{proof}
We observe
%by \eqref{JacobiGaussrelation} and \eqref{Gauss1a}, %and %\eqref{Gauss2a}, that
\begin{align*}
J(A,\overline B)
\overset{\eqref{JacobiGaussrelation}}{=}
&\frac{g(A)g(\overline B)}{g(A\overline B)}+(q-1)B(-1)\delta(A\overline B)\\
\overset{\eqref{Gauss1a}}{=}
&\frac{g(A)}{g(A\overline B)}
\(\frac{qB(-1)}{g( B)}+(q-1){\delta(B)}\)
+(q-1)B(-1)\delta(A\overline B)\\
%\overset{\eqref{Gauss1a}}{=}
%&\frac{g(A)}{g(A\overline B)}
%\frac{qB(-1)}{g( B)}+(q-1)\delta(B)+(q-1)B(-1)\delta(A\overline B)\\
\overset{\eqref{Gauss1a}}{=}& \frac{g(A)qB(-1)}{g( B)}\(\frac{AB(-1)}{q}{g(\overline A B)} -\frac{q-1}q \delta(\overline A B)\)\\
&+(q-1)\delta(B)+(q-1)B(-1)\delta(A\overline B)\\
=&A(-1)\frac{g(A)g(\overline A B)}{g(B)}+(q-1)\delta(B)\\
%&=&A(-1)J(A,\overline AB)-(q-1)B(-1)\delta(B)+(q-1)\delta(B)B(-1)\\
\overset{\eqref{JacobiGaussrelation}}{=}&A(-1)J(A,\overline AB).
\end{align*}
\end{proof}

\begin{Remark}
The proof of Lemma \ref{lem:1} demonstrates that one strategy for proving identities is to use the basic properties of Gauss and Jacobi sums while keeping track of the delta terms. Alternatively, one can reprove  the above by changing variables
\begin{multline*}A(-1)J(A,\overline AB)=A(-1)\sum_{x\in \F_q^\times}A(1-x)\ol AB(x)\\=A(-1)\sum_{x\in \F_q^\times}A\(\frac{1-x}x\)B(x)=A(-1)\sum_{x\in \F_q^\times}A\(\frac{1-x}x\)\ol B\(\frac 1x\)\\=\sum_{x\in \F_q^\times} A\(1-\frac 1x\)\ol B\(\frac 1x\)=J(A,\ol B).\end{multline*}
The technique  of changing variables  makes the proofs more  compact when used properly.
\end{Remark}

\subsection{Lagrange inversion}\label{2.3}
Below is a finite field analogue of the Lagrange inversion formula \eqref{ClassicalLagrange}, which is similar to the Fourier inversion formula given in \cite{BCM}.  We  state the version where the basis of complex valued functions on the finite field is comprised of all multiplicative characters in $\fqhat$, together with $\delta(x)$.

\begin{Theorem}[\cite{Greene87} Theorem 2.7]\label{FFlagrange}
Let $p$ be an odd prime, $q=p^e$, and suppose $f:\mathbb{F}_q\rightarrow \C$ and $g:\F_q\rightarrow \F_q$ are functions. Then
$$\sum_{\substack{y\in\F_q\\ g(y)=g(x)}}f(y)=\delta(g(x))  \sum_{\substack{y\in\F_q\\ g(y)=0}}f(y)+\sum_{\chiup\in\fqhat}f_{\chi}\chiup(g(x)),$$
where $$f_\chi=\frac{1}{q-1}\sum_{y\in\F_q}f(y)\overline{\chiup}(g(y)).
$$
\end{Theorem}
%\footnote{
\noindent
Compared with the classical formula \eqref{ClassicalLagrange},  the assumptions  $f(0)=0$,  $f'(0)\neq 0$, i.e. the map being one-to-one near $0$, are not required. Greene pointed out that it is also the reason why the finite field version cannot be used to determine coefficients when $f$ is not a one-to-one function.

We will use Theorem \ref{FFlagrange} in Section \ref{finitefield}  to develop the finite field analogue of the hypergeometric function.

\subsection{A dictionary between the complex and finite field  settings}\label{dictionary}
We now list a dictionary that we will use for convenience.  It is  well-known to experts.   Let $N\in \mathbb N$, and $a,b\in\mathbb{Q}$  with common denominator $N$.
$$
\bgroup
\def\arraystretch{1.4}
\begin{array}{cccc}
\frac 1N &\leftrightarrow & \text{an order $N$ character $\eta_N \in \widehat{\F_q^\times}$} \\
a=\frac i N, \; b=\frac j N &\leftrightarrow & A, B\in \widehat{\F_q^\times},\;\; A=\eta_N^i,\; B=\eta_N^j\\
x^a &\leftrightarrow & A(x)\\
x^{a+b} &\leftrightarrow &  A(x)B(x)=A B(x)\\
a+b &\leftrightarrow & A\cdot B\\
-a &\leftrightarrow &\ol{A}\\
\G(a) &\leftrightarrow & g(A)\\
(a)_n=\G(a+n)/\G(a)&\leftrightarrow & (A)_\chi =g(A\chiup)/g(A)\\
B(a,b) &\leftrightarrow & J(A,B)\\
\int_0^1 \quad \text{d}x &\leftrightarrow &  \sum_{x\in \F} \\
\G(a)\G(1-a)=\frac{\pi}{\sin a \pi}, \;\; a\notin \Z&\leftrightarrow & g(A)g(\ol{A})=A(-1)q,\;\; A\neq \eps\\
{\displaystyle(ma)_{mn}=m^{mn} \prod_{i=1}^m \( a+\frac{i}m \)_n}&\leftrightarrow& {\displaystyle(A^m)_{\psi^m}=\psi(m^m) \prod_{i=1}^m(A\eta_m^i)_\psi }
\end{array}
\egroup
$$
\begin{Remark}
One could use $\frac{1}{q}\sum_{x\in \F}$  for the finite field analogue of $\int_0^1  \text{d}x$  so that the `volume' $\frac{1}{q}\sum_{x\in \F} \, 1=1$.   Another option is to use $\frac{1}{q-1}\sum_{x\in \F^\times}$.  Our choice above  is instead $\sum_{x\in \F}$,   which makes point counting formulas more natural.
\end{Remark}

\begin{Remark}\label{rmk:integer_note}
It is useful to note here that since integers correspond to the trivial character, integers appearing in exponents will disappear when translating to the finite field setting.
\end{Remark}
\color{black}

\section{Classical Hypergeometric Functions}\label{setting}

In this section, we recall the definition and basic properties of classical hypergeometric functions, including transformation and evaluation formulas. We also discuss the hypergeometric differential equation and its relationship to the Schwarz map and the monodromy group.

\subsection{Classical development}\label{3.1}

We now recall the basic development of the classical (generalized) hypergeometric functions.  See \cite{AAR,Bailey,Slater} for a thorough treatment. For all classical formulas stated in this paper,  the assumption is always that they  hold when both sides of the formula are well-defined and converge.

The classical (generalized) \Index{hypergeometric functions $_{n+1}F_{n}$} with complex parameters $a_1, \ldots, a_{n+1}$, $b_1, \ldots, b_{n}$, and argument $z$ are defined by
\begin{equation}\label{eq:HGS}
\pFq{n+1}{n}{a_1& a_2& \cdots& a_{n+1}}{& b_1& \cdots& b_{n}}{z} := \sum_{k=0}^\infty \frac{(a_1)_k\cdots (a_{n+1})_k}{(b_1)_k\cdots (b_{n})_k}\frac{z^k}{k!},
\end{equation}and converge when $|z|<1$.
Each  $_{n+1}F_{n}$ function satisfies an order $n+1$ ordinary Fuchsian differential equation in the variable $z$ with three regular singularities at $0$, $1$, and $\infty$ \cite[\S 2.1.2]{Slater}\cite{Yoshida}.  It also satisfies the following inductive integral relation of Euler \cite[Equation (2.2.2)]{AAR}.  Namely, when $\text{Re}(b_n)>\text{Re}(a_{n+1})>0$,
\begin{multline}\label{EulerIntegral}
\pFq{n+1}{n}{a_1& a_2& \cdots& a_{n+1}}{& b_1& \cdots& b_{n}}{z} = {B(a_{n+1}, b_n-a_{n+1})}^{-1}\\
\cdot \int_0^1 t^{a_{n+1}-1}(1-t)^{b_n-a_{n+1}-1}\cdot \pFq{n}{n-1}{a_1& a_2& \cdots& a_{n}}{& b_1& \cdots& b_{n-1}}{zt}\text{d}t.
\end{multline}

To relate these functions to periods of algebraic varieties (see \cite{KZ} for useful discussions on periods),  we first observe that by \eqref{binomial}
%and the binomial theorem,
we  may  define  for general $a\in\mathbb{C}$,
\begin{equation}\label{1F0}
_1P_0  [a;z]:= (1-z)^{-a} = \sum_{k=0}^\infty \frac{(a)_k}{k!}z^k =\,  _1F_0 [a;z].
\end{equation}
Here we are using the notation $_1P_0$ to indicate a relationship to periods of  algebraic varieties when $a\in \Q$.   Next we let
\begin{multline}\label{Pintegral}
_2P_1 \left [\begin{array}{cc}a&b\\&c\end{array}  ;z\right ]:= \int_0^1 t^{b-1}(1-t)^{c-b-1} \,_1P_0  \left [a;zt \right ]  \text{d}t\\=\int_0^1 t^{b-1}(1-t)^{c-b-1}(1-zt)^{-a}\text{d}t,
\end{multline} or one can define it using the Pochhammer contour
described in \S \ref{gamma and beta}, see \cite[\S 1.6]{Slater}.
%when we recalled the definition of the beta function.
Up to an algebraic multiple, the value of this function at a given $z$
can be realized as a period of a corresponding generalized Legendre curve,  see \cite{Archinard, WIN3a, Wolfart} and \S \ref{GLC} below.  To relate it to the $_2F_1$ function, one uses \eqref{EulerIntegral} to see that when $\text{Re}(c)>\text{Re}(b)>0$,
\begin{align}\label{2P1to2F1}
_2P_1 \left [\begin{array}{cc}a&b\\&c\end{array}  ;z\right ]&= \int_0^1 t^{b-1}(1-t)^{c-b-1}\cdot\, _1P_0[a;zt] \text{d}t\\
&=\int_0^1 t^{b-1}(1-t)^{c-b-1}\cdot\, _1F_0[a;zt] \text{d}t   \nonumber \\
&=B(b,c-b) \cdot \pFq{2}{1}{a&b}{&c}{z}. \nonumber
\end{align}
Inductively one can define the (higher) periods $_{n+1}P_{n}$\index{period functions $_{n+1}P_{n}$} similarly by
\begin{multline}\label{inductiveintegral}
_{n+1}P_n \left [\begin{array}{ccccc}a_1 & a_2 & \cdots & a_{n+1}\\& b_1& \cdots & b_{n}\end{array}  ;z\right ]:= \\
\int_0^1 t^{a_{n+1}-1}(1-t)^{b_{n}-a_{n+1}-1} \, _{n}P_{n-1} \left [\begin{array}{ccccc}a_1 & a_2 & \cdots & a_{n}
\\& b_1& \cdots & b_{n-1}\end{array}  ;zt\right ] \text{d}t.
\end{multline}

Again using the beta function, one can show that
when $\text{Re}(b_i)>\text{Re}(a_{i+1})>0$ for each $i\geq 1$,
\begin{multline*}
\pFq{n+1}{n}{a_1 & a_2 & \cdots & a_{n+1}}{& b_1& \cdots & b_{n}}{z}=\\
\prod_{i=1}^n B(a_{i+1},b_i-a_{i+1})^{-1} \cdot\,_{n+1}P_n \left [\begin{array}{ccccc}a_1 & a_2 & \cdots & a_{n+1} \\ & b_1& \cdots & b_{n}\end{array} ;z\right ].
\end{multline*}

By the definition of $_{n+1}F_{n}$ in \eqref{eq:HGS}, any given $_{n+1}F_{n}$ function satisfies two nice properties:
\begin{itemize}
\item[1)] The leading coefficient is 1;

\item[2)] The roles of the upper entries $a_{i}$ (resp. lower entries $b_j$) are  symmetric.
\end{itemize}
Clearly, the $_{n+1}P_n$ period functions do not satisfy these properties in general. The hypergeometric functions can thus be viewed as periods that are `normalized' so that both properties 1) and 2) are satisfied.

 These two properties are important motivating factors for our definition of a finite field analogue of $_{n+1}F_{n}$ in \S \ref{ffdef}.

\subsection{Some properties of  hypergeometric functions with $n=1$}\label{sec:classcical 2F1-hyper}
Here we focus on some important properties of the $_{2}F_{1}$ functions.   For details see \cite{AAR, Beukers, Slater}. \bk

\subsubsection{The hypergeometric differential equation}\label{sec:HDE}
First, we recall that the hypergeometric function $$\pFq{2}{1}{a&b}{&c}{z}$$  is a solution of the \Index{hypergeometric differential equation}
\begin{equation}\label{classicalHDE}
HDE(a,b;c;z):z(1-z)F'' + [ c-(a+b+1)z ]F'  - abF = 0,
\end{equation} where  here $'$ means derivative with respect to $z$.
This is \bk a Fuchsian differential equation with $3$ regular singularities at $0$, $1$, and $\infty$ \cite{Yoshida}.

\begin{Proposition} \label{HDEsolutions}
The space of solutions to $HDE(a,b;c;z)$ has bases $\{ \beta_z, \gamma_z\}$ given below, where $\beta_z, \gamma_z$ are expanded near the singularity $z\in\{0,1,\infty\}$.
\begin{enumerate}
\item When $z=0$, and $c\notin \Z$,
\begin{align*}
\beta_0 &:= \pFq{2}{1}{a&b}{&c}{z} \\
\gamma_0 &:= z^{1-c}\pFq{2}{1}{1+a-c&1+b-c}{&2-c}{z}.
\end{align*}
\item When $z=1$, and $a+b-c \notin \Z$,
\begin{align*}
\beta_1 &:= \pFq{2}{1}{a&b}{&1+a+b-c}{1-z}\\
\gamma_1 &:= (1-z)^{c-a-b} \pFq{2}{1}{c-a&c-b}{&1+c-a-b}{1-z}.
\end{align*}
\item
When $z=\infty$, and  $a-b \notin \Z$,
\begin{align*}
\beta_\infty &:=  z^{-a}\pFq{2}{1}{a&1+a-c}{&1+a-b}{1/z} \\
\gamma_\infty &:= z^{-b} \pFq{2}{1}{b&1+b-c}{&1+b-a}{1/z}.
\end{align*}
\end{enumerate}
\end{Proposition}
We note that near the singularities $z=0$, $1$, or $\infty$, in the cases when $c\in \Z$, $a+b-c\in \Z$, or $a-b\in \Z$, respectively, the equation $HDE(a,b;c;z)$ has  solutions with logarithmic terms near the corresponding singularity.

\subsubsection{Schwarz map and Schwarz triangles}\label{ss:Schwarz}
When $a,b,c\in \Q$, there is  an explicit correspondence between a hypergeometric differential equation $HDE(a,b;c;z)$ and a Schwarz triangle $\Delta(p,q,r)$ with $p,q,r \in \Q$ due to the following theorem of Schwarz.\index{Schwarz map}\index{Schwarz triangle $\Delta(p,q,r)$}

 \begin{Theorem}[Schwarz, (see \cite{Yoshida})] \label{schwarz}\index{Schwarz theorem}
Let $f,g$ be two independent solutions to the differential equation  $HDE(a,b;c;z)$ at a point $z\in\H$,  the complex upper half-plane, and let $p= |1-c|$, $q=|c-a-b|$, and $r=|a-b|$.  Then the Schwarz map $D=f/g$ gives a bijection from $\H\cup \R$ onto a curvilinear triangle with vertices $D(0)$, $D(1)$, $D(\infty)$ and corresponding angles $p\pi$, $q\pi$, $r \pi$, as illustrated below. (Note that a change of the basis \{f,g\} corresponds to a fractional linear transformation which does not change the angles of the curvilinear triangle).
\end{Theorem}

\scriptsize
$$
\def\svgwidth{100mm}
%% Creator: Inkscape inkscape 0.91, www.inkscape.org
%% Accompanies image file 'SchwarzMap.pdf' (pdf, eps, ps)
\begingroup%
  \makeatletter%
  \providecommand\color[2][]{%
    \errmessage{(Inkscape) Color is used for the text in Inkscape, but the package 'color.sty' is not loaded}%
    \renewcommand\color[2][]{}%
  }%
  \providecommand\transparent[1]{%
    \errmessage{(Inkscape) Transparency is used (non-zero) for the text in Inkscape, but the package 'transparent.sty' is not loaded}%
    \renewcommand\transparent[1]{}%
  }%
  \providecommand\rotatebox[2]{#2}%
  \ifx\svgwidth\undefined%
    \setlength{\unitlength}{460.02955627bp}%
    \ifx\svgscale\undefined%
      \relax%
    \else%
      \setlength{\unitlength}{\unitlength * \real{\svgscale}}%
    \fi%
  \else%
    \setlength{\unitlength}{\svgwidth}%
  \fi%
  \global\let\svgwidth\undefined%
  \global\let\svgscale\undefined%
  \makeatother%
  \begin{picture}(1,0.34822893)%
    \put(0,0){\includegraphics[width=\unitlength,page=1]{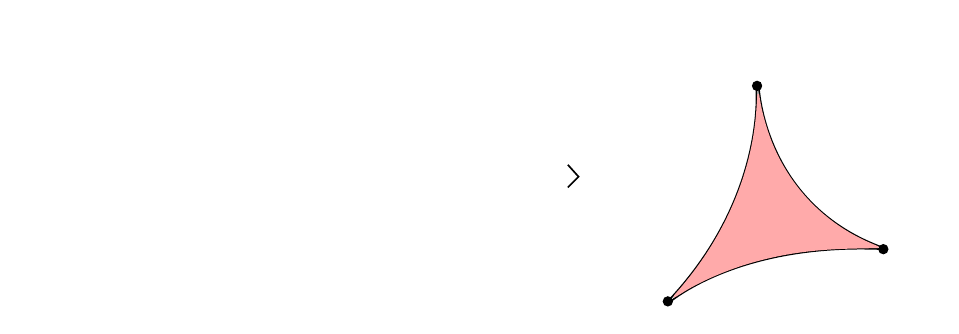}}%
    \put(-0.07961083,0.34295887){\color[rgb]{0,0,0}\makebox(0,0)[lb]{\smash{}}}%
    \put(0,0){\includegraphics[width=\unitlength,page=2]{SchwarzMap.pdf}}%
    \put(0.7374817,0.27832579){\color[rgb]{0,0,0}\makebox(0,0)[lb]{\smash{$D(0)$}}}%
    \put(0.64853059,0.00655584){\color[rgb]{0,0,0}\makebox(0,0)[lb]{\smash{$D(1)$\\ }}}%
    \put(0.90764007,0.05749929){\color[rgb]{0,0,0}\makebox(0,0)[lb]{\smash{$D(\infty)$}}}%
    \put(0,0){\includegraphics[width=\unitlength,page=3]{SchwarzMap.pdf}}%
    \put(0.83410794,0.26571228){\color[rgb]{0,0,0}\makebox(0,0)[lb]{\smash{$p\pi$}}}%
    \put(0.73784078,0.00361731){\color[rgb]{0,0,0}\makebox(0,0)[lb]{\smash{$q\pi$}}}%
    \put(0.9285118,0.13652805){\color[rgb]{0,0,0}\makebox(0,0)[lb]{\smash{$r\pi$}}}%
    \put(0,0){\includegraphics[width=\unitlength,page=4]{SchwarzMap.pdf}}%
     \put(0.52356887,0.18186672){\color[rgb]{0,0,0}\makebox(0,0)[lb]{\smash{$D$}}}%
    \put(0,0){\includegraphics[width=\unitlength,page=5]{SchwarzMap.pdf}}%
    \put(0.01304264,0.242484){\color[rgb]{0,0,0}\makebox(0,0)[lb]{\smash{$\mathfrak{H}$}}}%
  \end{picture}%
\endgroup%

$$
\normalsize

 Note that a Schwarz triangle with angles $p\pi$, $q\pi$, and $r\pi$ as described in Theorem \ref{schwarz} can be used to tile the sphere, the Euclidean plane, or the hyperbolic plane through reflections along its edges, depending on whether $p+q+r$ is equal to, greater than, or less than $1$, respectively.  Therefore, each Schwarz triangle $\Delta(p,q,r)$ can be associated to the symmetry group of this tiling, which we denote by $S_\Delta(p,q,r)$.

\subsubsection{Kummer relations}\label{ss:Kummer}
For generic $z$,  $HDE(a,b;c;z)$ has  group of symmetries $G(a,b;c;z)$ of order $24$. This symmetry group
acts projectively on its vector space of solutions. Let $\Sigma$  be the set of all ordered triples  $(a,b,c)$.  The group $G(a,b;c;z)$ is related both to the set of all 6 linear fractional transformations $$\{z,1-z,1/z, z/(z-1),1/(1-z),1-1/z\}$$ that permutes $0,1,\infty$, and to the set of finite order maps from $\Sigma$ to itself  preserving the (unordered) set of Schwarz angles $\{|1-c|, |c-a-b|,|a-b|\}$. For example, both $(a,b,c)\mapsto (c-a,b,c)$ and $(a,b,c)\mapsto (a,c-b,c)$ have order 2 and preserve the set of  Schwarz angles. So does the composition map $(a,b,c)\mapsto (c-a,c-b,c)$.

The group $G(a,b;c;z)$ is isomorphic to the symmetric group $S_4$ (see \cite{Kummer}, \cite[\S 2.3]{AAR})  and is an extension of $S_3$ (permuting the singular points $0$, $1$, $\infty$), by the Klein $4$-group $V_4$ via a short exact sequence
$$
1 \rightarrow V_4 \rightarrow G(a,b;c;z) \rightarrow S_3 \rightarrow 1.
$$
%On the whole, there are $24$ transformations in $G(a,b;c;z)$ which act on the solution space of $HDE(a,b;c;z)$.  Applying the elements of $G(a,b;c;z)$ to the solution $$\pFq{2}{1}{a&b}{&c}{z}$$ yields Kummer's $24$ solutions to $HDE(a,b;c;z)$.
In particular,  the action of $V_4$  is  via the following two identities (see   \cite[Theorem 2.2.5]{AAR}), which correspond to the order 2 maps on $\Sigma$ above.   The first one is due to Pfaff
\begin{align}
   \pFq{2}{1}{a&b}{&c}{z}
                        &= (1-z)^{-a}\pFq{2}{1}{a&c-b}{&c}{\frac z{z-1}} \label{eq:Pfaff-transf} \\
                         &= (1-z)^{-b}\pFq{2}{1}{ c-a&b}{&c}{\frac z{z-1}}, \nonumber
\end{align}
and can be obtained from the {\color{black} integral representation} of the $_{2}F_{1}$ and a change of variables.  The second is Euler's formula

\begin{align}\label{eq:Euler-transf}
   \pFq{2}{1}{a&b}{&c}{z}=& (1-z)^{c-a-b}\pFq{2}{1}{c-a&c-b}{&c}{z},
\end{align}
which can be deduced from the Pfaff transformations in \eqref{eq:Pfaff-transf}.\index{Pfaff transformation}\index{Euler transformation}

In addition to $\eqref{eq:Pfaff-transf}$ and $\eqref{eq:Euler-transf}$, the other Kummer   relations  describe the analytic continuation of hypergeometric functions at the other singularities.  For example, the analytic continuation of  $$\pFq{2}{1}{a&b}{&c}{z}$$  from near 0 to   $\infty$  can be expressed by following equation \cite[Theorem 2.3.2]{AAR},
\begin{multline}\label{eq:z->1/z}
\pFq{2}{1}{a&b}{&c}{z}=\frac{\G(c)\G(b-a)}{\G(c-a)\G(b)}\cdot (-z)^{-a} \pFq{2}{1}{a&a-c+1}{&a-b+1}{1/z}\\
+\frac{\G(c)\G(a-b)}{\G(c-b)\G(a)}\cdot (-z)^{-b} \pFq{2}{1}{b&b-c+1}{&b-a+1}{1/z}.
\end{multline}
For more details, see \cite{AAR, Kummer, Slater}.

\subsubsection{Monodromy}\label{monodromy}
Consider $\C P^1\setminus\{0,1,\infty\}$ as a topological space with a fixed base point $t_0$.  The fundamental group $\pi_1(\C P^1\setminus\{0,1,\infty\},t_0)$ is generated by three simple loops based at $t_0$, one around each of the points $\{0,1,\infty\}$, with the  relation that an appropriate composition of the three loops is  homotopic to a contractible loop. \bk

Given two linearly independent solutions $f,g$ of $HDE(a,b;c;z)$ near $t_0$, their analytic continuations along any oriented loop $C$ in $\C P^1\setminus\{0,1,\infty\}$ starting and ending at  $t_0$ will produce two linearly independent  solutions $\tilde{f},\tilde{g}$ near $t_0$.  This process corresponds to an action of the fundamental group $\pi_1(\C P^1\setminus\{0,1,\infty\},t_0)$ on the solution space $S(a,b;c;z)$ of $HDE(a,b;c;z)$ which results in a homomorphism

\begin{align*}
\pi_1(\C P^1\setminus\{0,1,\infty\},t_0) &\rightarrow GL(S(a,b;c;z))\\
[C] &\mapsto M(C),
\end{align*}
called the \Index{monodromy representation} of $HDE(a,b;c;z)$.
It is determined up to conjugation in $GL_2(\C)$  depending on where  $t_0$ is located  \cite{Beukers, Yoshida}.  The mondromy representation is a degree-$2$ representation whose image $M(a,b;c;z)$ is called the \Index{monodromy group} of $HDE(a,b;c;z)$.  The monodromy group $M(a,b;c;z)$ is generated by the images $M_0$, $M_1$, and $M_\infty$ of the three loops described above which generate the fundamental group $\pi_1(\C P^1\setminus\{0,1,\infty\},t_0)$, respectively.  The \Index{projective monodromy group} is defined to be the image of the monodromy group in $PGL_2(\C)$.

\begin{Example}Assume $c\notin \Z$, fix $t_0$ to be near $z=0$, and let $C_0$ be a simple clockwise-oriented loop based at $t_0$ which contains only the singularity $0$. Let
\begin{align*}
f &=\pFq{2}{1}{a&b}{&c}{z}\\
g &= z^{1-c}\, \pFq{2}{1}{1+a-c&1+b-c}{&2-c}{z}.
\end{align*}
  Then $$M_0:=M(C_0)=\M 100{e^{2\pi i (1-c)}}.$$ In the case $c\in \Z$, $M_0$ is conjugate to a matrix of the form $\left(\begin{smallmatrix}1 & * \\ 0 & 1\end{smallmatrix}\right)$.
\end{Example}

When $a,b,c\in \Q$, the monodromy group, denoted by $M(a,b;c)$, is a \Index{triangle group}, which, using the standard notation set by Takeuchi in \cite{Takeuchi-triangle}, is of the form
\begin{equation}\label{tri}
(e_1,e_2,e_3):=\la x,y \mid x^{e_1}=y^{e_2}=(xy)^{e_3}=1\bk \ra,
\end{equation}
where $e_i\in \mathbb N\cup \{\infty\}$ can be computed from $a,b,c$, and by symmetry  we can assume that $e_1\le e_2\le e_3$. Here $x^\infty=1$ means the order of $x$ is infinite.

Suppose $c$, $a+b-c$, $a-b\not\in \Z$. The monodromy group $M(a,b;c;z)$  is $GL_2(\C)$-conjugate to the triangle group generated by
\begin{align*}
M_0 &=\M 100{e^{2\pi i (1-c)}} ,\\
M_1 &=A \M 100{e^{2\pi i (c-a-b)}} A^{-1}, \\
M_\infty &=B \M {e^{2\pi i a}}00{e^{2\pi i b}} B^{-1},
\end{align*}
where
$$
  A=\M{\frac{\G(c)\G(c-a-b)}{\G(c-a)\G(c-b)}}{\frac{\G(c)\G(a+b-c)}{\G(a)\G(b)}}{\frac{\G(2-c)\G(c-a-b)}{\G(1-a)\G(1-b)}}{\frac{\G(2-c)\G(a+b-c)}{\G(a-c+1)\G(b-c+1)}},
$$
$$
  B=\M{\frac{\G(c)\G(b-a)}{\G(c-a)\G(b)}}{\frac{\G(c)\G(a-b)}{\G(c-b)\G(a)}}{\frac{(-1)^{c-1}\G(2-c)\G(b-a)}{\G(1-a)\G(b-c+1)}}{\frac{(-1)^{c-1}\G(2-c)\G(a-b)}{\G(a-c+1)\G(1-b)}},
$$
which are computed from   \eqref{eq:z->1/z}
and the  Kummer relation

\begin{multline}\label{eq:z->1-z}
\pFq{2}{1}{a&b}{&c}{z}=\frac{\G(c)\G(c-a-b)}{\G(c-a)\G(c-b)} \pFq{2}{1}{a&b}{&a+b+1-c}{1-z}\\
+\frac{\G(c)\G(a+b-c)}{\G(a)\G(b)}(1-z)^{c-a-b} \pFq{2}{1}{c-a&c-b}{&1+c-a-b}{1-z},
\end{multline}
(see Corollary 2.3.3 of \cite{AAR}).

%\begin{multline}\label{eq:z->1/z}
%\pFq{2}{1}{a&b}{&c}{z}=\frac{\G(c)\G(b-a)}{\G(c-a)\G(b)}(-z)^{-a} \pFq{2}{1}{a&a-c+1}{&a-b+1}{\frac 1z}\\
%+\frac{\G(c)\G(a-b)}{\G(a)\G(c-b)}(-z)^{-b} \pFq{2}{1}{b&b-c+1}{&b-a+1}{\frac 1z},
%\end{multline}
%respectively.

\begin{Example}\label{ex_monodromy}
Let $a=1/4$, $b=3/4$, $c=1/2$. Then the generators of the monodromy group $M(\frac14, \frac34, \frac12 ;z)$ can be computed using the above formulas. We simplify them into the following expressions using the reflection formula (Theorem \ref{thm:reflect}) and the multiplication formulas (Theorem \ref{thm:multiplication}) for the gamma function: \bk
$$
M_0=\M 100{-1},\quad M_1=\M 0{-1/2}{-2}0,\quad M_\infty=\M 0{-1/2}{2}0,
$$
which are each of finite order and satisfy
$$
M_0M_\infty=M_1^{-1} (=M_1) =M_\infty^{-1}M_0.
$$
Thus the monodromy group is isomorphic to the triangle group $(2,2,4)$, which is also isomorphic to the Dihedral group of order 8. Modulo the center, the projective monodromy group is isomorphic to the triangle group $(2,2,2)$, which is isomorphic to the Klein 4-group.
\end{Example}

\begin{Example}\label{(a,6,6)}
We work out the cases $(a,b,c)=(1/6,1/3,5/6)$ and
$(1/12,1/4,5/6)$.

Let $a=1/6$, $b=1/3$, $c=5/6$. The generators of the monodromy group $M(\frac16, \frac13, \frac56 ;z)$  computed using the above formulas are
$$
M_0=\M 100{\zeta_6},\quad M_1=A \M 100{\zeta_3} A^{-1}, \quad M_\infty=B \M {\zeta_6}00{\zeta_3} B^{-1},
$$
where matrices $A$ and $B$ can be computed from the formulas above and simplified to the following expressions using the functional equations of the gamma function \bk
 $$ A=\M{\frac{\G\(\frac56\)\G\(\frac 13\)}{\G\(\frac12\)\G\(\frac23\)}}{-3\frac{\G\(\frac56\)\G\(\frac 23\)}{\G\(\frac16\)\G\(\frac13\)}}{\frac16\frac{\G\(\frac16\)\G\(\frac 13\)}{\G\(\frac56\)\G\(\frac23\)}}{-\frac12\frac{\G\(\frac16\)\G\(\frac 23\)}{\G\(  \frac12  \)\G\(\frac13\)}},\quad
  B=\M{\sqrt 3}{-6\frac{\G\(\frac56\)^2}{\G\(\frac16\)\G\(\frac12\)}}{\frac{(-1)^{-1/6}}6\frac{\G\(\frac16\)^2}{\G\(\frac56\)\G\(\frac12\)}}{(-1)^{5/6}\sqrt 3}.
$$
Thus both the monodromy group and  the projective monodromy group are isomorphic to the triangle group $(3,6,6)$.

Let $a=1/12$, $b=1/4$, $c=5/6$. The  monodromy group $M(\frac1{12}, \frac14, \frac56 ;z)$  are generated by
$$
M_0=\M 100{\zeta_6},\quad M_1=A \M 100{-1} A^{-1}, \quad M_\infty=B \M {\zeta_{12}}00{\sqrt{-1}} B^{-1},
$$
where
$$
  A=\M{\frac{\G\(\frac56\)\G\(\frac 12\)}{\G\(\frac34\)\G\(\frac7{12}\)}}{-2\frac{\G\(\frac56\)\G\(\frac 12\)}{\G\(\frac1{12}\)\G\(\frac14\)}}{\frac16\frac{\G\(\frac16\)\G\(\frac 12\)}{\G\(\frac{11}{12}\)\G\(\frac34\)}}{-\frac13\frac{\G\(\frac16\)\G\(\frac 12\)}{\G\(\frac14\)\G\(\frac5{12}\)}}=\G\(\frac 12\)\M{\frac{\G\(\frac56\)}{\G\(\frac34\)\G\(\frac7{12}\)}}{-2\frac{\G\(\frac56\)}{\G\(\frac1{12}\)\G\(\frac14\)}}{\frac16\frac{\G\(\frac16\)}{\G\(\frac{11}{12}\)\G\(\frac34\)}}{-\frac13\frac{\G\(\frac16\)}{\G\(\frac14\)\G\(\frac5{12}\)}},$$ and $B=\M{\sqrt 2}{-6\frac{\G\(\frac56\)^2}{\G\(\frac1{12}\)\G\(\frac7{12}\)}}{\frac{(-1)^{-1/6}}6\frac{\G\(\frac16\)^2}{\G\(\frac{5}{12}\)\G\(\frac{11}{12}\)}}{(-1)^{5/6}\sqrt 2}.
$

Thus the monodromy group is isomorphic to the triangle group $(2,6,12)$. Modulo the center, the projective monodromy group is isomorphic to the triangle group $(2,6,6)$.

\end{Example}

We note that $S_\Delta(p,q,r)$ (see \S \ref{ss:Schwarz}) is not the same as the  monodromy group written in the notation of the triangle group $(e_1,e_2,e_3)$. The Schwarz map is defined as $f/g$, so $\Delta(p,q,r)$ is related to the projective monodromy group. In fact, the projective monodromy group of  $HDE(a,b;c;z)$  is isomorphic to the subgroup of index $2$ of the group $S_\Delta(p,q,r)$, consisting of all elements which are products of an even number of reflections. \index{Schwarz triangle $\Delta(p,q,r)$}\index{triangle group! $(e_1,e_2,e_3)$}\index{projective monodromy group}

\normalsize

\subsubsection{Schwarz's list and algebraic hypergeometric identities}

In \cite{Schwarz}, Schwarz determined a list of 15 (unordered)  triples $$\{|1-c|,|c-a-b|,|a-b|\},$$  for which the corresponding differential equations $HDE(a,b;c;z)$ have  finite  monodromy groups. %Generalizations of Schwarz's work include work of
Beukers and Heckman \cite[Theorem 4.8]{BH}
have generalized Schwarz's work by
finding general $_{n+1}F_n$ with rational parameters and finite monodromy groups.  It is well-known that  $$\pFq{2}{1}{a&b}{&c}{z}$$ is an algebraic function of $z$ when the monodromy group for $HDE(a,b;c;z)$ is a finite group; we give an example of this phenomenon below.\index{algebraic hypergeometric series}

\begin{Example}
For a classical example, let $r = \frac{i}{N}$ be reduced, with $\frac{N}{2} < i < N$.  Then, to compute the Schwarz triangle corresponding to $HDE(r-\frac12,r;2r;z)$, we see that
$p = |1-2r|$, $q= \frac 12$, $r = \frac 12$. Thus the corresponding triangle group is $(2,2,1/ |1-2r|)$\index{triangle group!$(2,2,n)$}, which can be identified with a Dihedral group. This corresponds to a tiling of the unit sphere, since $\frac12 + \frac12 + |1-2r| >1$. The corresponding algebraic expression for the hypergeometric functions is equation \eqref{eq:Slater}, as we have seen before.

In particular, when $N=3$, $i=2$,  i.e.  when $r=\frac 23$, by the algorithm described above, the monodromy group $M(\frac16,\frac23;\frac43)$ is generated by
$$
M_0=\M 100{\zeta_3^2},\, M_1=\M 0{2^{2/3}}{2^{-2/3}}0,\, M_\infty=\M 0{2^{2/3}}{2^{-2/3}\zeta_3}0,
$$
which  is isomorphic to the triangle group $(2,3,6)$.  To be more precise, we have
$$
  A=\M{\frac{\G\(\frac 43\)\G\(\frac 12\)}{\G\(\frac 76\)\G\(\frac 23\)}}{\frac{\G\(\frac 43\)\G\(-\frac 12\)}{\G\(\frac 16\)\G\(\frac 23\)}}{\frac{\G\(\frac 23\)\G\(\frac 12\)}{\G\(\frac 56\)\G\(\frac 13\)}}
 {\frac{\G\(\frac 23\)\G\(-\frac 12\)}{\G\(-\frac 16\)\G\(\frac 13\)}},
 \quad
  B=\M{\frac{\G\(\frac 43\)\G\(\frac 12\)}{\G\(\frac 76\)\G\(\frac 23\)}}{\frac{\G\(\frac 43\)\G\(-\frac 12\)}{\G\(\frac 16\)\G\(\frac 23\)}}{(-1)^{1/3}\frac{\G\(\frac 23\)\G\(\frac 12\)}{\G\(\frac 56\)\G\(\frac 13\)}}
 {(-1)^{1/3}\frac{\G\(\frac 23\)\G\(-\frac 12\)}{\G\(-\frac 16\)\G\(\frac 13\)}}
$$
and using the identities
$$
  \G(1+z)=z\G(z), \quad  \G(z)\G(1-z)=\pi/\sin{\pi z}, \quad \G(1/2)=\sqrt \pi
$$
and
$$
  \G\( \frac 13\)=\G\(2\cdot \frac 16\)=\frac{2^{1/3}}{2\sqrt{\pi}}\G\(\frac 16\)\G\(\frac 23\),
$$
we obtain
$$
   A=\M{2^{1/3}}{- \frac13 2^{ 1/3}}{2^{- 1/3}}{\frac 132^{- 1/3}}, \quad
   B=\M{2^{1/3}}{-\frac 13 2^{ 1/3}}{\zeta_6 2^{-1/3}}{\zeta_6 \frac 132^{-1/3}}
$$
and thus
\begin{align*}
  M_1&=A\M100{-1}A^{-1}= \M 0{2^{2/3}}{2^{-2/3}}0,\\
  M_\infty&=B\M{\zeta_6}00{\zeta_3^2}B^{-1}=\M 0{2^{2/3}}{2^{-2/3}\zeta_3}0.
\end{align*}

The   projective monodromy group  is isomorphic to the triangle group $(2,2,3)$, which can be identified with the Dihedral group of size $6$.  This corresponds to the tiling of the sphere depicted below.
$$
\includegraphics[scale=0.15]{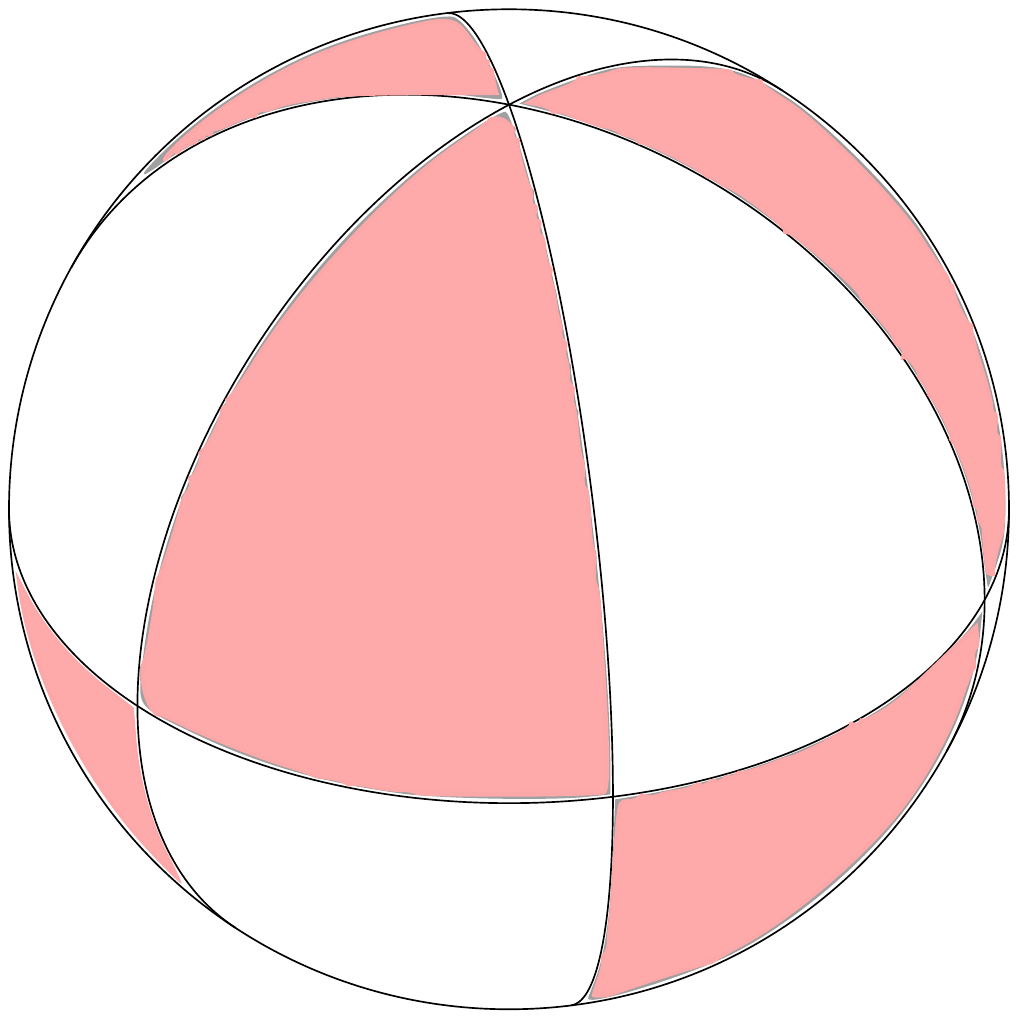}
$$
 In the above picture, one can think of each darker (respectively lighter) triangle as the image of the upper (respectively lower) half complex plane under the corresponding Schwarz map as described in Theorem \ref{schwarz}.

Furthermore, when $N=8$, $i=5$,  the monodromy group $M(\frac18,\frac58;\frac54;z)$  is $(2,4,8)$ and the  projective monodromy group  is the triangle group $(2,2,4)$, which can be identified with the Dihedral group of size $8$.
\end{Example}

\begin{Example}\label{eg:(2,3,3)} When $a=\frac 14, b=\frac 34, c=\frac 43$ or $\frac 23$ the corresponding Schwarz triangle is $(\pi/2,\pi/3,\pi/3)$ and the  monodromy  group is generated by

$$
   M_0=\M 100{\zeta_3^2},\quad M_1=\M {\frac{\sqrt{-3}}{3}}{\frac{\G\(\frac{1}{3}\)^2}{\G\(\frac 1{12}\)\G\(\frac 7{12}\)}\frac{8\sqrt 2}{\sqrt 3+i}}
   {\frac{\sqrt 3 -i}{12\sqrt 2}\frac{\G\(\frac 1{12}\)\G\(\frac 7{12}\)}{\G\(\frac 13\)^2}}{\frac{3+\sqrt{-3}}6},
$$
$$
  M_\infty=\M {-\frac{\sqrt{-3}}{3}}{\frac{\G\(\frac{1}{3}\)^2}{\G\(\frac 1{12}\)\G\(\frac {7}{12}\)}\frac{8\sqrt 2}{\sqrt 3-i}}
   {\frac{\G\(\frac{2}{3}\)^2}{\G\(\frac 5{12}\)\G\(\frac {11}{12}\)}\frac{i-\sqrt 3}{4\sqrt 2}}{\frac{\sqrt{-3}}{3}},
$$
with $M_0^3=M_1^3=id$, $M_\infty^2=-id$,  $M_\infty M_1M_0 =id$. This is  isomorphic to the triangle group $(3,3,4)$ and the  projective  monodromy group is isomorphic to $(2,3,3)$\index{triangle group! $(2,3,3)$}. In a private communication, Yifan Yang told us the following, which  he computed numerically.
Let
$$f(t):=\pFq{2}{1}{\frac 14&\frac 34}{&\frac 43}{t}.$$  Then
$$4096-4096f^2+1152tf^4-27t^2f^8 =0.$$ For a generic choice of $t$, the Galois group  of  the splitting field of the polynomial $4096-4096x+1152tx^2-27t^2x^4$ over $\Q(t)$  is isomorphic to the symmetric group $S_4$. For special choices of $t$,  the Galois group is smaller.  For instance when   $t=1/2$, the Galois group is isomorphic to $D_4$, the Dihedral group of order 8.

Similarly, Yifan Yang computed that
$$g(t) :=\pFq{2}{1}{\frac 14&\frac 34}{&\frac 23}{t}$$ satisfies
$$1+8g^2+18(1-t)g^4-27(1-t)^2g^8=0.$$ We have verified his claims by using the  recursive relations satisfied by the coefficients of $f$ and $g$.  See \S \ref{Sec:analogues}, Example \ref{eg:18}, specifically \eqref{FF-g^2},  for the finite field analogues of $f(t)^2$ and $g(t)^2$ with $t=\frac{\l}{\l-1}$.

Yang's first result is related, by a Pfaff transformation, to the first formula on page 165 of \cite{Vidunas2} of Vid{\=u}nas:
$$ \pFq{2}{1}{\frac 14&-\frac{1}{12}}{&\frac{2}3}{\frac{x(x+4)^3}{4(2x-1)^3}}=\frac1{(1-2x)^{1/4}}.$$ We give a numeric finite field analogue of this formula in Conjecture \ref{eq:Vidunas}.
\end{Example}

We note here that Vid{\=u}nas \cite{Vidunas2} gave a list of several other algebraic hypergeometric $_2F_1$ functions.  Moreover, Baldassarri and Dwork \cite{BD} determined general second order linear differential equations with only algebraic solutions.

\subsubsection{Evaluation Formulas}\label{ss:3.2.2}
There are many evaluation and transformation formulas for the classical $_{2}F_{1}$ hypergeometric functions that are incredibly useful for obtaining higher transformation formulas, computing periods, or proving supercongruences.  For example, this technique is employed in work of the second, third, and fourth authors \cite{Long, Long-Ramakrishna, Swisher}, among others. These evaluation formulas are obtained in various ways; for example, see \cite{AAR, Gessel-Stanton, Slater, Whipple}.

First, the Gauss summation formula \cite[Thm. 2.2.2]{AAR}  can be obtained by using the integral form of the $_{2}F_{1}$ from \eqref{Pintegral} and \eqref{2P1to2F1}, together with the beta integral from Definition \ref{def:beta}.  In particular  when $\text{Re}(c-a-b)>0$,
\begin{equation}\label{eq:Gauss}
\pFq{2}{1}{a&b}{&c}{1}=\frac{B(b,c-a-b)}{B(b,c-b)}=\frac{\G(c)\G(c-a-b)}{\G(c-a)\G(c-b)}.
\end{equation}\index{Gauss summation formula}

Next, we have the Kummer evaluation theorem \cite[Cor. 3.1.2]{AAR}, which gives that
\begin{equation}\label{eq:Kummer}
\pFq{2}{1}{c&b}{&1+c-b}{-1}=\frac{B(1+c-b,\frac c2+1)}{B(1+c,\frac c2+1-b)}.
\end{equation}\index{Kummer evaluation theorem}

Also, we have the very useful Pfaff-Saalsch\"utz formula \cite[Thm. 2.2.6] {AAR}.  This result follows from comparing coefficients on both sides of \eqref{eq:Euler-transf}. It states that for a positive integer $n$ and $a,b,c\in \C$,
\begin{multline}\label{eq:pf-s}
\pFq{3}{2}{a&b&-n}{&c&1+a+b-n-c}{1}=\frac{(c-a)_n(c-b)_n}{(c)_n(c-a-b)_n}\\
=\frac{\G(c-a+n)\G(c-b+n)\G(c)\G(c-a-b)}{\G(c-a)\G(c-b)\G(c+n)\G(c-a-b+n)}\\=\frac{B(c-a+n,c-b+n)B(c,c-a-b)}{B(c+n,c-a-b+n)B(c-a,c-b)}.
\end{multline} Note that the left side is a terminating series as $-n$ is a negative integer while the right side is a product. Such formulas are useful for proving
supercongruences as in \cite{Long-Ramakrishna, Swisher,Zudilin}.
\bk
In a more symmetric form, it can be written as the following when one of $a,b,d$ is a negative integer,
\begin{multline}\label{eq:pf-s2}
\pFq{3}{2}{a&b&d}{&c&1+a+b+d-c}{1}= \\
\frac{\G(c-a-d)\G(c-b-d)\G(c-a-b)\G(c)}{\G(c-a)\G(c-b)\G(c-d)\G(c-a-b-d)}.
\end{multline}\index{Pfaff-SaalSch\"utz evaluation formula}

There is a rich supply of evaluation formulas in the literature, see for instance \cite{AAR,Gessel-Stanton,   Whipple}, \cite[Appendix III]{Slater}.

\subsubsection{Higher transformation formulas}\label{ss:3.2.6}
There are also many quadratic or higher order  transformation formulas for $_2F_1$ (and a few for $_3F_2$ functions) which are obtained in various ways \cite{Goursat, Tu-Yang1,Vidunas1}.  An example is the following Kummer quadratic transformation formula \cite[Thm. 3.1.1]{AAR} or \cite[2.3.2.1]{Slater}  which gives that
\begin{equation*}
\pFq{2}{1}{c&b}{&c-b+1}{z} = (1-z)^{-c}\, \pFq{2}{1}{\frac{c}2&\frac{1+c}2-b}{&c-b+1}{\frac{-4z}{(1-z)^2}}.
\end{equation*}\index{Kummer quadratic transformation}

 From Theorem \ref{schwarz}, we observe that the Schwarz triangle corresponding to the hypergeometric function on the left is $\Delta(|b-c|,|b-c|, |1-2b|)$. It can be tiled by two copies of the Schwarz triangle $\Delta(|b-c|,\frac 12, \frac 12 |1-2b|)$, which corresponds to the hypergeometric function appearing on the right hand side.   In other words, the projective monodromy group corresponding to the left hand side is commensurable with  the projective monodromy group on the right hand side.  This is illustrated in the diagram below.  \\

$$
\def\svgwidth{50mm}
%% Creator: Inkscape inkscape 0.91, www.inkscape.org
%% Accompanies image file 'DoubleTiling.pdf' (pdf, eps, ps)
\begingroup%
  \makeatletter%
  \providecommand\color[2][]{%
    \errmessage{(Inkscape) Color is used for the text in Inkscape, but the package 'color.sty' is not loaded}%
    \renewcommand\color[2][]{}%
  }%
  \providecommand\transparent[1]{%
    \errmessage{(Inkscape) Transparency is used (non-zero) for the text in Inkscape, but the package 'transparent.sty' is not loaded}%
    \renewcommand\transparent[1]{}%
  }%
  \providecommand\rotatebox[2]{#2}%
  \ifx\svgwidth\undefined%
    \setlength{\unitlength}{126.09118004bp}%
    \ifx\svgscale\undefined%
      \relax%
    \else%
      \setlength{\unitlength}{\unitlength * \real{\svgscale}}%
    \fi%
  \else%
    \setlength{\unitlength}{\svgwidth}%
  \fi%
  \global\let\svgwidth\undefined%
  \global\let\svgscale\undefined%
  \makeatother%
  \begin{picture}(1,0.77869539)%
    \put(0,0){\includegraphics[width=\unitlength,page=1]{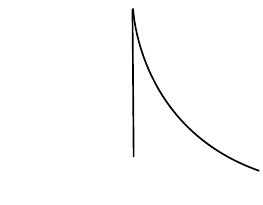}}%
    \put(-2.67230903,1.05439629){\color[rgb]{0,0,0}\makebox(0,0)[lb]{\smash{}}}%
    \put(0,0){\includegraphics[width=\unitlength,page=2]{DoubleTiling.pdf}}%
    \put(0.60526856,0.00972409){\color[rgb]{0,0,0}\makebox(0,0)[lb]{\smash{$|b-c|\pi$\\ }}}%
    \put(0.23514678,0.00748156){\color[rgb]{0,0,0}\makebox(0,0)[lb]{\smash{$|b-c|\pi$\\ }}}%
    \put(0,0){\includegraphics[width=\unitlength,page=3]{DoubleTiling.pdf}}%
    \put(0.60078233,0.75445388){\color[rgb]{0,0,0}\makebox(0,0)[lb]{\smash{$\frac12|1-2b|\pi$}}}%
  \end{picture}%
\endgroup%

$$

\section{Finite Field Analogues}\label{finitefield}

In this section we build on the discussion of period and hypergeometric functions for the classical complex setting described in Section \ref{setting}  by defining period and hypergeometric functions over finite fields which behave analogously. We conclude by comparing our definition to those previously defined by others.

\subsection{Periods in the finite field setting}

For any function $f:\, \F_q\longrightarrow \C$, the orthogonality of characters \cite{Berndt-Evans-Williams, IR} or Greene's Lagrange inversion formula over finite fields (Theorem \ref{FFlagrange}  with $g(x)=x$\bk)  implies that  $f$ has a unique representation
\begin{equation}\label{eq:f-f_chi}
  f(x)=\delta(x)f(0)+\sum_{\chi \in \widehat{\F_q^\times}}f_\chi\chiup(x),
\end{equation}
where
$$
  f_\chi=\frac 1{q-1}\sum_{y\in \F_q} f(y)\overline\chiup(y).
$$
Later, we will use the uniqueness in this statement to obtain identities.

It is also useful to note  that when $a\neq 0$,
\begin{equation}\label{eq:delta}
\delta(x-a)=\sum_{\chi \in \widehat{\F_q^\times}}\frac{\chi(x/a)}{q-1}.
\end{equation}

We now define a new version of hypergeometric functions for the finite field $\F_q$ parallel to \S \ref{3.1}. We start with a natural analogue to \eqref{1F0} by letting
\begin{equation}\label{1P0}
{}_1\mathbb P_0[A;x;q]:=\overline A(1-x),
\end{equation}
and then inductively defining
\begin{multline}\label{n+1Pn}
  {}_{n+1}\mathbb P_n
  \left[
   \begin{matrix}
    A_1 & A_2 & \cdots  & A_{n+1} \\
         & B_1 & \cdots & B_n \\
\end{matrix} ;\l;q
  \right]
 : = \\
  \sum_{y\in\F_q} A_{n+1}(y)\overline A_{n+1} B_n(1-y) \cdot
  {}_{n}\mathbb P_{n-1}
  \left[
   \begin{matrix}
    A_1 & A_2 & \cdots  & A_{n} \\
         & B_1 & \cdots & B_{n-1} \\
\end{matrix}
 ;\l y;q\right],
\end{multline}
which corresponds to \eqref{inductiveintegral} via the dictionary in \S \ref{dictionary} (and recalling Remark \ref{rmk:integer_note}).  \index{period functions $_{n+1}\mathbb P_n$ over finite fields}\index{finite field analogues!period function  $_{n+1}\mathbb P_n$} We note the asymmetry among the $A_i$ (resp. $B_j$) in the definition.  Part of the reason we start with the analogue of the period, rather than the hypergeometric function is because the periods are  related to point counting.
When there is no ambiguity in our choice of field $\mathbb{F}_q$, we will leave out the $q$ in this notation and simply write
$$
 {}_{n+1}\mathbb P_n
  \left[
   \begin{matrix}
    A_1 & A_2 & \cdots  & A_{n+1} \\
         & B_1 & \cdots & B_n \\
\end{matrix} ;\l
  \right]
:=
 {}_{n+1}\mathbb P_n
  \left[
   \begin{matrix}
    A_1 & A_2 & \cdots  & A_{n+1} \\
         & B_1 & \cdots & B_n \\
\end{matrix} ;\l;q
  \right].
$$

When $n=1$ we have,
\begin{equation}\label{Eq:30}
 \phgq {A_1}{A_2}{B_1}{\l}=\sum_{y\in\fq} A_2(y)\overline {A_2}B_1(1-y)\overline A_1(1-\l y). \end{equation}
In terms of Jacobi sums,
\begin{multline}
\label{eq:for6.7}
   \phgq {A_1}{A_2}{B_1}{\l}=            \begin{cases}
               J(A_2, \overline {A_2}{B_1}), &\mbox{ if } \l=0\\
               \displaystyle\frac 1{q-1}\sum_{\chi \in\fqhat}
               J(\overline {A_1}, \overline{\chiup})
               J(A_2\chiup,B_1\overline {A_2})\chiup(\l), & \mbox{ if } \l\neq 0
               \end{cases} .\end{multline}
          By Lemma \ref{lem:1} \begin{equation} \label{eq:for6.7-1}
             \phgq {A_1}{A_2}{B_1}{\l}= \begin{cases}  J(A_2, \overline {A_2}{B_1}),& \mbox{ if } \l=0,\\
                \displaystyle\frac {A_2(-1)}{q-1}\sum_{\chi \in\fqhat}
                J({A_1}\chiup, \overline \chiup)
                J(A_2\chiup,\overline {B_1\chiup})\chiup(\l),&  \mbox{ if } \l\neq 0.\end{cases}
          \end{equation} %broke (4.6) into two parts so when citing it won't causse problem

While they are equivalent, the second expression is more  symmetric  when the finite field analogue of binomial coefficients defined in \eqref{eq:FFbim} is used. A general formula for the $_{n+1}\mathbb{P}_n$ period function highlighting this symmetry is:

\begin{multline*}%\label{eq:PP}
 {}_{n+1}\mathbb P_n
  \left[
   \begin{matrix}
    A_1 & A_2 & \cdots  & A_{n+1} \\
         & B_1 & \cdots & B_n \\
\end{matrix} ;\l
  \right]\\
=\frac{(-1)^{n+1}}{q-1}\cdot \left(\prod_{i=1}^n A_{i+1}B_i(-1) \right) \sum_{\chi\in\fqhat} \binom{A_1\chiup}{\chiup}\binom{A_2\chiup}{B_1\chiup}\cdots\binom{A_{n+1}\chiup}{B_n \chiup}\chiup(\lambda)\\
+ \delta(\lambda)\prod_{i=1}^n J(A_{i+1},\overline{A_{i+1}}B_i).
\end{multline*}

This is similar to Greene's definition \cite[Def. 3.10]{Greene}, though we note that our binomial coefficient differs from Greene's by a factor of $-q,$  and we define the value at $\l=0$ differently. Our version  gives a natural reference point for the  normalization to come later,  as it does in the classical  case.

\subsection{Hypergeometric varieties}\label{ss:HV} It is well-known that Greene's hypergeometric functions can be used to count points on varieties over finite fields \cite[et al.]{Ahlgren,Barman,WIN3a,WIN3b,FOP,Fuselier,Koike,Lennon-count, Lennon2, Ono,Rouse,Salerno,Ve11} and evaluate twisted Kloosterman  sums \cite{Lin-Tu}.  For example, the following result of Ono (stated here in terms of our notation) gives the number of points on the Legendre curves $E_\l: y^2 = x(x-1)(x-\l)$ over $\F_q$.
\begin{Theorem}[Ono \cite{Ono}] \label{thm:Ono}For $\l \in \F_q$ with $\l \neq 0,1$,
$$
\#E_\l(\F_q) = q + 1 +  \phgq {\phi}{\phi}{\eps}{\l}.
$$
\end{Theorem}See also \cite{Rouse} by Rouse.

We use the period functions $_{n+1}\mathbb{P}_{n}$ for this purpose in a more general context.  Consider the family of \Index{hypergeometric algebraic varieties}
$$X_{\l}: \quad y^N=x_1^{i_1}\cdots x_{n}^{i_{n }}\cdot (1-x_1)^{j_{1}}\cdots(1-x_{n })^{j_{n }}\cdot (1-\l x_1\cdots x_{n })^k.$$
We can count the points on $X_{\l}$ in a manner similar to that in \cite{WIN3b}, where the case $N=n+1 $, $i_s=n $, $j_t=1$ for all $s,t$ and $k=1$  is  considered. Different algebraic models are used for the hypergeometric algebraic varieties over $\Q$ in \cite[Thm. 1.5]{BCM}.  We choose our model in order to make  the next statement straightforward.

\begin{Proposition}\label{point-count}
Let  $q=p^e\equiv 1\pmod{N}$ be a prime power, and $\eta_N \in \fqhat$  a primitive order $N$ character. Then
$$
\#X_{\l}(\fq) = 1+q^{n } + \sum_{m=1}^{N-1}  \,
_{n+1 }\mathbb{P}_{n }
\left[
\begin{matrix}
\eta_N^{-mk} & \eta_N^{mi_{n }} & \dots & \eta_N^{mi_1} \\
     & \eta_N^{mi_n+mj_{n }} & \dots & \eta_N^{mi_1+mj_1} \\
\end{matrix}
; \lambda;q
\right].$$
\end{Proposition}

\begin{proof}
We begin as in the proof of Theorem 2 in \cite{WIN3b}, to see that

\begin{multline*}
\#X_{\l}(\fq) = 1+q^n\\ +\sum_{m=1}^{N-1} \sum_{x_i\in\fq} \eta_N^m(x_1^{i_1}\cdots x_{n }^{i_{n }}(1-x_1)^{j_1}\cdots(1-x_{n })^{j_{n}}(1-\lambda x_1\cdots x_{n })^k).
\end{multline*}

Then, by applying \eqref{n+1Pn} $n$ times, we see that

\begin{multline*}
_{n +1 }\mathbb{P}_{n } \left[
\begin{matrix}
\eta_N^{-mk} & \eta_N^{mi_{n }} & \cdots & \eta_N^{mi_1} \\
     & \eta_N^{mi_{n }+mj_{n }} & \cdots & \eta_N^{mi_1+mj_1} \\
\end{matrix}
; \lambda ; q
\right]
\\=  \sum_{x_i\in\fq} \eta_N^{mi_1}(x_1)\eta_N^{mj_1}(1-x_1)\cdots \eta_N^{mi_{n }}(x_{n })\eta_N^{mj_{n }} (1-x_{n }) \eta_N^{mk}(1-\l x_1\cdots x_{n }).
\end{multline*}

This gives the result. \end{proof}

It is more convenient to obtain the major term $1+q^n$ by using $\eps(0)=1$. The total point count is independent of the choice of $\eps(0)$. Also, we are not dealing with the smooth model of $X_\l$ here. The point count here is related to Dwork's work on zeta functions of hypersurfaces which will be recalled in \S \ref{ss:zeta}.

\subsection{Hypergeometric functions over finite fields}\label{ffdef}
In the classical case (see \S \ref{3.1}), a hypergeometric function is normalized to have constant term $1$, obtained from the corresponding $_{n+1}P_{n}$ function  divided by its value  at $0$. Here we will  similarly  normalize the  finite field period functions.   We thus define
\begin{equation}\label{normalized2FF1}
_2\F_1\left[
\begin{matrix}
A_1 & A_2 \\
   & B_1 \\
\end{matrix}
; \l
\right] := \frac{1}{J(A_2,B_1\overline{A_2})} \cdot \phgq {A_1}{A_2}{B_1}{\l}.
\end{equation}\index{hypergeometric functions  $_{n+1}\mathbb F_n$ over finite fields}\index{finite field analogues!hypergeometric function $_{n+1}\mathbb F_n$}
The $_2\F_1$ function satisfies
\begin{itemize}
\item[1)] $_2\F_1\left[\begin{matrix}A_1 & A_2 \\   & B_1 \\\end{matrix}; 0  \right]=1$;
\item[2)] $_2\F_1\left[
\begin{matrix}
A_1 & A_2 \\   & B_1 \\\end{matrix}; \l\right]=\,_2\F_1\left[\begin{matrix}A_2 & A_1 \\   & B_1 \\
\end{matrix}; \l\right]$,  if $A_1$, $A_2\neq \eps$, and $A_1$, $A_2\neq B_1$.
\end{itemize}
Property  1) follows from the definition and \eqref{Eq:30}; 2) will be proved in Proposition \ref{prop: commute and conjugation} below. Intuitively, with the additional Jacobi sum factor one can rewrite  the right hand side of \eqref{normalized2FF1} using the finite field analogues of rising factorials  with the roles of $A_1$ and $A_2$ being symmetric.

More generally, we define
\begin{multline}\label{general_HGF}
  \pFFq{n+1}n {A_1&A_2&\cdots &A_{n+1}}{&B_1&\cdots&B_n}\l \\ :=
  \frac{1}{\prod_{i=1}^nJ(A_{i+1}, B_{i}\overline{A_{i+1}})}
  {}_{n+1}\mathbb P_n
  \left[
   \begin{matrix}
    A_1 & A_2 & \cdots  & A_{n+1} \\
         & B_1 & \cdots & B_n \\
\end{matrix} ; \l\right].
\end{multline}
\begin{Definition}\label{def:3}
A period function  ${}_{n+1}\mathbb P_n
  \left[
   \begin{matrix}
    A_1 & A_2 & \cdots  & A_{n+1} \\
         & B_1 & \cdots & B_n \\
\end{matrix} ; \l\right]$
or the corresponding ${}_{n+1}\mathbb F_n
  \left[
   \begin{matrix}
    A_1 & A_2 & \cdots  & A_{n+1} \\
         & B_1 & \cdots & B_n \\
\end{matrix} ; \l\right]$
is said to be \emph{primitive} if $A_i\neq \eps$ and $A_i\neq B_j$ for all $i,j$; otherwise it is said to be \emph{imprimitive}. \bk Similarly, we call the corresponding   classical period or hypergeometric function  \emph{primitive}  if $a_i,a_i-b_j\notin \Z$ for any $a_i,b_j$.
\end{Definition}\index{primitive hypergeometric/period functions over finite fields}\index{finite field analogues!primitive hypergeometric/period function}

As in the $n=1$ case,  the value of the $_{n+1} \mathbb F_n$ functions at $\l=0$ is $1$, and for primitive $_{n+1} \mathbb F_n$ the characters $A_i$ (resp. $B_j$) can be permuted without effect,  which can be seen from \eqref{eq:compMcCarthy} below.

\subsection{Comparison with other finite field hypergeometric functions}\label{S:comparison}

Alternative definitions for hypergeometric functions over finite fields have been { given in the papers of Greene \cite{Greene}, McCarthy \cite{McCarthy}, Katz \cite[Ch. 8.2]{Katz}, and Beukers, Cohen, Mellit \cite{BCM}}.  Sometimes, they are referred to as Gaussian hypergeometric functions or finite hypergeometric functions.
For the sake of consistency within this paper, below we use the notation $_{j}\F_k$ rather than $_{j}F_k$ which is used in  \cite{Greene} and \cite{McCarthy}.
Greene's version is defined by
\begin{multline*}_{n+1}\F_n  \(
   \begin{matrix}
    A_1 & A_2 & \cdots  & A_{n+1} \\
         & B_1 & \cdots & B_n \\
\end{matrix} ;\l \)^G := \\
\frac{(-1)^{n+1}}{q^n(q-1)}\cdot \sum_{\chi\in\fqhat} \binom{A_1\chiup}{\chiup}\binom{A_2\chiup}{B_1\chiup}\cdots\binom{A_{n+1}\chiup}{B_n \chiup}\chiup(\lambda).
\end{multline*}

\noindent McCarthy's version is defined by

\begin{multline*}_{n+1}\F_n  \(
   \begin{matrix}
    A_1 & A_2 & \cdots  & A_{n+1} \\
         & B_1 & \cdots & B_n \\
\end{matrix} ;\l
  \)^{M} := \\  \frac{1}{q-1} \sum_{\chiup\in\fqhat}\prod_{i=1}^{n+1} \frac{g(A_i\chiup)}{g(A_i)}\prod_{j=1}^{ n } \frac{g(\overline{B_j\chiup})}{g(\overline{B_j})} g(\overline{\chiup})\chiup(-1)^{n+1}\chiup(\l)  \end{multline*}

\noindent and is symmetric in the $A_i$'s and $B_j$'s.  %Alternatively, see  Definition 1.1 of \cite{BCM}.
It can be shown that the ``hypergeometric sum" defined by Katz in \cite{Katz}  can be written as

$$_{n}\F_m  \(
   \begin{matrix}
    A_1 & A_2 & \cdots  & A_{n} \\
    B_1  & B_2 & \cdots & B_m \\
\end{matrix} ;\l
  \)^{K} :=  \frac{1}{q-1}\sum_{\chiup\in\fqhat}\overline{\chiup}(\l) \prod_{i=1}^n g(A_i\chiup)\prod_{j=1}^mg(\overline{B_j\chiup})B_j\chiup(-1).  $$
  { In \cite{BCM}, the authors  used the following modified version of Katz's hypergeometric sum for the case $m=n$
  $$_{m}\F_m  \(
   \begin{matrix}
    A_1 & A_2 & \cdots  & A_{m} \\
    B_1  & B_2 & \cdots & B_m\\
\end{matrix} ;\l
  \)^{BCM} :=  \frac{1}{1-q}\sum_{\chiup\in\fqhat}{\chiup}(\l) \prod_{i=1}^m \frac{g(A_i\chiup)}{g(A_i)}\frac{g(\overline{B_j\chiup})}{g(\ol B_j)}\chiup(-1),  $$
  which is equivalent to McCarthy's version with $B_1=\varepsilon$.
 In the work of \cite{BCM}, the hypergeometric sum is the finite field version of the hypergeometric functions corresponding to hypergeometric motives over $\Q$, so we have additional conditions on the characters $A_i$, $B_j$} such that  the sets $\{A_i\}$ and $\{B_j\}$ are closed under Galois conjugates respectively.\bk

Our period functions are closely related to Greene's Gaussian hypergeometric functions and both  can be used to count points. The relationship  between the two  is

\begin{multline*}
    {}_{n+1}\mathbb P_n
  \left [
   \begin{matrix}
    A_1 & A_2 & \cdots  & A_{n+1} \\
         & B_1 & \cdots & B_n \\
\end{matrix} ;\l
  \right ] \\ = q^n\(\prod_{i=1}^n A_{i+1}B_{i}(-1)\)
               {}_{n+1} \F_n
  \(
   \begin{matrix}
    A_1 & A_2 & \cdots  & A_{n+1} \\
         & B_1 & \cdots & B_n \\
\end{matrix} ;\l
  \)^{G}\\
  +\delta(\l)\prod_{i=1}^n J(A_{i+1},\overline{A_{i+1}}B_{i}).
\end{multline*}

In the primitive case, the normalized  $_{n+1}\mathbb F_n$-hypergeometric function  defined in \eqref{general_HGF}  is the same as McCarthy's hypergeometric function over finite fields, when $\l\neq 0$. To be precise,  in the primitive case
\begin{multline}\label{eq:compMcCarthy}
  \pFFq{n+1}n {A_1&A_2&\cdots &A_{n+1}}{&B_1&\cdots&B_n}\l\\=
  \frac{1}{\prod_{i=1}^nJ(A_{i+1}, B_{i}\overline{A_{i+1}})}
  {}_{n+1}\mathbb P_n
  \left[
   \begin{matrix}
    A_1 & A_2 & \cdots  & A_{n+1} \\
         & B_1 & \cdots & B_n \\
\end{matrix} ; \l\right]\\
= \,_{n+1}\F_n  \(
   \begin{matrix}
    A_1 & A_2 & \cdots  & A_{n+1} \\
         & B_1 & \cdots & B_n \\
\end{matrix} ;\l
  \)^{M}+\delta(\l).
\end{multline}

%%%%%%To see how $_{n+1}\mathbb F_n$ relates to Katz's version,
We note that McCarthy's hypergeometric function is related to Katz's via (see Prop. 2.6 of \cite{McCarthy})

\begin{multline*}_{n+1}\F_n  \(
   \begin{matrix}
    A_1 & A_2 & \cdots  & A_{n+1} \\
         & B_1 & \cdots & B_n \\
\end{matrix} ;\l
  \)^{M} =  \left[\frac{1}{g(A_1)} \prod_{i=1}^n \frac{B_i(-1)}{g(A_{i+1})g(\overline{B}_i)} \right] \\
   \quad \cdot \, _{n+1}\F_{n+1}  \(
   \begin{matrix}
    A_1 & A_2 & \cdots  & A_{n+1} \\
     \varepsilon & B_1 & \cdots & B_n \\
\end{matrix} ;\frac{1}{\l}
  \)^{K}.
  \end{multline*}

\section{Some Related Topics on Galois Representations}\label{Gal-background}

In this section, we interpret the finite field analogues of periods and  hypergeometric functions using a Galois perspective. See books by Serre \cite{Serre1,Serre2} on basic representation theory and  Galois representations. Readers can choose to skip this section as most of the later proofs  can be obtained using the setup in the previous sections. The Galois interpretation gives us a global picture and allows us to predict the finite field analogues of classical formulas quite efficiently. The approach below is derived using work of Weil \cite{Weil} and is used  to reinterpret some results in \cite{WIN3a} on generalized Legendre curves.  See  \cite{Katz} by Katz and \cite{BCM} by Beukers, Cohen and Mellit for related more general discussions on the topic.

\subsection{Absolute Galois groups and Galois representations}\label{ss:Galois5.1}\index{Galois representation}
 We now recall some standard results in algebraic number theory.  Let $J/L$ be a finite Galois extension of number fields with Galois group $\text{Gal}(J/L)$, and let ${\mathcal O}_J$ and ${\mathcal O}_L$ be the respective rings of integers.   As ${\mathcal O}_J$ and ${\mathcal O}_L$ are Dedekind domains,  every  prime ideal $\mathfrak p$ of
 ${\mathcal O}_L$ factors in ${\mathcal O}_J$ into a product of prime ideals $\wp_i$ over $\mathfrak p$.  Since the Galois group $\text{Gal}(J/L)$ acts transitively on primes over $\mathfrak p$, this factorization has the form $\prod^g_{i=1} {\wp}^e_i$, where $e$ is called the ramification degree of $\mathfrak p$ in $L$ and is independent of $\wp_i$ as $J/L$ is Galois. The transitive action also implies the quantity  $f=[{\mathcal O}_{J}/{\wp}_i : {\mathcal O}_L/{\mathfrak p} ]$ is  independent of $i$.
%by the transitive action of $\text{Gal}(J/L)$ on $\{ {\wp_i}\}_{i=1,2,\dots, g}$.
For a prime ideal $\wp_i$ of ${\mathcal O}_J$, the \Index{decomposition group} of $\wp_i$  is defined  by  $D_{{\wp}_i} = \{ \sigma \in \text{Gal}(J/L) \mid \sigma({\wp}_i) = {\wp}_i\}$.
One easily sees from the transitive action
that the decomposition groups $D_{{\wp}_i}$ are conjugate to one another within $\text{Gal}(J/L)$. Thus if $\text{Gal}(J/L)$ is abelian we have that $D_{{\wp}_i}$ depends only on ${\mathfrak p}$.  Furthermore,  $\# D_{{\wp}_i} =ef$, and $efg=[J:L]$.

The discriminant $\text{Dis}_{J/L}$ ideal  is a product of powers of prime ideals of ${\mathcal O}_L$.  If a prime \color{black}  does not divide $\text{Dis}_{J/L}$ then it is a standard result that $e=1$. Such primes are called \Index{unramified}. 
Assume ${\mathfrak p}$ is unramified in $J/L$.
As the extension
$\({\mathcal O}_{J}/{\wp}_i\) \big /
\({\mathcal O}_L/{\mathfrak p}\)$ is a finite
extension of finite fields it has cyclic Galois group with generator
given by $x \mapsto x^{\# {\mathcal O_L}/{\mathfrak p}}$. Any element of $Gal(J/L)$ whose restriction to the residue field $ {\mathcal O_J}/{\wp}_i$ is the inverse of the above map is called a (geometric) \emph{Frobenius} automorphism.  In this case
$$\Z/f\Z \simeq \text{Gal}\(\({\mathcal O}_{J}/{\wp}_i\) \big /
\({\mathcal O}_L/{\mathfrak p}\)\) \simeq D_{{\wp}_i}  \subset
\text{Gal}(J/L).$$ Thus
for unramified primes $\fp$ in $J/L$ we have  the well-defined \Index{Frobenius conjugacy class}, $\text{Frob}_{\fp} \subset \text{Gal}(J/L)$.

\begin{Example}
Let $\zeta_5$ be a primitive $5$th root of unity and
set $J=\mathbb Q (\zeta_5)$ and $L=\mathbb Q$. Then $\text{Dis}_{J/L} =(5^3)$ and
all primes other than $(5)$ are unramified.
For an unramified  prime $p$, $f$ is the smallest integer satisfying
$p^f \equiv 1\pmod5$. One determines $g$ by solving
$efg = [\mathbb Q(\zeta_5):\mathbb Q]=4$.
The primes $p$ that are $1$ modulo $5$ have $(f,g)=(1,4)$.
For primes that are $2$ and $3$ modulo $5$ we have $(f,g)=(4,1)$.
Finally, primes that are $4$ modulo $5$ have $(f,g)=(2,2)$.

In this example $\text{Gal}(\mathbb Q(\zeta_5)/\mathbb Q) \simeq
\mathbb Z /4\Z$ is cyclic. In a biquadratic extension such as
$\mathbb Q(\sqrt{5},\sqrt{13})/\mathbb Q$ with Galois group $(\mathbb Z/2\Z)^2$
one will never have an unramified
prime $p$ with $\text{Frob}_p$ having order $4$. It will either be the case that $(f,g)=(2,2)$ or $(f,g)=(1,4)$.
\end{Example}

Arithmetic objects over a number field $L$ often have a representation of the absolute Galois group of $L$,
$G_L:=\text{Gal}(\ol{L}/L)$, attached to them. For instance given an elliptic curve $E_{/L}$  one can study its $m$-torsion $E[m]$  over $\ol L$.  Since $E(\mathbb C) \simeq S^1\times S^1$ we see $E[m] \simeq \mathbb Z/m \Z  \times \mathbb Z/m \Z $. More importantly, the abelian group structure of $E(\ol L) \subset E(\mathbb C)$ is defined over $L$ so $E(\ol L)$ comes equipped with an action
of $G_L$. Thus we have a continuous homomorphism $G_L \to \operatorname{Aut}(E[m]) \subset GL_2(\mathbb Z/m\Z)$. Now fix a prime $\ell$, set $m=\ell^n$ and take an inverse limit as $n \to \infty$ to get a homomorphism $\rho_{E,\ell}:G_L \to GL_2({\mathbb Z}_\ell)$.  By extending the scalars to $\Q_\ell$ or any of its extension fields by tensoring, one gets a 2-dimensional representation of  $G_L$ over an $\ell$-adic field.  This homomorphism has a number of important properties. First it is  unramified almost everywhere.
As the field fixed by the kernel of $\rho_{E,\ell}$ is an infinite extension
of $L$, this
last statement requires some interpretation.
Let $J$ be the field  fixed by  the kernel of $\rho_{E,\ell}$.  The extension $J/L$ is infinite, but for any finite $M/L$ with $M \subset J$
the relative discriminant of $\text{Dis}_{M/K}$ is relatively prime to all but a fixed finite
set of prime ideals of $L$ depending only on $E$ and $\ell$.
The primes of $L$ outside this set are called {\it unramified}.
Taking the inverse limit, this implies there
is, for every unramified prime ideal $\fp$ in $J/L$,
a conjugacy class $\text{Frob}_{\fp}$ associated to $\text{Gal}(J/L) \subset
GL_2({\mathbb Z}_\ell)$. Its characteristic polynomial,
$\text{det}(I-\text{Frob}_{\fp}X)$, is then well-defined. A priori
this polynomial is in ${\mathbb Z}_\ell[X]$ but one can in fact prove
it is in $\mathbb Z [X]$. It contains
important information about $E$, namely that  the coefficient of $X$, denoted
$-a_{\mathfrak p}$, determines the number of points of the elliptic curve
mod $\mathfrak p$, that is,
 $$\# E({\mathbb F}_{\mathfrak p}) = \# {\mathbb F}_{\mathfrak p} +1 -a_{\mathfrak p}.$$
These polynomials, as we vary $\mathfrak p$, determine the local at $\mathfrak p$ factors of the $L$-function of $E$ (except for the ramified primes where the local
$L$-factor is slightly different), which according to the  Birch and Swinnerton-Dyer Conjecture  determines the rank of the abelian group $E(L)$, of $L$-rational points of $E$.

\begin{Example}For the elliptic curve $E: y^2=x^3+1$, it is known that   for each prime $p>3$,
\begin{equation}\label{eqn: CM3}
\# E(\F_p)=\left \{ \begin{array}{lll}p+1+J(\phi, \eta_3)+\ol {J(\phi, \eta_3)}& \text{ if } p\equiv 1 \pmod 3\\
p+1 &\text{ if } p\equiv 2 \pmod 3 \end{array} \right. ,
\end{equation}  where $\phi$ is the unique quadratic character on
$\F^{\times}_p$ and  $\eta_3$ is a cubic character of $\F^{\times}_p$,
 see \cite[\S18.3]{IR}.
Furthermore $E$ has complex multiplication by $\mathbb Z
[\zeta_3]$, that is, its endomorphim ring is an order in $\mathbb Z[\zeta_3]$ properly containing $\mathbb Z$. These `extra' endomorphisms are generated over $\mathbb Z$ by $(x,y) \mapsto (\zeta_3 x,y)$. When viewing $E$ as a variety  over $\mathbb Q$, these endomorphisms are necessarily defined
over ${\mathbb Q}(\zeta_3)$. The image of
$G_{{\mathbb Q}(\zeta_3)}$ under
$\rho_{E,\ell}: G_{{\mathbb Q}(\zeta_3)} \to GL_2({\mathbb Z}_\ell)$ commutes with these extra endomorphisms. It is then easy to see that the image of
$\rho_{E,\ell}|_{G_{{\mathbb Q}(\zeta_3)}}$ is abelian. When $\ell \equiv 1 \pmod3$  (the ordinary case) one finds that $\rho_{E,\ell}|_{G_{{\mathbb Q}(\zeta_3)}}$ is a direct sum of two characters. For $\ell \equiv 2 \pmod 3$  one
must extend the coefficients from $\Z_{\ell}$ to a larger ring to get the direct sum.
\end{Example}

\subsection{Gr\"ossencharacters in the sense of Hecke}\label{ss:Grossencharacter}
We now recall a result of Weil which is relevant to our discussion below. Weil computed the local zeta functions for (homogeneous) Fermat curves of the form $X^n+Y^n=Z^n$ or special generalized Legendre curves of the form $y^N=x^{m_1}(1-x)^{m_2}$ (cyclic covers of $\C P^1$ only ramify at $0,1,\infty$). In both cases the local zeta functions can be expressed in terms of explicit Jacobi sums, see \cite{Weil1} by Weil or the textbook \cite{IR} by Ireland and Rosen. In \cite{Weil}, Weil  explained how to consider Jacobi sums as Gr\"ossencharacters (also written  `Gr\"ossencharakter' by Hecke and Weil and sometimes  called Hecke characters) which we recall  below. Here, we mainly use  Weil's notation below.

\begin{Definition}\label{def:Grossen}[See Weil \cite{Weil}]
Let $L$ be a  number field  with $r_1$ non-equivalent real embeddings and $r_2$ non-equivalent complex embeddings.  Fix an embedding of $L$ to $\C$  and let $\mathcal O_L$ be its ring of integers as before. Let $\frak m$ be a  nonzero ideal of $\mathcal O_L$ and use $\mathcal I(\frak m)$ to denote  the set of  ideals of $\mathcal O_L$ that are prime to $\frak m$. A \Index{Gr\"ossencharacter}   (or `Gr\"ossencharakter')  of $L$ with defining ideal $\frak m$, according to Hecke, is a complex-valued function $f$ on  the set of ideals of $\mathcal O_L$   such that
\begin{enumerate}
\item $f(\frak a)f(\frak b)=f(\frak a\frak b)$ for all $\frak a,\frak b\in \mathcal I(\frak m)$
\item There are rational integers $e_i$ and  rational  numbers $c_i$, with $1\le i\le  r_1+r_2$, such that if $a
\in \mathcal O_L$, $a\equiv 1 \pmod{\frak m}$ and $a$ is positive at all real
embeddings $L \hookrightarrow \mathbb R$,
then the value of $f$ at the principal ideal $(a)$ satisfies $$f((a))=\prod_{i=1}^{ r_1+r_2} a_i^{e_i}{|a_i|}^{c_i},$$  where $a_1=a,a_2,\cdots, a_{ r_1+r_2}$ are the non-equivalent embeddings of $a$ to $\C$, $|\cdot|$ denotes the complex absolute value,  and we use the principal branch for ${|a_i|}^{c_i}$.
\end{enumerate}The defining ideal ${\frak m}$ is not unique. The greatest common  divisor of all  defining ideals is called the \emph{conductor} of $f$ and is also a defining ideal.
\end{Definition}

\begin{Remark} In fact, the above two conditions being satisfied simultaneously  imposes many  restrictions  on the  $e_i$'s and $c_j$'s. For instance, the condition (1) implies that if $a$ is a unit which is 1 modulo $\frak m$, then $f((a))=1$. From this
one deduces a dependence relation on the $|a_i|$ with the
 $e_i$'s and $c_j$'s as coefficients.
\end{Remark}

 %Also, \color{blue} Gr\"ossencharacters thus defined can be rewritten as characters for the idele class groups, see for example chapter II, \S 9 of \cite{SilvermanII}, by the class field theory.   \color{black}

\begin{Example}\label{eg:Tate}Let $L=\Q$ and $\frak m=(1)$, then the set $\mathcal I(\frak m)$ contains all ideals of $\Z$, which are all principal. For $I \in \mathcal I (\frak m)$, write $I=(a)$ with $a>0$.
Define a function $\mathcal T$ on $\mathcal I(\frak m)$ by $\mathcal T(I)=a$. Then $\mathcal T$ is a Gr\"ossencharacter of $\Q$ with $ r_1=1, r_2=0$, $e_1=0$ and $c_1=1$.
\end{Example}

\begin{Example}\label{eg:Cond-64} Let $L=\Q(\sqrt{-1})$ with $r_1=0,r_2=1$  and let  { $\frak m=(4)$.} %$\frak m=(4+4\sqrt{-1})$.
The ring of integers  of $L$ is $\mathbb Z [i]$, a PID. Define a function $f$ on  prime ideals $(a+b\sqrt{-1})\in \mathcal I (\frak m)$ as follows
$$f((a+b\sqrt{-1} )) =(a+b\sqrt{-1})\chi(a+b\sqrt{-1}),$$
where

$$
  \chi(a+b\sqrt{-1})=
  \begin{cases}
     (-1)^{\frac{b+1}2}\sqrt{-1},&\, \mbox{ if } a\equiv 0 \pmod{2}, \; b\equiv 1 \pmod{2},\\
    (-1)^{\frac{a-1}2},&\, \mbox{ if } a\equiv 1\pmod{2}, \; b\equiv 0 \pmod{2},\\
      0,&\, \mbox{ otherwise.}
  \end{cases}
$$

It is straightforward to check that $f$ is a multiplicative function on ideals in $\mathcal I (\frak m)$. Also,  when $a+b\sqrt{-1}\equiv 1 \pmod {\frak m}$, $f((a+b\sqrt{-1}))=a+b\sqrt{-1}$.  So $f$ is a Gr\"ossenchacter of $L$ with defining ideal $(4)$. In this case, $e_1=1$ and $c_1=0$.
\end{Example}

\begin{Remark}\label{rem:GK-GL}For any fixed number field $L$, by definition, the set of Gr\"ossencharacters is closed under multiplication and division.
Two Gr\"ossencharacters are equivalent if  they agree on all ideals coprime to both defining ideals of the characters.  The classes of equivalent Gr\"ossencharacters are in  bijective  correspondence with the homomorphisms from the group of idele classes of $L$ to $\C^\times$. In general, these characters,  according to Definition \ref{def:Grossen} need  not take value in $S^1=\{z\in \C: |z|=1\}$.   By class field theory, finite order characters of $G_{L}=  \text{Gal}(\ol L/L)$ correspond to  finite order  Gr\"ossencharacters  of $L$. In this case, we sometime use the same notation for both the Gr\"ossencharacter and the corresponding character of $G_L$. \bk
\end{Remark}

\subsection{Notation for the $N$th power residue symbol}
\label{ss:mthpower}
Let $N$ be a natural number and set $K=\Q(\zeta_N)$. Recall that $\O_K$ denotes the ring of integers of $K$.
%$\Q( \zeta_N)$.
For each finite prime ideal $\frak p\subset \O_K$ that is coprime to $N$, let  $q(\frak p):=\#(\O_K/\frak p)$ and let $\F_{\fp}$ be the residue field. Necessarily, $q(\fp)\equiv 1\pmod N$. We define a map from $ \O_K$ to the set of $N$th roots of unity together with $0$ using the following \Index{$N$th power residue} symbol notation $\displaystyle \(\frac{x}{\frak p}\)_N$
\index{$\displaystyle \(\frac{x}{\frak p}\)_m$:   $N$th power residue}(here we simply call it the \Index{$N$th symbol}),  see pp. 240-241 of \cite{Milne-cft} by Milne.
\begin{Definition} For $x\in \O_K$
define $\displaystyle \(\frac{x}{\frak p}\)_N:=0$  if  $x\in \frak p$, and if $x\notin \frak p$, let the symbol take the value of the unique $N$th root of unity such that
\begin{equation}\(\frac{x}{\frak p}\)_N\equiv x^{(q(\frak p)-1)/N} \pmod{\frak p}.
\end{equation}
We can extend the definition to  $x \in K$,
provided $\mathfrak p$ does not appear to a negative power in the factorization of the fractional ideal $(x)$.
\end{Definition}

For explicit examples of the $N$th symbols with $N=2,3,4,6$, see \cite{IR}.\medskip

For a fixed unramified prime ideal $\frak p$, the $N$th symbol   induces a multiplicative map from the residue field  $\F_{\fp}$ to $\C$ which sends $0\in \F_{\fp}$ to 0  and thus is compatible with our notion of multiplicative characters on $\F_{\fp}$  introduced in \S \ref{ss:Gauss&Jacobi}.

\begin{Definition}\label{def:iota}\index{$\iota_{\frak p}(\cdot)$-function}
Now for any fixed rational number of the form $\frac iN$ with $i,N\in \Z$, we define a map $\iota_{(\cdot)}\!\(\frac iN\)$  from  the set of  unramified prime ideals $\fp$ of
$\O_K$  to multiplicative characters of the corresponding residue fields $\O_K/\fp$ by
\begin{equation}\label{msymbol}\iota_{\frak p}\(\frac iN\)(\cdot)
=\(\frac{\cdot}{\frak p} \)_N^i.
\end{equation}
\end{Definition}

By definition,  for any integers $i,j,N$,
\begin{equation*}\iota_{\frak p}\(\frac {i+j}N\)=\iota_{\frak p}\(\frac {i}N\)\iota_{\frak p}\(\frac {j}N\), \mbox{ and }  \quad \overline {\iota_{\frak p}\(\frac {i}N\)}=\iota_{\frak p}\(\frac {-i}N\),
\end{equation*} \bk where the bar denotes complex conjugation as before. \bk

\begin{Example}For each unramified prime ideal with residue field of odd characteristic, $\iota_{\frak p}(\frac 12)=\phi$, the quadratic character.
\end{Example}

Fix a $c \in K^\times$ and let $\mathfrak p$ be a prime ideal of
${\mathcal O}_{K}$
such that $\mathfrak p$ is prime to $N$ and $\rm{ord}_{\mathfrak p}(c)=0$.  Let Frob$_{\mathfrak p}$
denote the Frobenius automorphism associated to ${\mathfrak p}$ in the abelian Kummer extension $K(c^{1/N})/K$. Using the Artin symbol one easily checks that for varying $\mathfrak p$
\begin{equation}
\label{eq:m-Artin}\iota_{\frak p}\(\frac {1}N\)(c) = \frac{ \mbox{Frob}_{\mathfrak p}(c^{1/N})}{c^{1/N}}.
\end{equation}

\
For $c\in K^\times$,  we extend the $N$th symbol $\(\frac{c}{\cdot}\)_N$  multiplicatively to get a map from   the set of   ideals of  $\mathcal O_{K}$ that are coprime to both $c$
and $\text{Dis}_{K/\mathbb Q}$ to the multiplicative group $\mu_N$ as follows. For each ideal $\frak a$ of such, we decompose it as $\frak p_1^{e_1}\cdots \frak p_s^{e_s}$ where $\frak p_i$ are distinct prime ideals. Then we let $$\(\frac{c}{\frak a}\)_N:=\(\frac{c}{\frak p_1}\)_N^{e_1}\cdots \(\frac{c}{\frak p_s}\)_N^{e_s}.$$
See page 241 of \cite{Milne-cft}.  When $\frac iN=\frac 12$, $K=\Q$, and $c\in \Q^\times$, this step means we extend the Legendre symbol $\(\frac{c}{\cdot}\):=\(\frac{c}{\cdot}\)_2$ to  the Jacobi symbol.
% which applies to all  unramfied ideals of $\mathcal O_{K}$ that are  coprime to $c$, see \cite{IR}.
\begin{Proposition}\label{prop:2}For $N\in \mathbb N$ and
%field $K$ containing $\Q(\zeta_m)$, and
$c\in K^\times$ \bk the map  $\mathfrak a \mapsto \(\frac{c}{\mathfrak a}\)_N$ corresponds to a $1$-dimensional  representation of $G_K$  and the kernel is $G_{K(\sqrt[N]{c})}$.
\end{Proposition}
\begin{proof}By the construction, the map $\mathfrak a \mapsto \(\frac{c}{\mathfrak a}\)_N$ is a Gr\"ossencharacter with finite image and we may take as a defining ideal the discriminant ideal $\text{Dis}_{K(\sqrt[N]{c})/K}$. By  Kummer theory,  it corresponds to a $1$-dimensional  representation of $G_K$  with kernel $G_{K(\sqrt[N]{c})}$.
\end{proof}

If we fix $a=\frac iN$ and $c\in K^\times$, then  for  any prime ideal $\frak a$ coprime to $\text{Dis}_{K(\sqrt[N]{c})/K}$, the map
\begin{equation}\label{eq:chi,im} \frak a \mapsto
\chi_{\frac{i}N,c}(\frak a):= \(\frac{c}{\frak a}\)_N^i\end{equation}  also corresponds to a  1-dimensional representation of $G_K$  with kernel being $G_{K(\sqrt[N]{c})}$. \bk So the above map $\chi_{a,c}$ is a Gr\"ossencharacter, see Remark \ref{rem:GK-GL}. \bk

\begin{Example}\label{eg:5}
Let $K=\Q$, $c=-1$ in which case $ r_1=0, r_2=1$, $N=2$ and  $\frak m=(4)$. In this case, if $e_1=1$ and $c_1=-1$ then condition (2) of Definition~\ref{def:Grossen} holds for the map
$\chi_{\frac 12,-1}$.
\end{Example}\medskip

Conversely, for any character $A\in \widehat{\F_{\frak p}^\times}$ of order dividing $N$, we have another map  $\kappa_{\F_{\frak p}}$
which assigns a rational number  to $A$. In order to define this map, we first recall that if  $x \in \O_K \backslash \frak p$ then $A(x) \in \mu_N$. 
%In fact, $A(x)$  lies in the subgroup of $\mu_N$ generated by $\displaystyle \( \frac{x}{\frak p}\)_N$, 
We will show there exists $i\in \mathbb N$ such that $A(x) = \displaystyle \( \frac{x}{\frak p}\)^i_N$ for all $x\in \O_K$. To see that fix an $x_0$ whose image generates the cyclic group $\(\O_K / \frak p \)^\times$. As the order of $A(x_0)$ divides
$N$, which is the order of
$\displaystyle \( \frac{x_0}{\frak p}\)_N$, we have that
$A(x_0)= \displaystyle \( \frac{x_0}{\frak p}\)^i_N$ for some $i$. For any $x \in \O_K$ but not in $\frak p$ we have, for some $r$, that $x \equiv x^r_0 \pmod{\frak p}$.  Then for all $x$,
\begin{equation}
A(x)= A(x^r_0)= A(x_0)^r =\( \frac{x_0}{\frak p}\)^{ri}_N=
\( \frac{x^r_0}{\frak p}\)^i_N = \( \frac{x}{\frak p}\)^i_N.
\end{equation}
We may thus  define $\kappa_{{\mathbb F}_{\frak p}}(A)=\frac{i}{N}$.  Then for any $a=\frac iN$, and  any prime ideal $\fp$ of $\OqN$ coprime to $N$,
$$
\kappa_{\F_\fp}(\iota_\fp(a)) \equiv  a \pmod{\Z}.
$$

Next we will see that the $N$th symbol notation is compatible with field extensions.
%and thus the map $\iota_{(\cdot)}\!\(\frac im\)$  can be extended to finite extensions of $\Q(\zeta_m)$.
Recall that one can lift a multiplicative character $A\in \widehat{\F_q^\times}$ to any finite extension of $\F_q$ by using the norm map. Let $L$ be a finite extension of $\Q(\zeta_N)$ and let $\wp$ be a prime ideal in the ring of integers $\mathcal O_L$ of $L$ above $\fp$, with $\fp$ coprime to the discriminant of $L$. Then $\F_{\wp}:=\mathcal O_L/\wp$ \bk is a finite extension of $\F_{\frak p}:=\OqN/\fp$ and we denote the degree of the extension by $f$.
\begin{multline}\label{eq:mth-conjugate}
\(\frac{\text{N}^{\F_{\wp}}_{\F_{\frak p}}(x)}{\frak p}\)_N=
\( \frac{ x\cdot x^{q(\frak p)} \cdot x^{q(\frak p)^2} \cdots
x^{q(\frak p)^{f-1}}}{\frak p}\)_N =
\(  \frac{ x^{\frac{ q(\frak p)^f-1}{q(\frak p)-1} }}{\frak p}
\)_N
\\
\equiv
x^{(q(\frak p)^f-1)/N} \pmod{\frak p}.
\end{multline}
But $\displaystyle x^{(q(\frak p)^f-1)/N}
= x^{(q(\wp)-1)/N}
\equiv
\( \frac{x}{\wp}\)_N \pmod{\wp}$.
Thus on the residue field level, this means the $\(\frac{\cdot}{\wp}\)_N$ symbol  on $\F_{\wp}$  can be  computed from the norm map
$\text{N}^{\F_{\wp}}_{\F_{\frak p}}$
composed with the  map $\(\frac{\cdot}{\fp}\)_N$ on $\F_\fp$.

Using $\iota_\fp(\cdot)$ (see \eqref{msymbol}), one can associate to any hypergeometric function
 $$\pPq{n+1}{n}{a_1&a_2&\cdots&a_{n+1}}{&b_1&\cdots& b_n}{\l}$$
such that $a_i,b_j,\l\in \Q$, a collection of hypergeometric functions over finite residue fields $\F_{\fp}$ (varying in $\fp$)
$$\pPPq{n+1}{n}{\iota_\fp(a_1)&\iota_\fp(a_2)&\cdots&\iota_\fp(a_{n+1})}{&\iota_\fp(b_1)&\cdots &\iota_\fp(b_n)}{\l;q(\fp)},$$
where $\fp$  runs through all unramified prime ideals of $\Q(\zeta_N)$ with $N$ being the least positive common denominator of all $a_i$ and $b_j$. We will see this explicitly for the $n=1$ case  in \S \ref{ss:6.3}.\index{finite field analogues!period function  $_{n+1}\mathbb P_n$}\index{period functions $_{n+1}\mathbb P_n$}

\subsection{Jacobi sums and Gr\"ossencharacters}\label{ss:Jacobi-Grosse}

Let
%$K=\Q(\zeta_m)$, where
%$\zeta_N$ be a fixed primitive $N$th root of unity and let
$\frak p$ be an unramified prime ideal of $K=\Q(\zeta_N)$. We will now give Weil's result (for his $r=2$ case).
Let $\underline{a}=(\frac{a_1}N,\frac{a_2}N)$ with $a_i\in \Z$.  For each prime ideal $\frak p$ coprime to $N$, let
\begin{align*}
 \mathcal J_{\underline a}(\frak p)&:=-\(\frac{-1}\fp\)^{a_1+a_2}_N\sum_{x \in \mathcal O_K/\frak p} \(\frac{x}\fp\)_N^{a_1} \(\frac{1-x}\fp\)_N^{a_2}\\
 &=-\(\frac{-1}\fp\)^{a_1+a_2}_NJ\(\iota_\fp\(\frac{a_1}N\),\iota_\fp\(\frac{a_2}N\) \).
 \end{align*}
Alternatively, one can write
\begin{equation}\label{J->mathcal J}
\mathcal J_{\underline a}(\frak p) = -\chi_{\frac{a_1+a_2}N,-1}(\fp)\cdot J\(\iota_\fp\(\frac{a_1}N\),\iota_\fp\(\frac{a_2}N\) \).
\end{equation}%\index{Weil's jacobi sum $\mathcal J_{\underline a}(\frak p)$}
\index{$\mathcal J_{\underline a}(\frak p)$: Weil's Jacobi sum}

Next we extend $\mathcal J_{\underline a}$ to all  ideals in $\mathcal I ((N))$ by using
$$\mathcal J_{\underline a}(\frak a\frak b )=\mathcal J_{\underline a}(\frak a)\mathcal J_{\underline a}(\frak b),$$ if $\frak a, \frak b \in \mathcal I((N))$.
\begin{Theorem}[Weil, \cite{Weil}]\label{thm:Weil}The map $\mathcal J_{\underline a}$ is a Gr\"ossencharacter of $\Q(\zeta_N)$ with a defining ideal $\frak m=(N^2)$.
\end{Theorem}Namely, Weil showed that property  (2)  in Definition \ref{def:Grossen}  above also holds for this map when $\frak m=(N^2)$.  Note that the conductor is not $(N)$ in general. For example, when $\underline{a}=(\frac 12,1)$, $\mathcal J_{\underline a}(\frak p)=\chi_{\frac 12,-1}(\frak p)$ which was mentioned in Example \ref{eg:5}. Its conductor is (4) instead of $(2)$.

\medskip

Sometimes  quotients of Jacobi sums  take values that are roots of unity.

\begin{Example}By Example 2 of \cite{WIN3a}, we have that for $N=10$ and $\frak p$ any unramified prime ideal of $\mathcal O_{\Q(\zeta_{10})}$, the Gr\"ossencharacter 
$J_{(\frac1{10},\frac6{10})}/\mathcal J_{(\frac{2}{10},\frac5{10})}\(\fp\)$
of $\Q(\zeta_{10})$ satisfies $$\mathcal J_{(\frac1{10},\frac6{10})}/\mathcal J_{(\frac{2}{10},\frac5{10})}\(\fp\)=\(\frac{2}{\frak p}\)_{10}^8=\chi_{\frac 8{10},2}(\fp).$$
\end{Example}

\begin{Example}\label{eg:Yamamoto}\index{Yamamoto's example}It is shown by Yamamoto (see \cite[\S20]{Yamamoto}) that the  Gr\"ossencharacter $\mathcal J_{(\frac 2{12},\frac 5{12})}/\mathcal J_{(\frac 3{12},\frac 4{12})}(\fp)$  of $\Q(\zeta_{12})$   is not a multiplicative character of $\F_\fp$ for all primes $\fp$ coprime to 12, \bk but its square is. To be more precise, Yamamoto showed that
$$\(\mathcal J_{(\frac 2{12},\frac 5{12})}/\mathcal J_{(\frac 3{12},\frac 4{12})}\)^2(\fp)=\(\frac{27/4}{\frak p}\)_{12}=\chi_{\frac 1{12},\frac{27}4}(\fp).$$ One reason for the above phenomenon,  known as the sign ambiguity,  is that in Weil's Theorem \ref{thm:Weil}, a  defining ideal is $(N^2)$, but not $(N)$ in general.
\end{Example}

\section{Galois Representation Interpretation}\label{Gal}

In this section we use the theory discussed in Chapter  \ref{Gal-background}  to interpret our finite field period function in terms of Galois representations, starting with the $_1\mathbb P_0$ functions in \S \ref{ss:Galois0P1}.  Then we review the specific case associated to generalized Legendre curves in \S \ref{GLC}.  In \S \ref{ss:6.3}, we describe the Galois interpretation for the period functions $_2\mathbb P_1$ and prove Theorem \ref{thm:WIN3a}.  We then discuss the interpretation of the special cases of imprimitive $_2\mathbb P_1$ functions in \S \ref{ss:impure2P1}, and the Galois interpretation of the normalized $_2\mathbb F_1$ functions in \S \ref{ss:Galois_normalization}.  In \S \ref{ss:zeta} we discuss how finite field period functions can give information about local zeta functions for hypergeometric varieties and give a few examples.

\subsection{Galois interpretation for $_1\mathbb P_0$}\label{ss:Galois0P1}
To illustrate our ideas, we will first discuss the Galois representation background behind the $_1\mathbb P_0$ function defined in   \eqref{1P0}  using the $\iota_{(\cdot)}$  map (see \eqref{msymbol}).  Recall that by our notation  given in  \eqref{1F0},
$$_1P_0\left[\frac iN; \l\right]=(1-\l)^{-\frac iN}.$$
Let $\l\in  \Q$ and  $K=\Q(\zeta_N)$.  Now fix $\frac iN, \l \in {\Q}^\times$ and allow $\fp$ to vary.  For each  prime ideal $\fp$ of $\mathcal O_K$  coprime to the discriminant ideal of $K(\sqrt[N]{1-\l})/\Q$,
\bk we have
$$_1\mathbb P_0\left[\iota_\fp \(\frac iN\);\l; q(\fp)\right]=\ol{\iota_\fp \(\frac iN\)}\(1-\l\)={\iota_\fp \(-\frac iN\)}\(1-\l\)=\chi_{-\frac iN, 1-\l}(\fp),$$
by  the definition in  \eqref{1P0}.
On the other hand, if we fix $\fp$ and let $i$ vary, then by Proposition \ref{prop:2}, $_1\mathbb P_0[\iota_\fp(\cdot);\l;q(\fp)]$ provides a map  which sends  rational numbers of the form $\frac iN$ to finite order characters of the Galois group
$G_{\Q(\zeta_N)}$ whose kernel contains $G_{\Q(\sqrt[N]{1-\l},\zeta_N)}$. Further  one  can  compute  the corresponding Artin $L$-function\index{Artin $L$-function}
\begin{multline}\label{eq:L-A}
L\(\frac iN, \l;s\) \\
=\prod_{\text{good } \fp \text{ of } {\mathcal O}_K}
\(1-\, _1\mathbb P_0\left[\iota_\fp \(\frac iN\);\l; q(\fp)\right]q(\fp)^{-s}\)^{-1}.
\end{multline}
Note the right side
only includes those prime ideals $\fp$
that are coprime to the absolute discriminant of $\Q(\sqrt[N]{1-\l},\zeta_N)$.

\begin{Example}For $\frac iN=\frac 12$ and  $\l=-1$ \bk as in Example \ref{eg:5}, we have
$$L\( \frac 12, -1;s\)=\prod_{p \text{ odd prime}} \(1-\(\frac{-1}p\)p^{-s}\)^{-1}.$$ This is the Dirichlet L-function for the Dirichlet Character $\chi_{\frac 12, {-1}}$.
\end{Example}

\subsection{Generalized Legendre curves and their Jacobians}\label{GLC}
 We now describe the Galois interpretation of our finite field period functions for the specific setting corresponding to generalized Legendre curves and their Jacobians.

For any complex numbers $a$, $b$, $c$, $z$, with $\text{Re}(c)>\text{Re}(b)>0$, we have that the formula of Euler (see \cite{AAR})
$$
  \int_0^1 x^{b-1}(1-x)^{c-b-1}(1-zx)^{-a}dx
$$
converges, and the classical hypergeometric series $ \pFq{2}{1}{a&b}{&c}{z} $ can be expressed,
for $|z|<1$,
as
\begin{equation*}
  \pFq{2}{1}{a&b}{&c}{z}=\frac 1{ B(b,c-b)} \int_0^1 x^{b-1}(1-x)^{c-b-1}(1-z x)^{-a}dx,
\end{equation*}
 where the branch is determined by
$$
  \arg(x)=0, \, \arg(1-x)=0, \,\,\, |\arg(1-zx)|<\frac{\pi}2, \, x \in (0,1),
$$
and $B(\cdot,\cdot)$ is the beta function defined in Definition \ref{def:beta}.
As described in \S \ref{gamma and beta}, the restriction  $\text{Re}(c)>\text{Re}(b)>0$ needed for $B(\cdot,\cdot)$ can be dropped if we take the Pochhammer contour $\gamma_{01}$
(see Definition \ref{def:poch} for the notation for $\gamma_{ab}$) \bk as the integration path in the definition of $B(\cdot,\cdot)$.

If the parameters $a$, $b$, $c\in\Q$, and $a,b,a-c,b-c \notin \Z$,  which correspond to the non-degenerate cases,  Wolfart \cite{Wolfart} realized that the integrals
\begin{multline*}
  \frac1{\(1-e^{2\pi ib}\)\(1-e^{2\pi i(c-b)}\)}\int_{\gamma_{01}} x^{b-1}(1-x)^{c-b-1}(1-\l x)^{-a}dx\\=B(a,b) \pFq{2}{1}{a&b}{&c}{\l}=\pPq{2}{1}{a&b}{&c}{\l}
\end{multline*} and
\begin{multline}
\frac1{(1-e^{-2\pi ia})(1-e^{2\pi i (c-a)})} \int_{\gamma_{\frac 1{\l}\infty}}x^{b-1}(1-x)^{c-b-1}(1-\l x)^{-a}dx\\
=(-1)^{c-a-b-1}\l^{1-c}B(1+a-c,1-a)\pFq{2}{1}{1+b-c&1+a-c}{&2-c}{\l}\\
\\
=(-1)^{c-a-b-1}\l^{1-c}\pPq{2}{1}{1+b-c&1+a-c}{&2-c}{\l}
\end{multline}
are both \emph{periods}  (related to differential 1-forms of the form \eqref{eq:differentials} below)  of  a so-called \Index{generalized Legendre curve} of the form \begin{equation}\label{eq:GLC}y^N=x^i(1-x)^j(1-\l x)^k,
\end{equation}where
\begin{equation}\label{eq:abc->Nijk}N= \operatorname{lcd}(a,b,c),\quad i=N\cdot (1-b),\quad j=N\cdot (1+b-c),\quad k=N\cdot a,
\end{equation}
and lcd  means the  least (positive)  common  denominator.  By the assumption that $a,b,a-c,b-c\notin \Z$, we know $N\nmid i,j,k,i+j+k$.  This assumption is to require
that the curve be a cover of $\C P^1$ ramifying at
exactly  the four distinct points $0,1,1/\l, \infty$
($\l \neq 0,1$).  Also by changing variables if needed, one can assume $0<i,j,k<N$.

\begin{Example}\label{(a,6,6)-2}
  In Example \ref{(a,6,6)}, we see that for the order 2 hypergeometric differential equation with parameters $a=1/6,b=1/3,c=5/6$, the projective monodromy group is isomorphic to the arithmetic triangle group $(3,6,6)$. Using \eqref{eq:abc->Nijk}, one can compute that $N=6, i=4,  j=3,k=1$ for this case. Similarly, the projective monodromy group for the hypergeometric differential equation with parameters $a=1/12,b=1/4,c=5/6$ is isomorphic to  the arithmetic triangle group $(2,6,6)$. For the latter case, by \eqref{eq:abc->Nijk}, $N=12,i=9,j=5,k=1$.
\end{Example}
For $a=b=\frac 12,c=1$ and any fixed $\l\in {\Q}$  with $\l \neq 0,1$,  the corresponding curve is the well-known Legendre curve
\begin{equation}\label{eq:Llambda}L_\l: y^2=x(1-x)(1-\l x),
\end{equation}see \cite{Silverman}. We summarize a few relevant properties of the Legendre curves here.\index{Legendre curve}
\begin{itemize}
\item It is a double over of {$\C P^1$} which ramifies only at $0,1, \frac 1\l, \infty$ as demonstrated by the picture below.
 Going from right to left, first cut the torus twice including half of each of the indicated
 %\bl [I changed from colored to indicated as the colors may not shown in print. LL] \bk
 boundaries on each torus, to get two cylinders. Each cylinder can be
realized as the sphere on the  left   by pinching the ends together.
Gluing along the slits gives the double cover.
%Here  in the first picture we branch cut a $\C P^1$ along the indicated lines;  in the second picture  we take two copies of $\C P^1$ with the same cut,   and then in the last picture we glue them together along the corresponding branch cuts as indicated below.
\scriptsize
$$
 \def\svgwidth{100mm}
%% Creator: Inkscape inkscape 0.91, www.inkscape.org
%% Accompanies image file 'doublecover2.pdf' (pdf, eps, ps)
\begingroup%
  \makeatletter%
  \providecommand\color[2][]{%
    \errmessage{(Inkscape) Color is used for the text in Inkscape, but the package 'color.sty' is not loaded}%
    \renewcommand\color[2][]{}%
  }%
  \providecommand\transparent[1]{%
    \errmessage{(Inkscape) Transparency is used (non-zero) for the text in Inkscape, but the package 'transparent.sty' is not loaded}%
    \renewcommand\transparent[1]{}%
  }%
  \providecommand\rotatebox[2]{#2}%
  \ifx\svgwidth\undefined%
    \setlength{\unitlength}{439.17451859bp}%
    \ifx\svgscale\undefined%
      \relax%
    \else%
      \setlength{\unitlength}{\unitlength * \real{\svgscale}}%
    \fi%
  \else%
    \setlength{\unitlength}{\svgwidth}%
  \fi%
  \global\let\svgwidth\undefined%
  \global\let\svgscale\undefined%
  \makeatother%
  \begin{picture}(1,0.25826985)%
    \put(0,0){\includegraphics[width=\unitlength,page=1]{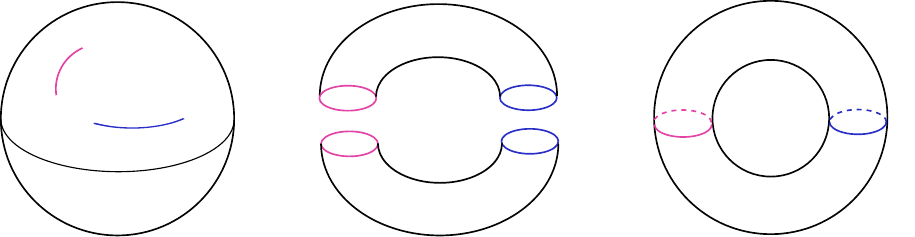}}%
    \put(0.07002161,0.10829233){\color[rgb]{0,0,0}\makebox(0,0)[lb]{\smash{}}}%
    \put(0.07490232,0.11875096){\color[rgb]{0,0,0}\makebox(0,0)[lb]{\smash{$0$}}}%
    \put(0.20497759,0.12846278){\color[rgb]{0,0,0}\makebox(0,0)[lb]{\smash{$1$}}}%
    \put(0.08117752,0.21148441){\color[rgb]{0,0,0}\makebox(0,0)[lb]{\smash{$\infty$}}}%
    \put(0.03910364,0.14262648){\color[rgb]{0,0,0}\makebox(0,0)[lb]{\smash{$\frac{1}{\lambda}$}}}%
    \put(0.51904662,0.09783369){\color[rgb]{0,0,0}\makebox(0,0)[lb]{\smash{$0$}}}%
    \put(0.51765213,0.14664069){\color[rgb]{0,0,0}\makebox(0,0)[lb]{\smash{$0$}}}%
    \put(0.62065454,0.09725599){\color[rgb]{0,0,0}\makebox(0,0)[lb]{\smash{$1$}}}%
    \put(0.61699911,0.14652112){\color[rgb]{0,0,0}\makebox(0,0)[lb]{\smash{$1$}}}%
    \put(0.41515726,0.1438517){\color[rgb]{0,0,0}\makebox(0,0)[lb]{\smash{$\infty$}}}%
    \put(0.41864347,0.09713636){\color[rgb]{0,0,0}\makebox(0,0)[lb]{\smash{$\infty$}}}%
    \put(0.78330199,0.12084273){\color[rgb]{0,0,0}\makebox(0,0)[lb]{\smash{$\infty$}}}%
    \put(0.9799744,0.11660965){\color[rgb]{0,0,0}\makebox(0,0)[lb]{\smash{$1$}}}%
    \put(0.87464106,0.12084278){\color[rgb]{0,0,0}\makebox(0,0)[lb]{\smash{$0$}}}%
    \put(0.3191423,0.1465707){\color[rgb]{0,0,0}\makebox(0,0)[lb]{\smash{$\frac{1}{\lambda}$}}}%
    \put(0.3191423,0.09825418){\color[rgb]{0,0,0}\makebox(0,0)[lb]{\smash{$\frac{1}{\lambda}$}}}%
    \put(0.67905112,0.12191935){\color[rgb]{0,0,0}\makebox(0,0)[lb]{\smash{$\frac{1}{\lambda}$}}}%
  \end{picture}%
\endgroup%

$$
\normalsize

\item  As  $\l\in  \Q \backslash \{0,1\}$, the curve $L_\l$ is an algebraic curve defined over $ \Q$ with genus  1.

\item Every holomorphic differential on $L_\l$ is a scalar multiple of $$\displaystyle \omega_\l:=\frac{dx}{\sqrt{x(1-x)(1-\l x)}}.$$

\item A period of $L_\l$ is $ 2 \cdot  \int_0^1 \omega_\l = 2\cdot \, \pPq{2}{1}{\frac 12&\frac 12}{&1}{\l}=2\pi\cdot \pFq{2}{1}{\frac 12&\frac 12}{&1}{\l}$.  As described in section \S\ref{sec:HDE}, the hypergeometric function $\pFq{2}{1}{\frac 12&\frac 12}{&1}{\l}$ satisfies the differential equation  $HDE(\frac 12,\frac 12;1;\l)$   and   thus the monodromy group of $HDE(\frac 12,\frac 12;1;\l)$ is isomorphic to the arithmetic triangle group $(\infty,\infty,\infty)$  as described in \S \ref{monodromy}.

\item  For simplicity fix $\l\in
\Q\setminus \{0,1\}$. As we recalled in \S\ref{ss:Galois5.1} (see also \cite{Serre2,Silverman}), there is a compatible family of 2-dimensional $\ell$-adic  representations $\{\rho_{L_\l,\ell}\}$ of $G_\Q$ constructed from the Tate module of $L_\l$.  For any prime $\ell$ and for any prime  $p\neq \ell$  not dividing the discriminant of the elliptic curve $L_\l$ (denoted by  $N(L_\l)$),  $\rho_{\lambda,\ell}$ is unramified at $p$. Thus it makes sense to consider the trace and determinant of $\rho_{\lambda,\ell}$ evaluated at the conjugacy class of Frobenius at $p$. In particular,
$$\text{Tr} \,\rho_{L_\l,\ell}(\text{Frob}_p)=-\sum_{x\in \F_p} \phi(x(1-x)(1-\l x))=-\pPPq{2}{1}{\phi&\phi}{&\eps}{\l}$$ and $\text{det}\, \rho_{L_\l,\ell}(\text{Frob}_p)=p$,  where $\text{Frob}_p$ stands for the   Frobenius conjugacy class  of $p$ in $G_\Q$.

\item For $\l\in  \Q\setminus \{0,1\}$, the L-function of $L_\l$ is
\begin{equation}\label{eq:L-L}L(L_\l,s)``=" \prod_{p\nmid N(L_\l)}\(1+\pPPq{2}{1}{\phi&\phi}{&\eps}{\l}p^{-s}+p^{1-2s}\)^{-1}.  \end{equation}Here, the quotations indicate that we are only giving a formula for the good $L$-factors.
\end{itemize}

\medskip

For more general cases, Archinard in \cite{Archinard} explained how to construct the
smooth model $X_\l^{[N;i,j,k]}$ of $C_\l^{[N;i,j,k]}$  for $\l\neq 0$, $1$ of \eqref{eq:GLC} and to compute all periods of first and second kind on $C_\l^{[N;i,j,k]}$ using hypergeometric functions. This  is also recast in  \cite{WIN3a}, and we follow that  development here.
Like $L_\l$, this curve is also a cyclic cover of $\C P^1$ ramifying at  $0$, $1$, $\frac{1}{\l}$, and  $\infty$.  When $\l=0$ or 1, the  number of ramification points of the covering map  is   at most $3$ and hence the covering curve  is a quotient  of a Fermat curve. This is the reason behind the degenerate situation happening at  $\l=0$ or 1.   Note that when $N\mid i,j,k,$ or $i+j+k$,  one can rewrite $C_\lambda^{[N;i,j,k]}$ as $y^N=x^i(1-x)^j$ by changing variables. We exclude this degenerate case below.

To consider the generic case of $C_\l^{[N;i,j,k]}$ being a cyclic cover of  $\C P^1$ ramifying at exactly 4 points, we will assume the following for the remaining of this section:
$$ N\nmid i,j,k,i+j+k, \quad\textnormal{gcd}(i,j,k,N)=1, \quad \l\neq 0,1,$$where $\gcd$ stands for greatest common divisor.

\begin{Remark}
Note that the assumptions $N\nmid i,j,k,i+j+k$ are satisfied if $i,j,k,N$ are computed by \eqref{eq:abc->Nijk} from any given  $(a,b,c)\in \Q^3$ satisfying $a,b,c-a,c-b\notin \Z$. Also, by changing variables, one can assume $i,j,k>0$.
\end{Remark}

Assume  $\l\in {\Q}\setminus \{0,1\}$.
Then $X_\l^{[N;i,j,k]}$ has genus
$$
g(N;i,j,k):= 1+N-\frac{\gcd(N,i+j+k)+\gcd(N,i)+\gcd(N,j)+\gcd(N,k)}2,
$$
see \cite{Archinard} by Archinard.  We use $J_\l^{[N;i,j,k]}$ to denote the Jacobian of $X_\l^{[N;i,j,k]}$, which is an abelian variety of dimension given by $g(N;i,j,k)$ and is defined over $\Q$.   For each proper divisor $d$ of $N$, there is a curve  $C_\l^{[d;i,j,k]}$  and a canonical map $C_\l^{[N;i,j,k]}\, \rightarrow \, C_\l^{[d;i,j,k]}$ sending $(x,y)$ to $(x,y^{N/d})$ which is generically of degree $N/d$. This canonical map induces a surjective homomorphism $\pi_d: J_\l^{[N;i,j,k]}\, \rightarrow \, J_\l^{[d;i,j,k]}$. We now use $J_\l^{\text{prim}}$ to denote the primitive part of $J_\l^{[N;i,j,k]}$, the identity component of $\displaystyle \bigcap_{d\mid N}\mbox{ker}\, {\pi_d}$.

 The curve  $C_\lambda^{[N;i,j,k]}$  admits an automorphism $A_{\zeta_N}:\,(x,y)\mapsto (x,\zeta_N^{-1} y)$ and this map induces {a representation of the finite group $\Z/N\Z$, depending on the choice of the primitive $N$th root $\zeta_N$,} on the vector space $H^0(X_\lambda^{[N;i,j,k]}, \Omega^1)$ of the holomorphic differential $1$-forms on  $X_\lambda^{[N;i,j,k]}$. When  $N,i,j,k,\l$ are fixed, we   denote this curve by $X(\l)$ below for simplicity.  A basis of $H^0(X(\lambda), \Omega^1)$ can be chosen by the regular pull-backs of differentials on $C_\l^{[N;i,j,k]}$ of the  form
\begin{equation}\label{eq:differentials}
 \omega= \frac{x^{b_0}(1-x)^{b_1}(1-\lambda x)^{b_2}dx}{y^n}, \quad 0\leq n\leq N-1, \, b_i\in\Z,
\end{equation}
satisfying the following conditions equivalent to the pullback of $\omega$ being regular at $0,1,\frac 1\l, \infty$ respectively,
\begin{align*}
  b_0 \geq \frac{ni+\gcd(N,i)}N-1,\;\;
  b_1 \geq \frac{nj+\gcd(N,j)}N-1,\;\;
  b_2 \geq \frac{nk+\gcd(N,k)}N-1,
  \end{align*}
\begin{equation*}
   b_0+b_1+b_2 \leq \frac{n(i+j+k)-{\gcd(N,i+j+k)}}{N}-1.
\end{equation*}
 For details  on this construction, see the work of Archinard and Wolfart in \cite{Archinard, Wolfart}.
For each $0\leq n<N$, we let $V_n$ denote the isotypical component of $H^0(X(\lambda), \Omega^1)$ associated to the character $\sigma_n: \, \zeta_N\mapsto\zeta_N^n$, where $\zeta_N$ is a primitive $N$th root of unity.  Then the space $H^0(X(\lambda), \Omega^1)$ is decomposed into a direct sum $\displaystyle \bigoplus_{n=0}^{N-1}V_n$.
If $\gcd(n,N)=1$, the dimension of $V_n$ is given by
 $$
  \dim V_n=\left \{\frac{ni}N\right \}+\left \{\frac{nj}N\right \}+\left \{ \frac{nk}N\right \}-\left \{ \frac{n(i+j+k)}N\right \},
$$
where $\left \{ x\right \} =x-\lfloor x \rfloor$ denotes the fractional part of $x$, see \cite{Archinard-exceptional}.  Furthermore,
$$
 \dim V_n +\dim V_{N-n}=2,
$$ when $\gcd(n,N)=1$. The elements of $V_n$ with $\gcd(n,N)=1$ are said to be \emph{new}. The subspace
 $$
  H^0(X(\lambda), \Omega^1)^{\mbox{new}}=\displaystyle \bigoplus_{\gcd(n,N)=1}V_n
 $$ is of dimension $\varphi(N)$,  Euler's totient function of $N$, see \cite{Archinard}.

Let $S$ be a basis of $ H^0(X(\lambda), \Omega^1)^{\mbox{new}}$ whose elements are of the form $ \omega_n={x^{b_0}(1-x)^{b_1}(1-\lambda x)^{b_2}dx}/{y^n}$  with $\gcd(N,n)=1$. Under our assumptions, $J_\l^{\text{prim}}$ is of dimension $\varphi(N)$, and is defined over $\Q$.

The Jacobian variety $J^{\text{prim}}_\l$  is isomorphic to the quotient of $\C^{\varphi(N)}$ by the lattice of periods and it is isogenous to the complex torus $\C^{\varphi(N)}/\Lambda(\l)$ with
$$
  \Lambda(\l)=\left\{ \(\sigma_n(u)\int_{\gamma_{01}}\omega+\sigma_n(v)\int_{\gamma_{\frac1\l\infty}}\omega\)_{\omega\in S}: u,v \in \Z[\zeta_N]\right\},
$$
by Archinard and Wolfart \cite{Archinard, Wolfart}.
Here $\gcd(n,N)=1$, and $\sigma_n$ is the automorphism of $\Z[\zeta_N]$ defined by $\zeta_N\mapsto \zeta_N^n$.

In the case of  $0<i,j,k<N$ and  $N<i+j+k<2N$, we have
$$
 \dim V_n =\dim V_{N-n}=1, \quad \gcd(N,n)=1,
$$
and
$\omega_n$ is  $ x^{-\{ni/N\}}(1-x)^{-\{nj/N\}}(1-\l x)^{-\{nk/N\}}dx$. For instance, the differential forms $\omega_1$ and $\omega_{N-1}$ are $\omega_1=dx/y$ and $\omega_{N-1}=\frac{x^{i-1}(1-x)^{j-1}(1-\l x)^{k-1}}{y^{N-1}} dx$. Thus, when $ 1 \leq i, j, k < N$, $\gcd(N, i, j, k) = 1$, $N \nmid i + j$ nor $i + j + k$, and $\l\neq 0$, $1$, the lattice $\Lambda(\l)$ can be expressed in terms of
$$
   _2P_1 \left [\begin{array}{cc} {\{\frac {nk}N\}}&{1-\{\frac{ni}N\}}\\ &{2-\{\frac{ni}N\}-\{\frac{nj}N\}}\end{array}  ;\l\right ].
$$

 However, when $0<i+j+k< N$ or $2N<i+j+k< 3N$,
we do not have such a general form for the vector space $V_n$.
\begin{Example}
The spaces $V_1$ and $V_3$ corresponding to the family $C_\l^{[4;1,1,1]}$ have dimension $0$ and $2$, respectively. The space $H^0\(X(\l),\Omega^1\)$ is spanned by
$$
   S=\{ dx/y^2,\, dx/y^3,\, xdx/y^3\},
$$
and hence
$$
  _2P_1 \left [\begin{array}{cc} \frac 12& \frac 12\\ & 1\end{array}  ;\l\right ],\, _2P_1 \left [\begin{array}{cc} \frac 34& \frac 14\\ & \frac 12\end{array}  ;\l\right ],\, _2P_1 \left [\begin{array}{cc} \frac 34& \frac 54\\ & \frac 32\end{array}  ;\l\right ]
$$
are periods of $C_\l^{[4;1,1,1]}$.

For the family  $C_\l^{[5;3,4,4]}$, we have $\dim V_1=2$ and $\dim V_2=1$. A basis of $H^0\(X(\l),\Omega^1\)$ is

$$
  S=\left\{ \frac{dx}y,\, \frac{xdx}y,\, \frac{x(1-x)(1-\l x)dx}{y^2},\,  \frac{x(1-x)^2(1-\l x)^2dx}{y^3}\right\}.
$$
\end{Example}

\subsection{Galois interpretation for $_2\mathbb P_1$}\label{ss:6.3}
In this section we prove
Theorem \ref{thm:WIN3a}.
Let  $\l \in \Q \setminus \{0,1\}$
and set $K=\Q(\z_N)$ and
$\ol K$ to be its algebraic closure.

 For any fixed prime $\ell$, similar to the elliptic curve case discussed in \S \ref{ss:Galois5.1}, the $\ell^n$-torsion points of the abelian variety $J_\l^{\text{prim}}$ gives rise to a continuous homomorphism  $\rho^{\text{prim}}_{\l,\ell}$ from group $\text{Gal}(\ol{K}/K)$ to $GL_{2\varphi(N)}(\Z_\ell)$.  %when the powers $\ell^n$ are passing to the inverse limit.
For simplicity,  we extend the scalar  rings to $\ol {\Q}_\ell$ and note that $\rho^{\text{prim}}_{\l,\ell}$ only ramifies at finite many places.  Recall that we write $\eta_N$ to denote any  character of order $N$ on $\F_q^\times$ and extend $\eta_N$  to be defined on $\F_q$ by setting $\eta_N(0)=0$. Evaluating the traces of the representation $\rho^{\text{prim}}_{\l,\ell}$ at $\text{Frob}_\fp$  for any unramified prime $\fp$ of $\mathcal O_K$   with $q(\fp)=q$, one has
\begin{multline}\label{eq:count-Jprime}\text{Tr} \, \rho^{\text{prim}}_{\l,\ell}(\text{Frob}_\fp)= -\sum_{m\in (\Z/N\Z)^\times} \(\sum_{x\in \F_q} \eta_N^m\( x^i(1-x)^j(1-\l x)^k\) \)\\
%%=-\sum_{m\in (\Z/N\Z)^\times} \pPPq{2}{1}{\eta_N^{m(k-i)}&\eta_N^{mj}}{&\eta_N^{-mj}}{\l;q}
=-\sum_{m\in (\Z/N\Z)^\times} \pPPq{2}{1}{\eta_N^{-mk}&\eta_N^{mi}}{&\eta_N^{m(i+j)}}{\l;q}.
\end{multline}
For a related discussion, see Proposition \ref{point-count}; for more details, see \cite{WIN3a}. Essentially,  the equation (\ref{eq:count-Jprime}) can be proved by using induction on $N$.\\

We are now ready to prove Theorem \ref{thm:WIN3a}.

\begin{proof}[Proof  of Theorem \ref{thm:WIN3a}]
Fix an embedding of $\zeta_N$ to $\ol\Q_\ell$. Observe that the map $A_{\zeta_N}$, which is defined over $K$, induces an order $N$ automorphism $A_{\zeta_N}^*$ on $J_\l^{\text{prim}}$ and hence the representation spaces of $\rho^{\text{prim}}_{\l,\ell}$ over $\ol \Q_\ell$ by our assumption. Consequently,  by the construction of $J_\l^{\text{prim}}$  $$\rho^{\text{prim}}_{\l,\ell}|_{G_{K}}\cong \bigoplus_{\gcd(m,N)=1} \sigma_{\l,m},$$ where $\sigma_{\l,m}$ corresponds to the $\zeta_N^m$ eigenspace of $A_{\zeta_N}^*$. Due to the symmetric roles of $\sigma_{\l,m}$, they have the same dimension, which has to be 2.

For any $c\in K^\times$, fix $\sqrt[N]{c}$ an $N$th root. We consider the smooth model of the following twisted generalized Legendre curve $$C_{\l,c}^{[N;i,j,k]}: \quad y^N=cx^i(1-x)^j(1-\l x)^k.$$ It is isomorphic to $C_\l^{[N;i,j,k]}$  via the map $T_c: C_{\l,c}^{[N;i,j,k]} \rightarrow C_{\l}^{[N;i,j,k]}$ defined by $(x,y)\mapsto (x,\sqrt[N]{c}y)$.   Note that the primitive part of its Jacobian  variety $J_{\l,c}^{\text{prim}}$ is also $\varphi(N)$ dimensional defined over $\Q$. Let $\rho^{\text{prim}}_{\l,c,\ell}$ denote the corresponding Galois representation of $G_K$ over $\ol \Q_\ell$, and note its restriction over $G_{K}$ also decomposes into a direct sum of 2-dimensional subrepresentations, denoted by $\sigma_{\l,c,m}$ like the case $c=1$ before. To proceed, we fix a large prime $\ell$ and consider the $\ell^n$ division points  on the Jacobian $J_\l^{\text{prim}}$.  Identify the Jacobian variety $J_\l$ of the curve $C_\l^{[N;i,j,k]}$ with the group $Pic^0(C_\l^{[N;i,j,k]})$ of the classes of divisors of degree zero on $C_\l^{[N;i,j,k]}$.  Assume that $P$ represents the class $\sum_{i=1}^s n_i(x_i,y_i)$ where $(x_i,y_i)\in C_{\l,c}^{[N;i,j,k]}$ and degree $\sum_{i=1}^s n_i=0$ and $P$ is an $\ell^n$-division point on $J_{\l,c}^{\text{prim}}$. We further assume that as an element in the group algebra $\ol \Q_\ell \left [J_{\l,c}[\ell^n] \right]$, $P$ lies in the $\zeta_N^m$-eigenspace of the automorphism on $\rho^{\text{prim}}_{\l,c,\ell}[\ell^n]$ induced from  $A_{\zeta_N}:(x,y)\mapsto (x,\zeta_N^{-1}y)$ where $(m,N)=1$.  The isomorphism $T_c$ sends $\ell^n$-division points on $J_{\l,c}^{\text{prim}}$ to $\ell^n$-division points on $J_{\l}^{\text{prim}}$.   For any  $\text{Frob}_\fp\in G_{K}$ with $\fp$ coprime to the discriminant of $K(\sqrt[N]{c})$,
\begin{align*}
\text{Frob}_\fp(T_c(P)) & =  \text{Frob}_\fp\sum_{i=1}^s n_i\(x_i,\sqrt[N]{c} \cdot y_i\) \\
& \overset{\eqref{eq:m-Artin}}= \sum_{i=1}^s n_i \(\text{Frob}_\fp(x_i), \(\frac c{\fp}\)_N \sqrt[N]{c} \cdot\text{Frob}_\fp(y_i)\) \\
& = T_c\(\sum_{i=1}^s n_i \(\text{Frob}_\fp(x_i), \(\frac c{\fp}\)_N \cdot\text{Frob}_\fp(y_i)\)\) \\
& = T_c\(\frac c{\fp}\)_N^{-m}P.
\end{align*}
This means  $\sigma_{\l,c,m}\cong \sigma_{\l,m}\otimes \chi_{-\frac{m}N,c}$ where $\chi_{-\frac{m}N,c}$ is as in \eqref{eq:chi,im}. \bk Summing up all pieces, we have
\begin{equation*}\text{Tr} \rho^{\text{prim}}_{\l,c,\ell}(\text{Frob}_\fp)=\sum_{m\in (\Z/N\Z)^\times}\(\frac c{\fp}\)_N^{-m} \cdot \text{Tr} \sigma_{\l,m} (\text{Frob}_\fp).
\end{equation*}

Meanwhile, in terms of explicit point counting using characters,
\begin{multline*}\text{Tr} \rho^{\text{prim}}_{\l,c,\ell}(\text{Frob}_\fp)=-  \sum_{m\in (\Z/N\Z)^\times } \sum_{x\in \F_{\fp}}\iota_\fp\(\frac mN\) (c x^i(1-x)^j(1-\l x)^k) \\
=-  \sum_{m\in (\Z/N\Z)^\times } \(\frac{c}{\fp}\)_N^m \cdot \sum_{x\in \F_{\fp}}\iota_\fp\(\frac mN\) ( x^i(1-x)^j(1-\l x)^k).
\end{multline*} As the  above two equations hold for arbitrary $c \in K$, we have
\begin{multline*}\text{Tr} \sigma_{\l,m} (\text{Frob}_\fp)=-\sum_{x\in \F_{\fp}}\iota_\fp\(\frac {-m}N\) ( x^i(1-x)^j(1-\l x)^k)\\= -\pPPq{2}{1}{\iota_\fp(\frac{mk}N)& \iota_\fp(\frac{-mi}N)}{&\iota_\fp(\frac{-m(i+j)}N)}{\l}.\end{multline*} Using \eqref{eq:abc->Nijk}, if we let $\sigma_{\l,1}$ above  be the 2-dimensional representation $\sigma_{\l,\ell}$ stated in Theorem \ref{thm:WIN3a}, then $\iota_\fp(\frac{mk}N)=\iota_\fp(a)$, $\iota_\fp(\frac{-mi}N)=\iota_\fp(b)$, and $\iota_\fp(\frac{-m(i+j)}N)=\iota_\fp(c)$ respectively, which concludes the proof of Theorem \ref{thm:WIN3a}.

\end{proof}

\begin{Remark}\label{rem:1}In other words, for fixed rational numbers $a,b,c$, the corresponding $$-_{2}{\mathbb P}_{1} \left[ \begin{smallmatrix}\iota_\fp(a) &\iota_\fp(b)\\&\iota_\fp(c)\end{smallmatrix}; \l; q(\fp) \right]$$
functions are the traces (or characters) of an explicit 2-dimensional Galois representation $\sigma_{\l,\ell}$ of $G_K$ at the Frobenius elements.  When $a,b,c-a,c-b\notin \Z$, the representation is \emph{pure} in the sense that for each good prime ideal $\fp$ such that $\l\in \Z$ and $\l\neq 0,1 \pmod{p}$, the characteristic polynomial of $\sigma_{\l,\ell}(\text{Frob}_\fp)$ is of the form
\begin{equation}
\label{eq:Hp}H_\fp(T)=T^2- \text{Tr}\sigma_{\l,\ell}(\text{Frob}_\fp)T+\text{det}\sigma_{\l,\ell}(\text{Frob}_\fp)
\end{equation}
 and  has two roots of the same absolute value $\sqrt{q}$ where $q=\#(\mathcal O_K/\fp)$. See Corollary \ref{cor:det} for how to compute  $\text{det}\sigma_{\l,\ell}(\text{Frob}_\fp)$. When either $a,b,c-a,$ or $c-b$ in $\Z$, the corresponding representation {
 degenerates}. We will give some examples in \S \ref{ss:impure2P1}.
\begin{Example}\index{finite field analogues!Kummer evaluation formula}
The following analogue of Kummer's evaluation (see \eqref{eq:Kummer})  expresses the value of the $_2\mathbb P_1$ function at $-1$ in terms of not one but two Jacobi sums. To be more precise, let $B,D,\phi\in \widehat{\F_q^\times}$ where $\phi$ is of order 2.  Then, for $C=D^2$,
  \begin{equation}\label{eq:FF-Kummer}
  \phgq {B}{C}{C\ol B}{-1;q} ={J(D,\ol B)}+{J(D\phi, \ol B)}.
 \end{equation} This is proved by Greene in  \cite[(4.11)]{Greene}. For given $b,c\in \Q$  such that $b,c,c-2b\notin \Z$ and letting $M$ be the least positive common denominator of $\frac c 2,\frac{c+1}2,b$,  there is a 2-dimensional $\ell$-adic Galois representation $\sigma_{-1,\ell}$  of $G_{\Q(\zeta_N)}$ corresponding to $\pFq{2}{1}{b&c}{&c-b}{-1}$ via   Theorem \ref{thm:WIN3a} such that at each good unramified prime $\fp$,
 %\begin{multline*}
 $$\text{Tr}\sigma_{-1,\ell} (\text{Frob}_\fp)=\chi_{\frac c2-b,-1}(\fp)\mathcal J_{(\frac c2,-b)}(\fp)+\chi_{\frac {c+1}2-b,-1}(\fp)\mathcal J_{(\frac {c+1}2, -b)}(\fp), $$%\end{multline*}
 where the notation $\mathcal J_{(a,b)}(\fp)=-\iota_\fp(a)\iota_\fp(b)(-1)J(\iota_\fp(a),\iota_\fp(b))$  as \eqref{J->mathcal J} in \S \ref{ss:Jacobi-Grosse}.
\end{Example}

 When $\l=0$ or 1, the $_2\mathbb{P}_1$ function corresponds to  a dimension 1 (instead of 2) compatible family of Galois representations  and in these cases the corresponding generalized Legendre curves have smaller genus.  When $\l=1$, the curve $C_{\l}^{[N;i,j,k]}$ becomes $y^N=x^i(1-x)^{j+k}$ which is a quotient of a Fermat curve.
 The following analogue of Gauss' evaluation formula \eqref{eq:Gauss} follows from \eqref{Eq:30}
 \begin{equation}\label{eq:Gauss-FF}\index{finite field analogues!Gauss  evaluation formula}
 \phgq {A}{B}{C}{1;q}=\sum_{y \in \F_q} B(y)\ol{AB}C(1-y)=J(B,\ol{AB}C).
 \end{equation}

They allow one to determine the Galois representation up to semisimplification and compute the determinants of the representations at Frobenius elements, see \eqref{eq:det} below. Consequently, one can use these period
functions to compute the Euler $p$-factors of the L-function of the corresponding Galois representations, as in equations (\ref{eq:L-A}) and (\ref{eq:L-L}). Watkins has written a \texttt{Magma} program which  computes these  $p$-factors under additional assumptions (namely when the hypergeometric motives are defined over $\Q$, see \cite{Watkins}. Also, in \cite{RRW}, there are discussions on the properties of such  $p$-factors, hypergeometric L-functions, their conductors and functional equations.)  \bk
\end{Remark}

\begin{Remark}
From the representation point of view, the products (resp. sums) of $_{2}{\mathbb P}_{1}$  functions over the same finite field correspond to tensor products (resp. direct sums) of the corresponding representations. See \cite{Serre1}.
\end{Remark}

\subsection{Some special cases of ${}_2\mathbb P_1$-functions}\label{ss:impure2P1}
 In Remark \ref{rem:1}, it is mentioned that when $a,b,a-c,b-c\notin \Z$, i.e. when $\pFq{2}{1}{a&b}{&c}{\l}$ is primitive (see Definition \ref{def:3}), the representation $\sigma_{\l,\ell}$ stated in Theorem \ref{thm:WIN3a} is pure.
We now  consider the cases below  giving identities for the imprimitive  cases, which correspond to $\pFq{2}{1}{a&b}{&c}{\l}$ with either $a,b,a-c,b-c\in \Z$.  {Indeed, in the first two formulas below without $\delta$ terms, the right sides can be thought of as `$q$ plus a unit' and thus resemble Eisenstein series corresponding to reducible representations. The last two formulas are also of this form when $\l \neq 1$ and resemble characters when $\l=1$.}

\begin{Proposition}\label{prop: 2F1-imprimitive} Suppose $\l \neq 0$.  Then,
\begin{align*}
\phgq {\eps}BC{\l} &=
                J(B, \overline BC)-\overline C(\l) \overline BC(\l-1); \\
\phgq {A}BB{\l} &=\overline B(\l)J(B,\overline A)-\overline A(1-\l); \\
\phgq {A}BA{\l} &=\overline B(\l-1)J(B,\overline A)-B(-1)\overline A(\l) +(q-1)\delta(1-\l)\delta(B);\\
\phgq {A}{\eps}C{\l} &=\overline C(-\l)\overline AC(1-\l)J(C,\overline A)-1+(q-1)\delta(1-\l)\delta(\overline AC).
\end{align*}
\end{Proposition}

\begin{proof}
When $\l \ne 0$, by \eqref{Eq:30}
\begin{multline*}
  \phgq {\eps}BC{\l}=\sum_y B(y)\ol BC(1-y)\eps(1-\l y)=\sum_{y\neq 1/\l}B(y) \ol BC(1-y)\\
               = \sum_{y}B(y) \ol BC(1-y)-\ol C(\l)\ol BC(\l-1) = J(B,\overline BC)-\ol C(\l)\ol BC(\l-1);
\end{multline*}

\begin{multline*}
 \phgq ABB{\l}=\sum_y B(y)\eps(1-y)\overline A(1-\l y)= \sum_{y\neq 1}B(y)\overline A(1-\l y) \\
 =\overline B(\l)J(B,\overline A)-\overline A(1-\l).
\end{multline*}
%Due to the relation
%$$
%J(A,\overline B)J(C,\overline A)=B(-1)J(C,\overline B)J(C\overline B, \overline A B)-\delta(A)(q-1)+\delta(B\overline C)(q-1),
%$$
Taking $A_1=B_1=A$ and  $A_2=B$ in \eqref{eq:for6.7}
we have
the function $\phgq {A}BA{\l}$ can be written as
%\begin{multline*}
$$
      \phgq {A}BA{\l}=\frac{B(-1)}{q-1}\sum_\chi J(A\chiup,\overline {\chiup})J(B\chiup,\overline {A\chiup})\chiup(\l).
$$
%      \end{multline*}

The relation
$$
J(R,\overline S)J(T,\overline R)=S(-1)J(T,\overline S)J(T\overline S, \overline R S)-\delta(R)(q-1)+
\delta(S\overline T)(q-1)
$$
can  be verified using
\eqref{Gauss1} through \eqref{JAeps} and routine algebra.
It is simplest to deal with the cases where the various $\delta$ terms are nonzero first.

Now take
$R=A\chiup$, $S=\chiup$ and $T=B\chiup$ to see our
$_2{\mathbb P}_1$ is
\begin{multline*}
\frac{B(-1)}{q-1}J(B,\overline A)
\left[ \left(\sum_\chi J(B\chiup,\overline{\chiup})\chiup(-\l)\right] \right)-B(-1)\overline A(\l)+\delta(B)\delta(1-\l)(q-1)
\end{multline*}
which, upon replacing $J(B\chiup,\overline{\chiup})$ by its
defining sum and interchanging the order of
summation in the resulting double sum, becomes
\begin{equation*}
B(-1)J(B,\overline A)\overline B(1-\l) -B(-1)\overline A(\l) +\delta(B)\delta(1-\l)(q-1)
\end{equation*}
$$
=J(B,\overline A)\overline B(\l-1) -B(-1)\overline A(\l)+\delta(B)\delta(1-\l)(q-1).
$$
\bk
Using the same  type of  argument, we get the last claim.
\end{proof}

\subsection{Galois interpretation for $_{n+1} \mathbb F_{n}$}\label{ss:Galois_normalization}
We now use Theorem \ref{thm:WIN3a} to give a description of the Galois representations associated to the period functions $_{2} \mathbb P_{1}$. Recall $K=\Q(\zeta_N)$.

As mentioned earlier, when the parameter $\l \in  {\Q}$, we have that the period function $_{n+1} \mathbb P_{n}$ corresponds to the trace of a Galois representation of degree at most $n$. The normalized period, which is obtained by dividing $_{n+1} \mathbb P_{n}$ by a product of Jacobi sums, corresponds to the tensor product of the  Galois representation associated with $_{n+1} \mathbb P_{n}$ and  a linear character associated with the Gr\"ossencharacter arising from the Jacobi sums as in \S \ref{ss:Jacobi-Grosse}.

When $n=1$  and $\l\in  \Q\setminus \{0,1\}$, by Theorem \ref{thm:WIN3a} we have under the same assumptions, that there is a 2-dimensional $\ell$-adic  Galois representation  $$\tilde{\sigma}_{\l,\ell}:=\sigma_{\l,\ell}\otimes (\chi_{c,-1}\cdot \mathcal J_{(b,c-b)})^{-1}$$  over $\ol \Q_\ell$ such that for any good prime ideal $\fp$  of $\mathcal O_{K}$,
%where $K$ is the Galois closure of $\Q(\l,\zeta_N)$,
$$\text{Tr}\,\tilde \sigma_{\l,\ell} (\text{Frob}_\fp)=\, _2\F_1\left[\begin{matrix}\iota_\fp(a) &\iota_\fp(b) \\   & \iota_\fp(c) \\\end{matrix}; \l; q(\fp)\right].$$ See the proof of Theorem \ref{thm:WIN3a} for the construction of ${\sigma}_{\l,\ell}$, and \eqref{eq:chi,im} (resp.  \eqref{J->mathcal J}) for the notation for $\chi_{c,-1}$ (resp. $\mathcal J_{b,c-b}$).
Note that the $-$ sign is not needed for the above formula due to the relation between $\mathcal J_{(\cdot, \cdot)}$ and $J(\cdot,\cdot)$ given in \eqref{J->mathcal J}. Moreover,  when $c-b,b,c\notin \Z$,  the characteristic polynomial of $\tilde \sigma_{\l,\ell} (\text{Frob}_\fp)$ is degree 2 with two roots of  absolute value $1$. This will be particularly convenient when taking tensor products.

\subsection{Zeta functions and hypergeometric functions over finite fields}\label{ss:zeta}
In this section, we relate our finite field period functions $_{n+1}\mathbb P_n$ to the local zeta functions of hypergeometric varieties, and demonstrate how Theorem \ref{thm:WIN3a} can be used to compute the L-function of associated Galois representations.

Given any hypersurface $H_f$ defined by an algebraic affine (or projective) equation $f(x_1,\cdots,x_{n})=0$ over a finite field $\F_q$, use $N_s$ to denote the number of affine (projective) solutions of $f$ over $\F_{q^s}$. The \Index{zeta function} of $H_f$ over $\F_q$ is defined by  $$Z_q(H_f,T):=\exp\(\sum_{s=1}^\infty \frac{N_s T^s}{s}\).$$
Now we recall the following theorem of Dwork.
\begin{Theorem}[Dwork, see \cite{Koblitz-book-p-adic}]The zeta function of any affine or projective hypersurface given by $f(x_1,\cdots,x_{n})=0$ is a ratio of two polynomials with coefficients in $\Z$ and constant 1.
\end{Theorem}

For example, if $C$ is a smooth projective  irreducible genus $g$ curve defined over $\F_q$, then Weil showed that $$Z_q(C,T)=\frac{P(T)}{(1-T)(1-qT)}$$ where $P(T)\in \Z[T]$ has degree $2g$ with all roots of absolute value $1/\sqrt{q}$. See \cite{IR} for more details.

In this perspective, Proposition \ref{point-count} implies that the finite field period functions $_{n+1}\mathbb P_n$ are related to the local zeta functions of  hypergeometric varieties given by algebraic equations of the form \eqref{eq:hyp-var}. In the following discussion, we focus on the case $n=1$.
Along this line, we relate  Theorem \ref{thm:WIN3a} to some result concerning twisted exponential sums.
These involve the finite field period functions $_2\mathbb P_1$ whose parameters are characters on finite extensions of $\F_q$. \bk Given $A,B,C\in \widehat{\F_q^\times}$, by our notation
$$\phgq {A}{B}{C}{\l;q}=\sum_{y\in \F_q} B(y)\overline {B}C(1-y)\overline A(1-\l y).$$
As mentioned in \S \ref{ss:mthpower}
any character $A$ on $\F_q^\times$ can be extended to a multiplicative character $A_r$ on the finite extension $\F_{q^r}$ using the norm map $\text{N}_{\F_q}^{\F_{q^r}}$, i.e.  for $x\in \F_{q^r}$, $$A_r(x)=A(\text{N}_{\F_q}^{\F_{q^r}}(x)).$$See (\ref{eq:mth-conjugate}).
For instance, if $x\in \F_q$, then
 \begin{equation}\label{eq:A_r}A_r(x)=A(x\cdot x^q\cdots x^{q^{r-1}})=A(x\cdot x\cdots x)=A(x^r)=\(A(x)\)^r.\end{equation}
Thus we can define $_{2}{\mathbb P}_{1} \left[ \begin{smallmatrix}A_r &B_r \\&C_r \end{smallmatrix}; \l; q^r \right]$ accordingly. Using the perspective of twisted Jacobi sums \cite{AS1,AS2}, the generating function
\begin{equation}\label{eq:gen-2P1}
Z[A,B;C;\l;q;T]:=\text{exp}\( \sum_{r\ge 1} \(\phgq {A_r}{B_r}{C_r}{\l;q^r} \)\cdot \frac {T^r}r\)
\end{equation} is also a rational function, which is originally due to Dwork.

In view of Theorem \ref{thm:WIN3a} and the third \bk
Weil conjecture (proved by Deligne) for local zeta functions of smooth curves (cf. \cite[\S11.3]{IR}),  we know for $a,b,c\in \Q$ and $a,b,c-a,c-b\notin \Z$,  the corresponding Galois representation $\sigma_{\l,\ell}$  is pure and at each unramified prime ideal $\fp$ and $\l\neq 0,1 \pmod{\fp}$, the function $$Z[\iota_\fp(a),\iota_\fp(b);\iota_\fp(c);\l;q(\fp);T]=H_\fp\(\frac 1T\)$$ where $H_\fp(T)$ is the characteristic polynomial of the Frobenius $\text{Frob}_\fp$ under $\sigma_{\l,\ell}$, as in \eqref{eq:Hp}. Thus it is a polynomial with two roots of the same absolute value $1/\sqrt{q}$.

To give an example, we first recall  another Hasse-Davenport relation (see  Section 11.5 in \cite{Berndt-Evans-Williams} or Section 11.4 of \cite{IR})  which
relates  the Gauss sums $g(A)$ over $\F_q$ and  $g(A_r)$ over $\F_{q^r}$ by
\begin{equation}\label{eq:H-D2}\index{Hasse-Davenport relation}
  g(A_r)=(-1)^{r-1}g(A)^r.
\end{equation}
\begin{Example}By \eqref{eq:H-D2} and the finite field analogue of Kummer's evaluation \eqref{eq:FF-Kummer}, one has when $C=D^2$, and $B^2\neq C$, that
\begin{equation*}
  Z(B,C;C\ol{B};-1;q;T)=(1+J(D,\ol B)T)(1+J(D\phi, \ol B)T).
\end{equation*}
In comparison, the finite field analogue of the Gauss evaluation formula \eqref{eq:Gauss-FF} implies
$$
 Z[A,B;C;1;q;T]={1+J(B,\ol{AB}C)T},\quad  \mbox{ if } A\neq C.
 $$
\end{Example}

As a corollary of Theorem \ref{thm:WIN3a}, one can describe the determinants of the 2-dimensional Galois representation $\sigma_{\l,\ell}$ of  Theorem \ref{thm:WIN3a} at Frobenius elements explicitly by a simple computation of eigenvalues of Frobenius elements.
\begin{Corollary}\label{cor:det}Let the notation be as in Theorem \ref{thm:WIN3a}. Then
\begin{equation}\label{eq:det}
\text{det}\, \sigma_{\l,\ell} (\text{Frob}_\fp)=\frac{1}2\( \(_{2}{\mathbb P}_{1} \left[ \begin{smallmatrix}\iota_\fp(a) &\iota_\fp(b)\\&\iota_\fp(c)\end{smallmatrix}; \l; q(\fp) \right]\)^2+\, _{2}{\mathbb P}_{1} \left[ \begin{smallmatrix}(\iota_\fp(a))_2 &(\iota_\fp(b))_2\\&(\iota_\fp(c))_2\end{smallmatrix}; \l; q(\fp)^2 \right] \),
\end{equation}
{utilizing the notation described in \eqref{eq:A_r}.}
Together, we are able to compute the L-function $L(\sigma_{\l,\ell}, s)$ of $\sigma_{\l,\ell}$ if we omit the ramified prime factors. It will take the form of
$$\prod_{\fp \text{ good prime}} \(1+\pPPq{2}{1}{\iota_\fp(a)&\iota_\fp(b)}{&\iota_\fp(c)}{\l;q(\fp)}q(\fp)^{-s}+ \text{det}\, \sigma_{\l,\ell} (\text{Frob}_\fp)\cdot q(\fp)^{-2s}\)^{-1}. $$
\end{Corollary}

%\subsection{Evaluation formulas}

\subsection{Summary}
To us, the Galois representation interpretation, i.e. that  the $_{n+1}\mathbb P_n$ or  $_{n+1}\F_n$ functions are, up to sign, the character values of  Galois representations at corresponding Frobenius conjugacy classes,  provides a useful guideline when one translates classical results to the finite field setting. A few principles are summarized as follows.

First, let $\l\in  \Q$, and $N$ the least common  multiple  of the orders of the $A_i$ and $B_j$. Assume $q \equiv 1 \pmod N$ and that $\l$ can be embedded into  $\F_q$. Then:
\begin{itemize}
\item
The function $\pFFq{n+1}{n}{A_1& A_2 &\cdots&A_{n+1}}{&B_1&\cdots&B_{n}}{\l;q}$ corresponds, up to a sign, to the trace of a degree at most $n+1$
Galois representation of $G_K$ at the Frobenius conjugacy class at a prime ideal with residue of size $q$. {When $\l=0,1$, the imprimitve cases, %the degree is usually less than $n+1$ and 
the representation is often degenerate.} \\

\item
When  $\l\neq 0,1$, the primitive $_2\mathbb P_1$  or $_2\mathbb F_1$     function corresponds to a $2$-dimensional Galois representation, which is pure.\\

\item
Products (resp. sums) of $_{n+1}{\mathbb P}_{n}$ or $_{n+1}\F_n$  functions correspond to tensor products (resp. direct sums) of the corresponding representations.

\end{itemize}

\begin{Remark}
In relation to the first item above, see \cite{AS1,AS2, Katz} and the paper on hypergeometric motives by Roberts and Rodriguez-Villegas \cite{RRW}, as well as a different approach to realize the hypergeometric motives defined over $\Q$ by Beukers,  Cohen, and Mellit in \cite{BCM} based on Katz's result in \cite{Katz}. See \cite{BCM,LTYZ} for some explicit examples in this approach.
\end{Remark}

\section{A finite field Clausen formula and an application}\label{Ramanujan}

In this section, we use the previous discussion of Galois representations from Sections \ref{Gal-background} and \ref{Gal} to discuss a finite field version of the classical Clausen formula, due to Evans and Greene, and how its geometric interpretation sheds light on Ramanujan type formulas for $1/\pi$.  Below we continue to use $\eps$ to denote the trivial character and  $\phi$ to denote the quadratic character. The discussion below is closely related to Weil's result  describing Jacobi sums as Gr\"ossencharacters, which was recalled in \S \ref{ss:Grossencharacter}.

\subsection{A finite field version of the  Clausen formula by Evans and Greene}\label{sec:ClausenFF}
One version of the Clausen formula \cite{Bailey}\index{Clausen formula} states that
\begin{equation}\label{eq:Clausen}
 \pFq{2}{1}{c-s-\frac 12&s}{&c}{\l}^2=\, \pFq{3}{2}{2c-2s-1&2s&c-\frac 12}{&2c-1&c}{\l}.
\end{equation}
In terms of differential equations, this means that the symmetric square of the 2-dimensional solution space of  $HDE(c-s-\frac 12,s;c;\l)$  (see \eqref{classicalHDE})   is the 3-dimensional solution space of the hypergeometric differential equation satisfied by the $_3F_2$ occurring on the right side of \eqref{eq:Clausen}. In \cite{Evans-Greene}, Evans and Greene obtained the following  analogue of \eqref{eq:Clausen} written here in our notation.
\begin{Theorem}[\cite{Evans-Greene}, Theorem 1.5]\label{thm:FF-Clausen}
    Let $C,S\in \widehat{\F_q^\times}$. Assume that $C\neq \phi$, \bk and $S^2\not\in \left\{\eps, C, C^2\right\}$. Then for $\l\neq 1$,
\begin{multline*}
      {\pFFq 21{C\overline S\phi& S}{&C}{\l}}^2= \pFFq{3}{2}{C^2\ol S^2&S^2&C\phi}{&C^2&C}{\l}\\
      +\phi(1-\l)\ol C(\l)\(\frac{J(\ol S^2 ,C^2)}{J(\ol C,\phi)}+\delta(C)(q-1)\).
\end{multline*}
 \end{Theorem}\index{finite field analogues!Clausen formula}

 %
 %{\color{blue}[To include $x=1$, we will need an addtional $\delta(1-x)$ term as in this case, LHS coerresponds to 1-dimensional %%representation while the 3F2 at 1 should be two dimensional. I think the extra delta term can be computed from from Greene %%Theorem 4.38. Ling ]}

This theorem  captures the well-known fact from representation theory that the tensor square of a  2-dimensional representation (associated with the $_2\mathbb F_1$ on the left) equals  its symmetric square (3-dimensional representation, associated with the $_3\mathbb F_2$ on the right) plus an additional linear representation from the exterior power.  The  expression here using our notation is closer to the complex setting than the version of the statement given in Theorem 1.5 of \cite{Evans-Greene} which is in terms of period functions.
We remark that when $\l=1$, by Theorem 4.38-(i) in \cite{Greene}  we have
\begin{multline}\label{special-claussen@1}
\pFFq{3}{2}{C^2\ol S^2&S^2&C\phi}{&C^2&C}{1}=
\sum_{D \in \{S,S\phi\}}
\frac{\phi(-1)J(D,C\ol S^2)J(\ol C S^2,\phi\ol D)}{J(\phi ,\phi C)J(S^2,C^2\ol S^2)}\\
=\sum_{ D\in \{S,S\phi\}}
\frac{J(D,C\ol S^2)J(\ol C S^2,\phi\ol D)}{J(\phi S,\ol C)J(S,\ol C)},
  %\frac{J(S,C\ol S^2)J(\ol C S^2,\phi\ol S)}{J(\phi S,\ol C)J(S,\ol C)}+\frac{J(\phi S,C\ol S^2)J(\ol C S^2,\ol S)}{J(\phi S,\ol C)J(S,\ol C)},
\end{multline}
which corresponds to a 2-dimensional Galois representation  that can be described by Gr\"ossencharacters,   while Gauss' evaluation theorem says that
$$
\pFFq 21{C\overline S\phi& S}{&C}1^2=\frac{J(\phi, S)^2}{J(S,\ol C)^2}.
$$

A nice example which realizes the Clausen formula geometrically is given by Ahlgren, Ono and Penniston \cite{AOP}.  In their work, they consider the $K3$ surfaces defined by
$$
 X_\l :\, s^2=xy(1+x)(1+y)(x+\l y), \quad \l \neq 0, -1.
$$
In particular, the  point counting on $X_\l$ over $\mathbb{F}_q$ is related to a $_3{\mathbb P}_2$ by the following equation,
\begin{equation}\label{eqn:counting}
\sum_{x,y\in\mathbb{F}_q} \phi(xy(1+x)(1+y)(x+\l y))= 
 \left[
   \begin{matrix}
    \phi & \phi & \phi \\
         & \eps & \eps \\
\end{matrix};-\l \right].
%\phgthree \phi\phi\phi\eps\eps{-\l}.
\end{equation}
One establishes the above equality by expanding the right hand side via \eqref{n+1Pn} and employing the change of variables $(x,y) \mapsto (-x,-y)$ followed by $x \mapsto 1/x$.

The point counting on the following elliptic curve 
$$
  E_\l:\, y^2=(x-1)(x^2-1/(1+\l)), \quad \l \neq 0,-1,
$$
over $\mathbb{F}_q$ is given by
$$
  a(\l,q):=-\sum_{x\in\mathbb{F}_q}\phi(x-1)\phi(x^2-1/(1+\l)).
$$
Then we have the equality (\cite[Theorem 2.1]{AOP})
\begin{equation}\label{eq:3F2-ap-square}
%  \phgthree \phi\phi\phi\eps\eps{-\l}
   {}_3\mathbb P_2
 \left[
   \begin{matrix}
    \phi & \phi & \phi \\
         & \eps & \eps \\
\end{matrix};-\l \right]  
  =\phi(1+\l)( a(\l,q)^2-q).
\end{equation}
A geometric interpretation of the relation \eqref{eq:3F2-ap-square} between the K3 surfaces $X_\lambda$ and the elliptic curve $E_\l$ in terms of the so-called Shioda-Inose structure has been given in \cite{Long04} by Long.
When $q\equiv 1\pmod{4}$,
$a(\l,q)$ is essentially  $\phgq{\eta_4}{\eta_4} {\eps}{-\l}$ where $\eta_4$ is an order 4 character.
\bk To be more precise, if $1+\l=b^2$ is a square in the finite field $\F_q$, then \bk
$$
  a(\l,q)=\phi(b-1)\phgq{\eta_4}{\eta_4}{\eps}{-\l}.
$$
Thus {Theorem 1.1} of \cite{AOP} is equivalent to the Clausen formula over the finite field $\mathbb{F}_q$ with $S=\eta_4$ and $C=\eps$.

For special choices of $\l\in \Q$ such as $1$, $8$, $1/8$, $-4$, $-1/4$, the corresponding elliptic curve $E_\l$ has complex multiplication (CM). For these $\l$ values, the period functions $_2\mathbb P_1$ can be written in terms of Jacobi sums.
\begin{Example}\label{eg:7}When $\l=1/8$,  the elliptic curve $E_{1/8}$ is
 $\Q$-isogenous to the curve $y^2=x^3-9x$  which has CM and is of conductor $288$  \cite{lmfdb:288.d3}. From this information and the explicit description of  Gr\"ossencharacter $\mathcal J_{(\frac 12,\frac 14)}$, which is given in Example \ref{eg:Cond-64}, one has

 \begin{align*}
  -a(1/8,q)%=&\phi\(3\)\sum_u\phi(u)\phi(1- u)\phi(1+ u)
  %\\
   %   =&\phi(3)\cdot \phgq{\phi}{\phi}{\eps}{-1}\\
   %   =&{\color{red} \phi(3) \cdot
   %   \begin{cases}
   %   0,& \mbox{ if } q\equiv -1 \pmod{4},\\
   %    \phi(2)\phgq{\eta_4}{\eta_4}{\eps}{1}+J(\phi,\ol{\eta_4}), & \mbox{ if } q\equiv 1 \pmod{4}
   %   \end{cases}
   %   }\\
     % =&{\color{red} \ast \cdot
     % \begin{cases}
     % 0,& \mbox{ if } q\equiv -1 \pmod{4},\\
     %  \phi(2)\phgq{\eta_4}{\eta_4}{\eps}{-1/8}+J(\phi,\ol{\eta_4}), & \mbox{ if } q\equiv 1 \pmod{4}
     % \end{cases}
      %}\\
      =&\phi(3) \cdot
      \begin{cases}
      0,& \mbox{ if } q\equiv -1 \pmod{4},\\
      J(\phi, \eta_4)+J(\phi, \ol\eta_4), & \mbox{ if } q\equiv 1 \pmod{4}.
      \end{cases}
\end{align*}
\end{Example}

\begin{Example}\label{eg:8}
When $\l=-1/4$, the elliptic curve $E_{-1/4}$ is $\Q$-isogenous to the curve $ y^2=x^3-1$, which has complex multiplication and  conductor $144$ \cite{lmfdb:144.a3}.  It follows that
$$-a(-1/4,q)=\phi(-1)\cdot
  \begin{cases}
      0,& \mbox{ if } q\equiv -1 \pmod{3}\\
      J(\phi, \eta_3)+J(\phi, \ol{\eta}_3), & \mbox{ if } q\equiv 1\pmod{3},
      \end{cases}
$$
where
$-\phi(-1)J(\phi,\eta_3)$ is a  Hecke character
with   conductor $\frak m \bk =(4\sqrt{-3})$ over $\Q(\zeta_6)$ and it can be given as follows. Suppose that a prime ideal $\fp$ of $\Z(\zeta_6)$ is generated by $\alpha\in\Z[\zeta_6]$.  Then we have
  $$
    \mathcal J_{(\frac 12, \frac 13)}(\fp)=\chi(\alpha)\cdot \alpha,
  $$
 where $\chi(-1)=-1$, $\chi(5)=-1$, and $\chi(\zeta_6)=-\zeta_6$. Here, the unit group $\(\Z[\zeta_6]/\frak m\)^\ast$ has order $24$ and is generated by $-1$, $5$, and $\zeta_6$.
\end{Example}

\subsection{Analogues of Ramanujan type formulas for $1/\pi$}\label{ss:Ram}

In 1914, Ramanujan \cite{Ramanujan} gave a list of infinite series formulas for $1/\pi$, where the series are values of hypergeometric functions. One example is
\begin{equation}\label{eq:RR1/pi}
\sum_{k=0}^{\infty}(6k+1)\(\frac{(\frac 12)_k}{k!}\)^3\(\frac 14\)^k=\frac{4}{\pi}.
\end{equation}

Later in the 1980's, Borwein-Borwein \cite{BB} and Chudnovsky-Chudnovsky \cite{CC} proved these formulas using essentially the theory of elliptic curves with complex multiplication. The idea of their proof was recast in \cite{WIN2} and we will give some related discussion here. Roughly speaking, the family of elliptic curves $E_\l$ defined in \S \ref{sec:ClausenFF} has unique holomorphic differentials $\omega_\l$, up to  scalars, similar to the case for the Legendre curve $L_\l$ that we mentioned earlier in \S \ref{GLC}, see \eqref{eq:Llambda}. Integrating $\omega_\l$ along a suitably chosen path on $E_\l$ leads to
a period of the first kind on $E_\l$ given by  $$p(\l)=\G\left(\frac 14\right)\G\left(\frac 34\right)(1-\l)^{1/4}\,\pFq{2}{1}{\frac 14&\frac 14}{&1}{-\l}.$$
We know $\G(\frac 14)\G(\frac 34)=\sqrt{2}\pi$
by the reflection formula, Theorem \ref{thm:reflect}. One can compute  periods of the second kind on $E_\l$ from   $\frac{\text{d} }{\text{d}\l}p(\l)$. Together, differentials of the first and second kind for $E_\l$ form a 2-dimensional vector space $V(\l)$.  When $E_\l$ has CM, its endomorphism ring $R:=R(\l)$,  which induces an action  on the space $V(\l)$,   is larger than $\Z$.   In particular, $\omega_\l$ is an eigenfunction of $R$. Let $\eta_\l$ be a differential of the second kind that is also
 an eigenfunction of $R$.
By Chowla-Selberg,  for any cycle $\gamma_\l$ of $E_\l$, $$\int_{\gamma_\l} \omega_\l\cdot \int_{\gamma_\l} \eta_\l\sim \pi,$$
where $\sim$ means equality up to a multiple in $\ol{\Q}$. Picking  $\gamma_\l$ in the right way, one obtains a Ramanujan type formula for $1/\pi$ corresponding to the CM value $\l$. In this way, we obtain Ramanujan type formulas for $1/\pi$ using products of two distinct periods  that are determined by $R$.  In the spirit of the Clausen formula, the  product of two periods lies in the symmetric square of the period space for $E_\l$; more explicitly, it is a special element in the solution space of the $_3F_2$.

While periods of the second kind  obtained from derivation do not have exact finite field or Galois analogues, there are analogues of `eigenfunctions' of the endomorphism ring on the Galois side. For {{suitable}} $\l$ the elliptic curve $E_{\l}$ has CM by a field we denote $K_\l$.
%denote the CM field of $E_\l$.
Fix a prime $\ell$ and let $\rho_{\l,\ell}$ denote the family of 2-dimensional $\ell$-adic Galois representations  of $\text{Gal}(\ol \Q/\Q)$ constructed from the elliptic curve $E_\l$ tensoring with $\Q_\ell$. As $E_\l$ has CM,
$$\rho_{\l,\ell}|_{\text{Gal}(\ol \Q/K_\l)}\simeq\sigma_{1,\l}\oplus \sigma_{2,\l},$$
where   $\sigma_{1,\l}$ and $\sigma_{2,\l}$ are $1$-dimensional representations that are invariant under $R$.  Each of them corresponds to a Gr\"ossencharacter of $K_\l$, as we described explicitly for two cases above (Examples \ref{eg:7} and \ref{eg:8}). The product $\sigma_{1,\l}\sigma_{2,\l}$ lies in the  symmetric square of $\rho_{\l,\ell}|_{\text{Gal}(\ol \Q/K_\l)}$. The Clausen formula implies it is also a linear factor of the 3-dimensional Galois representation corresponding to $_3\mathbb P_2$. For the special cases of $\l$ (say $\l=-\frac 14, \frac 18$),  one can  describe the values of the $_3\mathbb P_2$ functions explicitly.\\

Using the above computation for $a(\l, q)$ and the Clausen formula, one has the following proposition based on Examples \ref{eg:7} and \ref{eg:8}.

\begin{Proposition}\label{prop:FF-Ramanujan}
 We have the following two period function evaluation formulas.
\begin{itemize}
\item[1.] For any prime power $q$ that is coprime to $6$, we have
$$
%   \phgthree {\phi}{\phi}{\phi}{\eps}{\eps}{1/4}
   {}_3\mathbb P_2
 \left[
   \begin{matrix}
    \phi & \phi & \phi \\
         & \eps & \eps \\
\end{matrix};1/4 \right]
   =\phi(-1)
   \begin{cases}
      q,& \mbox{ if } q\equiv -1 \pmod{3},\\
      J(\phi, \eta_3)^2+J(\phi, \ol\eta_3)^2+q, & \mbox{ if } q\equiv 1\pmod 3,
      \end{cases}
$$ 
where $\eta_3$ is a character of order 3.
\bigskip
\item[2.] For any prime power $q$ that is coprime to $4$, we have
$$
%  \phgthree {\phi}{\phi}{\phi}{\eps}{\eps}{-1/8}
   {}_3\mathbb P_2
 \left[
   \begin{matrix}
    \phi & \phi & \phi \\
         & \eps & \eps \\
\end{matrix};-1/8 \right]
  =\phi(-2)
    \begin{cases}
      q,& \mbox{ if } q\equiv -1 \pmod4,\\
      J(\phi, \eta_4)^2+J(\phi, \ol\eta_4)^2+q, & \mbox{ if } q\equiv 1\pmod4,
      \end{cases}
$$
where $\eta_4$ is a character of order 4.
\medskip
\end{itemize}
\end{Proposition}

 Here we remark that the values \bk $\phi(-1)$ and $\phi(-2)$ respectively appearing on the right hand sides of the above equalities can be seen from \eqref{eq:3F2-ap-square} and the corresponding CM structures. See also  Theorem 1 of \cite{WIN2}. Proposition \ref{prop:FF-Ramanujan}  yields the following corollary.

\begin{Corollary}When $\l=-1/4$ (resp. $\l=1/8$), $\sigma_{1,\l}\sigma_{2,\l}$ extends to a 1-dimensional representation of $G_\Q$. It  corresponds to the following Gr\"ossencharacter  of $\Q$:  $$\chi_{\frac 12,-1}\cdot \mathcal T \quad  \text{ (resp. } \,\, \chi_{\frac 12,-2}\cdot \mathcal T\text{)},$$ where $\mathcal T$ is defined in Example \ref{eg:Tate} and $\chi_{\frac 12,(\cdot)}$   is defined in \eqref{eq:chi,im}.
\end{Corollary}

This is compatible with the $p$-adic analogues of Ramanujan-type formulas for $1/\pi$.  For instance, the following supercongruence related to \eqref{eq:RR1/pi}  was conjectured by van Hamme \cite{vanHamme} and proved by  the second author in \cite{Long}.  For any prime $p>3$,

\begin{equation*}
\sum_{k=0}^{\frac{p-1}2}(6k+1)\(\frac{(\frac 12)_k}{k!}\)^3\(\frac 14\)^k\equiv \(\frac{-1}p\) \pmod{p^4}.
\end{equation*}

See \cite{WIN2} for a general result, and \cite{Long-Ramakrishna}  by the second and third authors, \cite{Swisher} by the fourth author, \cite{OZ} by Osburn and Zudilin  for some recent developments on Ramanujan-type supercongruences.
\bk

\section{Translation of Some Classical Results}\label{translations}

In this section, we use our primary method to translate several classical results to their finite field analogues, while also making note of information gleaned from the corresponding Galois interpretations as appropriate.  The results we translate include Kummer's 24 relations and the well-known Pfaff-Saalsch\"utz formula, which were considered previously by Greene \cite{Greene}. \color{black}  We additionally use our method to obtain a few finite field analogues of algebraic hypergeometric identities.

\subsection{Kummer's $24$ Relations}\label{ss:Kummer24}

In this section we address the relations between the period ${}_2\mathbb P_1$ hypergeometric functions corresponding to independent solutions of the hypergeometric differential equation, and Kummer's 24 relations.   Statements below without proof or further references are due to Greene, especially Theorem 4.4 in \cite{Greene}. See also \cite{Greene93} by Greene for an interpretation in terms of representations of $\text{SL}(2,q)$.

The following proposition corresponds to  the discussion in  \S \ref{ss:Kummer} for the classical setting.

\begin{Proposition}\label{prop: Greene's independent solutions}
For any characters $A$, $B$, $C \in \widehat{\F_q^\times}$, and $\l\in \F_q$, we have
\begin{eqnarray*}
\phgq ABC\l
    &=&ABC(-1) \overline C(\l)\, \phgq {\overline CB}{\overline CA}{\overline C}\l+\delta(\l)J(B,C\overline B),\\
%\phgq ABC\l
    &=&ABC(-1)\overline A(\l )\, \phgq {A}{\overline CA}{\overline BA}{1/\l}+\delta(\l)J(B,C\overline B),\\
%\phgq ABC\l
   & =& B(-1)\, \phgq AB{AB\overline C}{1-\l}.
\end{eqnarray*}
\end{Proposition}

\begin{Remark}In  the first two equalities, the extra delta term on the right is necessary as when $\l=0$ the left hand side is a non-zero Jacobi sum by definition, while the first term on the right hand side has value 0 by our convention. % Rewrite the last entry as $$\phgq AB{\ol C}{1-\l}  = B(-1)\, \phgq AB{AB C}{\l}.$$
In the last entry, by  replacing characters $A$, $B$, and $C$ with characters $AD$, $BD$, and $\ol C D$, repectively,  one gets $$\phgq {AD}{BD}{\ol C D}{1-\l}  = BD(-1)\, \phgq {AD}{BD}{ABCD}{\l}.$$
Taking  $\l=0$ and applying \eqref{eq:for6.7-1} to both sides we obtain
$$
  \frac{BD(-1)}{q-1}\sum_{\chi}J(AD\chi,\ol\chi)J(BD\chi,C\ol {D\chi})=BD(-1)J(BD,AC).
$$
By a variable change $\chi \mapsto \ol D\chi$ and the identity between the Gauss sums and Jacobi sums (\ref{JacobiGaussrelation}), we can deduce the identity of Helversen-Pasotto \cite{Helversen-Pasotto}, which says
that for any multiplicative characters $A,B,C,D$ of $\F_q$
 \begin{multline}\index{Helversen-Pasotto formula}
 \frac1{q-1}\sum_{\chi} g(A\chi)g(B\chi)g(C\ol \chi)g(D\ol \chi)=\\
 \frac{g(AC)g(AD)g(BC)g(BD)}{g(ABCD)}+q(q-1)AB(-1)\delta(ABCD).
 \end{multline}    A representation theoretic proof of this formula is given in \cite{L-SA} by Li and Soto-Andrade. Using the representation theoretic approach, Li further obtained the $p$-adic field version of this identity in \cite{Li-Barnes-II}.
\end{Remark}

\begin{Remark}\label{rem:6}Let $\l\in  \Q\setminus \{0,1\}$, and $a,b,c\in \Q$ such that $a$, $b$, $a-c$, $b-c$, $a+b-c \notin  \Z$.  Use $\sigma_{\l,1}$ (resp. $\sigma_{1-\l,2}$) to denote a  2-dimensional $\ell$-adic Galois representation corresponding to $\pPq{2}{1}{a&b}{&c}{\l}$ (resp. $\pPq{2}{1}{a&b}{&a+b-c}{1-\l}$) via Theorem \ref{thm:WIN3a}. The last equality in Proposition \ref{prop: Greene's independent solutions}  implies that if we use the  $\chi_{(\cdot),(\cdot)}$  notation in \eqref{eq:chi,im}, then
$$\phgq {\iota_\fp(a)}{\iota_\fp(b)}{\iota_\fp(c)}{\l; q(\fp)}
    =  \chi_{b,-1}(\fp)\cdot  \, \phgq  {\iota_\fp(a)}{\iota_\fp(b)}{\iota_\fp(a+b-c)}{1-\l;q(\fp)}.$$This means $ \sigma_{\l,1}$ is isomorphic to $\chi_{b,-1}\otimes \sigma_{1-\l,2}$ up to semisimplification.  When $\l=0$ or $1$, both  representations $\sigma_{\l,1}$ and $\sigma_{1-\l,2}$ have degree 1 instead of 2. Thus no extra delta terms will be needed in this case.  Many other formulas in this section have similar Galois interpretations and most of them require additional delta terms.

\end{Remark}

The next proposition corresponds to the Pfaff and Euler transformations   in \S \ref{ss:Kummer}.
%See, respectively,
%\eqref{eq:Pfaff-transf},
%{\color{red}\eqref{eq:Pfaff-transf} and \eqref{eq:Euler-transf}.}

\begin{Proposition}\label{prop: Greene's Euler and Pfaff}
For any characters $A$, $B$, $C \in \widehat{\F_q^\times}$, and $\l\in \F_q$, we have
\begin{align*}
\phgq ABC\l
&=\overline A(1-\l)\, \phgq A{\overline BC}C{\frac \l{\l-1}}+\delta(1-\l)J(B, C\ol{AB}), \\
%\phgq ABC\l
&=\overline B(1-\l)\, \phgq { C\overline A}BC{\frac \l{\l-1}}
+\delta(1-\l)J(B, C\ol{AB}),\\
%\phgq ABC\l
&=\overline {AB}C(1-\l)\, \phgq {\overline AC}{\overline BC}C{\l}+\delta(1-\l)J(B, C\ol{AB}).
\end{align*}
\end{Proposition}

The following relations between the normalized ${}_2\mathbb F_1$-hypergeometric functions follow immediately from  Propositions \ref{prop: Greene's independent solutions} and \ref{prop: Greene's Euler and Pfaff}.

\begin{Proposition}\label{prop: normalized independent solutions} For any characters $A$, $B$, $C \in \widehat{\F_q^\times}$, and $\l\in \F_q$, we have
\begin{eqnarray*}
\pFFq {2}{1} {A&B}{&C}{\l} &=& ABC(-1) \overline C(\l) \frac{J(\ol CA,\ol A)}{J(B,C\ol B)} \pFFq {2}{1} {\overline CB&\overline CA}{&\overline C}\l+\delta(\l), \\
% \pFFq 21 {A&B}{&C}\l
   & =& ABC(-1)\overline A(\l) \frac{J(\ol CA,\ol BC)}{J(B,C\ol B)}\pFFq 21{A&\overline CA} {&\overline BA}{1/\l}+\delta(\l),\\
%\pFFq 21 {A&B}{&C}\l
  &  =&\frac{J(B, C\ol{AB})}{J(B, C\ol B)}\pFFq 21{A&B}{&AB\overline C}{1-\l},\\
%\pFFq 21 {A&B}{&C}\l
   & =&\overline A(1-\l)\pFFq 21 {A&\overline BC}{&C}{\frac \l{\l-1}}+\delta(1-\l)\frac{J(B, C\ol{AB})}{J(B, C\ol B)},\\
%\pFFq 21 {A&B}{&C}\l
    &=&\overline B(1-\l)\pFFq 21{\overline AC&B}{&C}{\frac \l{\l-1}}+\delta(1-\l)\frac{J(B, C\ol{AB})}{J(B, C\ol B)},\\
%\pFFq 21 {A&B}{&C}\l
    &=&\overline {AB}C(1-\l)\pFFq 21{\overline AC &\overline BC}{&C}{\l}+\delta(1-\l)\frac{J(B, C\ol{AB})}{J(B, C\ol B)}.
\end{eqnarray*}
\end{Proposition}

% We obtain evaluation formulas for the ${}_2\mathbb F_1$-functions %depending on whether they are primitive or imprimitive.

Using Proposition \ref{prop: 2F1-imprimitive} and Lemma \ref{lem:1}, we obtain the following proposition, which gives the evaluations in the imprimitive  (see Definition \ref{def:3}) cases.

\begin{Proposition}\label{prop: 2FF1-imprimitive} Suppose $\l \neq 0$, and $A$, $B$, $C$ are nontrivial. Then,
\begin{eqnarray*}
\pFFq21 {\eps &B}{&C}{\l} &=&   1-\frac{\overline C(\l) \overline BC(\l-1)}{J (B, \overline BC)},\\
\pFFq21 {A&B}{&B}{\l}&=&\overline A(1-\l)-\overline B(\l)J(B,\overline A),    \\
\pFFq21 {A&B}{&A}{\l}&=&\overline B(1-\l)-\frac{\overline A(\l)}{J(\ol A,B)},\\
\pFFq21 {A&\eps}{&C}{\l}&=&
     1-\overline C(-\l)\overline AC(1-\l)J(C,\overline A)
     -\delta(1-\l)\delta(\overline AC)(q-1).
\end{eqnarray*}
\end{Proposition}
Next, we consider the primitive  (see Definition \ref{def:3}) cases.

\begin{Proposition}\label{prop: commute and conjugation}
If $A, B, C \in \widehat{\F_q^\times}$,  $A,B\neq \eps$ and $A,B\neq C$,  then  we have the following.
\begin{enumerate}
\item[(1)]  In general, we have
\begin{align*}
J(A,\overline AC)\cdot \phgq ABC\lambda &= J(B,\overline B C)\cdot \phgq BAC\lambda, \\
\pFFq{2}{1}{A&B}{&C}\lambda &=
\pFFq 21{B&A}{&C}\lambda;
\end{align*}
\item[(2)] For $\l\neq 0$, $1$, we have
%\begin{align*}
%\phgq ABC\lambda &=
%\overline C(\lambda)C\overline{AB}(\lambda-1)\frac{J(\ol A,\ol C A)}{J(\ol B, \ol CB)}\phgq {\overline A}{\overline B}{\overline C}\lambda, \\
% \pFFq{2}{1}{A&B}{&C}\lambda &=
%\overline C(\lambda)C\overline{AB}(\lambda-1)\frac{J(\ol A,\ol C A)}{J( B, \ol BC)}\pFFq{2}{1} {\overline A &\overline B}{&\overline C}\lambda.
%\end{align*}

\begin{align*}
\phgq ABC\lambda &=
\overline C(\lambda)C\overline{AB}(\lambda-1)\frac{J(B,C\ol B )}{J(A, C\ol A)}\phgq {\overline A}{\overline B}{\overline C}\lambda, \\
 \pFFq{2}{1}{A&B}{&C}\lambda &=
\overline C(\lambda)C\overline{AB}(\lambda-1)\frac{J(\ol B,\ol C B)}{J( A, C\ol A)}\pFFq{2}{1} {\overline A &\overline B}{&\overline C}\lambda.
\end{align*}

\end{enumerate}
\end{Proposition}
\begin{proof}
Since $A,B$ are not equal to $\eps, C$, part (1) follows from the definition of the $_2\mathbb P_1$ function (see \eqref{Eq:30}) \bk and the relations
$J(A,C\ol A)=A(-1)J(A,\ol C)$, and $g(\ol CA)J(A,\ol C)=g(A)g(\ol C)$.

To prove part (2), we first use the Kummer relations stated in \bk  Propositions \ref{prop: Greene's independent solutions}  and \ref{prop: Greene's Euler and Pfaff}. More precisely,  for $\lambda\neq0$, $1$,  we have
  \begin{align*}
     \phgq ABC\lambda&
     \overset{\text{Prop.} \ref{prop: Greene's independent solutions}, \text{part 3}}{=}
     B(-1)\phgq AB{AB\overline C}{1-\lambda}\\
             &
             \overset{\text{Prop.} \ref{prop: Greene's Euler and Pfaff}, \text{part 3}}{=}
             B(-1)\overline C(\lambda)\phgq {B\overline C}{A\overline C}{AB\overline C}{1-\lambda}\\
             &
     \overset{\text{Prop.} \ref{prop: Greene's independent solutions}, \text{part 3}}{=}
             \overline C(\lambda)ABC(-1)\phgq {B\overline C}{A\overline C}{\overline C}{ \l }\\
             &
     \overset{\text{Prop.} \ref{prop: Greene's Euler and Pfaff}, \text{part 3}}{=}
             \overline C(\lambda)\overline{AB}C(\lambda-1)\phgq{\overline B}{\overline A}{\overline C}{\lambda}.
  \end{align*}
Next by part (1), one has
$$
   \phgq ABC\lambda= \overline C(\lambda)\overline{AB}C(\lambda-1)\frac{J(\ol A,\ol C A)}{J(\ol B, \ol CB)}\phgq{\overline A}{\overline B}{\overline C}{\lambda}.
$$
\end{proof}
\begin{Remark} Using the Galois perspective,
one can interpret the above equalities  as in Remark \ref{rem:6}.  {Let $a,b,c\in \Q$, with $a$, $b$, $a-c$, $b-c \notin \Z$ and $\l\in  \Q$. Let $N$ be the least common denominator of $a,b$ and $c$.}
Then for each prime $\ell$,  by part (2) of the above proposition, the 2-dimensional $\ell$-adic Galois representation $\sigma_{\l,\ell}$ of 
$G_{\Q(\zeta_N)}$ associated with $\pPq{2}{1}{a&b}{&c}{\l}$ via Theorem \ref{thm:WIN3a} satisfies the following property: up to semisimplification,  $$\sigma_{\l,\ell} \cong \psi \otimes \ol \sigma_{\l,\ell},$$ where $\ol \sigma_{\l,\ell}$ is the complex conjugation of $\sigma_{\l,\ell}$ and $\psi$ is the linear representation of $G_K$ associated with the Gr\"ossencharacter $\chi_{-c,\l}\cdot \chi_{c-a-b, \l-1}\cdot \mathcal J_{(b,c-b)}/\mathcal J_{(a,c-a)}$ of $K$; see \eqref{eq:chi,im} and \eqref{J->mathcal J} for $\chi_{(\cdot),(\cdot)}$ and $\mathcal J_{(\cdot,\cdot)}$.
\end{Remark}

The next result corresponds to  the classical equation \eqref{eq:z->1/z} in \S \ref{ss:Kummer}.
\begin{Corollary}
Suppose  $A,B\neq \eps$, and $A, B\neq C$.
If $\l\neq 0$, we have
\begin{multline*}
2\cdot \phgq ABC{\l}= ABC(-1)\overline A(\l)\phgq A{A\overline C}{A\overline B}{\frac 1\l}\\
+ABC(-1)\overline B(\l)\frac{J(B,\overline BC)}{J(A,\overline AC)}\phgq B{B\overline C}{B\overline A}{\frac 1\l},
\end{multline*}
\begin{multline*}
2\cdot \pFFq{2}{1} {A&B}{&C}{\l}= ABC(-1)\overline A(\l)\frac{J(\ol CA,\ol BC)}{J(B,C\ol B)}\pFFq{2}1 {A& A\overline C}{&A\overline B}{\frac 1\l}\\
+ABC(-1)\overline B(\l)\frac{J(C\ol A,\overline CB)}{J(A,\overline AC)}\pFFq 21 {B&B\overline C}{&B\overline A}{\frac 1\l}.
\end{multline*}
\end{Corollary}
Note that the appearance of a factor of 2 on the left hand sides corresponds to the fact that the right hand  sides are the  traces of 4-dimensional Galois representations at Frobenius elements.
\begin{proof}
Using Proposition \ref{prop: commute and conjugation} part (1)
and then Proposition \ref{prop: Greene's independent solutions} part (2)
we obtain that if $\l\neq 0$,
\begin{multline*}
\phgq ABC\lambda=\frac{J(B,\overline B C)}{J(A,C\ol A)}\phgq BAC\lambda\\
= ABC(-1)\overline B(\l)\frac{J(B,\overline B C)}{J(A,C\ol A)}\phgq {B}{\overline CB}{\overline AB}{\frac 1\l}.
\end{multline*}
Combining this with the result obtained by using the second part of Proposition \ref{prop: Greene's independent solutions} directly,
we get the  desired relations.
\end{proof}

\subsection{A  Pfaff-Saalsch\"utz  evaluation formula}

In this section we review an analogue of the Pfaff-Saalsch\"utz formula obtained by Greene in \cite{Greene}, and provide geometric interpretations in terms of Galois representations.

To review how Greene obtained his analogue, we first relabel the Pfaff-Saalsch\"utz formula \eqref{eq:pf-s} as
\begin{equation}\label{PSrelabeled}
\pFq{3}{2}{a&b&-n}{&d&1+a+b-n-d}{1}=\frac{(d-a)_n(d-b)_n}{(d)_n( d -a-b)_n}.
\end{equation}
Next, take $A,B$ and $C$ to be $\ol A D,\ol B D$ and $D$ in the third identity of Prop. \ref{prop: Greene's Euler and Pfaff} to obtain
\begin{equation}\label{first_step}
{AB}\ol D(1-\l)\, \phgq ABD\l
= \phgq {\overline AD}{\overline BD}D{\l}
-\delta(1-\l)J(D\ol B,\ol A)
\end{equation}
where we use Lemma \ref{lem:1} to get the $\delta$-term.

{Now, expand the left side of \eqref{first_step} using the defining formula \eqref{n+1Pn}, invoking the rewritten third identity of Prop. \ref{prop: Greene's Euler and Pfaff} above, then expand the $_2{\mathbb P}_1$ on the right via its definition \eqref{Eq:30} to get a double sum. The $\delta$-term
becomes $BD(-1)J(D\ol B,\ol A)$ and upon interchanging the order of summation one realizes the main term as
$J(B,\overline{CB} D) J(C,A \ol D)$. Finally, an application of Lemma \ref{lem:1} to $J(B,\overline{CB} D)$ establishes the following analogue to the Pfaff-Saalsch\"utz formula \eqref{PSrelabeled},}

\begin{equation}\label{eq:pf-s-FF}\index{finite field analogues!Pfaff-Saalsch\"utz  evaluation formula}\index{Pfaff-Saalsch\"utz  evaluation formula!in terms of $_3\mathbb P_2$}
   {}_3\mathbb P_2
 \left[
   \begin{matrix}
    A & B & C \\
         & D & ABC\overline D \\
\end{matrix};1\right] =B(-1)J(C,A\overline D)J(B, C\overline D)-BD(-1)J(D\overline B,\overline A).
\end{equation} 
Equivalently,  for $a,b,c,d\in \Q$ with least common denominator $N$ and any unramified prime ideal $\fp$ of $\mathcal O_{\Q(\zeta_N)}$, \bk using $\mathcal J_{(a,b)}(\fp)=-\iota_\fp(a)\iota_\fp(b)(-1)J(\iota_\fp(a),\iota_\fp(b))$  as \eqref{J->mathcal J} in \S \ref{ss:Jacobi-Grosse} (See Definition \ref{def:iota} for  $\iota_\fp(\cdot)$), one can rewrite {the analogue \eqref{eq:pf-s-FF} as}
\begin{multline*}
   {}_3\mathbb P_2
 \left[
   \begin{matrix}
    \iota_\fp(a) & \iota_\fp(b) & \iota_\fp(c) \\
         & \iota_\fp(d) & \iota_\fp(a+b+c-d) \\
\end{matrix};1\right] \\ =\chi_{a,-1}(\fp)\left [ \mathcal J_{(c,a-d)}(\fp)\mathcal J_{(b,c-d)}(\fp)+\mathcal J_{(d-b,-a)}(\fp)\right ].
\end{multline*}
%This means the Galois representation corresponding to  the  $_3\mathbb P_2$ function evaluated at 1 is a direct sum of two Gr\"ossencharacters.

In terms of Jacobi sums, \eqref{eq:pf-s-FF} can be written as \index{finite field analogues!Pfaff-Saalsch\"utz  evaluation formula}\index{Pfaff-Saalsch\"utz  evaluation formula!in terms of Jacobi sums}
\begin{multline}\label{eq:P.S.-FF}
\frac1{q-1}\sum_{\chi \in \widehat{\F_q^\times}} C\chiup(-1)J(A\chiup,\overline \chiup)J(B\chiup,\overline{D\chiup})J(C\chiup,D\overline{ABC\chiup})\\=J(C,A\overline D)J(B,C\overline D)-D(-1)J(D\overline B,\overline A).
\end{multline}
\begin{Remark}\label{rem:3}  We note that  from the Galois perspective, the $_3\mathbb P_2$ in \eqref{eq:pf-s-FF}  corresponds to a 2-dimensional Galois representation that can be described using two Gr\"ossencharacters.  Furthermore,  the additional term $-BD(-1)J(D\ol B, \ol A)$ is due to the additional term involving $\delta(1-\l)$ appearing in the third identity in Proposition \ref{prop: Greene's Euler and Pfaff}, together with formula \eqref{eq:delta}. \end{Remark}
For further examples of evaluation formulas over finite fields, see \cite{Evans-Greene2, Greene, Greene-Stanton, McCarthy}.

\subsection{A few analogues of algebraic hypergeometric formulas}\label{Sec:analogues}
In this section, we will give finite field analogues of \eqref{eq:Slater} and the following identity \cite[(1.5.19)]{Slater}
\begin{equation}\label{eval2}
\pFq{2}{1}{a&a+\frac 12}{&\frac12}z=\frac{1}2\( (1+\sqrt{z})^{-2a}+(1-\sqrt{z})^{-2a} \),
\end{equation}
which generalize a few recent results of Tu and Yang in \cite{Tu-Yang2} that are proved using  quotients of Fermat curves.

In the complex setting, to equate two formal power series, it suffices to compare  the coefficients of each. For a proof of  the identity \eqref{eval2} we note that both sides are  functions of $z$. The coefficient of $z^n$ on the right hand side is $\binom{-2a}{2n}$,  while on the left it is
 $$\frac{(a)_n(a+1/2)_n}{(1/2)_n(1)_n}\overset{\text{Thm. \ref{thm:double}}}=\frac{(2a)_{2n}}{(1)_{2n}}=\binom{-2a}{2n}.$$%changed from (2)_{2n} to (2a)_{2n}, 9/24/22
The evaluation identity  \eqref{eq:Slater}  mentioned in the introduction can be proved similarly.

In the finite field setting, one can obtain identities in a similar way  using \eqref{eq:f-f_chi} which we recall says that  any function $f(x):\F_q \rightarrow \C$ can be written uniquely in the form

$$
 f(x)=\delta(x)f(0)+\sum_{\chi \in \widehat{\F_q^\times}}f_\chi\cdot \chiup(x).
$$
We now state our finite field analogues of  \eqref{eval2} and \eqref{eq:Slater}.

\begin{Theorem}\label{thm:FF-Dihedral}
Let $q$ be an odd prime power, $z\in \F_q^\times$, and $A\in \widehat{\F_q^\times}$ have order larger than 2. Then
\begin{eqnarray*}
\pFFq{2}{1}{A&A\phi}{&\phi}{z}
&=& \( \frac{1+\phi(z)}2\) \(\overline A^2(1+\sqrt{z})+\overline A^2(1-\sqrt{z})\), \\
\pFFq{2}1{A&A\phi}{&A^2}{z}&=&\( \frac{1+\phi(1-z)}2\) \(\overline A^2\(\frac{1+\sqrt{1-z}}2\)+\overline A^2\(\frac{1-\sqrt{1-z}}2\)\).
\end{eqnarray*}
\end{Theorem}
\begin{Remark}We first note that the above formulas are well-defined. When $z\neq 0$ does not have a square root in $\F_q$,  then  $\(\frac{1+\phi(z)}2\)=\frac{1-1}2=0$. When $z\neq 0$ has a square root,  we have  $\( \frac{1+\phi(z)}2\)=1$ and the right hand side is a sum of two characters.  
\end{Remark}
\begin{Remark}
We now interpret the first identity in terms of  a global Galois perspective. Assume $z\in  \Q$ and $a=\frac mN$ for some integer $m$ coprime to $N$. Let $\sigma_{z,\ell}$ be the 2-dimensional Galois representation  associated with $\pFq{2}{1}{a&a+\frac12}{&\frac 12}{z}$  as in Theorem \ref{thm:WIN3a}.
If $\sqrt z\notin \Q(\zeta_N)$, then for any unramified prime ideal $\fp$
of ${\mathcal O}_{\Q(\zeta_N)}$ that is inert in $\Q({\sqrt z},\zeta_N)$,  $\text{Tr}\sigma_{z,\ell} (\text{Frob}_\fp)=0$. This means  $\sigma_{z,\ell}$ is  induced from a finite order character of $G_{\Q(\sqrt z, \zeta_N)}$ which is an index-2 subgroup of $G_{\Q(\zeta_N)}$.  There is a similar interpretation for the second result.
\end{Remark}
\begin{proof}
By the first equality of Proposition \ref{prop: normalized independent solutions}, we have
\begin{equation*}
\pFFq{2}{1}{A&A\phi}{&\phi}{z}=\phi(z)\pFFq{2}{1}{A&A\phi}{&\phi}{z}+\delta(z).
\end{equation*}
Hence, if $z$ is not a square in $\F_q^\times$, the evaluation is equal to zero.

  We now suppose $z=a^2$ for some $a \in \F_q^\times$. Then one has
  \begin{align*}
    \pFFq{2}{1}{A&A\phi}{&\phi}{z}
    \overset{\eqref{eq:for6.7-1}}{=}
    & \frac{A\phi(-1)}{(q-1)J(A\phi,\overline A)}\sum_{\chi\in \fqhat} J(A\chiup,\overline \chiup)J(A\phi\chiup,\phi\overline \chiup)\chiup(a^2)\\
    \overset{\eqref{JacobiGaussrelation}}{=}
    &\frac{\phi A(-1)}{(q-1)J(A\phi,\overline A)}\sum_\chi\frac{g(A\chiup)g(\overline \chiup)}{g(A)}\frac{g(A\phi\chiup)g(\phi\overline{\chiup})}{g(A)} \chiup(a^2)\\
     \overset{\eqref{Hasse-Dav-special}}{=}
     &\frac{\phi(-1)}{q-1}\sum_\chi \frac{\overline{A\chiup}(4)g(A^2\chiup^2)g(\phi)^2g(\overline\chiup^2)\chiup(4)}{q \overline A(4)g(A^2)}\chiup(a^2)\\
     \overset{\eqref{Gauss1} \text{ and } \eqref{JacobiGaussrelation}}{=}
     &\frac1{q-1}\sum_\chi\(J(A^2\chiup^2,\overline\chiup^2)-(q-1)\delta(A^2)\)\chiup(a^2)\\
    =&\overline A^2(1+a)+\overline A^2(1-a).
  \end{align*}
   We can use {Proposition \ref{prop: normalized independent solutions}} and the duplication formula \eqref{Hasse-Dav-special} to deduce that
\begin{align*}
     \pFFq{2}1{A&A\phi}{&A^2}{z}&=\frac{J(\phi,\phi  A)}{J(\phi A,\phi A)}\pFFq{2}{1}{A&A\phi}{&\phi}{1-z}\\&=A(4)\pFFq{2}{1}{A&A\phi}{&\phi}{1-z},
\end{align*}
which leads to the desired second statement.
\end{proof}

Now  we recast some of the  previous  discussion in terms of representations. Recall equation \eqref{eq:Slater} mentioned in the introduction, which says
$$
\pFq{2}{1}{a-\frac 12&a}{&2a}z=\( \frac{1+\sqrt{1-z}}2\)^{1-2a}.
$$
Writing  the rational number $1-2a$ in reduced form, let $N$ be its denominator. We first notice that when $z$ is viewed as an indeterminate, the Galois group $G$ of  the Galois closure of  $\ol \Q\(\( \frac{1+\sqrt{1-z}}2\)^{1-2a}\)$ over $\ol \Q(z)$ is a Dihedral group. This is compatible with the corresponding monodromy group, which can be computed from the method explained in \S \ref{monodromy}.

As a consequence of the second part of Theorem \ref{thm:FF-Dihedral}, we have the following corollary.
\begin{Corollary}\label{Cor: product of 2F1}
Let $A,B \in \widehat {\F_q^\times}$  such that $A,B, AB, A\ol B$ have orders larger than 2. Then,
\begin{multline*}
  \pFFq{2}{1}{A&\phi{A}}{& A^2}{z} \pFFq{2}{1}{B&\phi B}{&B^2}{z}\\
  =\pFFq{2}{1}{AB& \phi AB}{&(AB)^2}{z} +\ol B^2\(\frac z4\)\pFFq{2}{1}{A\ol B & \phi A\ol B}{&(A\ol B)^2}{z}-\delta(1-z)AB(4),
\end{multline*}
\begin{multline*}
  \pFFq{2}{1}{A&\phi{A}}{& \phi}{z} \pFFq{2}{1}{B& \phi B}{&\phi}{z}\\
  =\pFFq{2}{1}{AB& \phi AB}{&\phi}{z} +\ol B^2\(1-z\)\pFFq{2}{1}{A\ol B & \phi A\ol B}{&\phi}{z}-\delta(z).
\end{multline*}
\end{Corollary}
{Geometrically, this means the tensor product of two 2-dimensional representations of a dihedral group corresponding to the left side above is isomorphic to the direct sum of the two 2-dimensional representations corresponding to the terms on the right side of the above identities.}
In other words, it describes a fusion rule for 2-dimensional representations of Dihedral groups.

\begin{proof}
 Note that  $\(\frac{1+\sqrt{1-z}}2\)\( \frac{1-\sqrt{1-z}}2 \)=\frac{z}4$ and
 $$
   \(\frac{1+\phi(1-z)}2\)^2=\(1-\frac{\delta(1-z)}2\)\(\frac{1+\phi(1-z)}2\).
 $$
For any element $z\neq 1 \in \F_q$, we have
  \begin{multline*}
  \pFFq{2}{1}{A&\phi{A}}{& A^2}{z}\pFFq{2}{1}{B&\phi B}{&B^2}{z}
  =\\ \(\frac{1+\phi(1-z)}2\) \cdot\(\ol{A}^2\(\frac{1+\sqrt{1-z}}2\)+\ol{A}^2\(\frac{1-\sqrt{1-z}}2\)\)
  \\ \cdot\(\ol{B}^2\(\frac{1+\sqrt{1-z}}2\)+\ol B^2\(\frac{1-\sqrt{1-z}}2\)\)\\
  =\(\frac{1+\phi(1-z)}2\)\(\ol A^2 \ol B^2\(\frac{1+\sqrt{1-z}}2\)+\ol A^2 \ol B^2\(\frac{1-\sqrt{1-z}}2\)\)\\
  +\ol B^2\(\frac z4\)\(\frac{1+\phi(1-z)}2\)\(\ol A^2 B^2\(\frac{1+\sqrt{1-z}}2\)+\ol A^2 B^2\(\frac{1-\sqrt{1-z}}2\)\)\\
  =\pFFq{2}{1}{AB& \phi AB}{&(AB)^2}{z} +\ol B^2\(\frac z4\)\pFFq{2}{1}{A\ol B& \phi A\ol B}{&(A\ol B)^2}{z}.
\end{multline*}
{When $z=1$, 
$$ \pFFq{2}{1}{\chi&\phi{\chi}}{& \chi^2}{1}=\frac12\cdot 2 \ol\chi^2\(\frac12\)=\chi(4),$$
 for any character $\chi$. Therefore, 
 \begin{multline*}
  \pFFq{2}{1}{A&\phi{A}}{& A^2}{1}\pFFq{2}{1}{B&\phi B}{&B^2}{1}=A(4)B(4)\\
  =\frac 12\(\pFFq{2}{1}{AB& \phi AB}{&(AB)^2}{1} +\ol B^2\(\frac 14\)\pFFq{2}{1}{A\ol B& \phi A\ol B}{&(A\ol B)^2}{1}\).
\end{multline*}
The first statement follows. 

The second statement follows in a way similar to the first,
using the first part of Theorem \ref{thm:FF-Dihedral} or the transformation 
$$
  \pFFq{2}{1}{A& \phi A}{&A^2}{z}=A(4)\pFFq{2}{1}{A& \phi AB}{&\phi}{1-z}
$$
used in the proof of Theorem \ref{thm:FF-Dihedral}.
}
\end{proof}

Additionally, consider Slater's equation (1.5.21) in \cite{Slater},
\begin{equation*}\label{eq:Slater-1.5.21}
\pFq{2}{1}{2a&a+1}{&a}{z}= \frac{1+z}{(1-z)^{2a+1}}.
\end{equation*}
We have the following   finite field analogue: for $A\neq \eps$,
$$
\pFFq{2}{1}{A^2& A}{&A}{z}={\overline A^2(1-z)}-\overline A(z)J(A,\overline A^2).
$$
This follows from the second part of Proposition \ref{prop: 2FF1-imprimitive}.

\begin{Remark}The  two formulas stated in Theorem \ref{thm:FF-Dihedral} are direct consequences of the duplication formula  in Theorem \ref{thm:double}. By Theorem \ref{thm:multiplication}, we can generalize \eqref{eval2} using the same argument  to get the following formula representing the $m$th multiplication formula
\begin{equation}\label{eq:mth}
\pFq{m}{m-1}{a&a+\frac 1m&\cdots& a+\frac{m-1}m}{&\frac 1m&\cdots&\frac{m-1}m}{z}=\frac{1}m \(\sum_{i=1}^m (1-\zeta_m^i \sqrt[m]{z})^{-ma}\).
\end{equation}   Geometrically, the monodromy group for the degree-$m$ differential equation satisfied by the function on the left is a finite group (see the criterion in \cite{BH}), which is isomorphic to the Galois group of $\ol \Q\( (1+\sqrt[m]{z})^{-ma}\)$ over the function field $\ol \Q(z)$. When $m=2$ it is a Dihedral group. In general, it is a cyclic group extended by $\Z/m\Z$. Its finite field analogue is of the following form: let $q\equiv 1\pmod m$ and $\eta_m$ be a primitive order $m$ character, then for $A\in \widehat{\F_q^\times}$ such that $A^m\neq \eps$,
\begin{multline}\label{eq:mth-mul}
\pFFq{m}{m-1}{A&A\eta_m&\cdots& A\eta_m^{m-1}}{&\eta_m&\cdots&\eta_m^{m-1}}{z} \\
=\frac 1m \(1+\sum_{i=1}^{m-1} \eta_m(z)\) \sum_{i=1}^m \ol{A}^m\(1-\zeta_m^i \sqrt[m]{z}\).
\end{multline}
Equation (\eqref{eq:mth-mul} follows from the definition of $_m\F_{m-1}$ and the multiplication formula for Jacobi sums \eqref{eq:mul-Jacobi}.  \end{Remark}

We conclude this section by continuing a discussion about $g(t)^2=\pFq{2}{1}{\frac 14&\frac 34}{&\frac 23}{t} ^2$ in Example \ref{eg:(2,3,3)}.

\begin{Example}\label{eg:18}
By the Clausen formula \eqref{eq:Clausen},
\begin{eqnarray*}
g(t)^2&\overset{\eqref{eq:Pfaff-transf}}=& (1-t)^{-\frac 12} \pFq{2}{1}{\frac 14&-\frac 1{12}}{&\frac 23}{\frac t{t-1}}^2\\
&\overset{\eqref{eq:Clausen}}=& (1-t)^{-\frac 12}\pFq{3}{2}{-\frac 16&\frac 16&\frac 12}{&\frac 13&\frac 23}{\frac t{t-1}}\\
&\overset{\eqref{eq:mth}}=& \frac{(1-t)^{-\frac 12}}{3}\(\sum_{i=1}^3\sqrt{1- \zeta_3^i \sqrt[3]{\frac t{t-1}}}\).
\end{eqnarray*}
For the relation between $f(t)^2=\pFq{2}{1}{\frac 14&\frac 34}{&\frac 43}{t} ^2$ and $g(t)^2$, if we let  $x:=x(t)=\sqrt{3+6g(t)^2}$, then
  \begin{align*}
  f(t)^2 & =-8-\frac{4x(x+1)(x-3)}{3}\frac{t-1}t\\
   & = 8g(t)^2(2-\sqrt{3+6g(t)^2})\frac{t-1}t-\frac 8t.
 \end{align*}

% $$
%  u(t)= \frac{(1-t)^{\frac 12}}{3}\(\sum_{i=1}^3\sqrt{1- \zeta_3^i %\sqrt[3]{t}\),
% $$
% and
% $$
%   v=\frac{8u(2+3u-\sqrt{3+6u})}{1+6u+9u^2}.
% $$
% Set . Then we have

% [If we set  $u(t)=g(t/(t-1))^2$, $v(t)=\pFq{2}{1}{\frac 14&\frac 34}{&\frac 43}{t/(t-1)} ^2$, then we have $v=8u(2-%x)/t-8(t-%1)/t$. ]

%If we let $t=\frac{\l}{\l-1}$, then
%$$g\(\frac{\l}{\l-1}\)^2=\text[details to be filled.]$$

Now we consider the $\F_q$ analogue with argument $t=\frac{\l}{\l-1}$ so that the results are easier to state. Let $q\equiv 1\pmod {12}$ be a prime power and $\eta$ be a primitive order 12 character on $\F_q^\times$. Then the finite field versions of  $f\(\frac{\l}{\l-1}\)^2$ and $g\(\frac{\l}{\l-1}\)^2$ can be stated as follows: for $\l\neq0$, $1$,
\begin{multline}\label{FF-g^2}
\eta^8(\l) \pFFq 21{\eta^3&\ol\eta^3}{&\eta^4}{\frac \l{\l-1}}^2=\pFFq 21{\eta^3&\ol\eta^3}{&\ol\eta^4}{\frac \l{\l-1}}^2\\
  =
  \begin{cases}
     1+\displaystyle\sum_{0\leq i<j\leq 2} \eta^6\((1-\zeta_3^ia)(1-\zeta_3^ja)\), &\,\mbox{ if } \l =a^3\\
     \eta^4(\l), &\,\mbox{ otherwise}.
  \end{cases}
\end{multline}

This is obtained by using the following finite field Pfaff transformation given in item 4 of Proposition \ref{prop: normalized independent solutions}, 

$$
  \pFFq 21{\eta^3&\ol\eta^3}{&\ol\eta^4}{\frac \l{\l-1}}^2
  =\eta^6(1-\l)\pFFq 21{\eta^3&\ol\eta}{&\ol\eta^4}{\l}^2,
$$the finite field Clausen formula (Theorem \ref{thm:FF-Clausen} with $S=\eta^3,C=\ol \eta^4$)

$$
  \pFFq 21{\eta^3&\ol\eta}{&\ol\eta^4}{\l}^2
  =\pFFq 32{\ol\eta^2&\eta^2&\eta^6}{&\eta^4&\ol\eta^4}{\l}+\eta^6(1-\l)\eta^4(\l), \quad \l\neq 1,
$$ equation \eqref{eq:mth-mul} which gives that

$$
  \pFFq 32{\ol\eta^2&\eta^2&\eta^6}{&\eta^4&\ol\eta^4}{\l}
  =\frac 13\(1+\eta^4(\l)+\ol\eta^4(\l)\)\(\sum_{i=0}^2\eta^6\(1-\zeta_3^i\sqrt[3]\l\)\),
$$
the last identity of Proposition \ref{prop: commute and conjugation},

$$
  \pFFq 21{\eta^3&\ol\eta^3}{&\ol\eta^4}{\frac \l{\l-1}}
  = \eta^4(\l)\frac{J(\eta^5,\eta^3)}{J(\eta,\eta^3)}\pFFq 21{\eta^3&\ol\eta^3}{&\eta^4}{\frac \l{\l-1}},
$$
and the fact $J(\eta^5,\eta^3)=J(\eta,\eta^3)$ when $\eta$ has order $12$.
\end{Example}

\section{Quadratic or Higher Transformation Formulas}\label{higher}

In this section we {first consider} finite field analogues of some higher degree hypergeometric transformation formulas { which are} related to elliptic curves. { These} highlight the role geometric correspondences like isogeny and isomorphism play in some transformation formulas.

Next, we obtain several finite field analogues of classical formulas satisfying the ($\ast$) condition { given on page 1} by using our main technique.  Our first example is a quadratic formula, which demonstrates that our technique has the capacity to produce analogues that are satisfied by all values in $\F_q$.  We prove analogues of the Bailey cubic $_3F_2$ formulas and an analogue of a formula by Andrews and Stanton.

We then use a different approach to obtain a finite field analogue of a cubic formula of Gessel and Stanton. { As a corollary of this cubic formula, Gessel and Stanton obtained an evaluation formula, with a proof that cannot be translated directly.  However, using the Galois perspective we can predict a finite field analogue which we then prove using a different approach.}  \bk This application illustrates that the Galois perspective is helpful in a greater context. \bk

%we can predict beyond what we can see with the ($\ast$) condition and then prove these predictions using other approaches, which are less systematic than our main technique.

\subsection{Some results related to elliptic curves}
In \S \ref{ss:Kummer24}, we discussed the finite field analogues of Kummer's 24 relations, which are  between $_2\mathbb P_1$ (resp. $_2\F_1$) functions linked via linear fractional transformations. These are the cases that can be deduced from Greene's finite field version of the Lagrange theorem {stated in Theorem \ref{FFlagrange}}.  However, Greene's  theorem does not apply to  higher degree transformation formulas.

Borwein and Borwein (\cite{Borweins})
 proved
%{ they give the following cubic formula which states that }
that for real $z\in (0,1)$,
\begin{equation*}
\pFq{2}{1}{\frac 13&\frac 23}{&1}{1-z^3}=\frac{3}{1+2z} \, \pFq{2}{1}{\frac 13&\frac 23}{&1}{\(\frac{1-z}{1+2z}\)^3}.
\end{equation*}Geometrically, this corresponds to the fact that the two elliptic curves
$$
   y^2+xy+\frac{1-z^3}{27}y=x^3,\qquad
   y^2+xy+\frac 1{27}\frac{(1-z)^3}{(1+2z)^3}y=x^3
$$
are 3-isogenous over $\Q(z,\zeta_3)$, which can be verified  using the degree-3 modular equation satisfied by their corresponding $j$-invariants \cite{drew}
\begin{multline*}
X^4+Y^4+36864000(X^3+Y^3)+452984832000000(X^2+Y^2) \\ 
+1855425871872000000000(X+Y) -X^3Y^3+2587918086X^2Y^2 \\
-770845966336000000XY + 2332(X^3Y^2+X^2Y^3)\\
-1069956(X^3Y+XY^3)+8900222976000(X^2Y+XY^2).
\end{multline*}   
Hence, when $q\equiv 1\pmod 3$ and $z$ can be embedded in $\F_q$, the local zeta functions of these curves over $\F_q$ are the same, i.e.  if $\eta_3$ is a primitive cubic character in $\widehat{\F_q^\times}$, then
\begin{equation}\label{eq:BB-FF}
\pPPq{2}{1}{\eta_3&\eta_3^2}{&\eps}{1-z^3}=\pPPq{2}{1}{\eta_3&\eta_3^2}{&\eps}{\(\frac{1-z}{1+2z}\)^3}.
\end{equation}When one replaces $_2\mathbb P_1$ by $_2\F_1$, the above formula still holds {by our definition of $_2\F_1$ \eqref{normalized2FF1}}  as the normalizing Jacobi sums are the same.

Similarly, we consider the relation between the universal elliptic curves
$$E_t:\, y^2=4x^3-\frac{27x}{1-t}-\frac{27}{1-t},$$
and the Legendre curves
$$L_\l:\, y^2=x(1-x)(1-\l x).$$
 The $j$-invariant of $E_t$ is $\frac{1728}t$ and  the $j$-invariant  of $L_\l$ is $\frac{256(\l^2-\l+1)^3}{\l^2(\l-1)^2}$.
 Stienstra and Beukers
(\cite[Theorem 1.5]{S-B}) proved
\begin{equation}\label{eqn:S-B}
   \pFq 21{\frac 1{12}&\frac5{12}}{&1}{\frac{27\l^2(\l-1)^2}{4(\l^2-\l+1)^3}}=(1-\l+\l^2)^{1/4} \pFq 21{\frac 12&\frac12}{&1}{\l}.
\end{equation}
%Here the $j$-invariant of $E_t$ is $\frac{1728}t$ and  the $j$-invariant  of $L_\l$ is $\frac{256(\l^2-\l+1)^3}{\l^2(\l-1)^2}$.
{Geometrically,  
$L_\l$ is the pullback of the universal elliptic curve along the natural map of modular curves $Y(2) \rightarrow Y(1)$.}

%equation \eqref{eqn:S-B} relates two isomorphic elliptic curves with the same $j$-invariants. 
To see the finite field analogue of \eqref{eqn:S-B}, we have the following relation
by the work of the first author \cite[Theorem 1.2]{Fuselier}.
\begin{Theorem}
For $p\equiv 1 \pmod{12}$, $\l\in\F_p$ with $\l\neq - 1$, $2$, $1/2$ and $\l^2-\l+1\neq 0$, we have
$$
  \pPPq21{\eta_{12}&\eta_{12}^5}{&\eps}{\frac{27\l^2(\l-1)^2}{4(\l^2-\l+1)^3}}=\eta_{12}(-1)\eta_{12}^3(1-\l+\l^2) \pPPq21{\phi&\phi}{&\eps}{\l},
$$
where $\eta_{12}$ is a primitive character of order $12$ in $\widehat{\F_p^\times}$.    To write it in terms of $_2\F_1$,
\begin{equation}\label{eq:2F1forEC}
  \pFFq21{\eta_{12}&\eta_{12}^5}{&\eps}{\frac{27\l^2(\l^2-1)^2}{4(\l^2-\l+1)^3}}=\eta_{12}^3(1-\l+\l^2) \pFFq21{\phi&\phi}{&\eps}{\l}.
\end{equation}
\end{Theorem}

\begin{proof}
For any fixed $\l\in\Q\setminus \{0,1\}$, the elliptic curve $L_\l$ is isomorphic to $E_t$ with $t=\frac{27\l^2(\l-1)^2}{4(\l^2-\l+1)^3}$ over $\Q$. To be more explicit, when $t=\frac{27\l^2(\l-1)^2}{4(\l^2-\l+1)^3}$, the elliptic curve $E_t$ can be written as
\begin{multline*}
  \frac{(\l-2)^2(2\l-1)^2(\l+1)^2}4y^2=\\
  \((2\l^2+\l-1)x+3(\l^2-\l+1)\)\((\l^2-\l-2)x-3(\l^2-\l+1)\) \\
  \cdot\((2\l^2-5\l+2)x+3(\l^2-\l+1)\).
\end{multline*}
It is $\Q$-isomorphic to
$$
  Y^2=-\frac{\l^2-\l+1}{(\l-2)(\l+1)(2\l-1)}X(X-1)(X-(1-\l))
$$
by  the change of variables
$$
  (X,Y)=\(-\frac{3(\l^2-\l+1)}{2\l^2+\l-1}\(\frac {3x}{\l-2}+1\), \frac {(\l-2)(2\l-1)(\l+1)}{2^4(\l^2-\l+1)}y\).
$$

For any prime $p\equiv 1 \pmod{12}$, let $a_p(t)$ denote the trace of the Frobenius endomorphism on $E_t$.  According to the work of  the first author \cite{Fuselier}, when $t\in\F_p^\times$ and $t\neq 1$, i.e. $\l\in\F_p^\times$ with $\l\neq \pm 1$, $2$, $1/2$,  we have
$$
  \pPPq21{\eta_{12}&\eta_{12}^5}{&\eps}{t}=-\phi(2)\eta_{12}(-1)\ol\eta_{12}^3(1-t)a_p(t).
$$
By the above discussion
$$
a_p(t)=-\phi((1-\l+\l^2)(\l-2)(2\l-1))\pPPq{2}{1}{\phi&\phi}{&\eps}{\l}.
$$
It follows that when $\l\neq 0$, $\pm 1$, $2$, $1/2$ in $\F_p$,
$$
\pPPq21{\eta_{12}&\eta_{12}^5}{&\eps}{\frac{27\l^2(\l-1)^2}{4(\l^2-\l+1)^3}}=\eta_{12}(-1)\eta_{12}^3(1-\l+\l^2) \pPPq21{\phi&\phi}{&\eps}{\l}.
$$
One can check this identity also holds for $\l=0$, $1$.
We  deduce \eqref{eq:2F1forEC} by
noting that  $$J(\eta_{12}^5,\eta_{12}^7)=\sum_{x\in \F_q} \eta_{12}^5\(x\) \eta_{12}^7(1-x)=\sum_{x\in \F_q \backslash \{1\}} \eta_{12}^5\(\frac{x}{1-x}\)=-\eta_{12}^5(-1)$$ as stated in \eqref{eq:J(A,Abar)}. Similarly, $J(\phi,\phi)=-\phi(-1)$.
\end{proof}

%where $\eta_{12}$ is a primitive character of order $12$ in $\widehat{\F_p^\times}$.

We wish to keep in mind that geometrically, transformation formulas are sometimes related to `correspondences' like isogenies or isomorphisms, however in general the underlying geometric objects are complicated. More abstractly, {these transformation formulas} describe `correspondences' between two hypergeometric motives. {From this perspective,} finite field analogues of these transformation formulas describe relations between the Galois representations associated with the motives.  The two examples above demonstrate that when the involved monodromy groups are arithmetic and hence have moduli interpretations, one may obtain $\F_q$ transformation formulas geometrically. The techniques used later in this section are mainly based on character sums and they can handle general cases regardless of whether there exist moduli interpretations.

\subsection{A Kummer quadratic transformation formula}

In this section, we illustrate our dictionary method with a finite field analogue of Kummer's quadratic formula.  Recall from \S \ref{ss:3.2.6} the  Kummer quadratic transformation formula%, which was mentioned in \S \ref{ss:3.2.6}
\begin{equation}\label{eq:quadratic}
 (1-x)^{-c}\, \pFq{2}{1}{\frac{1+c}2-b&\frac{c}2}{&c-b+1}{\frac{-4x}{(1-x)^2}}=\pFq{2}{1}{b&c}{&c-b+1}{x}.
\end{equation}
 We now outline a proof  using the multiplication and reflection formulas, along with the Pfaff-Saalsch\"utz formula.  Observe that our proof is different from the one in \cite[pp. 125-126]{AAR}. We first recall an inversion formula \cite[(2.2.3.1)]{Slater}.
 \begin{equation}\label{eq:Slater-Inversion}
 (a)_{n-r}=(-1)^r\frac{(a)_n}{(1-a-n)_r}.
 \end{equation}
\begin{proof}
To begin, note that
\begin{multline}\label{Eq:40}\binom{-c-2k}{n-k}=(-1)^{n-k}\frac{(c+2k)_{n-k}}{(1)_{n-k}}\overset{\eqref{eq:Slater-Inversion}}=(-1)^{n-k}\frac{(c+2k)_n(-n)_k}{(1-c-2k-n)_k(1)_n}\\=(-1)^{n-k}\frac{\G(c+2k+n)\G(1-c-2k-n)}{\G(c+2k)\G(1-c-k-n)}\frac{(-n)_k}{n!}\\
\overset{\text{reflection}}=(-1)^n\frac{\G(c+2k+n)\G(c+k+n)}{\G(c+2k)\G(c+2k+n)}\frac{(-n)_k}{n!}=(-1)^n\frac{(c)_{n+k}(-n)_k}{(c)_{2k}n!} \\
=(-1)^n\frac{(c)_n(c+n)_k(-n)_k}{(c)_{2k}n!}\overset{\eqref{eq:double-rising}}=(-1)^n\frac{(c)_n(c+n)_k(-n)_k}{4^k(\frac c2)_{k}(\frac {c+1}2)_kn!}.
\end{multline}
The left hand side  of \eqref{eq:quadratic} can be expanded as
\begin{eqnarray*}
&&\sum_{k\ge 0}\frac{(\frac{1+c}2-b)_k(\frac{c}2)_k}{k!(c-b+1)_k}(-4x)^k(1-x)^{-c-2k}\\
&=&\sum_{k\ge 0}\frac{(\frac{1+c}2-b)_k(\frac{c}2)_k}{k!(c-b+1)_k}(-4x)^k\sum_{i\ge 0}\binom{-c-2k}{i}(-x)^i\\
&\overset{{n=k+i}}=&\sum_{k,n\ge 0}\frac{(\frac{1+c}2-b)_k(\frac{c}2)_k}{k!(c-b+1)_k}4^k\binom{-c-2k}{n-k}(-x)^n\\
&\overset{\eqref{Eq:40}}=&\sum_{k,n\ge 0}\frac{(\frac{1+c}2-b)_k(\frac{c}2)_k}{k!(c-b+1)_k}4^k\frac{(c)_n(c+n)_k(-n)_k}{4^k(\frac c2)_{k}(\frac {c+1}2)_kn!}x^n\\
&=&\sum_{n\ge 0}\frac{(c)_n}{n!}\, \pFq{3}{2}{\frac{1+c}2-b&c+n&-n}{&\frac{c+1}2&c-b+1}{1}x^n.
\end{eqnarray*}
By the Pfaff-Saalsch\"utz formula \eqref{eq:pf-s}, the above equals
$$
 \sum_{n\ge 0}\frac{(c)_n(b)_n(\frac{1-c}2-n)_n}{n!(\frac{1+c}2)_n(b-c-n)_n}x^n
\overset{\text{reflection}}=\sum_{n\ge 0}\frac{(b)_n(c)_n}{n!(c-b+1)_n}x^n.
$$
\end{proof}

Now we obtain  the corresponding  finite field analogue of \eqref{eq:quadratic},  starting with an  analogue of \eqref{Eq:40}. We {recall} that $\phi$ denotes the quadratic character,  and note that under the  assumptions below, $\frac{c}{2}$ in the classical setting corresponds to a character $D$ below.
\begin{Lemma}\label{lem:quad}
Let $D,K, \chiup \in \widehat{\F_q^\times}$ and set $C=D^2$.  Then,
\begin{multline}\label{eq:double}
J(\overline{CK^2},\overline\chiup K)=\chiup D\phi (-1) \frac{g(\overline \chiup K)g(CK\chiup)g(D)g(D\phi)}{K(4)g(DK)g(C)}\frac{g(\overline {DK}\phi) }q\\+\frac{(q-1)^2}q\delta(D \chiup\phi)\delta(DK\phi)+(q-1)\delta(DK).
\end{multline}
\end{Lemma}

\begin{proof}
\begin{align*}
J(\ol{CK^2},\ol \chiup K)
&\overset{\text{Lem.} \ref{lem:1}}=\chiup K(-1)J(\ol \chiup K, C\chiup K)\\
&=\chiup K(-1)\frac{g(\ol \chiup K)g(C\chiup K)}{g(CK^2)}\cdot \frac{g(D^2)}{g(C)}+(q-1)\delta(CK^2)\\
\overset{\text{duplication}}=& \chiup K(-1)\frac{g(\ol \chiup K)g(C\chiup K)g(D)g(D\phi)}{K(4)g(DK)g(DK\phi)g(C)}+(q-1)\delta(CK^2)\\
&\overset{\eqref{Gauss1a} }=\chiup K(-1)\frac{g(\ol \chiup K)g(C\chiup K)g(D)g(D\phi)}{K(4)g(DK)g(C)} \\
&\cdot \( \frac{DK\phi(-1)}{q}g(\ol{DK\phi})-\frac{q-1}{q}\delta(DK\phi)\) +(q-1)\delta(CK^2).
\end{align*}
Breaking this into three terms, the middle term is
\begin{multline*}
-\chiup K(-1)\frac{g(\ol \chiup K)g(C\chiup K)g(D)g(D\phi)}{K(4)g(DK)g(C)}\frac{(q-1)}q\delta(DK\phi)\\
\overset{\text{duplication}}=-\chiup K(-1)\frac{(q-1)}q g(\ol{\chiup D}\phi) g(\chiup D \phi)\delta(D\phi K)\\
\overset{\eqref{Gauss1}}=-(q-1)\delta(DK\phi)+\frac{(q-1)^2}q\delta(D\phi K)\delta(D\phi \chiup).
\end{multline*}  Recombining the three terms,
using the easily established identity
$\delta(R^2) - \delta(R\phi) =\delta(R)$,
gives the Lemma.
\end{proof}

\begin{Remark}\label{rem:4}Among the three terms on the right hand side of \eqref{eq:double}, the first  is the major term predicted by the classical case. The other two  correspond to
degenerate cases.
%and the other two terms {address } the possibility that $\ol{CK^2}$ and/or $\ol \chiup K$ is the trivial character.
\end{Remark}

\begin{Theorem}\label{thm:quad-2F1}\index{finite field analogues!Kummer quadratic transformation}
 Let $B, D \in \widehat{\F_q^\times}$, and set  $C=D^2$.  When $D\neq \phi$ and $B\neq D$, we have,  for all $x\in\fq $
\begin{multline*}
\overline C(1-x) \, \pFFq{2}{1}{D\phi\overline B&D}{&C\overline B}{\frac{-4x}{(1-x)^2}} \\
= \pFFq{2}{1}{B&C}{&C\ol{B}}{x} -\delta(1-x)\frac{J(C,\overline B^2)}{J(C,\overline B)}-\delta(1+x) \frac{J(\ol B,D\phi)}{J(C,\ol B)}.
\end{multline*}
\end{Theorem}

\begin{Remark}\label{rem:denom0}
In the above statement, we see the character $\overline{C}(1-x)$ multiplied by a hypergometric function being evaluated at a rational expression which is undefined when $x=1$. In this instance (and others similar to it), we take the convention that the value of the product is $0$ when $x=1$.
\end{Remark}

 We will use the following identities, which can be checked directly,  in the proof of Theorem \ref{thm:quad-2F1}.   When $B\neq C$, and $\phi D\ol B\neq \eps$,
\begin{equation}\label{eq:44}
  \overline D(4)\frac{J(\phi \ol{B},D)}{J(D,D\ol B)}=\frac{J(C,\overline B^2)}{J(C,\overline B)};\quad
  \frac{J(\ol B,D\phi)}{J(C,\ol B)}=\frac{J(C\ol B,D\phi)}{J(C,D\ol B \phi)}.
\end{equation}

\begin{proof}
When $x=0$, both sides are equal to 1 by  definition.  When $x=1$, both sides take value 0 as well.  Thus, we assume $x\neq 0,1$.

We first consider the case when $C=D^2\neq \eps$, and $B^2\neq C$.  We deal with this case by considering separately when $B=C$ and when $B\neq C$.  Suppose  $B \neq C$.
Taking  $A_1=D\phi \ol B$, $A_2=D$, $B_1=C\ol B$ and $K=\chiup$ in \eqref{eq:for6.7-1},
\begin{multline*}
\overline C(1-x) \, \pFFq{2}{1}{D\phi\overline B&D}{&C\overline B}{\frac{-4x}{(1-x)^2}}\\
=\frac1{(q-1)J(D,D\overline B)}\sum_{K\in \fqhat}D(-1)J(D\phi \ol{B}K,\ol{K})J(DK,B\ol{CK}) K(-4x)\ol{CK^2}(1-x)\\
=\frac1{(q-1)^2J(D,D\overline B)}\sum_{K\in \fqhat}D(-1)J(D\phi \ol{B}K,\ol{K})J(DK,B\ol{CK})\\
\cdot K(-4x) \sum_{\varphi\in \fqhat}J(\ol{CK^2},\ol{\varphi})\varphi(x).
\end{multline*}
Letting $\chiup:=K\varphi$, the above equals
\begin{multline}\label{Eq:56}
\frac1{(q-1)^2J(D,D\overline B)} \\
\cdot \sum_{K,\chiup\in \fqhat}D(-1)J(D\phi \ol{B}K,\ol{K})J(DK,B\ol{CK})K(-4)J(\ol{CK^2},\ol{\chiup}K)\chiup(x).
\end{multline}
By Lemma \ref{lem:quad}  and \eqref{eq:44}  we can rewrite this as
\begin{equation}\label{eqn:S-breakdown}
S_1+S_2-\ol{D}(4)\frac{J(\phi \ol{B},D)}{J(D,D\ol B)}\delta(1-x)
=S_1+S_2-\delta(1-x)\frac{J(C,\overline B^2)}{J(C,\overline B)},
\end{equation}
where the $S_i$  are obtained by applying \eqref{eq:double}.     Explicitly,
\begin{multline*}
S_1=\frac1{(q-1)^2J(D,D\overline B)}\sum_{K,\chiup\in \fqhat}D(-1)J(D\phi \ol{B}K,\ol{K})J(DK,B\ol{CK})K(-4) \\
\cdot \chiup D\phi (-1) \frac{g(\overline \chiup K)g(CK\chiup)g(D)g(D\phi)}{K(4)g(DK)g(C)}\frac{g(\overline {DK}\phi) }q\chiup(x);
\end{multline*} and
\begin{align*}
S_2=\frac{\ol D(4)J(\phi,\phi B\ol D)}{qJ(D,D\ol B)}\phi(-1)J(\ol B,\phi D)\phi\ol D(x).
\end{align*}

 We first analyze $S_1$.
\begin{multline*}
S_1=\frac{g(C\ol{B})}{(q-1)^2 g(D)g(D\overline B)}\sum_{K,\chiup\in \fqhat}\chiup \phi(-1) J(D\phi \ol{B}K,\ol{K})\frac{g(DK)g(B\ol{CK})}{g(B\ol{D})} \\
\cdot K(-4)
 \frac{g(\overline \chiup K)g(CK\chiup)g(D)g(D\phi)}{K(4)g(DK)g(C)}\frac{g(\overline {DK}\phi) }q\chiup(x)\\
\overset{\text{combine Gauss sums}}=\frac{BD(-1)g(C\ol{B})g(D\phi)}{(q-1)^2q^2g(C)}\sum_\chi g(B\ol{C\chiup})g(D\phi \chiup)\\
\cdot \(\sum_{K} K\chiup \phi(-1)J(D\phi \ol{B}K,\ol{K})\frac{g(B\ol{CK})g(\overline \chiup K)}{g(B\ol{C\chiup})}\frac{g(CK\chiup)g(\overline {DK}\phi)}{g(D\phi \chiup)}\)\chiup(x).
\end{multline*}
Letting $\Omega:=\frac{BD(-1)g(C\ol{B})g(D\phi)}{q^2g(C)}$,  we have
\begin{multline*}
S_1=\frac{\Omega}{(q-1)^2}\sum_{\chi\in \fqhat} g(B\ol{C\chiup})g(D\phi \chiup)\\
\cdot \(\sum_{K\in \fqhat} K\chiup \phi(-1)J(D\phi \ol{B}K,\ol{K})\frac{g(B\ol{CK})g(\overline \chiup K)}{g(B\ol{C\chiup})}\frac{g(CK\chiup)g(\overline {DK}\phi)}{g(D\phi \chiup)}\)\chiup(x)\\
\overset{\text{Jacobi sums}}=\frac{\Omega}{(q-1)^2}\sum_\chi\phi(-1)g(B\ol{C\chiup})g(D\phi \chiup)\chiup(x)\sum_{K} K\chiup (-1)J(D\phi \ol{B}K,\ol{K})\\
\cdot \left [J(B\ol{CK},\ol{\chiup}K)-(q-1)BCK(-1)\delta(B\ol{C\chiup})\right ] \\
\cdot \left [J(CK\chiup,\ol{DK}\phi)-(q-1)DK\phi(-1)\delta(D\phi \chiup)\right].
\end{multline*}
To continue we separate the delta terms,
\begin{multline*}
S_1=\frac{\Omega}{(q-1)}\sum_\chi\phi(-1)g(B\ol{C\chiup})g(D\phi \chiup)\chiup(x)\cdot\\
\sum_{K}\frac{ K\chiup (-1)}{q-1}J(D\phi \ol{B}K,\ol{K})
 J(B\ol{CK},\ol{\chiup}K) J(CK\chiup,\ol{DK}\phi)+ E_1+E_2\\
\overset{\eqref{eq:P.S.-FF}}=\frac{\Omega}{(q-1)}\sum_\chi  \phi(-1)g(B\ol{C\chiup})g(D\phi \chiup)\chiup(x)\\
\cdot (J(C\chiup,\ol{D}\phi)J(\overline \chiup,B\chiup)
-B(-1)J(C\ol{B}\chiup,\ol{D}\phi B))+ E_1 +E_2,
\end{multline*}
where
\begin{multline*}
E_1:=-\sum_{\chi\in \fqhat}\frac{\Omega\phi(-1)}{(q-1)}g(\eps)g(B\ol{D} \phi)\sum_{K\in \fqhat} J(D\phi \ol{B}K,\ol{K})J(BK,\ol{DK}\phi)\cdot \delta(B\ol{C\chiup}) \chiup(x)\\
\overset{\text{Gauss eval.}}=\sum_\chi\Omega\phi(-1)g(B\ol{D} \phi)J(B,\ol{B})\delta(B\ol{C\chiup}) \chiup(x)\\
= -\phi B(-1)\Omega g(B\ol{D} \phi){B\ol{C}}(x),
\end{multline*}
\begin{multline*}
E_2:=-\sum_{\chi\in \fqhat}\frac{\Omega\phi(-1)}{(q-1)}g(\eps)g(B\ol{D} \phi)\sum_{K\in \fqhat}J(D\phi \ol{B}K,\ol{K})J(B\overline {CK}, D\phi K )\delta(D\phi\chiup)\chiup(x)\\
\overset{\text{Gauss eval.}}=\sum_\chi\frac{g(C\ol{B})g(D\phi)}{q^2g(C)}\phi(-1)g(B\ol{D} \phi)J(D\phi, \ol{D}\phi)\delta(D\phi\chiup)\chiup(x)\\
=-D(-1) \Omega g(B\ol{D} \phi)\ol{D}\phi(x).
\end{multline*}

Now we continue our analysis of $S_1$.  We separate  $S_1-(E_1+E_2)$  into two terms.
The first term is
\begin{multline*}
\frac{\Omega}{(q-1)}\sum_{\chi\in \fqhat}\phi(-1)g(B\ol{C\chiup})g(D\phi \chiup)J(C\chiup,\ol{D}\phi)J(\overline \chiup,B\chiup)\chiup(x) \\
=\frac{\Omega}{(q-1)}\sum_\chi\phi(-1)g(B\ol{C\chiup})g(D\phi \chiup)\\
\cdot \(\frac{g(C\chiup)g(\ol{D}\phi)}{g(D\phi\chiup)}+(q-1)D\phi(-1)\delta(D\phi\chiup)\)J(\overline \chiup,B\chiup)\chiup(x).
\end{multline*}
By definition, this equals
\begin{multline*}
\pFFq{2}{1}{B&C}{&C\ol{B}}{x}-\sum_{\chi\in \widehat{\F_q^\times}}\Omega \, g(B\ol{D} \phi)D(-1)J(D\phi, B\ol{D}\phi)\delta(D\phi \chiup)\chiup(x)\\
=\pFFq{2}{1}{B&C}{&C\ol{B}}{x}-\Omega \,g(B\ol{D} \phi)\phi(-1)J(D\phi, \ol B)\ol{D}\phi (x)\\= \pFFq{2}{1}{B&C}{&C\ol{B}}{x}-S_2.
\end{multline*}

The  second  term {of $S_1-(E_1+E_2)$ } is
\begin{multline*}
-\frac{\Omega}{(q-1)}\sum_{\chi\in \fqhat}\phi(-1)g(B\ol{C\chiup})g(D\phi \chiup)B(-1)J(C\ol{B}\chiup,\ol{D}\phi B)\chiup(x)\\
=-\frac{\Omega}{(q-1)}\sum_\chi\phi(-1)g(B\ol{C\chiup})g(D\phi \chiup)B(-1)CB\chiup(-1)J(C\ol{B}\chiup,\ol{D\chiup}\phi )\chiup(x)\\
=-\frac{\Omega}{(q-1)}\sum_\chi\phi\chiup(-1)g(B\ol{C\chiup})g(D\phi \chiup)\frac{g(C\ol{B}\chiup)g(\ol{D\chiup}\phi )}{g(D\ol{B}\phi)}\chiup(x)\\
=-\frac{\Omega}{(q-1)}\sum_\chi\frac{\phi \chiup(-1)}{g(D\ol{B}\phi)}g(B\ol{C\chiup})g(C\ol{B}\chiup)g(D\phi \chiup){g(\ol{D\chiup}\phi )}\chiup(x).
\end{multline*}
Applying the finite field reflection formula \eqref{Gauss1}, the above becomes
\begin{multline*}
-\frac{\Omega}{(q-1)}\sum_{\chi\in \fqhat}\frac{\phi\chiup(-1)}{g(D\ol{B}\phi)}\(BC\chiup(-1)q-(q-1)\delta(B\ol{C\chiup})\)\\
\cdot(D\phi\chiup(-1)q-(q-1)\delta(D\phi \chiup))\chiup(x)\\
=-\frac{q^2\Omega}{(q-1)g(D\ol{B}\phi)}DB(-1)\sum_\chi\chiup(-1)\chiup(x)-E_1-E_2\\
\overset{\eqref{eq:delta}}=-\frac{q^2\Omega}{g(D\ol{B}\phi)}DB(-1)\delta(1+x)-E_1-E_2\\
%=-\frac{J(C\ol B,D\phi)}{J(C,D\ol B \phi)}%\delta(1+x)-E_1-E_2
=-\delta(1+x)\frac{J(\ol B,D\phi)}{J(C,\ol B)}-E_1-E_2.
\end{multline*}
Thus we have that
$$S_1 = \pFFq{2}{1}{B&C}{&C\ol{B}}{x}  -\delta(1+x)\frac{J(\ol B,D\phi)}{J(C,\ol B)}-  S_2. $$

Thus by \eqref{eqn:S-breakdown}, putting these all together gives that when $B\neq C$,
\begin{multline*}
\overline C(1-x) \, \pFFq{2}{1}{D\phi\overline B&D}{&C\overline B}{\frac{-4x}{(1-x)^2}} =  \\
\pFFq{2}{1}{B&C}{&C\ol{B}}{x} -\delta(1-x)\frac{J(C,\overline B^2)}{J(C,\overline B)} -\delta(1+x)\frac{J(\ol B,D\phi)}{J(C,\ol B)},
\end{multline*}
as desired.

When $B=C$, we can obtain by similar means that
\begin{multline*}
\overline C(1-x) \, \pFFq{2}{1}{\phi\overline D&D}{&\eps}{\frac{-4x}{(1-x)^2}} \\
= \pFFq{2}{1}{C&C}{&\eps}{x} -\delta(1-x)\frac{J(C,\overline C^2)}{J(C,\overline C)} -\delta(1+x) \frac{J(\ol C,D\phi)}{J(C,\ol C)}.
\end{multline*}

We next consider the case when $D^2 = \eps$.  Since we know $D\neq \phi$ from our hypothesis, we have $D=\eps$. Note that under the assumption of $B\neq D$, we also have $B\neq \eps$.  In this case, by Proposition \ref{prop: 2FF1-imprimitive}, the left hand side {of the equation in Theorem \ref{thm:quad-2F1}} is
\begin{align*}
\eps(1-x) \, &   \pFFq{2}{1}{\phi\overline B&\eps}{&\overline B}{\frac{-4x}{(1-x)^2}}\\
=& -\(B\(\frac{4x}{(1-x)^2}\)\phi\(\frac{(1+x)^2}{(1-x)^2}\)J(\ol B, B\phi)-\eps(1-x)\)\\
=& -B\(\frac{4x}{(1-x)^2}\)J(\ol B, B\phi)(1-\delta(1+x))+1\\
=& -B\(\frac{4x}{(1-x)^2}\)J(\ol B, B\phi)+1+B(-1)J(\ol B, B\phi)\delta(1+x)\\
=& -B\(\frac{4x}{(1-x)^2}\)J(\ol B, B\phi)+1+J(\ol B, \phi)\delta(1+x)\\
=& -B\(\frac{4x}{(1-x)^2}\)J(\ol B, B\phi)+1-\frac{J(\ol B, \phi)}{J(\eps, B)}\delta(1+x).\\
\end{align*}
Moreover, by Proposition \ref{prop: 2FF1-imprimitive}, and Lemma \ref{lem:1} we have that
\begin{multline*}
\pFFq{2}{1}{ B&\eps}{&\overline B}{x} =-\(B(-x)\ol{B^2}(1-x)J(\ol B, \ol B)-1\)\\
\overset{\eqref{eq:double-Jac}}=-B\(\frac{-4x}{(1-x)^2}\)J(\ol B, \phi)+1 =-B\(\frac{4x}{(1-x)^2}\)J( \ol B, B\phi)+1.
\end{multline*}
 Recalling that $\delta(1-x)=0$ since $x\neq1$, this gives the desired result that
\begin{multline*}
\eps(1-x)  \pFFq{2}{1}{\phi\overline B&\eps}{&\overline B}{\frac{-4x}{(1-x)^2}} \\
= \pFFq{2}{1}{ B&\eps}{&\overline B}{x} -\delta(1-x)\frac{J(\eps,\overline B^2)}{J(\eps,\overline B)}-\delta(1+x)\frac{J(\ol B,\phi)}{J(\eps,\ol B)}.
\end{multline*}

We now consider the case when  $B^2=C$, and thus  $B=D\phi$. The left hand side becomes
\begin{multline*}
\ol C(1-x) \pFFq{2}{1}{\eps&D}{&D\phi} {\frac{-4x}{(1-x)^2}}= \ol C(1-x)-\frac{\ol {D\phi}(-4x)}{J(D,\phi)}\\=\ol C(x-1)-\frac{\ol {D\phi}(-4x)}{J(D,\phi)} (1-\delta(1+x)),
\end{multline*}
while the right hand side is
 $$\ol C(x-1)-\frac{\ol D \phi(x)}{J(C, \ol D\phi)}=\ol C(x-1)-\frac{\ol D \phi(x)}{J(\ol D\phi, \ol D\phi)}.$$
Using the duplication formula, we obtain the result.
\end{proof}

\subsubsection{A summary of the  the proof of Theorem \ref{thm:quad-2F1}}\label{ss:quad-summary}
In the previous proof of the quadratic formula,  the delta terms produce two kinds of byproducts:  the extra terms like $\delta(1\pm x)$, which are due to the degeneracy of the corresponding Galois representations at the corresponding $x$ values, and the extra characters like the $E_i$ terms described above. While the first kind of extra terms are unavoidable, but can be {hidden by placing extra assumptions on the values of $x$}, the second kind of extra terms might be eliminated.  In particular, some of them are `path dependent'  in the following sense.  In the primitive case  by Definition \eqref{normalized2FF1} and Proposition \ref{prop: commute and conjugation}, the  $_2\F_1$ function can be computed from two $_2\mathbb P_1$ functions whose upper parameters are  related by  a permutation.   For instance, in  this proof, if we swap  the upper characters of the left hand side, then  the delta terms do not involve $B\ol{C}(x)$.

The above approach can be adopted to give finite field analogues of many classical quadratic or higher degree formulas relating two primitive $_{n+1}F_{n}$ functions with rational parameters and arguments of the form $c_1x$ and  $c_2x^{n_1}(1-c_3x)^{n_2}$, respectively, that satisfies the $(\ast)$ condition, where $n_1,n_2\in \Z, n_1>0$  and $c_1,c_2, c_3\in \Q^\times$. See \cite{Bailey2} for a collection of such formulas.  As we assume the classical formula satisfies the ($\ast$) condition, we will similarly mimic the proof of the classical formula to get the finite field analogues. We  expand one side (typically  the more complicated side) of the formula as a double summation involving two characters $K$ and $\chi$, as in \eqref{Eq:56}. Then  we reduce the double sum to  obtain the other  side of the formula. The dictionary described in \S \ref{dictionary} predicts that the major terms for the finite field analogues, i.e. the $_{n+1}\F_{n}$ terms   predicted on both sides,  agree.  Though extra delta terms are technical  to compute in general, they correspond to the discrepancy between two finite dimensional Galois representations.  There is one subtlety we should be aware of in the finite field translation: we might need to extend the current finite field $\F_q$ in order to use an evaluation formula to bridge both sides.

%For quite a few cases here, this is the Pfaff-SaalSch\"utz formula %over $\F_q$.  One way to get around the need for a field extension %is to use other evaluation formulas, such as the quadratic Pfaff-%SaalSch\"utz formula over $\F_q$ (Corollary \ref{cor:quad-pf-s} %below). See the proof of Theorem \ref{Thm:Bailey-FF} for such an %example.

\subsection{The quadratic formula in connection with the Kummer relations}
In this section, we give another finite field version of Kummer's quadratic formula in which the arguments on both sides of the transformation are rational functions of $z$ of degree two.

Together with the Kummer relations discussed in  \S \ref{ss:Kummer24}, one can also obtain many equivalent versions of quadratic formulas over  finite fields which are useful in many ways. A different approach to eliminate the extra characters of the second kind  as described in \S \ref{ss:quad-summary}  is as follows.
Recall $D$ is chosen so that $D^2=C$. Applying a Pfaff transformation (the fifth identity of Proposition \ref{prop: normalized independent solutions}) to the right hand side of the quadratic formula in Theorem \ref{thm:quad-2F1}, we obtain
\begin{multline}
\overline C(1-x) \, \pFFq{2}{1}{D\phi\overline B&D}{&C\overline B}{\frac{-4x}{(1-x)^2}} \\
= \overline C(1-x) \pFFq{2}{1}{C\ol{B}^2&C}{&C\ol{B}}{\frac{x}{x-1}}-\delta(1+x) \frac{J(\ol B,D\phi)}{J(C,\ol B)}.
\end{multline}
By letting $z=\frac{x}{x-1}$ and hence $1-z=\frac{1}{1-x}$, when $z\neq \frac 12,1$, we get
\begin{equation}
 \pFFq{2}{1}{D\phi\overline B&D}{&C\overline B}{4z(1-z)}=   \pFFq{2}{1}{C\ol{B}^2&C}{&C\ol{B}}{z},
\end{equation}
or equivalently when also $A^2$, $B^2 \neq \eps$, then
\begin{equation}
 \pFFq{2}{1}{A&B}{&AB\phi}{4z(1-z)}=   \pFFq{2}{1}{A^2&B^2}{&AB\phi}{z}.
\end{equation}
For example, it is easy to see that the terms do not differ by extra characters as both sides are invariant under the involution $z\mapsto 1-z$.

%\subsubsection{Using Kummer relations to obtain other quadratic formulas}

 Below is the $\F_q$ version of another Kummer quadratic transformation formula.

\begin{Theorem}\label{thm: kummer's quad trans}\index{finite field analogues!Kummer quadratic transformation}
Let $A$, $B\in \fqhat$  such that %$A$ is a square,
$B$, $B^2\ol A$, $A\ol B\phi \neq \eps$. Then we have
$$
 \pFFq21{A&B}{&B^2}{\frac{4z}{(1+z)^2}}=A^2(1+z)\pFFq21{A&A\phi\ol B}{&B\phi}{z^2},
$$
when $z\neq \pm1$.
\end{Theorem}
This is the finite field analogue of the following quadratic transformation of Kummer (see \cite[(3.1.11)]{AAR})
$$
   \pFq21{a&b}{&2b}{\frac{4z}{(1+z)^2}}=(1+z)^{2a}\pFq21{a&a+1/2-b}{&b+1/2}{z^2}.
$$

\begin{proof}[Proof of Theorem \ref{thm: kummer's quad trans}]Following arguments similar to the proof of the classical case, one can derive Kummer's quadratic transformation from the quadratic formulas

\begin{equation}\label{eqn: quad with Pfaff}
   \pFFq21{B&D^2}{&D^2\ol B}{z}={\ol D}^2(1+z)\pFFq21{\phi D& D}{&D^2\ol B}{\frac{4z}{(z+1)^2}}, \quad z\neq \pm 1,
\end{equation}
provided $B\neq \eps$ and $D^2\neq B^2$,
and
\begin{equation}\label{eqn: quad with Kummer and duplication}
  \pFFq21{K&E^2}{& K^2}{1-z} = E^2(2){\ol E}^2(1+z)\pFFq21{\phi E& E}{&\phi K}{\frac{(z-1)^2}{(z+1)^2}}, \quad z\neq \pm 1,
\end{equation}
provided  $K\neq \eps$ and $E^2\neq K^2$.

The first equation is obtained by applying Pfaff's transformation (Proposition \ref{prop: normalized independent solutions} - item 4) to the quadratic formula in Theorem \ref{thm:quad-2F1}. Using Proposition \ref{prop: normalized independent solutions} - item 3 and the duplication formula on Gauss sums, one can observe the second identity from (\ref{eqn: quad with Pfaff}).

Replace $1-z$ with $4z/(1+z)^2$ in (\ref{eqn: quad with Kummer and duplication}) to derive the formula
$$
  \pFFq21{K&E^2}{& K^2}{\frac{4z}{(1+z)^2}} =
  {\ol E}^2(1+z^2)E^4(1+z)\pFFq21{\phi E& E}{&\phi K}{\(\frac{2z}{1+z^2}\)^2}.
$$

Replace the arguments $\frac{4z}{(1+z)^2}$ by $\(\frac{2z}{1+z^2}\)^2$ and set the characters $D=E$, $B={\phi E^2\ol K}$ in (\ref{eqn: quad with Pfaff}) to obtain the formula
$$
   \pFFq21{\phi E^2\ol K&E^2}{&\phi K}{z^2}=\ol E^2(1+z^2)\pFFq21{\phi E& E}{&\phi K}{\(\frac{2z}{1+z^2}\)^2}.
$$
Therefore, one has
$$
  \pFFq21{E^2&K}{& K^2}{\frac{4z}{(1+z)^2}}=E^4(1+z)\pFFq21{E^2&\phi E^2\ol K}{&\phi K}{z^2},
$$
since the ${}_2\mathbb F_1$-functions are primitive  under our assumptions.  One can remove the condition on the finite field $\F_q$ and replace the character $E^2$ by any character $A$.
\end{proof}

As an immediate corollary of  Theorem \ref{thm: kummer's quad trans},  we  obtain the finite field version of Theorem 3 in \cite{Tu-Yang1}, which says for real numbers $a$ and $b$ such that neither $b+3/4$ nor $2b+1/2$ is a non-positive integer, then in a neighborhood of $z=0$ one has
$$
 \pFq21{a+b&b+\frac14}{&2b+\frac12}{\frac{4z}{(1+z)^2}}= (1+z)^{2a+2b}\pFq21{a+b&a+\frac14}{&b+\frac34}{z^2}.
$$

\begin{Corollary}
 Let $\eta$ be a primitive character of order $4$  so $\phi=\eta^2$.  Let $A$, $B$ characters so that $A\eta$, $B\eta$, $A\ol B\phi \neq \eps$.
Then we have
$$
 \pFFq21{AB&B\eta}{&\phi B^2}{\frac{4z}{(1+z)^2}}= A^2B^2(1+z)\pFFq21{AB&A\eta}{&B\ol\eta}{z^2},
$$
when $z\neq 0$, $\pm1$.
\end{Corollary}
See Appendix  \S \ref{ss:conj} for some numeric observations related to other  higher transformation formulas in \cite{Tu-Yang1}.

\subsubsection{Two other immediate corollaries of Theorem \ref{thm:quad-2F1}}
From a transformation formula, one can obtain evaluation formulas either by specifying values or by comparing coefficients on both sides. Here we will show how to obtain the finite field analogue of the Kummer evaluation formula \eqref{eq:Kummer}, which has been obtained by Greene in \cite{Greene} using a different approach.

\begin{Corollary}\label{cor: Kummer evaluation}\index{finite field analogues!Kummer evaluation formula}
Assume  $B,C\in \fqhat$ and set $C=D^2$,  then
  $$
   \pFFq{2}{1}{B&C}{&C\ol B}{-1}=\frac{J(D,\ol B)}{J(C,\overline B)}+\frac{J(\ol B,D\phi)}{J(C,\ol B)}.
  $$
\end{Corollary}
\begin{proof}
   By Theorem \ref{thm:quad-2F1}, if $C=D^2\neq \eps$,  $B^2\neq C$,  $x=-1$, one has
   $$
     \overline C(2) \, \pFFq{2}{1}{D\phi\overline B&D}{&C\overline B}{1}= \pFFq{2}{1}{B&C}{&C\ol{B}}{-1}-\frac{J(\ol B,D\phi)}{J(C,\ol B)}.
   $$
Hence,
\begin{multline*}
\pFFq{2}{1}{B&C}{&C\ol{B}}{-1} =  \overline C(2) \, \pFFq{2}{1}{D\phi\overline B&D}{&C\overline B}{1}+\frac{J(\ol B,D\phi)}{J(C,\ol B)} \\
= \ol D(4)\frac{J(\phi, D)}{J(D,D\ol B)}+\frac{J(\ol B,D\phi)}{J(C,\ol B)},
\end{multline*}
where if $C\neq B$,
$$
\ol D(4)\frac{J(\phi, D)}{J(D,D\ol B)}=\frac{\ol D(4)g(\phi)g(D)g(C\ol B)}{g(\phi D)g(D)g(D\ol B)}=\frac{J(\ol B,D)}{J(C,\ol B)},
$$
and if $C= B$,
$$
\ol D(4)\frac{J(\phi, D)}{J(D,D\ol B)}=\frac{\ol D(4)J(\phi, D)}{J(D,\ol D)}=\frac{J(D,D)}{J(D,\ol D)}=\frac{J(\ol C,D)}{J(C,\ol C)},
$$
since $J(\phi, \chiup)=\chiup(4)J(\chiup,\chiup)$ if $\chiup\neq \eps$.

When $C=\eps$, the claim follows   Proposition \ref{prop: 2F1-imprimitive}.
\end{proof}

Next we will get a `quadratic' version of the finite field Pfaff-Saalsch\"utz formula \eqref{eq:P.S.-FF} to be used later. This version includes cases to which \eqref{eq:P.S.-FF} is not applicable. Switching $x$ and $1-x$, we first rewrite  the formula in Theorem \ref{thm:quad-2F1} as

\begin{multline}\label{eqn:87}
\overline C(x) \, \pFFq{2}{1}{D\phi\overline B&D}{&C\overline B}{\frac{-4(1-x)}{x^2}}= \pFFq{2}{1}{B&C}{&C\ol{B}}{1-x}\\-\delta(x)\frac{J(C,\overline B^2)}{J(C,\overline B)}-\delta(2-x) \frac{J(\ol B,D\phi)}{J(C,\ol B)}.
\end{multline}

We first use one of the Kummer relations to write the first term on the right in terms of $x$ and assume $x\neq 0$. Now we look at the left hand side which is
%\begin{multline*}
%\overline C(x) \, \pFFq{2}{1}{D\phi\overline B&D}{&C\overline B}{\frac{-4(1-x)}{x^2}}\\
%=\frac{D(-1)}{(q-1)J(D,D\ol{B})}\sum_{K\in \fqhat}J(D\phi \ol{B}K,\ol K)J(DK,B\ol{CK})K(-4) \ol{CK^2}(x)K(1-x)\\
%+\ol C(x)\delta\(\frac{-4(1-x)}{x^2}\)\\
%\overset{\eqref{eq:delta}}=\frac{D(-1)}{(q-1)^2J(D,D\ol{B})}\sum_{K}J(D\phi \ol{B}K,\ol K)J(DK,B\ol{CK})K(-4) \ol{CK^2}(x) \\
%\cdot \sum_{\varphi\in \fqhat}J(K,\overline \varphi)\varphi(x) + \ol C(x)\delta\(1-x\)\\
%\overset{\chi=\overline{CK^2}\varphi}=\frac{D(-1)}{(q-1)^2J(D,D\ol{B})} \sum_{\chi \in \fqhat}\sum_{K\in \fqhat}J(D\phi \ol{B}K,\ol K)J(DK,B\ol{CK})K(-4) J(K,\overline {C\chi K^2})\chi(x)\\
%+\sum_{\chi\in \fqhat}\frac{1}{q-1}\chi(x).
%\end{multline*}
\begin{align*}
\overline C&(x) \, \pFFq{2}{1}{D\phi\overline B&D}{&C\overline B}{\frac{-4(1-x)}{x^2}}\\
\overset{\eqref{eq:for6.7}}{=}
&\frac{D(-1)}{(q-1)J(D,D\ol{B})}\sum_{K\in \fqhat}J(D\phi \ol{B}K,\ol K)J(DK,B\ol{CK})K(-4) \ol{CK^2}(x)K(1-x)\\
&+\ol C(x)\delta\(\frac{-4(1-x)}{x^2}\)\\
=&\frac{D(-1)}{(q-1)^2J(D,D\ol{B})}\sum_{K}J(D\phi \ol{B}K,\ol K)J(DK,B\ol{CK})K(-4) \ol{CK^2}(x) \\
&\cdot \sum_{\varphi\in \fqhat}J(K,\overline \varphi)\varphi(x) + \ol C(x)\delta\(1-x\)\\
\overset{\chi=\overline{CK^2}\varphi}=&\frac{D(-1)}{(q-1)^2J(D,D\ol{B})} \sum_{\chi \in \fqhat}\sum_{K\in \fqhat}J(D\phi \ol{B}K,\ol K)J(DK,B\ol{CK})K(-4) J(K,\overline {C\chi K^2})\chi(x)\\
&+\sum_{\chi\in \fqhat}\frac{1}{q-1}\chi(x),
\end{align*}\bk where in the last step we use \eqref{eq:delta} to expand $\delta(1-x)$. \bk

Since $C$ is a square, $C(-1)=1$. Under the assumption that $x\neq 0$, the right hand side  of \eqref{eqn:87} is
\begin{eqnarray*}
&&\pFFq{2}{1}{B&C}{&C\ol{B}}{1-x}-\delta(2-x) \frac{J(\ol B,D\phi)}{J(C,\ol B)}\\
&=&\frac{J(C,\ol{B}^2)}{J(C,\ol B)}\pFFq{2}{1}{B&C}{&B^2}{x}-\delta\(x-2\)  \frac{J(\ol B,D\phi)}{J(C,\ol B)}\\
&\overset{\eqref{eq:delta}}=&\frac 1{(q-1)J(C,\ol B)}\(\sum_{\chi\in \fqhat}J(B\chi,\ol \chi)J(C\chi,\ol{B^2\chi})\chi(x)
-J(\ol B,D\phi)\sum_{\chi\in \fqhat}\ol{\chi}(2)\chi(x)\).
\end{eqnarray*}

By equating the coefficient of $\chi$ on both sides and assuming $B^2\neq C$ we get

\begin{align*}
 & \left[  \(J(B\chi,\ol\chi)J(C\chi,\ol{B^2\chi})-J(\ol B, D\phi)\ol\chi(2)\)/J(C,\ol B) \right] -1\\
=&\frac{D(-1)}{(q-1)J(D,D\ol{B})}\sum_{K\in \fqhat}J(D\phi \ol{B}K,\ol K)J(DK,B\ol{CK}) J(K,\overline {C\chi  K^2})K(-4)\\
=&\frac{D\chi(-1)}{(q-1)J(D,D\ol{B})}\sum_{K}J(D\phi \ol{B}K,\ol K)J(DK,B\ol{CK}) {J(\overline {C\chi  K^2}, CK\chi)}K(-4)\\
=&\frac{D\chi(-1)}{(q-1)J(D,D\ol{B})}\sum_{K}\frac{g(D\phi \ol{B}K)g(\ol K)g(DK)g(B\ol{CK})}{g(D\phi \ol B)g(B\ol D)}\\
&\cdot \(\frac{{g(\overline {C\chi K^2})g( CK\chi)}}{g(\ol K)}+(q-1) \chi(-1) \delta(K)\) K(-4)\\
=&\left[\frac{D\chi(-1)}{(q-1)J(D,D\ol{B})}\sum_{K}\frac{g(D\phi \ol{B}K)g(DK)g(B\ol{CK})g(\overline {C\chi K^2})g( CK\chi)}{g(D\phi \ol B)g(B\ol D)}K(-4)\right]-1.
\end{align*}

Assume $C=D^2$, $D\neq \phi$ and $B^2\neq C$. Then
\begin{multline}\label{PS-quad}\sum_{K\in \fqhat} g(D\phi \ol{B}K)g(DK)g(B\ol{CK}){g(\overline {C\chi K^2})g( CK\chi)}K(-4)=\\
(q-1)\frac{D\chi(-1)g(D\phi \ol B)g(B\ol D)J(D,D\ol B)}{J(C,\ol B)}\(J(B\chi,\ol\chi)J(C\chi,\ol{B^2\chi})-J(\ol B, D\phi)\ol\chi(2)\).
%\frac{D(-1)}{(q-1)J(D,D\ol{B})}\sum_{K}J(D\phi \ol{B}K,\ol K)J(DK,B\ol{CK}) J(K,\chi\overline {C K^2})K(-4)\\
%=\(J(B\chi,\ol\chi)J(C\chi,\ol{B^2\chi})-J(\ol B, D\phi)\ol\chi(2)\)/J(C,\ol B)-1.
\end{multline}

In conclusion, we obtained the following quadratic version of Pfaff-SaalSch\"utz formula over $\F_q$.
\begin{Corollary}\label{cor:quad-pf-s}\index{Pfaff-Saalsch\"utz  evaluation formula! quadratic version}\index{finite field analogues!Pfaff-Saalsch\"utz  evaluation formula}
   Let $A$, $B$, $C\in \fqhat$  not simultaneously trivial, such that $CA, CB\neq \phi$.   Then
  \begin{multline*}
     \frac{AB(-1)C(4)J(\phi AC,\ol{ABC^2})}{q(q-1)g(\phi)}\sum_{\chi\in \fqhat} g(A\chi^2)g(B \chi)g(C\ol \chi)g(\phi \ol{ABC\chi}) g(\ol \chi)\ol\chi(-4)\\
     =J(\ol{ABC^2}, AC^2)J(A,B^2C^2)-J(\ol{ABC^2},\phi AC){AC^2}(2).
   \end{multline*}

If  $CB=\phi$  (or $CA=\eps$), then
\begin{multline*}
     \frac 1{(q-1)g(\phi)}\sum_{\chi\in \fqhat} g(A\chi^2)g(B \chi)g(C\ol \chi)g(\phi \ol{ABC\chi}) g(\ol \chi)\ol\chi(-4)\\
     =qB(4)AB(-1)J(B, B\ol A)-qAB(-1)A(2)-\delta(A)(q-1)B(-1)J(B,\phi);
   \end{multline*}
If $CA=\phi$  (or $CB=\eps$), then
\begin{multline*}
     \frac 1{(q-1)g(\phi)}\sum_{\chi\in \fqhat} g(A\chi^2)g(B \chi)g(\phi \ol A\ol \chi)g(\ol{B\chi}) g(\ol \chi)\ol\chi(-4)\\
     =-B(4)J(\phi, B)J(A\ol B^2,\phi B\ol A)-qAB(-1)A(2)\\
     +\delta(A)q(q-1)B(-1)-\delta(\phi B)(q-1)AB(-1).
   \end{multline*}
\end{Corollary}

\begin{proof}{The first case can be obtained from replacing the character $K$ by $K\ol{C\chi}$ in equation (\ref{PS-quad}) and relabeling the characters.}  The other cases are straightforward verifications from the relation between Gauss sums and Jacobi sums.
\end{proof}

\begin{Remark}When $A$ is a square of another character, the left hand side  of the formulas in the above Corollary can be evaluated using the finite field Pfaff-SaalSch\"utz formula \eqref{eq:P.S.-FF}. However, one does not need this assumption in the above Corollary.
\end{Remark}

\subsection{A finite field analogue of a theorem of Andrews and Stanton}

In this section, we will prove a finite field analogue of the following transformation formula given by Andrews and Stanton \cite{Andrews-Stanton}:

\begin{Theorem}[Theorem 5 of \cite{Andrews-Stanton}]\index{Andrews-Stanton cubic formula}\label{thm:AS}For $a,b,c, x\in \C$ such that both sides converge,
\begin{multline*}
(1-x)^{a+1}\pFq{3}{2}{a+1&b+1&a-b+\frac12}{&2b+2&2a-2b+1}{{4x}(1-x)}\\=(1-x)^{-a-1}\pFq{3}{2}{a+1&b+1&a-b+\frac12}{&2b+2&2a-2b+1}{\frac{-4x}{(1-x)^2}}.
\end{multline*}
\end{Theorem}

To prove this result, Andrews and Stanton used two Bailey cubic $_3F_2$ transformation formulas, both of which satisfy  the $(\ast)$ condition. Consequently, our dictionary method allows us to translate the proofs to obtain the following result. \bk

\begin{Theorem}\label{thm:A-S-FF}\index{finite field analogues!Andrews-Stanton cubic formula} For a given finite field $\F_q$  of odd characteristic,  let $A$, $B \in \fqhat$  such that none of $A$, $B$, $A\ol B^2$, $A^2\ol B^3$, $\phi A\ol B$, $\phi A\ol B^3$ is trivial  and $x$  an  element in $\F_q$ with $x\neq \pm 1$, $\frac 12$. We have
$$
  \pFFq 32 {A & B& A\phi\ol B }{&B^2&A^2\ol B^2}{4x(1-x)}
  =\ol A^2(1-x) \pFFq 32 {A & B& A\phi\ol B }{ &B^2&A^2\ol B^2}{\frac{-4x}{(1-x)^2}}.
$$
\end{Theorem}

 We note that the conditions in Theorem \ref{thm:A-S-FF} are used in order to avoid the need for delta terms.  The result could be stated for a less restricted class of characters, but would not appear as clean.  Before proving Theorem \ref{thm:A-S-FF}, we require some additional results.  We start by using our approach to obtain a finite field analogue of the Bailey cubic transformations.

\subsubsection{Finite field analogues of Bailey cubic transformations}
There are two Bailey cubic transforms (see \cite[(4.06)]{Bailey2}):
\begin{multline}\label{Bailey2}\index{Bailey cubic transformation}
\(1-x\)^{-a} \pFq{3}{2}{\frac a3&\frac{a+1}3&\frac{a+2}3}{ & b& a-b+\frac 32 }{\frac{27x^2}{4(1-x)^3}}=\\ \pFq{3}{2}{a&b-\frac12&a+1-b}{& 2b-1 &2a+2-2b}{4x},
\end{multline} %changed from 2b to 2b-1, 9/24/2022
and also (see \cite[(4.05)]{Bailey2})
\begin{multline}\label{Bailey}
(1-x)^{-a} \pFq{3}{2}{\frac a3&\frac{a+1}3&\frac{a+2}3}{&a-b+\frac 32& b}{\frac{-27x}{4(1-x)^3}}=\\ \pFq{3}{2}{a&2b-a-1&a+2-2b}{&b&a+\frac32 -b}{\frac x4}.
\end{multline}
We note that  both \eqref{Bailey2} and \eqref{Bailey} satisfy the ($\ast$) condition (see Appendix \S \ref{appendix:Bailey} for more details).

We first consider the $\F_q$ version of \eqref{Bailey2}. To do this, we apply the dictionary from \S \ref{dictionary}. We use $B$ for the $\F_q$ analogue of $b-\frac 12$ and $A$ for the analogue of $a/3$.
\begin{Theorem}\label{Thm:Bailey-FF}\index{finite field analogues!Bailey cubic transformation}
  Let  $q\equiv 1 \pmod{6}$ be a prime power. Let $ \eta_3 $ be a primitive character of order $3$ on $\F_q$, and $A$, $B\in \widehat{\F_q^\times}$ satisfying $A^3$, $B$, $A^3\ol B^2$, $A^6\ol B^3$, $\phi A^3\ol B$, $\phi A^3\ol B^3\neq\eps$. Then we have
\begin{multline*}
 \ol A^3(1-x) \pFFq 32 {A &A \eta_3  &A\ol{ \eta_3}  }{&\phi B&A^3\ol B}{\frac{27x^2}{4(1-x)^3}}=
\pFFq 32{A^3& B& \phi A^3 \ol B}{& B^2 & A^6\ol B^2}{4x}
\\ -\delta(x+2)\frac {g(\phi)g(\ol A^3)}{g(\phi \ol B)g(\ol A^3 B)}-\delta(1-x)\sum_{\substack{\chi\in \widehat{\F_q^\times}\\ \chi^3=\ol A^3}}\frac{g(\ol\chi) g(\phi \ol {B\chi})g(B\ol{A^3\chi})}{g(\phi \ol B)g(\ol A^3 B)g(A^3)}\ol\chi\(-4\).
\end{multline*}
\end{Theorem}

 This is a special case of the following identity
\begin{Theorem}\label{Thm:Bailey-FF-char}
Let $q$ be an odd prime power and let $A$, $B\in \widehat{\F_q^\times}$ be such that $A$, $B$, $A \ol B^2$, $A^2\ol B^3$, $\phi A \ol B$, $\phi A \ol  B^3\neq\eps$. Then we have
\begin{align*}
&\pFFq 32{A& B& \phi A \ol B}{& B^2 & A^2\ol B^2}{4x}\\
=&\frac{1}{(q-1)^2}\frac{\ol A(1-x)}{g(\phi\ol B)g(\ol AB)}\sum_{\chi\in \widehat{\F_q^\times}}\(A\)_{\chi^3}g(\ol\chi)g(\phi\ol{B\chi})g(B\ol{A\chi})\cdot\chi\(\frac{x^2}{4(x-1)^3}\)\\
&+\delta(x) +\delta(x+2)\frac {g(\phi)g(\ol A)}{g(\phi \ol B)g(\ol A B)}+\delta(1-x)\sum_{\substack{\chi\in \widehat{\F_q^\times}\\ \chi^3=\ol A}}\frac{g(\ol\chi) g(\phi \ol {B\chi})g(B\ol{A\chi})}{g(\phi \ol B)g(\ol A B)g(A)}\ol\chi\(-4\).
\end{align*}
\end{Theorem}

  As an immediate consequence, one has the following evaluation formula by letting $x=1$.
\begin{Corollary}\label{cor:cubic-eva&4}Given the assumptions as in Theorem \ref{Thm:Bailey-FF-char},
$$\pFFq 32{A& B& \phi A \ol B}{& B^2 & A^2\ol B^2}{4}=\sum_{\substack{\chi\in \widehat{\F_q^\times}\\ \chi^3=\ol A}}\frac{g(\ol\chi) g(\phi \ol {B\chi})g(B\ol{A\chi})}{g(\phi \ol B)g(\ol A B)g(A)}\ol\chi\(-4\).$$
\end{Corollary}

The proof of Theorem \ref{Thm:Bailey-FF-char} is mainly based on Corollary \ref{cor:quad-pf-s}, the quadratic version of Pfaff-SaalSch\"utz
formula. Before we begin the proof,   we introduce some notation and prove two short lemmas which will assist in the body of the proof.  If $A,B,K \in\fqhat$, define
\begin{equation}\label{eq:I_K}
  I_{A,B}(K):=\frac{1}{q-1}\sum_{\chi\in \fqhat}g(A K \chi)g(\phi \ol{B \chi})g(\ol A B\ol \chi)g(\ol K\chi^2) g(\ol \chi)\ol\chi(-4),
\end{equation}
and
\begin{equation*}\label{eq:J_K}
 J_{A,B}(K):=K(4)J(AK,\ol K)J(BK,\ol{B^2K})J(\phi A\ol B K, B^2\ol{A^2 K}).
\end{equation*}

With this notation in place, we state and prove two lemmas. The first  evaluates the above character sums for special choices of $K$.

\begin{Lemma}\label{lemma: I-K-sp}
 Under the same assumptions on $A,B$  as Theorem  \ref{Thm:Bailey-FF-char},  for any characters $K$ and a fixed finite field $\fq$, the following identities hold.
\begin{enumerate}
 \item If $K=\ol B$,
 \begin{align*}
  I_{A,B}(\ol B)=&-g(\phi)qA(-1) B(2)-g(\phi)A\ol B(4)J(\phi,A\ol B)J( B^3\ol A^2, \phi \ol B^2 A)\\
   J_{A,B}(\ol B)=&-\ol B(4)J( B, A\ol B)J(\phi \ol B^2 A, B^3\ol{A^2}).
 \end{align*}
\item
If  $K= \ol A B$, or $\phi \ol B$, then
\begin{align*}
   I_{A,B}(K)=&g(\phi)qAK(4)A(-1)J(AK,AK^2)-qA(-1)\ol K(2)g(\phi)\\
      &-\delta(K)AK(-1)(q-1)g(\phi)J(\phi, AK);\\
  J_{A,B}(\ol A B) =&B\ol A(4)A(-1)J( B, A\ol B)J( B^2\ol A,B)J(\phi, B\ol A),\\
  J_{A,B}(\phi \ol B)=&\ol B(4)J( \phi A\ol B,\phi B)J(\phi,\phi \ol B)J( \ol B^2A,\phi B^3\ol A).
\end{align*}
\end{enumerate}
\end{Lemma}
\begin{proof}
Following the special cases of the Corollary \ref{cor:quad-pf-s}, we can get $I_{A,B}(K)$ immediately. In particular,
\begin{multline*}
  \frac{I_{A,B}(\ol B)}{g(\phi)}=-qA(-1)B(2)-A\ol B(4)J(\phi,A\ol B)J(\ol A^2B^3,\phi A\ol B^2)\\
           -\delta(B)q(q-1)AB(-1)-\delta(\phi A\ol B)A(-1)(q-1)\\
           =-qA(-1)B(2)-A\ol B(4)J(\phi,A\ol B)J(\ol A^2B^3,\phi A\ol B^2),
\end{multline*}
since $B$, $\phi A\ol B$ are non-trivial characters.

When $K=\ol A B$,
\begin{align*}
  J_{A,B}(K)=&B\ol A(4)J( B, A\ol B)J( B^2\ol A,A\ol{B}^3)J(\phi, B\ol A)\\
     =&B\ol A(4)A(-1)J( B, A\ol B)J( B^2\ol A,B)J(\phi, B\ol A);
\end{align*} the other two can be computed from the definition of $J_{A,B}(K)$.
\end{proof}

The next lemma relates $I_{A,B}(K)$ and $J_{A,B}(K)$ under special assumptions. \bk
\begin{Lemma}\label{Lem:forBaileyCubic}Under the same assumptions on $A,B$  as Theorem \ref{Thm:Bailey-FF-char},  for any characters $K$ of a fixed finite field $\fq$ one has
$$ I_{A,B}(K)={A(-4)}g(A)g(\ol B)g(\phi B\ol A)\cdot J_{A,B}(K) -A(-1)\ol K(2)qg(\phi).$$
Consequently,\begin{align*}
  \frac{1}{q-1}&\sum_{K\in \fqhat} K(-x)\cdot I_{A,B}(K)=-\delta(x+2)A(-1)qg(\phi)\\
   &+\frac{ A(4)\phi(-1) }qg(A)g(\ol B)g(\phi B\ol A) \pPPq 32{A& B& \phi A \ol B}{& B^2 & A^2\ol B^2}{4x}.
\end{align*}%\color{blue} The  $_3\mathcal P_2$ should be $_3\mathbb P_2$. \color{black}
\end{Lemma}
\begin{proof}When $K^2\neq \ol B^2$, $K\neq \ol AB$, by letting  $A=\ol K$, $B=AK$, and $C=\phi\ol B$ in Corollary \ref{cor:quad-pf-s}, one has
\begin{multline*}
I_{A,B}(K) = A(-1)B(4)g(\phi){J({BK}, A\ol B^2)}J(\ol AB^2,\ol{B^2 K})J(\ol K, A^2\ol B^2 K^2)\\
-A(-1)\ol K(2)qg(\phi).
\end{multline*}
Using the next equality, obtained from the duplication and reflection formulas of Gauss sums, \bk
$$
J(\ol K, A^2\ol B^2 K^2)=A\ol BK(4)J(A\ol BK,A\ol B)J(\phi A\ol B K,\ol K)/J(\phi, A\ol B),
$$
one can derive
\begin{align*}
I_{A,B}(K)=& A(-1)B(4)g(\phi)\cdot AK(-1)J(BK,\ol AB\ol K)J(\ol {B^2K},{AK})\\
  &\cdot \frac{A\ol BK(4)}{J(\phi,A\ol B)}J(A\ol B K,B^2\ol{A^2K})J(\phi A\ol BK,\ol K)-A(-1)\ol K(2)qg(\phi)\\
  =& A(-1)  B(4)  g(\phi)\(\frac{AB(-1)A\ol B(4)}{J(\phi, A\ol B)} \frac{g(A)g(\ol B)g(\phi B\ol A)}{g(B\ol A)g(\phi A\ol B)}-\delta(B\ol A)A(-1)(q-1)\)\\
  &\cdot J_{A,B}(K)   -A(-1)\ol K(2)qg(\phi).
\end{align*}
The claim follows from straightforward computation. For the other case, the claim follows from the  Lemma \ref{lemma: I-K-sp}.
\end{proof}

We are now ready to prove Theorem \ref{Thm:Bailey-FF-char}. \bk
\begin{proof}[Proof of Theorem \ref{Thm:Bailey-FF-char}]  When $x=0$, both sides of the desired equation have value 1. \bk
Assume $x\neq 0$. The character sum
$$
 \frac{1}{(q-1)^2}\frac{\ol A(1-x)}{g(\phi\ol B)g(\ol AB)}\sum_{\chi\in \widehat{\F_q^\times}}\(A\)_{\chi^3}g(\ol\chi)g(\phi\ol{B\chi}g(B\ol{A\chi})\chi\(\frac{x^2}{4(x-1)^3}\)
$$
can be written as
\begin{align*}
\frac{1}{(q-1)^2}&\frac{\ol A(1-x)}{g(\phi\ol B)g(\ol AB)}\sum_{\chi\in \widehat{\F_q^\times}}\(A\)_{\chi^3}g(\ol\chi)g(\phi\ol{B\chi}g(B\ol{A\chi})\chi\(\frac{x^2}{4(x-1)^3}\)\\
 =&\frac {1}{(q-1)^2}\frac 1{g(\phi \ol B)g(\ol A B)} \\
 &\cdot \sum_{\chi,\varphi \in \widehat{\F_q^\times}} (A)_{\chi^3} g(\ol\chi)g(\phi\ol{B\chi}g(B\ol{A\chi})\ol\chi\(4\)J(A\chi^3\varphi,\ol\varphi)\cdot \varphi\chi^2(-x)\\
  =&\frac {1}{(q-1)^2}\frac 1{g(\phi \ol B)g(\ol A B)}\\
  &\cdot \sum_{\chi,\varphi} (A)_{\chi^3} g(\ol\chi)g(\phi \ol {B\chi})g(B\ol{A\chi})\ol \chi\(-4\)(A\chi^3)_{\varphi} g(\ol\varphi)\cdot \varphi\chi^2(-x)\\
  &+\frac {1}{(q-1)}\frac 1{g(\phi \ol B)g(\ol A B)}\\
  &\cdot \sum_{\chi}\delta(A\chi^3) (A)_{\chi^3}g(\ol\chi) g(\phi \ol {B\chi})g(B\ol{A\chi})\ol\chi\(-4\)\chi^2(x)\sum_{\varphi}\varphi(x).
\end{align*}
By using \eqref{eq:22}, which says $(A)_{\chi_1\chi_2}=(A)_{\chi_1}(A\chi_1)_{\chi_2}$, \bk the above equals
\begin{multline*}
\frac {1}{(q-1)^2}\frac 1{g(\phi \ol B)g(\ol A B)}\sum_{\chi,\varphi \in \fqhat} (A)_{\chi^3\varphi} g(\ol\chi) g(\phi \ol {B\chi})g(B\ol{A\chi})\ol\chi\(-4\) g(\ol\varphi)\cdot \varphi\chi^2(-x)\\
 -\delta(1-x)\frac 1{g(\phi \ol B)g(\ol A B)g(A)}\sum_{\chi}\delta(A\chi^3)g(\ol\chi) g(\phi \ol {B\chi})g(B\ol{A\chi})\ol\chi\(-4\)\\
 \overset{K:=\varphi\chi^2}=\frac {1}{(q-1)^2}\frac 1{g(\phi \ol B)g(\ol A B)}\sum_{\chi,K\in \widehat{\F_q^\times}} (A)_{\chi K} g(\ol\chi)g(\phi \ol {B\chi})g(B\ol{A\chi})\ol\chi\(-4\) g(\ol K\chi^2)\cdot K(-x)\\
 -\delta(1-x)\sum_{\chi^3=\ol A}\frac{g(\ol\chi) g(\phi \ol {B\chi})g(B\ol{A\chi})}{g(\phi \ol B)g(\ol A B)g(A)}\ol\chi\(-4\).
 \end{multline*}
We write the first term above as
$$\displaystyle \frac {1}{(q-1)}\frac 1{g(\phi \ol B)g(\ol A B)g(A)}\sum_{K\in \fqhat} I_{A,B}(K)\cdot K(-x),$$
where
  $I_{A,B}(K)$ is defined in \eqref{eq:I_K}. By applying Lemma \ref{Lem:forBaileyCubic}, we have
\begin{align*}
\frac{1}{(q-1)^2}&\frac{\ol A(1-x)}{g(\phi\ol B)g(\ol AB)}\sum_{\chi\in \widehat{\F_q^\times}}\(A\)_{\chi^3}g(\ol\chi)g(\phi\ol{B\chi}g(B\ol{A\chi})\chi\(\frac{x^2}{4(x-1)^3}\)\\
 =&\pFFq 32{A& B& \phi A \ol B}{& B^2 & A^2\ol B^2}{4x}-\delta(x)\\
 &-\delta(x+2)\frac {g(\phi)g(\ol A)}{g(\phi \ol B)g(\ol A B)}-\delta(1-x)\sum_{\chi^3=\ol A}\frac{g(\ol\chi) g(\phi \ol {B\chi})g(B\ol{A\chi})}{g(\phi \ol B)g(\ol A B)g(A)}\ol\chi\(-4\),
\end{align*}
since
\begin{align*}
   &\frac {A(4)\phi(-1)}{q}\frac {g(\ol B)g(\phi B\ol A)}{g(\phi \ol B)g(\ol A B)}J(B, B)J(\phi A\ol B,\phi A \ol B)\\
   =&\frac {A(-4)}{q}\frac {g(\ol B)g(\phi B\ol A)}{g(\phi \ol B)g(\ol A B)}J(B,\ol B^2)J(\phi A\ol B,\ol A^2B^2)\\
   =&1.
\end{align*}
\end{proof}

{
\begin{proof}[Proof of Theorem \ref{Thm:Bailey-FF}]
When $x=0$, both sides take value $1$. Thus, assume $x\neq 0$.
From the definition, we derive that
\begin{align*}
 \ol A^3(1-x)& \pFFq 32 {A &A \eta_3  &A\ol{ \eta_3 } }{&\phi B&A^3\ol B}{\frac{27x^2}{4(1-x)^3}}\\
 =&\frac{ \ol A^3(1-x) }{(q-1)^2J(A \eta_3 ,\phi \ol B)J(A\ol{ \eta_3 }, \ol A^3 B)}\\
 &\sum_{\chi,\varphi\in \fqhat}J(A\chi,\ol\chi)J(A \eta_3 \chi,\phi\ol{B\chi})J(A\ol{ \eta_3}\chi,B\ol{A^3\chi})\chi(27)\chi\(\frac{x^2}{4(x-1)^3}\)\\
 =&\frac {1}{(q-1)^2}\frac {\ol A^3(1-x) }{g(\phi \ol B)g(\ol A^3 B)}\\
  &\cdot \sum_{\chi,\varphi} (A)_\chi (A \eta_3 )_\chi (A\ol{ \eta_3 })_\chi g(\ol\chi)g(\phi \ol {B\chi})g(B\ol{A^3\chi})\chi\(27\)
  \chi\(\frac{x^2}{4(x-1)^3}\)\\
  =&\frac {1}{(q-1)^2}\frac {\ol A^3(1-x) }{g(\phi \ol B)g(\ol A^3 B)}\cdot \sum_{\chi,\varphi} (A^3)_{\chi^3} g(\ol\chi)g(\phi \ol {B\chi})g(B\ol{A^3\chi})\ol \chi\(\frac{x^2}{4(x-1)^3}\).
\end{align*}
The theorem follows from Theorem \ref{Thm:Bailey-FF-char} by replacing $A$ with $A^3$.
\end{proof}

Theorem \ref{Thm:Bailey-FF-char} provides a finite field analogue for \eqref{Bailey2}, and we will now use it to obtain a finite field analogue for \eqref{Bailey}.

\begin{Theorem}\label{Thm:Bailey-FF2-char}
  Let $q$ be a power of odd prime, and $A$, $B\in \widehat{\F_q^\times}$ satisfying $A$, $B$, $A\ol B^2$, $A^2\ol B^3$, $\phi A\ol B$, $\phi A\ol B^3\neq\eps$. Then we have
\begin{align*}
& \pFFq 32{A &A\ol B^2 &\ol A B^2 }{&A\ol B &\phi B}{\frac x4}\\
=&\frac{1}{(q-1)^2}\frac{\ol A\(1-x\)}{g(\phi\ol B)g(\ol AB)}\sum_{\chi\in \widehat{\F_q^\times}}\(A\)_{\chi^3}g(\ol\chi)g(\phi\ol{B\chi})g(B\ol{A\chi})\cdot\chi\(\frac{x}{4(1-x)^3}\)\\
&+\delta(x) -\delta(x+1/2)A(2)\frac {g(\phi)g(\ol A)}{g(\phi \ol B)g(\ol A B)}-\delta(1-x)\sum_{\substack{\chi\in \widehat{\F_q^\times}\\ \chi^3=\ol A}}\frac{g(\ol\chi) g(\phi \ol {B\chi})g(B\ol{A\chi})}{g(\phi \ol B)g(\ol A B)g(A)}\ol\chi\(4\).
\end{align*}
\end{Theorem}
As an immediate consequence, one has the following evaluation formula by letting $x=1$, comparable to Corollary \ref{cor:cubic-eva&4}. \bk
\begin{Corollary}\label{cor:cubic-3F2}Given the assumptions as in Theorem \ref{Thm:Bailey-FF2-char}, we have
$$\pFFq 32{A &A\ol B^2 &\ol A B^2 }{&A\ol B &\phi B}{\frac 14}=\sum_{\substack{\chi\in \widehat{\F_q^\times}\\ \chi^3=\ol A}}\frac{g(\ol\chi) g(\phi \ol {B\chi})g(B\ol{A\chi})}{g(\phi \ol B)g(\ol A B)g(A)}\ol\chi\(4\).$$
\end{Corollary}

\begin{proof}[Proof of Theorem \ref{Thm:Bailey-FF2-char}]
If we do a variable change $\chi\mapsto \ol{A\chi}$ in the {character sum corresponding to the} $_3\mathbb P_2$-period function on  the left  hand side of the identity in Theorem \ref{Thm:Bailey-FF-char}, we obtain that when $x\neq 0$,
\begin{equation}\label{eq:Bailey-FF2}
\pPPq 32{A& B& A\phi\ol B }{& B^2&A^2\ol B^2}{4x} =
\phi(-1)\ol A(4x) \pPPq 32  {A& A\ol B^2& \ol AB^2 }{&A\ol B&\phi B}{\frac1{4x}}.
\end{equation}

By letting $x\mapsto \frac 1x$ on the right hand sides of Theorem  \ref{Thm:Bailey-FF-char} and \eqref{eq:Bailey-FF2} we see that ignoring the delta terms gives
\begin{align*}
& \pFFq 32{A &A\ol B^2 &A\ol B^2 }{&A\ol B &\phi B}{\frac x4} \\
=& \frac{J(B,B)J(A\phi\ol B,A\phi\ol B)}{J(A\ol B^2,B)J(\ol AB^2,\phi A\ol B)}\cdot \phi(-1) A\(\frac 4x\)\\
&\cdot \frac{1}{(q-1)^2}\frac{\ol A\(\frac{x-1}x\)}{g(\phi\ol B)g(\ol AB)}\sum_{\chi\in \widehat{\F_q^\times}}\(A\)_{\chi^3}g(\ol\chi)g(\phi\ol{B\chi})g(B\ol{A\chi})\cdot\chi\(\frac{x}{4(1-x)^3}\)\\
=& \phi A(-1)\ol A(4) \phi(-1) A\(\frac 4x\)\\
&\cdot \frac{1}{(q-1)^2}\frac{\ol A\(\frac{x-1}x\)}{g(\phi\ol B)g(\ol AB)}\sum_{\chi\in \widehat{\F_q^\times}}\(A\)_{\chi^3}g(\ol\chi)g(\phi\ol{B\chi})g(B\ol{A\chi})\cdot\chi\(\frac{x}{4(1-x)^3}\)\\
=&  \frac{1}{(q-1)^2}\frac{\ol A\(1-x\)}{g(\phi\ol B)g(\ol AB)}\sum_{\chi\in \widehat{\F_q^\times}}\(A\)_{\chi^3}g(\ol\chi)g(\phi\ol{B\chi})g(B\ol{A\chi})\cdot\chi\(\frac{x}{4(1-x)^3}\)
\end{align*}
\end{proof}

As a corollary of Theorem \ref{Thm:Bailey-FF2-char}, we can obtain  a finite field analogue for Equation \eqref{Bailey}.
}
\begin{Theorem}\label{Thm: Bailey FF2}\index{finite field analogues!Bailey cubic transformation}
Let  $q\equiv 1 \pmod6$ be a prine power. Let $\eta_3 $ be a primitive character of order $3$ on $\F_q$, and $A$, $B$ be   characters satisfying $A^3$, $B$, $A^3\ol B^2$, $A^6\ol B^3$, $\phi A^3\ol B$, $\phi A^3\ol B^3\neq\eps$. Then we have
\begin{multline*}
\ol A^3\({1-x}\)\pFFq 32 {A &A\eta_3  &A\ol{\eta_3} }{&\phi B&A^3\ol B}{\frac{-27x}{4(1-x)^3}}= \pFFq 32{A^3 &A^3\ol B^2 &\ol A^3 B^2 }{&A^3\ol B &\phi B}{\frac x4}\\
 -\delta(x+1/2)A^3(2)\frac {g(\phi)g(\ol A^3)}{g(\phi \ol B)g(\ol A^3 B)}-\delta(1-x)\sum_{\substack{\chi\in \widehat{\F_q^\times}\\ \chi^3=\ol A^3}}\frac{g(\ol\chi) g(\phi \ol {B\chi})g(B\ol{A^3\chi})}{g(\phi \ol B)g(\ol A^3 B)g(A^3)}\ol\chi\(4\).
\end{multline*}
\end{Theorem}

We are now in the position to prove  Theorem \ref{thm:A-S-FF}, the finite field analogue of  Theorem \ref{thm:AS}. We follow the ideas given by Andrews and Stanton in \cite{Andrews-Stanton}.

\begin{proof}[Proof of Theorem \ref{thm:A-S-FF}]
%Assume $x\neq 0$, $1$, $-1/2$.
Following the ideas of { the proof of Theorem \ref{thm:AS} by  Andrews and Stanton} \cite{Andrews-Stanton}, let $x=y^2/(y-1)$ in Theorem \ref{Thm:Bailey-FF2-char} and $x=y(1-y)$ in Theorem \ref{Thm:Bailey-FF-char}. Notice that the function \bk
\begin{align*}
  F(y):=&\pFFq 32{A &A\ol B^2 &\ol A B^2 }{&A\ol B &\phi B}{\frac{y^2}{4(y-1)}}\\
  =&\frac{1}{(q-1)^2}\frac{\ol A\(1-\frac {y^2}{y-1}\)}{g(\phi\ol B)g(\ol AB)}\sum_{\chi\in \widehat{\F_q^\times}}\(A\)_{\chi^3}g(\ol\chi)g(\phi\ol{B\chi})g(B\ol{A\chi})\cdot\chi\(\frac{y^2(y-1)^2}{4(y-1-y^2)^3}\)
\end{align*}
is equal to
$$
  A(1-y)\pFFq 32{A &B &A\phi\ol B }{&B^2 &A^2\ol B^2}{4y(1-y)},
$$
if $y\neq 0$, $\pm 1$, $\frac 12$.
As the rational function $\frac{y^2}{4(y-1)}$ is invariant  under the change of variable $y\mapsto y/(y-1)$, i.e. $F(y)=F(y/(y-1))$, we have
$$
 \pFFq 32{A &B &A\phi\ol B }{&B^2 &A^2\ol B^2}{4y(1-y)}= \ol A((1-y)^2)\pFFq 32{A &B &A\phi\ol B }{&B^2 &A^2\ol B^2}{\frac{-4y}{(1-y)^2}},
$$
for any element $y\neq 0$, $\pm 1$, $\frac 12$.
\end{proof}

\subsection{Another application of Bailey cubic transformations}

In this section, we consider a degenerate case of \eqref{Bailey} corresponding to the triangle group (see \eqref{tri}) $(2,3,4)$ and prove a finite field version of this identity. We close the section by providing conjectures, based on numerical evidence, for finite field versions of identities related to the triangle group $(2,3,3)$.

Letting $b=\frac{a+2}3$ in \eqref{Bailey} gives the following degenerate case of the Bailey cubic transformation formula:
\begin{equation}\label{eq:cl-cubic}\index{degenerate Bailey cubic transformation}
\pFq{2}{1}{a&\frac{1-a}3}{&\frac{4a+5}6}x=(1-4x)^{-a} \, \pFq{2}{1}{\frac a3&\frac {a+1}3}{& \frac{4a+5}6}{\frac{-27x}{(1-4x)^3}},
\end{equation}
which we note satisfies condition ($\ast$) as it  is a degenerate  (from $_3F_2$ to $_2F_1$)  case of   \eqref{Bailey} (see Appendix \S \ref{appendix:Bailey}).
Further specifying $a=-\frac18$ in \eqref{eq:cl-cubic}, gives
$$
\pFq{2}{1}{-\frac1{24}&\frac{7}{24}}{&\frac 34}{\frac{-27 x }{(1-4x)^3}}= (1-4x)^{-\frac18} \pFq{2}{1}{-\frac18& \frac38,}{&\frac34}x,
$$
and then using \eqref{eq:Slater} with $a=-\frac 38$ to evaluate the right hand side gives the following \bk algebraic identity corresponding to the triangle group $(2,3,4)$:\bk
\begin{equation}\label{eq:2nd-cubic}
\pFq{2}{1}{-\frac1{24}&\frac{7}{24}}{&\frac 34}{\frac{-27 x }{(1-4x)^3}} = \left(\frac{\left(1+\sqrt{1-x}\right)^{2}}{4(1-4x)}\right)^{\frac 18}.
\end{equation}
Our goal here is to derive the finite field analogue of \eqref{eq:2nd-cubic}.

Recall that in \S \ref{ss:quad-summary} we outlined how our approach can be used to prove finite field versions of classical formulas satisfying the ($\ast$) condition. We first apply these ideas to obtain the following finite field version of  \eqref{eq:cl-cubic}. Here, the dictionary in \S \ref{dictionary} translates $\frac{a}{3}$ to $E$. By taking $B=\phi A\ol\eta_3$ in Theorem \ref{Thm: Bailey FF2}, we obtain the following:
\begin{Corollary}\label{remark13}\index{finite field analogues!degenerate Bailey cubic transformation}
Let  $q\equiv 1\pmod 3$ be a prime power, $E\in \fqhat$ such that $E^6\neq \eps$. Let  $\eta_3$ be a primitive cubic character. Then for $x\neq 0,1,\frac 14,-\frac 18$, \begin{equation*}
\pFFq{2}{1}{E^3&\eta_3 \ol{E}}{&E^2\phi\eta_3}x = \quad \ol{E^3}(1-4x)\, \pFFq{2}{1}{E&E\eta_3}{& E^2\phi\eta_3}{\frac{-27x}{(1-4x)^3}}.\end{equation*} \end{Corollary}
%\noindent Note that this Corollary corresponds to the $B=\phi A\ol\eta$ case  of Theorem \ref{Thm: Bailey FF2}.

We now establish the finite field analogue of \eqref{eq:2nd-cubic}  using an idea that is parallel to the proof of \eqref{eq:2nd-cubic}  described above.

\begin{Theorem}\label{thm:15}
Let  $q\equiv 1\pmod{24}$  be a prime power,  and let  $x\neq  -\frac 18,\frac 14$  in $\F_q$.  Let $\eta$ be an order 24 character on  $\F_q^\times$. Then
\begin{multline*}
\pFFq{2}{1}{\overline \eta&{\eta^7}}{&{\eta^{18}}}{\frac{-27x}{(1-4x)^3}}= \\ \(\frac{1+\phi(1-x)}2\)\( \eta^3\left ( \frac{(1+\sqrt{1-x})^{2}}{2^{2}(1-4x)}\right)+ \eta^3\left ( \frac{(1-\sqrt{1-x})^{2}}{2^{2}(1-4x)}\right)\).
\end{multline*}
\end{Theorem}

\begin{proof}It is easy to see that when $x=0$, both sides equal to 1.

When $x=1$, using Gauss evaluation formula and  \cite[Theorem 3.6.6]{Berndt-Evans-Williams}, one has
$$
\pFFq{2}{1}{\eta^7&\ol{\eta}}{&\eta^{18}}1=\frac{J(\phi,\ol\eta)}{J(\ol \eta,\ol \eta^5)}=\eta^3\(-\frac 43\)=\ol\eta^3\(-3\cdot 2^{2}\),
$$
since $2$ is a square in $\F_q$. Hence the claim holds when $x=1$.

For $x\neq 0,1,\frac 14,-\frac 18$, by Corollary \ref{remark13}, we have
$$\pFFq{2}{1}{\ol{\eta}&\eta^7}{&\eta^{18}}{\frac{-27x}{(1-4x)^3}}
= \ol{\eta}^3(1-4x)\pFFq{2}{1}{\ol{\eta}^3&\eta^9}{&\eta^{18}}{x}.
$$

By Theorem \ref{thm:FF-Dihedral},
$$
\pFFq{2}{1}{\ol{\eta}^3&\eta^9}{&\eta^{18}}{x}=\(\frac{1+\phi(1-x)}2\)\(\eta^6\(\frac{1+\sqrt{1-x}}2\)+\eta^6\(\frac{1-\sqrt{1-x}}2\)\).
$$

Composing this  with the proceeding cubic formula, we have
\begin{multline*}
\pFFq{2}{1}{\ol{\eta}&\eta^7}{&\eta^{18}}{\frac{-27x}{(1-4x)^3}}
= \ol{\eta}^3(1-4x)\pFFq{2}{1}{\ol{\eta}^3&\eta^9}{&\eta^{18}}{x}\\
=\(\frac{1+\phi(1-x)}2\)\ol{\eta}^3(1-4x)\(\eta^6\(\frac{1+\sqrt{1-x}}2\)+\eta^6\(\frac{1-\sqrt{1-x}}2\)\)\\
={\(\frac{1+\phi(1-x)}2\)}\(\eta^3\(\frac{(1+\sqrt{1-x})^2}{2^2(1-4x)}\)+\eta^3\(\frac{(1-\sqrt{1-x})^2}{2^2(1-4x)}\)\),
\end{multline*}
which concludes the proof.
\end{proof}

\begin{Remark} We obtain another algebraic hypergeometric identity corresponding to the triangle group $(2,3,3)$  from the  following \bk transformation
$$
  \pFq 21{\frac {3a}2& \frac{3a-1}2}{& a+\frac 12}{-\frac{z^2}3}=(1+z)^{1-3a}\pFq 21{a-\frac 13& a}{& 2a}{\frac{2z(3+z^2)}{(1+z)^3}}
$$
\cite[pp. 114 (39)]{Erdelyi} (or see \cite{Goursat}) with $a=\frac 14$,  namely,
\begin{equation}\label{anotherhf}
\pFq{2}{1}{-\frac{1}{12}&\frac 14}{&\frac 12}{\frac{2z(3+z^2)}{(1+z)^3}}=\( \frac{1+\sqrt{1+z^2/3}}{2(1+z)} \)^{1/4}.
\end{equation}

Numerical evidence suggests that the finite field analogue of \eqref{anotherhf} is the following.

\begin{Conjecture}\label{lastconj}

Let $q\equiv 1\pmod{12}$ and $\eta$ be any order 12 character.   Then for  $z\in \F_q$ such that $z\neq 0$, $\pm 1$, and $z^2+3\neq 0$,
\begin{multline*}
\pFFq{2}{1}{\overline \eta&\eta^3}{&\phi}{\frac{2z(3+z^2)}{(1+z)^3}}\overset{?}=\\ \( \frac{1+\phi(1+z^2/3)}2\)\(\eta^3\( \frac{1+\sqrt{1+z^2/3}}{2(1+z)} \)+ \eta^3\( \frac{1-\sqrt{1+z^2/3}}{2(1+z)} \)\).
\end{multline*}

\end{Conjecture}

Related to  the triangle group $(2,3,3)$,  Vid{\=u}nas gives the following formula, see \cite[pp. 165]{Vidunas2},
\begin{equation}\label{vidunas}
     \pFq 21{\frac 14& -\frac 1{12}}{& \frac 23}{\frac{x(x+4)^3}{4(2x-1)^3}}=\(1-2x\)^{-1/4}.
\end{equation}
We observe numerically  the following $\F_q$ analogue:

\begin{Conjecture}\label{eq:Vidunas}
 For $q\equiv 1 \pmod{12}$ and $\eta$ any primitive order 12 character:
\begin{equation*}
    \pFFq{2}{1}{\overline \eta&\eta^3}{&\ol\eta^4}{\(\frac{4u(u^3+1)}{8u^3-1}\)^3}\overset{?}=2\cdot \ol\eta^3(1-8u^3),
\end{equation*}
   if $\({4u(u^3+1)}/(8u^3-1)\)^3\neq 0$, $1$, and $8u^3\neq 1$ in $\F_q$.

\end{Conjecture}
  Here, we replace  $x$ by $4u^3$ in \eqref{vidunas} so that the pattern for the $\F_q$ version is cleaner. In \cite{Vidunas2}, there are several other algebraic hypergeometric identities. It will be nice to obtain their $\F_q$ analogues.
We leave these to the interested reader.

\end{Remark}

\subsection{Another cubic $_2\F_1$ formula and a corollary}\label{cubic1}
In this section, we will use a different approach to derive a finite field analogue of the following cubic formula by Gessel and Stanton \cite[(5.18)]{Gessel-Stanton},
\begin{equation}\label{(5.18)Gessel-Stanton}\index{Gessel-Stanton cubic formula}
\pFq{2}{1}{a&-a}{&\frac 12}{\frac{27x(1-x)^2}4}=\, \pFq{2}{1}{3a&-3a}{&\frac 12}{\frac{3x}4},
\end{equation}which satisfies the ($\ast$) condition.
 The finite field analogue of this formula is stated as Theorem \ref{thm:cubic1} in {\color{black}the introduction.}
Then we will obtain a consequential $_3\F_2$ evaluation formula, namely Theorem \ref{thm4},  in \S\ref{subsub:evalapplication}.

\subsubsection{Another cubic $_2\F_1$ formula}
A proof of \eqref{(5.18)Gessel-Stanton} illustrating the ($\ast$) condition is discussed in  Appendix \S \ref{ss:13.3}.
Though  our dictionary method is applicable,  we choose to give an alternative proof below, which is applicable to all characters $A$.

%We now use a different approach to derive a finite field analogue of the following cubic formula by Gessel and Stanton \cite[(5.18)]{Gessel-Stanton},
%\begin{equation}\label{(5.18)Gessel-Stanton}\index{Gessel-Stanton cubic formula}
%\pFq{2}{1}{a&-a}{&\frac 12}{\frac{27x(1-x)^2}4}=\, \pFq{2}{1}{3a&-3a}{&\frac 12}{\frac{3x}4},
%\end{equation}which satisfies the ($\ast$) condition. A proof of \eqref{(5.18)Gessel-Stanton} illustrating the ($\ast$) condition is discussed in  Appendix \S \ref{ss:13.3}. \bl The finite field analogue of this formula is stated as Theorem \ref{thm:cubic1} in
%the introduction.  Though  our dictionary method is applicable,  we choose to give an alternative proof below. Then we will obtain a consequential $_3\F_2$ evaluation formula, namely Theorem \ref{thm4},  in \S\ref{subsub:evalapplication}. \bk

\begin{Theorem}\label{thm:cubic11}\index{finite field analogues!Gessel-Stanton cubic formula}
Let $q$ be an odd prime power, and $A\in \fqhat$. Then for all $x\in \F_q$,
\begin{multline*}
\pFFq{2}{1}{A&\ol{A}}{&\phi}{\frac{27x(1-x)^2}4}=\\ \pFFq{2}{1}{A^3&\ol{A^3}}{&\phi}{\frac{3x}4}-\phi(-3)\delta(x-1)-\phi(-3)A(-1)\delta(x-{4}/3).
\end{multline*}
 \end{Theorem}

\begin{proof} It is straightforward to check that the formula holds when $x=0,1,\frac 43$. We thus assume $x\neq 0,1,\frac 43$.
 Let $\l:=\frac{27x(1-x)^2}4$, and $z:=\frac \l{\l-1}=\frac{-3x}{4-3x}\(\frac{3x-3}{3x-1}\)^2$, so that $1-z=\frac{1}{1-\l}$. For any multiplicative character $A$, we obtain from Proposition \ref{prop: normalized independent solutions} that
  \begin{multline}\label{eq: cubic-A-trans}
\pFFq{2}{1}{A&\ol{A}}{&\phi}{\frac{27x(1-x)^2}4}\\
=\ol A(1-\l)\pFFq{2}{1}{A&\phi A}{&\phi}z+\delta(1-\l)\frac{J(\phi,\ol A)}{J(\ol A,\phi A)}.\\
%%=\ol A\(\frac 1{1-z}\)\pFFq{2}{1}{A&\phi A }{&\phi}z+\delta(1-\l)A(-1)\\
\end{multline}
Furthermore, if $A^2\neq\eps$, we have
$$
  \pFFq{2}{1}{A&\phi A}{&\phi}z=\(\frac{1+\phi(z)}2\)\(\ol A^2(1+\sqrt z)+\ol A^2(1-\sqrt z)\),
$$
which is given by Theorem \ref{thm:FF-Dihedral}.

 If $z$ is not a square, i.e. $\phi(z)=-1$, then the value $\pFFq{2}{1}{A&\phi A}{&\phi}z$ is $0$.
When $\phi(z)=1$, we write $\sqrt z=t\(\frac{t^2+3}{1+3t^2}\)$ with $T:=t^2=\frac{-3x}{4-3x}$. Then
$$
  1+\sqrt z= \frac{(1+t)^3}{3t^2+1}, \quad 1-\sqrt z= \frac{(1-t)^3}{3t^2+1},
$$
and thus
\begin{multline}\label{eq: cubic-A}
\ol A\(\frac 1{1-z}\)\pFFq{2}{1}{A&\phi A}{&\phi}z=\ol A^3\(\frac 1{1-t^2}\)\(\ol A^6(1+t)+\ol A^6(1-t )\)\\
=\ol A^3\(\frac 1{1-T}\)\(\ol A^6(1+\sqrt T)+\ol A^6(1-\sqrt T )\).
\end{multline}
Note that if $A$ is a primitive character of order $6$ and $\phi(z)=\pm 1$, then

\begin{multline*}
  \pFFq{2}{1}{A&\ol{A}}{&\phi}{\frac{27x(1-x)^2}4}=\(\frac{1+\phi(z)}2\)\cdot \phi\(\frac 1{1-T}\)\(\eps(1+\sqrt T)+\eps(1-\sqrt T)\)\\
  =\(\frac{1+\phi(z)}2\)\cdot \phi\(4-3x\) 2,
\end{multline*}
since $T$ won't be $1$. Moreover, when $\phi(z)\neq 0$, we can write
$$
   \phi(z)=\phi\(\frac{-3x}{4-3x}\(\frac{3x-3}{3x-1}\)^2\)=\phi(-3x)\phi(4-3x).
$$
Therefore,
\begin{align*}
  \pFFq{2}{1}{A&\ol{A}}{&\phi}{\frac{27x(1-x)^2}4} &=\(1+\phi(-3x)\phi(4-3x)\)\phi(4-3x)\\
  &=\phi(-3x)+\phi(4-3x),
\end{align*}
when $\phi(z)\neq 0$.

%%\begin{multline*}
%%  \pFFq{2}{1}{A&\ol{A}}{&\phi}{\frac{27x(1-x)^2}4}=\%%(\frac{1+\phi(z)}2\)\cdot \phi(1-T) 2\\
%%  =\(1+\phi(-3x)\phi(4-3x)\)\phi(4-3x).
%%\end{multline*}
According to Proposition \ref{prop: 2FF1-imprimitive}, for each $x$, one has
$$
\pFFq{2}{1}{A^3&\ol{A^3}}{&\phi}{\frac{3x}4}=\pFFq{2}{1}{\phi &\phi}{&\phi}{\frac{3x}4}
=\phi(4-3x)+\phi(-3x),
$$
and thus we can obtain
\begin{multline*}%\label{eq:63}
\pFFq{2}{1}{A&\ol{A}}{&\phi}{\frac{27x(1-x)^2}4}=\\ \pFFq{2}{1}{A^3&\ol{A^3}}{&\phi}{\frac{3x}4}-\phi(-3)\delta(x-1)-\phi(3)\delta(x-1/3).
\end{multline*}

Similarly, if $A$ is a primitive character of order $3$ and $\phi(z)=\pm 1$, then
\begin{multline*}
  \pFFq{2}{1}{A&\ol{A}}{&\phi}{\frac{27x(1-x)^2}4}=\(\frac{1+\phi(z)}2\)\cdot  2\eps(3x-4)\\
  =\(1+\phi(3x)\phi(3x-4)\)(1-\delta(3x-4))\\
  =1+\phi(3x)\phi(3x-4)-\delta(3x-4),
\end{multline*}
and
$$
\pFFq{2}{1}{A^3&\ol{A^3}}{&\phi}{\frac{3x}4}=\pFFq{2}{1}{\eps &\eps}{&\phi}{\frac{3x}4}
=1+\phi(3x)\phi(3x-4).
$$
This leads to the identity in the theorem.

%For the other cases, we just give the key ideas here and omit the  details.  For characters $A^2=\eps$, using the formulas of imprimitive ${}_2\mathbb F_1$-functions in  Proposition \ref{prop: 2FF1-imprimitive}, we can derive the identities. For characters $A$ with $A^6\neq \eps$, the identity follows  from  equation (\ref{eq: cubic-A}), Theorem \ref{thm:FF-Dihedral}, and the same formulas in Proposition \ref{prop: normalized independent solutions} that we use to observe (\ref{eq: cubic-A-trans}).

{{\color{red}} For the case $A^2=\eps$, using the formulas of imprimitive ${}_2\mathbb F_1$-functions in  Proposition \ref{prop: 2FF1-imprimitive}, we can derive the desired identities.} For example, when $A=\eps$,
\begin{multline*}
  \pFFq{2}{1}{A&\ol{A}}{&\phi}{\frac{27x(1-x)^2}4}=1-\phi(3x)\phi(3x-4)\eps((1-x)(3x-1))\\
  =1-\phi(3x)\phi(3x-4)+\delta(1-x)\phi(-3)+\delta(3x-1)\phi(-3).
\end{multline*}

For the case $A^6\neq \eps$, we first observe that  $\phi(z)=1$ (resp. $-1$) if and only if  $\phi(T)=1$ (resp. $-1$). Hence, if $\phi(z)\neq 0$ (equivalently, $\phi(T)\neq 0$), we have from equation (\ref{eq: cubic-A}) and  Theorem \ref{thm:FF-Dihedral} that
 \begin{multline*}
\ol A\(\frac 1{1-z}\)\pFFq{2}{1}{A&\phi A}{&\phi}z \\
=\(\frac{1+\phi(T)}2\)\ol A^3\(\frac 1{1-T}\)\(\ol A^6(1+\sqrt T)+\ol A^6(1-\sqrt T )\)\\
= A^3\({1-T}\)\pFFq{2}{1}{A^3&\phi A^3}{&\phi}T.
  \end{multline*}
Thus, for all $x$, we can derive that
 \begin{multline*}
\ol A\(\frac 1{1-z}\)\pFFq{2}{1}{A&\phi A}{&\phi}z= A^3\({1-T}\)\pFFq{2}{1}{A^3&\phi A^3}{&\phi}T\\
+\delta(x-1)\(1-\(\frac{1+\phi(T)}2\)\ol A^3\(\frac 1{1-T}\)\(\ol A^6(1+\sqrt T)+\ol A^6(1-\sqrt T )\)\)\\
+\delta(3x-1)\(0-\(\frac{1+\phi(T)}2\)\ol A^3\(\frac 1{1-T}\)\(\ol A^6(1+\sqrt T)+\ol A^6(1-\sqrt T )\)\)\\
= A^3\({1-T}\)\pFFq{2}{1}{A^3&\phi A^3}{&\phi}T-\delta(x-1)\phi(-3)-\delta(3x-1)\(1+\phi(-3)\)A(-1),
  \end{multline*}
  since $(1\pm \sqrt{-3})^6=64$, and  $(1\pm\sqrt{-1/3})^6=-64/27$.

Using Proposition \ref{prop: normalized independent solutions}, we observe that
 \begin{align*}
&\pFFq{2}{1}{A&\ol{A}}{&\phi}{\frac{27x(1-x)^2}4}\\
&\hspace*{0.2in}=\delta((3x-1)(3x-4))A(-1)-\delta(x-1)\phi(-3)-\delta(3x-1)\(1+\phi(-3)\)A(-1)\\
&\hspace*{0.5in}+A^3\({1-T}\)\(\ol A^3(1-T)\pFFq{2}{1}{A^3&\ol A^3}{&\phi}{\frac T{T-1}}+\delta(1-T)\cdot \pFFq{2}{1}{A^3&\phi {A^3}}{&\phi}{1}\)\\
&\hspace*{0.2in}=\delta(3x-4)A(-1)-\delta(x-1)\phi(-3)-\delta(3x-1)\phi(-3)A(-1)\\
&\hspace*{0.5in}+\eps\({1-T}\)\pFFq{2}{1}{A^3&\ol A^3}{&\phi}{\frac{3x}4}\\
&\hspace*{0.2in}=\pFFq{2}{1}{A^3&\ol A^3}{&\phi}{\frac{3x}4}-\delta(x-1)\phi(-3)-\delta(3x-1)\phi(-3)A(-1).
\end{align*}
%where $J$ is the evaluation of $\pFFq{2}{1}{A^3&\phi {A^3}}{&\phi}{x}$ at $x=1$.
\end{proof}

\subsubsection{An application to an evaluation formula}\label{subsub:evalapplication}
In the next few pages we will obtain a finite field analogue of an evaluation formula by Gessel and Stanton. Our main point for the discussion is as follows. In addition to what can be done for the cases satisfying ($\ast$), there are classical results obtained by methods that have no direct translations to finite fields.  However, sometimes the Galois perspective allows us to make predictions which can be verified numerically. To prove them, we need different methods that might appear to be  ad hoc when compared with the  more systematic approaches we have used  so far.

As a corollary of \eqref{(5.18)Gessel-Stanton},  Gessel and Stanton showed  that for $n\in \mathbb N$, $a\in \C$
\begin{multline}\label{gessel_cor}\pFq{3}{2}{1+3a&1-3a&-n}{&\frac 32& -1-3n}{\frac 34}=\frac{(1+a)_n(1-a)_n}{(\frac 23)_n(\frac43)_n}\\=\frac{B(1+a+n,\frac 23)B(1-a+n,\frac 43)}{B(1+a,n+\frac 23)B(1-a,n+\frac 43)}.
\end{multline}
This result follows from applying Theorem 1 of \cite{Gessel-Stanton} to \eqref{(5.18)Gessel-Stanton} with an additional factor of $x^{-n-2}(1-x)^{-2n-3}(1-3x)$, instead of the factor $x^{-n-2}(1-x)^{-3n-2}(1-3x)$ which is mentioned on page 305 of \cite{Gessel-Stanton}.  Although Greene gave a Lagrange inversion formula over finite fields (see Theorem \ref{FFlagrange} for a special case), he pointed out it cannot be used to determine coefficients when the change of variable function is not one-to-one. This means we cannot use Theorem \ref{FFlagrange} to obtain a finite field analogue of \eqref{gessel_cor} directly.

%As Theorem 1 of \cite{Gessel-Stanton} implies the Lagrange inversion formula, {\color{magenta}it has no analogue} over finite fields.

Thus to look for a finite field analogue of  \eqref{gessel_cor}, we instead observe that the corresponding Galois representations are expected to be 3-dimensional. Also, if $3a$ corresponds to a character $A^3$, then there are 3 candidates for the character analogue of $a$,  these being  $A$, $A\eta_3$, and $A\eta_3^2$.  This was the language we used when we initially  phrased the statement of Theorem \ref{thm4} and then verified numerically.

 We now give a proof of Theorem \ref{thm4}, which  states that if $q$ is a prime power with $q\equiv 1 \pmod 6$, and $A$, $\chiup,\eta_3\in \widehat{\F_q^\times}$ such that $\eta_3$ has order 3, and none of $A^6, \chiup^6,(A\chiup)^3,(\ol{A}\chiup)^3$ are the trivial character, then
$$
\pFFq{3}{2}{A^3&\ol{A}^3& \ol{\chiup}}{&\phi& \ol{\chiup}^3}{\frac 34}=  \sum_{\substack{B\in \widehat{\F_q^\times}\\ B^3=A^3}}\frac{J(B\chiup, \eta_3)J(\ol{B}\chiup,\ol{\eta_3})}{J(B,\chiup \eta_3)J(\ol{B},\chiup\ol{\eta_3})}= A(-1)\sum_{B^3=A^3}\frac{J(B\chiup, \ol B\chiup)}{J(\eta_3 \chiup,\chiup \ol\eta_3)}.
$$

\begin{proof}[Proof of Theorem \ref{thm4}]
Let $\eta:=\eta_3$ and $\F:=\F_q$ for convenience. Multiplying $\ol{\chiup}(x(1-x)^2)$ by the right hand side of Theorem \ref{thm:cubic11} and taking the sum over all $x$, we see
 \begin{multline*}
 \sum_{x\in \F} \ol{\chiup}(x(1-x)^2) \pFFq{2}{1}{A^3&\ol{A^3}}{&\phi}{\frac {3x}4}-\ol{\chiup}\(\frac 4{27}\)\phi(-3)A(-1)\\
=\frac 1{J(\ol{A}^3, \phi A^3)}\, _3 \mathbb P_2 \left [\begin{array}{ccc}A^3&\ol{A}^3&\ol{\chiup}\\&\phi&\ol{\chiup}^3 \end{array}; \frac 34\right ]-\ol{\chiup}\(\frac 4{27}\)\phi(-3)A(-1).
 \end{multline*}
 We now let
\begin{multline*}
I:= \pFFq{3}{2}{A^3&\ol{A}^3& \ol{\chiup}}{&\phi& \ol{\chiup}^3}{\frac 34}-\frac{\ol{\chiup}(\frac 4{27})\phi(-3)A(-1)}{J(\ol{\chiup}, \ol{\chiup}^2)}\\
= \sum_{x\in \F} \frac{\ol{\chiup}(x(1-x)^2)}{J(\ol{\chiup}, \ol{\chiup}^2)} \pFFq{2}{1}{A^3&\ol{A^3}}{&\phi}{\frac {3x}4}-\frac{\ol{\chiup}(\frac 4{27})\phi(-3)A(-1)}{J(\ol{\chiup}, \ol{\chiup}^2)}.
\end{multline*}
By Theorem \ref{thm:cubic11},
\begin{align*}
I &=\sum_{x\in \F} \frac{\ol{\chiup}(x(1-x)^2)}{J(\ol{\chiup}, \ol{\chiup}^2)} \pFFq{2}{1}{A&\ol{A}}{&\phi}{\frac {27x(1-x)^2}{4}}\\
 &= \sum_{x\in \F} \frac{\ol{\chiup}(x(1-x)^2) }{J(\ol{A}, \phi A)J(\ol{\chiup}, \ol{\chiup}^2)} \phgq A{\ol A}{\phi}{\frac {27x(1-x)^2}4}\\
 &= \sum_{x\in \F} \frac{\ol{\chiup}(x(1-x)^2) }{J(\ol{A}, \phi A)J(\ol{\chiup}, \ol{\chiup}^2)} \frac{A(-1)}{q-1} \sum_{\varphi \in \fqhat} J(A\varphi, \ol{\varphi}) J(\ol A\varphi, \phi\ol\varphi)\varphi\(\frac {27}4\)\varphi (x(1-x)^2) \\
  &= \frac{A(-1)}{(q-1)J(\ol{A}, \phi A)J(\ol{\chiup}, \ol{\chiup}^2)}\sum_{\varphi \in \fqhat}   {J(A\varphi, \ol{\varphi}) } J(\ol A\varphi, \phi\ol\varphi) \varphi\(\frac {27}4\)J(\ol\chiup\varphi, \ol\chiup^2\varphi^2).
\end{align*}
 Using Theorem \ref{Hasse Davenport},
we observe that when $\chiup^6\neq \eps$,
\begin{multline*}
  \frac{ J(\ol{\chiup} \varphi, \ol{\chiup}^2 \varphi^2)}{J(\ol \chiup,\ol\chiup^2)}
  =\frac{g(\ol\chiup^3)}{g(\ol\chiup)g(\ol\chiup^2)}\frac{g(\varphi\ol\chiup)g(\varphi^2\ol\chiup^2)}{g(\varphi^3\ol\chiup^3)}
   +\delta\(\varphi^3\ol\chiup^3\)\frac{q-1}{J(\ol\chiup,\ol\chiup^2)} \\
  = \frac{g(\eta\ol\chiup)g(\ol\eta\ol\chiup)}{g(\phi\ol\chiup)g(\ol\chiup)}
     \frac{g(\phi\varphi\ol\chiup)g(\varphi\ol\chiup)}{g(\varphi\ol\chiup\eta)g(\varphi\ol\chiup\ol\eta)}
     \varphi\(\frac 4{27}\)
     +\delta\(\varphi^3\ol\chiup^3\)\frac{q-1}q \frac{g(\eta\ol\chiup)g(\ol\eta\ol\chiup)g(\phi)}{g(\phi\ol\chiup)g(\ol\chiup)}\ol\chiup\(\frac{27}4\) .
\end{multline*}
Thus the sum $I$ can be divided into two parts, $I_1$ and $I_2$, where
$$
  I_1=\frac1{J(\ol{A}, \phi A)}\frac{A(-1)}{q} \sum_{\overset{\varphi\in \fqhat}{\varphi^3=\chi^3}}J(A\varphi, \ol{\varphi}) J(\ol{A}\varphi,  \phi\ol\varphi)
   \frac{g(\eta\ol\chiup)g(\ol\eta\ol\chiup)g(\phi)}{g(\phi\ol\chiup)g(\ol\chiup)}\ol\chiup\varphi\(\frac {27}4\)
$$
and
$$
  I_2=\frac1{J(\ol{A}, \phi A)}\frac{A(-1)}{q-1} \sum_{\varphi\in \fqhat}
  J(A\varphi, \ol{\varphi}) J(\ol{A}\varphi,  \phi\ol\varphi)\frac{g(\eta\ol\chiup)g(\ol\eta\ol\chiup)}{g(\phi\ol\chiup)g(\ol\chiup)}
     \frac{g(\phi\varphi\ol\chiup)g(\varphi\ol\chiup)}{g(\varphi\ol\chiup\eta)g(\varphi\ol\chiup\ol\eta)}.
$$
Then under our assumptions, we have
  \begin{multline*}
I_1=\frac{A(-1)}{q^2} \sum_{\varphi^3=\chiup^3}g(\eta\ol\chiup)g(\ol\eta\ol\chiup)g(A\varphi)g(\ol A\varphi)\frac{g(\ol\varphi)g(\phi\ol\varphi)}{g(\phi\ol\chiup)g(\ol\chiup)}\ol\chiup\varphi\(\frac {27}4\)\\
=\frac{A(-1)}q\sum_{\varphi^3=\chiup^3} J(A\varphi,\ol A\varphi )J(\eta\ol\chiup, \ol{\eta\chiup})\ol\chiup\varphi\({27}\)
= A(-1)\sum_{\varphi^3=\chiup^3} \frac{J(A\varphi,\ol A\varphi )}{J(\ol\eta\chiup, {\eta\chiup})}.
  \end{multline*}
  For the sum $I_2$, we will divide it into 3 parts. First, we observe that
  \begin{multline*}
     I_2=\frac1{J(\ol{A}, \phi A)}\frac{A(-1)}{q-1} \frac{g(\eta\ol\chiup)g(\ol\eta\ol\chiup)}{g(\phi\ol\chiup)g(\ol\chiup)}\\ \cdot \sum_{\varphi \in \fqhat}
  J(A\varphi, \ol{\varphi}) J(\ol{A}\varphi,  \phi\ol\varphi){g(\phi\varphi\ol\chiup)g(\varphi\ol\chiup)}{g(\ol\varphi\chiup\ol\eta)g(\ol\varphi\chiup\eta)} \\
 \cdot    \(\frac{\varphi\chiup(-1)}q+\frac{q-1}q\delta\(\chiup\ol{\varphi\eta}\)\)\(\frac{\varphi\chiup(-1)}q+\frac{q-1}q\delta\(\chiup\ol\varphi\eta\)\) \\
= \frac{A(-1)}{q-1}\frac1{J(\ol{A}, \phi A)} \frac{g(\eta\ol\chiup)g(\ol\eta\ol\chiup)}{g(\phi\ol\chiup)g(\ol\chiup)} g(\eta)g(\phi\ol\eta)\\ \cdot \sum_{\varphi}
  J(A\varphi, \ol{\varphi}) J(\ol{A}\varphi,  \phi\ol\varphi)J(\phi\varphi\ol\chiup,\ol\eta\chiup\ol\varphi)J(\ol\chiup\varphi,\chiup\eta\ol\varphi)\\
\cdot  \(\frac 1{q^2}+\frac{q-1}{q^2}\delta\(\chiup\ol{\varphi\eta}\)+\frac{q-1}{q^2}\delta\(\chiup\ol\varphi\eta\)\).
  \end{multline*}
 The first delta term gives us
  \begin{multline*}
  -\frac{A(-1)}{q^2}\frac{J(\ol{\eta}, \ol\eta)}{J(\phi,\ol\eta)}g(\eta\ol\chiup)g(\ol\eta\ol\chiup)g(A\chiup\ol\eta)g(\ol A\chiup\ol\eta)\frac{g(\eta\ol\chiup)g(\phi\eta\ol\chiup)}{g(\phi\ol\chiup)g(\ol\chiup)}\\
    =-\frac{A(-1)}{q^2}\frac{\ol\eta(4)J(\ol{\eta}, \ol\eta)}{J(\phi,\ol\eta)}J(\eta\ol\chiup,\ol\eta\ol\chiup)J(A\chiup\ol\eta,\ol A\chiup\ol\eta)g(\eta\chiup^2)g(\ol\eta\ol\chiup^2)\\
    =-A(-1)\frac{J(A\chiup\ol\eta,\ol A\chiup\ol\eta)}{J(\eta\chiup,\ol\eta\chiup)},
  \end{multline*}
 using \eqref{eq:double-Jac}.   The second delta term contributes
  $$
     -A(-1)\frac{J( A\eta\chiup,\ol A\eta\chiup)}{J(\eta\chiup,\ol\eta\chiup)}.
  $$
  Therefore, we have
  \begin{multline*}
  I= A(-1)\frac{J(A\chiup,\ol A\chiup)}{J(\eta\chiup,\ol\eta\chiup)}\\
   +\frac{\phi(-1)}{q^2}\frac{g(\eta)g(\phi\ol\eta)}{J(\ol A,\phi A)}\frac{g(\eta\ol\chiup)g(\ol\eta\ol\chiup)}{g(\ol\chiup)g(\phi\ol\chiup)}
     {}_4\mathbb P_3 \[\begin{array}{cccc}A&\ol{A}&\ol{\chiup} \phi& \ol{\chiup}\\ & \phi& \ol{\chiup}\eta& \ol{\chiup\eta} \end{array}; 1\].
\end{multline*}

 We are now in the position to evaluate $ {}_4\mathbb P_3 \[\begin{array}{cccc}A&\ol{A}&\ol{\chiup} \phi& \ol{\chiup}\\ & \phi& \ol{\chiup}\eta& \ol{\chiup\eta} \end{array}; 1\]$.  For this purpose, we observe that for $x\neq 0$,
 \begin{align*}
     \phgq {A}{\ol A}\phi{\frac 1x} \phgq {\eta}{\ol{\eta}}{\phi}{\frac 1x}
    =\frac{\phi(-1)}{q-1}\sum_\chi{}_4\mathbb P_3 \[\begin{array}{cccc}A&\ol{A}&\ol{\chiup} \phi& \ol{\chiup}\\ & \phi& \ol{\chiup}\eta& \ol{\chiup\eta} \end{array}; 1\]\ol \chiup(x),
  \end{align*}
  and by Proposition \ref{prop: Greene's independent solutions} and Corollary \ref{Cor: product of 2F1}, the left hand side of the above equals
\begin{multline*}
 A\eta(x)J(\phi A,\phi A)J(\phi \eta, \phi\eta)\pFFq{2}{1}{A&\phi{A}}{& A^2}{x} \pFFq{2}{1}{\eta&\phi{\eta}}{&\ol\eta}{x}\\
   =A\eta(x)J(\phi A,\phi A)J(\phi \eta, \phi\eta)\\
 \cdot  \(\pFFq{2}{1}{A\eta& \phi A\eta}{&(A\eta)^2}{x} +\eta(2x)\pFFq{2}{1}{A\ol\eta& \phi A\ol\eta}{&(A\ol\eta)^2}{x}-\delta(1-x) A\eta(4)\).
 \end{multline*}

 From \eqref{eq:double-Jac}  and the relation
$$
J(\phi,\phi \chi)=\phi(-1)J(\ol\chi, \phi), \mbox{ if } \chi\neq \eps,
$$
we obtain
\begin{multline*}
  A\eta(x)J(\phi A,\phi A)J(\phi \eta, \phi\eta)\pFFq{2}{1}{A\eta& \phi A\eta}{&(A\eta)^2}{x}\\
  = A\eta(x)\phi(-1)\frac{J(\phi,\ol A)J(\phi,\ol\eta)}{J(\phi,\ol {A\eta})}\phgq{A\eta}{\phi A\eta}{A^2\ol\eta}x\\
  =\frac{J(\phi,\ol A)J(\phi,\ol\eta)}{J(\phi,\ol {A\eta})}\frac{A(-1)}{q-1}
  \sum_{K\in \fqhat} J(A\eta \ol K,K)J(\phi A\eta\ol K, \ol A^2\eta K)\ol K A\eta(x)\\
  =\frac{J(\phi,\ol A)J(\phi,\ol\eta)}{J(\phi,\ol {A\eta})}\frac{A(-1)}{q-1}
  \sum_{\chi \in \fqhat} J( \ol \chi, A\eta \chi)J(\phi \ol \chi, \ol A\ol\eta \chi)\ol \chi(x),
  \end{multline*}
  \begin{multline*}
  A\eta(x)J(\phi A,\phi A)J(\phi \eta, \phi\eta)\eta(2x)\pFFq{2}{1}{A\ol\eta& \phi A\ol\eta}{&(A\ol\eta)^2}{x}\\
  =\frac{J(\phi,\ol A)J(\phi,\ol\eta)}{J(\phi,\ol A\eta)}\frac{A(-1)}{q-1}
  \sum_{\chi\in \fqhat} J( \ol \chi, A\ol\eta \chi)J(\phi \ol \chi, \ol A\eta \chi)\ol \chi(x),
\end{multline*}
and
$$
   A\eta(x)J(\phi A,\phi A)J(\phi \eta, \phi\eta)A\eta(4)\delta(1-x)
   =\frac 1{q-1}J(\phi, \ol A)J(\phi, \ol\eta)\sum_{\chi\in \fqhat}\ol\chi(x).
$$

Therefore, by comparing the coefficients for $\ol{\chi}$, we get
  \begin{multline*}
  \phi(-1){}_4\mathbb P_3 \[\begin{array}{cccc}A&\ol{A}&\ol{\chi} \phi& \ol{\chi}\\ & \phi& \ol{\chi}\eta_3& \ol{\chi\eta_3} \end{array}; 1\]\\
    =A(-1)J(\phi,\ol A)J(\phi,\ol\eta)\( \frac{J( \ol \chi, A\eta \chi)J(\phi \ol \chi, \ol A\ol\eta \chi)}{J(\phi,\ol{A\eta})}+\frac{J( \ol \chi, A\ol\eta \chi)J(\phi \ol \chi, \ol A\eta \chi)}{J(\phi,\ol A\eta)}\)\\
    -J(\phi,\ol A)J(\phi,\ol\eta),
  \end{multline*}
  and when $\chi^2\neq \eps$, this is equal to
  $$
    J(\phi,\ol A)J(\phi,\ol\eta)\( \chi(4)J(A\eta\chi,\ol{A\eta}\chi)+\chi(4)J(A\ol\eta\chi,\ol A\eta\chi)-1\).
  $$

Our claim follows from these results and the fact that $\phi(-3)=1$ when $p\equiv 1 \pmod 6$.
\end{proof}

\begin{Remark}Gessel and Stanton obtained in \cite{Gessel-Stanton} many other interesting evaluation formulas. Many of them appear to have finite field analogues.  Interested readers are encouraged to take a look.
\end{Remark}

\section{An application to Hypergeometric Abelian Varieties}\label{ss:application}

In this section, we give an explicit application of the use of finite field formulas in computing the arithmetic invariants of hypergeometric varieties.  Specifically, we use the finite field quadratic transformation from Theorem \ref{thm:quad-2F1} to obtain the decomposition of a generically 4-dimensional abelian variety arising naturally from the generalized Legendre curve $y^{12}=x^9(1-x)^5(1-\l x)$.

In \cite{WIN3a},  based on  \cite{Wolfart} by Wolfart and \cite{Archinard} by Archinard, \bk the authors use generalized Legendre curves (see \S \ref{GLC}) to construct families of 2-dimensional abelian varieties with quaternionic multiplication (QM) that are parametrized by Shimura curves associated with arithmetic triangle groups with compact fundamental domains. For  example, the  primitive part of the Jacobian varieties of (the smooth models of) $y^6=x^4(1-x)^3(1-\l x)$ gives rise to a family of 2-dimensional abelian varieties parametrized by the Shimura curve associated with the arithmetic triangle group (3,6,6).   Another construction of  such a  family of abelian varieties has been obtained by Petkova-Shiga \cite{Petkova-Shiga} using Picard curves. In general the strategy in \cite{WIN3a}  yields  computable families of generalized Legendre curves giving rise to  hypergeometric abelian varieties of dimension larger than 2. For instance, the arithmetic triangle group (2,6,6) can be realized using the periods of the generalized Legendre curve $y^{12}=x^9(1-x)^5(1-\l x)$. (See Examples \ref{(a,6,6)} and \ref{(a,6,6)-2} for relations between the Legendre curves and the arithmetic triangle groups (3,6,6) and (2,6,6).)   By the discussion in \S \ref{GLC}, \bk for any $\l\in {\Q}\setminus\{0,1\}$, the \bk primitive part of the Jacobian variety of the smooth model of $y^{12}=x^9(1-x)^5(1-\l x)$ is of dimension $\varphi(12)=4$. It is natural to ask whether such a 4-dimensional abelian variety is simple or not. Meanwhile the two arithmetic groups (3,6,6) and (2,6,6) are commensurable,  as seen in  \cite{Takeuchi} by Takeuchi. Here, we \bk will use Theorem \ref{thm:quad-2F1}, which is a quadratic formula for $_2\F_1$, to prove the following theorem.

\begin{Theorem}\label{thm:isogeny}Let $\l \in  \Q$ such that $\l \neq 0,\pm 1$. Let $J_{\l,1}^{\text{prim}}$ (resp. $J_{\frac{-4\l}{(1-\l)^2},2}^{\text{prim}}$) be the primitive part of the Jacobian variety of $y^6=x^4(1-x)^3(1-\l x)$ (resp. $y^{12}=x^9(1-x)^5(1+\frac{4\l}{(1-\l)^2}x)$). Then $J_{\frac{-4\l}{(1-\l)^2},2}^{\text{prim}}$ is isogenous to $J_{\l,1}^{\text{prim}}\oplus J_{\l,1}^{\text{prim}}$ over some {cyclotomic} number field depending on $\l$.
\end{Theorem}

The proof is based on the  following  famous  result of Faltings \cite{Faltings} (see \S 5, Kor. 1 to Satz 4).

\begin{Theorem}[Faltings \cite{Faltings}] \label{thm:Faltings}
Let $A$ and $B$ be abelian varieties over a number field $L$.
Suppose that the  corresponding $\ell$-adic representations  $\rho_{A,\ell}
\simeq \rho_{B,\ell}
$
as $\overline{\Q}_{\ell}[G_L]$-modules. Then $A$ is isogenous to $B$.
\end{Theorem}
It is well-known that  the $\rho_{A,\ell}$ and $\rho_{B,\ell}$ in Theorem \ref{thm:Faltings} are semisimple and that Frobenius elements for $\rho_{A,\ell}$ and $\rho_{B,\ell}$ form a dense set of $G_L$. \bk
 From this we have the following consequence.

\begin{Corollary}\label{cor:Faltings}  If $A$ and $B$ are two abelian varieties over a number field  $L$
%[changed $K$ to $L$ in this Cor. -L]\bk
such that their corresponding Galois representations $\rho_{A,\ell},\rho_{B,\ell}$  have the same trace for almost all Frobenius elements of $G_L$, then $A$ is isogenous to $B$ over $L$.
\end{Corollary}

\begin{proof}[Proof of Theorem \ref{thm:isogeny}]
 For any fixed $\l\in \Q, \l \neq 0,\pm 1$, \bk use $\{\rho_{1,\ell}\}$ (resp. $\{\rho_{2,\ell}\}$) to denote the compatible family of 4-dimensional  (resp. 8-dimensional) Galois representations of $G_{\Q}$ arising from  $J_{\l,1}^{\text{prim}}$ (resp. $J_{\frac{-4\l}{(1-\l)^2},2}^{\text{prim}}$) respectively. By Corollary \ref{cor:Faltings}  we  only need to find a finite extension $L$ of $\Q$ such that  the traces of $\rho_{2,\ell}|_{G_L}$ and $\(\rho_{1,\ell}\oplus\rho_{1,\ell}\)|_{G_L}$ agree at almost all Frobenius elements of $G_L$.

Let $\fp$ be a good prime ideal of the ring of integers of $\Q(\zeta_{12})$ with residue field  of size  $q$. Note that $q\equiv 1\pmod{12}$. Then the trace of the  Frobenius element $\text{Frob}_\fp$ under $\{\rho_{1,\ell}\}$ and $\{\rho_{2,\ell}\}$ can be computed using \eqref{eq:count-Jprime}.
All the fractions on the right sides below are negative of those given by
\eqref{eq:count-Jprime}.
\color{black}As the sets we are summing over are stable
under negation, our formulas are correct.
\bk
\begin{equation}\label{eq: pt 366}
 \text{Tr} \rho_{1,\ell}\( \text{Frob}_\fp\)=-\sum_{m=1,5}\phgq{ \iota_{\fp} (\frac{m}6)} { \iota_{\fp} (\frac{2m}6)}{ \iota_{\fp} (\frac{-m}6)} {\l;q},
\end{equation}
and
\begin{equation}\label{eq: pt 266}
 \text{Tr} \rho_{2,\ell} \(\text{Frob}_\fp\)=-\sum_{m=1,5,7,11}\phgq{ \iota_{\fp} (\frac{m}{12})} { \iota_{\fp} (\frac{3m}{12})}{ \iota_{\fp} (\frac{-2m}{12})} {\frac{-4\l}{(1-\l)^2};  q},
\end{equation}respectively.
(For the notation $\iota_{\fp} (\frac ab)$, see Definition \ref{def:iota}.) \bk Assume $\l\neq 0,\pm 1$ in the residue field. We now let $\eta_{12}=\iota_\fp ( \frac m{12})$ with $(m,12)=1$ be any  primitive multiplicative character of order $12$.
Then
$\eta_{12}^6=\phi$ is the unique quadratic character.
 By applying Theorem \ref{thm:quad-2F1} with $D=\eta_{12}$, $C=\eta_{12}^2$ and $B=\eta_{12}^4$ in the $_2{\mathbb F}_1$ expression, invoking the commutativity { in \S \ref{ffdef},} and using the defining formula (\ref{Hasse-Dav-special})
 for $_2{\mathbb P}_1$, we see for $\l \neq 0,\pm 1$ that
$$
  \phgq {\eta_{12}^2}{\eta_{12}^4}{\ol\eta_{12}^2}{\l}= \ol\eta_{12}^2(1-\l)\frac{J(\eta_{12}^4,\phi)}{J(\eta_{12}^3,\ol\eta_{12}^5)}\phgq{\eta_{12}}{\eta_{12}^3}{\ol\eta_{12}^2}{\frac{-4\l}{(1-\l)^2}}.
$$
Replacing all characters by their $5$th powers gives
\begin{align*}
  \phgq {\ol\eta_{12}^2}{\ol\eta_{12}^4}{\eta_{12}^2}{\l}%=&\eta_{12}^2(1-\l)\frac{J(\ol\eta_{12}^4,\eta_{12}^2)\eta_{12}^2(4)}{J(\eta_{12}^3,\ol\eta_{12}^5)}\phgq{\eta_{12}^5}{\eta_{12}^3}{\eta_{12}^2}{\frac{-4\l}{(1-\l)^2}}\\
 =&\eta_{12}^2(1-\l)\frac{J(\ol\eta_{12}^4,\phi)}{J(\eta_{12}^3,\ol\eta_{12})}\phgq{\eta_{12}^5}{\eta_{12}^3}{\eta_{12}^2}{\frac{-4\l}{(1-\l)^2}}.
\end{align*}

Let $R(\eta_{12}):=\frac{J(\eta_{12}^4,\phi)}{J(\eta_{12}^3,\ol\eta_{12}^5)}$. Next, we will argue that  $R(\eta_{12})$
is a root of unity. As $R(\eta_{12})$ and $R(\eta_{12}^5)$ are Galois conjugates in $\mathbb Q (\zeta_{12})$ it suffices to show $R(\eta_{12}^5)$ is a root of unity.

{
First we note that by the reflection formula \eqref{Gauss1}, we can show that  $$g(\eta_{12}^k)g(\eta_{12}^{12-k})=\eta_{12}(-1)^k q.$$  Thus further using \eqref{Gauss1} we see that

\begin{multline}
R(\eta_{12}^5)
=\frac{g(\eta_{12}^8)g(\phi)}{g(\eta_{12}^3)g(\eta_{12}^{11})}
=\frac{g(\phi)g(\eta_{12}^8)g(\eta_{12})}{\eta_{12}(-1)qg(\eta_{12}^3)} \\
=\eta_{12}(-1)\frac{g(\phi)g(\eta_{12})}{g(\eta_{12}^3)g(\eta_{12}^4)}
=\eta_{12}(-1)\frac{g(\phi)g(\eta_{12})}{g(\eta_{12}^3)g(\eta_{12}^4)}\frac{g(\eta_{12}^7)g(\eta_{12}^5)}{\eta_{12}(-1)q}.
\end{multline}

Now letting $A=\eta_{12}$ in the multiplication formula \eqref{Hasse-Dav-special} then using \eqref{Gauss1} and \eqref{JacobiGaussrelation}, we obtain
$$
\eta_{12}(4)R(\eta_{12}^5) = \frac{g(\phi)^2g(\eta_{12}^2)g(\eta_{12}^5)}{qg(\eta_{12}^3)g(\eta_{12}^4)}=
\frac{g(\eta_{12}^2)g(\eta_{12}^5)}{g(\eta_{12}^3)g(\eta_{12}^4)}=\frac{J(\eta_{12}^2,\eta_{12}^5)}{J(\eta_{12}^3,\eta_{12}^4)}.
$$
}

By Example \ref{eg:Yamamoto}, the value of $J(\eta_{12}^2,\eta_{12}^5)/J(\eta_{12}^3,\eta_{12}^4)$ is also a root of unity. From a more global view point following Weil's Theorem \ref{thm:Weil}, we know now the corresponding Gr\"ossencharacter $\mathcal J_{(\frac13,\frac 12)}/\mathcal J_{(\frac14,\frac 7{12})}$ has finite image. By class field theory, it corresponds to a finite order character $\psi$. Hence there is a natural number $M$, corresponding to the intersections of the kernels of $\( \frac{-(1-\l)^2}{\cdot} \)_{12}$ and $\psi$,  such that for each good prime $\fp$ whose   residue field  has cardinality $1$ modulo $M$, then  $\eta_{12}^{2}(1-\l)=1, \eta_{12}(-1)=1,$ and $\frac{J(\eta_{12}^4,\phi)}{J(\eta_{12}^3,\ol\eta_{12}^5)}=1$.

Set $L=\Q(\zeta_{12M})$.
%\bl [changed $J$ to $L$ as we called it $L$ earlier in the proof. -L] 
\bk Then for each good prime ideal $\fp$ of $\mathcal O_L$, its residue field has size $q\equiv 1\pmod{12M}$. In particular, letting $\eta_{12}=\iota_\fp(\frac{1}{12})$ and $\eta_{12}=\iota_\fp(\frac{7}{12})$, respectively, in the above equations \eqref{eq: pt 366} and \eqref{eq: pt 266}, respectively,
yields $$\text{Tr} \rho_{2,\ell}|_{G_L} \(\text{Frob}_\fp\)=2\text{Tr} \rho_{1,\ell}|_{G_L} \(\text{Frob}_\fp\).$$  The claim then follows from Corollary \ref{cor:Faltings}.
\end{proof}

%Set $J=\Q(\zeta_{12M})$. Then for each good prime ideal $\fp$ of $\mathcal O_J$, its residue field has size $q\equiv 1\pmod{12M}$. In particular, letting $\eta_{12}=\iota_\fp(\frac{1}{12})$ and $\eta_{12}=\iota_\fp(\frac{7}{12})$, respectively, in the above equations \eqref{eq: pt 366} and \eqref{eq: pt 266}, respectively,
%yields $$\text{Tr} \rho_{2,\ell}|_{G_J} \(\text{Frob}_\fp\)=2\text{Tr} \rho_{1,\ell}|_{G_J} \(\text{Frob}_\fp\).$$  The claim then follows from Corollary \ref{cor:Faltings}.

\section{Open Questions and Concluding Remarks}\label{ss:formulas}

We have seen now various ways in which the finite field hypergeometric series defined in \eqref{general_HGF} can be leveraged. By slightly changing the definitions given in \cite{Greene} by Greene and \cite{McCarthy} by McCarthy, we are able to accomplish our goals of aligning with the underlining geometry and matching the classical setting as closely as possible. This gives us a systematic method for converting many classical results to the finite field setting, with the benefit of predictions motivated by the Galois perspective.  \\

In particular, we have used our methods to prove finite field analogues of numerous classical formulas including 9 quadratic or higher transformation formulas, 11 evaluation formulas, and 3 algebraic identities, among other formulas.

In addition, we will discuss some numeric observations in this section.

\subsection{Numeric observations}\label{ss:conj}
%We first state a numeric observation made by Henri Cohen {\color{magenta} [Was this observation described in [74] or was the observation part of private communication? -H]} , which is related to the Calabi-Yau threefold labeled by $[4]+[3]-[2]-5[1]$ in \cite{RV2}. For any $p\equiv 1 \pmod{12}$, let $\eta_{12}$ denote a primitive order 12 character of $\F_p^\times$. Then
%\begin{equation}
%\pFFq{4}{3}{\eta_{12}^4&\eta_{12}^8&\eta_{12}^3,\eta_{12}^9}{&\eps&\eps&\eps}{1;\;p}\overset{?}=-J(\eta_{12}^4,\eta_{12}^4)^3-J(\eta_{12}^8,\eta_{12}^8)^3+\eta_{12}(-1)\cdot p.
%\end{equation}

 Numerically, we have observed finite field analogues of the algebraic transformations of $_2F_1$-hypergeometric series mentioned in work of  the fifth author and  Yang \cite{Tu-Yang1} using \texttt{Magma}. We have the following conjecture, for which we use the convention that a rational function $f(x)=p(x)/q(x)$ is said to take value $\infty$ at $x=a$ if $q(a)=0$ but $p(a)\neq0$.\bk

\begin{Conjecture}\label{conj:FFT-Y}
Let $p\equiv 1\pmod{24}$, {$\alpha$ a root of $x^2+3$ and $\beta$ a root of $x^2+2$ in $\F_p$,} and $\eta_{24}$ a primitive multiplicative character of order $24$.   Define
\begin{align*}
f(z) &:=\frac{12\alpha z(1-z)^2(1-9z^2)}{(1+\alpha z)^6},\\
g(z) &:=-\frac{4(1+\beta)^4z(1+(4\beta-7)z^2/3)^4}{(1+z)(1-3z)(1+(4+2\beta)z-(1+2\beta)z^2)^4},
\end{align*}
and assume that $z\in \F_p$ satisfies $f(z)$, $g(z)\neq 0$, $1$, $\infty$.  Then,
\begin{multline*}
  \eta_{24}^3((1+z)(1-3z)) \ol\eta_{24}^6(1+\alpha z)\pFFq21{\eta_{24}^5&\eta_{24}^9}{&\ol\eta_{24}^6}{f(z)}\\
  \overset{?}{=} \phi(1+(4+2\beta)z-(1+2\beta)z^2)\pFFq21{\eta_{24}^3&\eta_{24}^9}{&\ol\eta_{24}^6}{g(z)}.
\end{multline*}
\end{Conjecture}

 Conjecture \ref{conj:FFT-Y} is a finite field analogue of the following transformation\footnote{We note that the left hand side of the second formula in \cite[Theorem 1]{Tu-Yang1} contains a typo. We are stating the corrected version here.} due to  the fifth author and  Yang \cite{Tu-Yang1}.
\begin{Theorem}\cite[Theorem 1]{Tu-Yang1}\label{thm:Tu-Yang}
Let $\alpha$ be a root of $x^2+3$, $\beta$ a root of $x^2+2$, and define
 \begin{align*}
 f(z) & := \frac{12\alpha z(1-z)^2(1-9z^2)}{(1+\alpha z)^6}\\
 g(z)& := -\frac{4(1+\beta)^4z(1+(4\beta-7)z^2/3)^4}{(1+z)(1-3z)(1+(4+2\beta)z-(1+2\beta)z^2)^4}.
 \end{align*}
Then,
\begin{multline*}
  \frac{(1+z)^{1/8}(1-3z)^{1/8}}{(1+\alpha z)^{5/4}} \pFq21{\frac 5{24}&\frac38}{&\frac 34}{f(z)}\\
  =(1+(4+2\beta)z-(1+2\beta)z^2)^{-1/2}\pFq21{\frac 18&\frac38}{&\frac 34}{g(z)},
\end{multline*}
and
\begin{multline*}
  \frac{{(1-z)^{1/4}(1+3z)^{1/4}(1+z)^{5/8}(1-3z)^{5/8}}}{(1+\alpha z)^{11/4}}\pFq21{\frac {11}{24}&\frac58}{&\frac 54}{f(z)}\\
  =\frac{1+(4\beta-7)z^2/3}{(1+(4+2\beta)z-(1+2\beta)z^2)^{3/2}}\pFq21{\frac 38&\frac58}{&\frac 54}{g(z)}.
\end{multline*}
\end{Theorem}

 By \S \ref{sec:classcical 2F1-hyper}, the projective monodromy groups of the differential equations satisfied by $\pFq21{\frac 5{24}&\frac38}{&\frac 34}{z}$ and  $\pFq21{\frac {11}{24}&\frac58}{&\frac 54}{z}$ are both triangle groups. In fact they are isomorphic to two commensurable arithmetic triangle groups $(4,6,6)$ and $(4,4,4)$ respectively.
 In \cite{Tu-Yang1}, the authors obtain the transformations by interpreting hypergeometric series as modular forms on Shimura curves. In the case of the above theorem, all the modular forms  for  the arithmetic triangle group $(4,6,6)$ can be expressed in terms of
$$
  \pFq21{\frac 5{24}&\frac38}{&\frac 34}{t} \quad\mbox{and}\quad t^{1/4}\pFq21{\frac {11}{24}&\frac58}{&\frac 54}{t},
$$
and  for  the Shimura curve associated to $(4,4,4)$ can be expressed in terms of
$$
  \pFq21{\frac 18&\frac38}{&\frac 34}{u}  \quad\mbox{and}\quad u^{1/4}\pFq21{\frac 38&\frac58}{&\frac 54}{u},
$$
 where  both $t$ and $u$  are  suitable meromorphic modular functions.
All of them can be regarded as modular forms for the intersection group $\Gamma$ of $(4,6,6)$ and $(4,4,4)$, which is an arithmetic group  generated by six elliptic elements of order $4$ with a single relation. The algebraic transformations between these hypergeometric series come from the identities among the modular forms with respect to $\Gamma$.

We  have also observed numerically a finite field version of the second formula in Theorem \ref{thm:Tu-Yang}.  Namely,
\begin{multline}\label{eqn:conj2}
  \eta_{24}^6((1-z)(1+3z)(1+\alpha z))\ol\eta_{24}^9((1+z)(1-3z)) \pFFq21{\eta_{24}^{11}&\ol\eta_{24}^9}{&\eta_{24}^6}{f(z)}\\
 \overset{?}{=}\phi(1+(4+2\beta)z-(1+2\beta)z^2)\pFFq21{\eta_{24}^9&\ol\eta_{24}^9}{&\eta_{24}^6}{g(z)}.
\end{multline}
Equation \eqref{eqn:conj2} is equivalent to  Conjecture \ref{conj:FFT-Y}, which can be seen  by applying the first item of Proposition \ref{prop: normalized independent solutions} and the  multiplication formula in Theorem \ref{Hasse Davenport}. In fact this is predicted by the Galois perspective as the finite field versions correspond to the traces of two dimensional Galois representations.

 Furthermore, we have the following conjectures, based on numerical evidence computed using \texttt{Magma}, for $\F_q$ analogues of the algebraic transformations stated in Theorems 2, 4, 5 of  \cite{Tu-Yang1}.

\begin{Conjecture}\label{eqn:1conj}
Let $\eta_{20}$ be a primitive character of order $20$ and define $$f(z):=\frac{64z(1-z-z^2)^5}{(1-z^2)(1+4z-z^2)^5}.$$   Then when $f(z)\neq 0,1,\infty$, %\bl [remove `and $z\neq 0$,' as it is already included]\bk
\begin{multline*}\\
  \pFFq21{\eta_{20}&\eta_{20}^5}{&\ol\eta_{20}^4}{f(z)}
  %{\frac{64z(1-z-z^2)^5}{(1-z^2)(1+4z-z^2)^5}}
=\eta_{20}(1-z^2)\eta_{20}^5(1+4z-z^2)\pFFq21{\eta_{20}^6&\eta_{20}^8}{&\ol\eta_{20}^2}{z^2}.
\end{multline*}
\end{Conjecture}

\noindent Conjecture \ref{eqn:1conj} is an analogue of the following theorem.

\begin{Theorem}[Theorem 2, \cite{Tu-Yang1}] For $z\in \C$ such that both sides converge,\bk
\begin{multline*}
  \pFq21{\frac 1{20}&\frac14}{&\frac 45}{\frac{64z(1-z-z^2)^5}{(1-z^2)(1+4z-z^2)^5}}\\
  =(1-z^2)^{1/20}(1+4z-z^2)^{1/4}\pFq21{\frac3{10}&\frac25}{&\frac9{10}}{z^2};
\end{multline*}
\begin{multline*}
  (1-z-z^2)\pFq21{\frac 9{20}&\frac14}{&\frac 65}{\frac{64z(1-z-z^2)^5}{(1-z^2)(1+4z-z^2)^5}}\\=(1-z^2)^{1/4}(1+4z-z^2)^{5/4}\pFq21{\frac12&\frac25}{&\frac{11}{10}}{z^2}.
\end{multline*}
\end{Theorem}

\noindent  Next, we have the following conjecture.

\begin{Conjecture}\label{eqn:2conj}
Let $\eta_{6}$ be a primitive character of order $6$ and $A$ a character with $A^6\neq \eps$. Then when $z\neq \pm 1,\pm 3$,
\begin{multline*}
A\eta_{6}(1+z)A^3\phi(1-z/3)\pFFq21{A^2\eta_{6}^2&A\eta_{6}^2}{&A^3}{z^2}\\
=\pFFq21{A\eta_{6}&A\phi}{&A^2 }{\frac{16z^3}{(1+z)(3-z)^3}}.
\end{multline*}
\end{Conjecture}

\noindent  Conjecture \ref{eqn:2conj} is an analogue of the following theorem.

\begin{Theorem}[Theorem 4, \cite{Tu-Yang1}]
For a rational number $a$ such that neither $3a+1$ nor $2a+1$ is a nonpositive integer, in a neighborhood of $z=0$,
\begin{multline*}
  (1+z)^{a+1/6}(1-z/3)^{3a+1/2}\pFq21{2a+\frac13&a+\frac13}{&3a+1}{z^2}\\=\pFq21{a+\frac16&a+\frac12}{&2a+1}{\frac{16z^3}{(1+z)(3-z)^3}}.
\end{multline*}
\end{Theorem}

\noindent  Our final conjecture is stated below.

\begin{Conjecture}\label{eqn:3conj}
Let $\eta_{12}$ be a primitive character of order $12$ and $A$ a character with $A^{12}\neq \eps$. Then when neither $-\frac{27z^2(1-z)}{1-9z}$ nor $-\frac{64z^3}{(1-z)^3(1-9z)}$ equals $0$, $1$, or $\infty$,
\begin{multline*}
  A^9\eta_{12}^9(1-z)\pFFq21{A^4\eta_{12}^4&A^2\eta_{12}^4}{&A^6}{-\frac{27z^2(1-z)}{1-9z}}\\
 =A\eta_{12}(1-9z)\pFFq21{A^3\eta_{12}^3&A\eta_{12}^3}{&A^4 }{-\frac{64z^3}{(1-z)^3(1-9z)}}.
\end{multline*}
\end{Conjecture}

\noindent  Conjecture \ref{eqn:3conj} is an analogue of the following theorem.

\begin{Theorem}[Theorem 5, \cite{Tu-Yang1}]
For a real number $a$ such that neither $6a+1$ nor $4a+1$ is a non-positive integer, then in a neighborhood of $z=0$,
\begin{align*}
  (1-z)^{9a+3/4}&\pFq21{4a+\frac13&2a+\frac13}{&6a+1}{-\frac{27z^2(1-z)}{1-9z}}\\
  &=(1-9z)^{a+1/12}\pFq21{3a+\frac14&a+\frac14}{&4a+1}{-\frac{64z^3}{(1-z)^3(1-9z)}}.
\end{align*}
\end{Theorem}

\noindent Note that it is unclear whether the transformation theorems of \cite{Tu-Yang1} cited here  satisfy the ($\ast$) condition or not as they are proved using automorphic forms and the proofs do not have direct finite field translations. \bk

%\appendix
%  Include appendix "chapters" here.
\section{Appendix}\label{appendix}

In this appendix we address a few remaining topics not otherwise discussed in the bulk of this work. In \S \ref{appendix:Bailey} and \S \ref{ss:13.3} we offer alternate proofs and/or outlines of proofs for some classical results which demonstrate that the ($\ast$) condition is satisfied.
%We conclude {\color{magenta}in \S \ref{sec:Stanton}} with a discussion of an interesting question posed by Stanton in private communication.

\subsection{Bailey $_3F_2$ cubic transforms}\label{appendix:Bailey}
 Here we show that \eqref{Bailey2} and \eqref{Bailey} satisfy the ($\ast$) condition.  See Bailey \cite{Bailey2} for the original proofs of these transformations.

First observe that from \eqref{classicbinom} we have
\begin{multline}\label{eq:binom-Bailey}
\binom{-a-3i}{n-i}=(-1)^{n}\frac{(a)_n(-n)_i(a+n)_{2i}}{n!(a)_{3i}}\\=(-1)^{n}\frac{(a)_n(-n)_i((a+n)/2)_i((a+n+1)/2)_i}{n!(a/3)_{i}((a+1)/3)_{i}((a+2)/3)_{i}}\(\frac4{27}\)^i,
\end{multline}
\begin{multline}\label{eq:binom-Bailey2}
\binom{-a-3i}{n-2i}=(-1)^{n}\frac{(a)_n(a+n)_i(-n)_{2i}}{n!(a)_{3i}}\\=(-1)^{n}\frac{(a)_n(a+n)_i(-n/2)_{i}((1-n)/2)_{i}}{n!(a/3)_{i}((a+1)/3)_{i}((a+2)/3)_{i}}\(\frac4{27}\)^i.
\end{multline}

\begin{proof}[proof of \eqref{Bailey}]
Note that the left hand side of \eqref{Bailey} is
\begin{multline*}
\sum_{i\ge 0}\frac{\(\frac a3\)_i\(\frac {a+1}3\)_i\(\frac {a+2}3\)_i}{i!(b)_i\(a+\frac 32-b\)_i}\(\frac{-27}{4}\)^ix^{i}(1-x)^{-a-3i}\\
=\sum_{i\ge 0}\frac{\(\frac a3\)_i\(\frac {a+1}3\)_i\(\frac {a+2}3\)_i}{i!(b)_i\(a+\frac 32-b\)_i}\(\frac{-27}{4}\)^ix^{i}\sum_{k\ge0}\binom{-a-3i}{k}(-x)^k\\
\overset{n=i+k}=\sum_{n,i\ge 0,n\ge i}\frac{\(\frac a3\)_i\(\frac {a+1}3\)_i\(\frac {a+2}3\)_i}{i!(b)_i\(a+\frac 32-b\)_i}\(\frac{27}{4}\)^i\binom{-a-3i}{n-i}(-x)^n\\
\overset{\eqref{eq:binom-Bailey}}=\sum_{n\ge 0}\frac{(a)_n}{n!}\pFq{3}{2}{-n&(a+n)/2&(a+n+1)/2}{&b&a+3/2-b}{1}x^n.
\end{multline*}
Applying the Pfaff-SaalSch\"utz formula \eqref{eq:pf-s}, this becomes
\begin{multline*}
\sum_{n\ge 0}\frac{(a)_n}{n!}\frac{(b-(a+n)/2)_n(b-(a+n+1)/2)_n}{(b)_n(b-a-n-1/2)_n}x^n\\
\overset{\eqref{eq:double-rising}}= \sum_{n}\frac{(a)_n}{n!}\frac{(2b-a-n-1)_{2n}}{(b)_n(b-a-n-1/2)_n}\(\frac x4\)^n\\
= \sum_{n}\frac{(a)_n\G(2b-a+n-1)\G(b-a-n-1/2)}{n!(b)_n\G(2b-a-n-1)\G(b-a-1/2)}\(\frac x4\)^n\\
= \sum_{n}\frac{(a)_n(2b-a-1)_n\G(2b-a-1)\G(b-a-n-1/2)}{n!(b)_n\G(2b-a-n-1)\G(b-a-1/2)}\(\frac x4\)^n\\
\overset{\text{Thm.} \ref{thm:reflect}}= \sum_{n}\frac{(a)_n(2b-a-1)_n\G(2-2b+a+n)\G(3/2-b+a)}{n!(b)_n\G(2-2b+a)\G(3/2-b+a+n)}\(\frac x4\)^n,
\end{multline*}
which agrees with the right hand side of \eqref{Bailey}.
\end{proof}

To prove \eqref{Bailey2} we note that one of $-n/2$ and $(1-n)/2$ is a nonpositive integer when $n\ge 0$.  Thus using \eqref{eq:binom-Bailey2} the proof of \eqref{Bailey2} follows similarly to that of \eqref{Bailey}.

\subsection{A proof of a formula by Gessel and Stanton}\label{ss:13.3}
Recall that equation \eqref{(5.18)Gessel-Stanton} states that
\begin{equation*}
\pFq{2}{1}{a&-a}{&\frac 12}{\frac{27x(1-x)^2}4}=\, \pFq{2}{1}{3a&-3a}{&\frac 12}{\frac{3x}4}.
\end{equation*}
 We give a proof here which demonstrates that \eqref{(5.18)Gessel-Stanton} satisfies the $(\ast)$ condition.

\begin{proof}[proof of \eqref{(5.18)Gessel-Stanton}]
Note that using the inversion formula \eqref{eq:Slater-Inversion} one has
\begin{equation}\label{lem:GS-cubic}
\binom{2n-2i}{i}  =\frac{(-1)^i}{i!}\frac{(-2n)_{3i}}{(-2n)_{2i}}=(-1)^i \frac{27^i}{4^i \cdot i!}\frac{(\frac{-2n}3)_i(\frac{1-2n}3)_i(\frac{2-2n}3)_i}{(-n)_i(-n+\frac 12)_i}.\end{equation}
Observe that
\begin{multline*}
\pFq{2}{1}{a &-a}{ & \frac 12}{\frac{27x(1-x)^2}4} =\sum_{k\ge 0}\frac{(a)_k(-a)_k}{(1)_k(\frac 12)_k}\(\frac{27x}4\)^{k}\sum_{i\ge 0}\binom{2k}{i}(-x)^i\\
\overset{n=k+i} =\sum_{n,i\ge 0}\frac{(a)_{n-i}(-a)_{n-i}}{(1)_{n-i}(\frac 12)_{n-i}}\(\frac{27}4\)^{n-i}\binom{2n-2i}{i} (-1)^{i} x^n\\
=\sum_{n \ge 0}\(\frac{27}4\)^{n}\frac{(a)_{n}(-a)_{n}}{(1)_{n}(\frac 12)_{n}}
\(\sum_{i}\frac{(-n)_{i}(\frac 12-n)_{i}}{(1-a-n)_{i}(1+a-n)_{i}}\(-\frac{27}4\)^{-i}\binom{2n-2i}{i} \) x^n.
\end{multline*}
By \eqref{lem:GS-cubic}, the above equals
\begin{multline*}
\sum_{n\ge 0}\frac{(a)_n(-a)_n}{(1)_n(\frac 12)_n}\pFq{3}{2}{\frac{-2n}3 & \frac{1-2n}3 & \frac{2-2n}3 }{ & 1+a-n & 1-a-n}{1}\(\frac{27x}4\)^n\\
\overset{\eqref{eq:pf-s2}}{=}\sum_{n\ge 0}\frac{(a)_n(-a)_n}{(1)_n(\frac 12)_n}\frac{\G(a+\frac {n+2}3)\G(a+\frac {n+1}3)\G(a+\frac n3)\G(1+a-n)}{\G(1+a-\frac n3)\G(a+\frac {2-n}3)\G(a+\frac {1-n}3)\G(a+n)}\(\frac{27x}4\)^n.\end{multline*}
By Theorem \ref{thm:multiplication}, we thus have
\begin{align*}
\pFq{2}{1}{a &-a}{ & \frac 12}{\frac{27x(1-x)^2}4} &= 3\sum_{n\ge 0}\frac{(a)_n(-a)_n}{(1)_n(\frac 12)_n}\frac{\G(3a+n)\G(1+a-n)}{\G(1+3a-n)\G(a+n)}\(\frac{3x}4\)^n\\
&=3\sum_{n\ge 0}\frac{(3a)_n(-a)_n}{(1)_n(\frac 12)_n}\frac{\G(1+a-n)\G(3a)}{\G(1+3a-n)\G(a)}\(\frac{3x}4\)^n\\
&=\sum_{n\ge 0}\frac{(3a)_n(-a)_n}{(1)_n(\frac 12)_n}\frac{\G(1+a-n)\G(3a+1)}{\G(1+3a-n)\G(a+1)}\(\frac{3x}4\)^n\\
&\overset{\text{Thm.} \ref{thm:reflect}}= \sum_{n\ge 0}\frac{(3a)_n(-a)_n}{(1)_n(\frac 12)_n}\frac{\G(-3a+n)\G(-a)}{\G(-3a)\G(-a+n)}\(\frac{3x}4\)^n\\
& = \pFq{2}{1}{3a&-3a}{&\frac 12}{\frac{3x}4}
\end{align*}
as desired.
\end{proof}

%\backmatter

%\printindex

\printindex


\begin{thebibliography}{999}

\bibitem{AS1}
Alan Adolphson and Steven Sperber.
\newblock On twisted exponential sums.
\newblock {\em Math. Ann.}, 290(4):713--726, 1991.

\bibitem{AS2}
Alan Adolphson and Steven Sperber.
\newblock Twisted exponential sums and {N}ewton polyhedra.
\newblock {\em J. Reine Angew. Math.}, 443:151--177, 1993.

\bibitem{AO}
Scott Ahlgren and Ken Ono.
\newblock A Gaussian hypergeometric series evaluation and {A}p\'ery number congruences.
\newblock {\em Journal fur die Reine und Angewandte Mathematik}, 187--212, 2000.

\bibitem{Ahlgren}
Scott Ahlgren, Gaussian hypergeometric series and combinatorial congruences, \emph{Symbolic computation, number theory, special functions, physics and combinatorics} (Gainesville, FL, 1999) Dev. Math., vol. 4, Kluwer Acad. Publ., Dordrecht,  pp. 1--12, 2001.

\bibitem{AOP}
Scott Ahlgren, Ken Ono, and David Penniston.
\newblock Zeta functions of an infinite family of {$K3$} surfaces.
\newblock {\em Amer. J. Math.}, 124(2):353--368, 2002.

\bibitem{AAR}
George~E. Andrews, Richard Askey, and Ranjan Roy.
\newblock {\em Special functions}, volume~71 of {\em Encyclopedia of
  Mathematics and its Applications}.
\newblock Cambridge University Press, Cambridge, 1999.

\bibitem{Andrews-Stanton}
George~E. Andrews and Dennis~W. Stanton.
\newblock Determinants in plane partition enumeration.
\newblock {\em European J. Combin.}, 19(3):273--282, 1998.


\bibitem{Archinard}
Nat{\'a}lia Archinard.
\newblock Hypergeometric abelian varieties.
\newblock {\em Canad. J. Math.}, 55(5):897--932, 2003.

\bibitem{Archinard-exceptional}
Nat{\'a}lia Archinard.
Exceptional sets of hypergeometric series.
\emph{J. Number Theory,} 101(2):244--269, 2003.

\bibitem{Bailey2}
Wilfrid~Norman Bailey. Products of generalized hypergeometric series, \emph{Proc. London Math. Soc. Ser.}
2, 28, 242--254, 1928.

\bibitem{Bailey}
Wilfrid~Norman Bailey.
\newblock {\em Generalized hypergeometric series}.
\newblock Cambridge Tracts in Mathematics and Mathematical Physics, No. 32.
  Stechert-Hafner, Inc., New York, 1964.

\bibitem{BD}
Francesco Baldassarri and Bernard Dwork.
\newblock On second order linear differential equations with algebraic
  solutions.
\newblock {\em Amer. J. Math.}, 101(1):42--76, 1979.

\bibitem{Barman}
Rupam Barman and Gautam Kalita.
\newblock Hypergeometric functions and a family of algebraic curves.
\newblock {\em Ramanujan J.}, 28(2):175--185, 2012.

\bibitem{Berndt-Evans-Williams}
Bruce~C. Berndt, Ronald~J. Evans, and Kenneth~S. Williams.
\newblock {\em Gauss and {J}acobi sums}.
\newblock Canadian Mathematical Society Series of Monographs and Advanced
  Texts. John Wiley \& Sons, Inc., New York, 1998.
\newblock A Wiley-Interscience Publication.

\bibitem{Beukers}
Frits Beukers.
\newblock Notes of differential equations and hypergeometric functions.
\newblock unpublished notes.


\bibitem{BCM}
Frits Beukers, Henri Cohen, and Anton Mellit, Finite hypergeometric functions, \emph{Pure Appl. Math. Q.} \textbf{11} (2015), no.~4, 559--589.


\bibitem{BH}
Frits Beukers and Gert Heckman.
\newblock Monodromy for the hypergeometric function {$_nF_{n-1}$}.
\newblock {\em Invent. Math.}, 95(2):325--354, 1989.

\bibitem{Borweins}
Jonathan~M. Borwein and Peter~B. Borwein.
\newblock A cubic counterpart of {J}acobi's identity and the {AGM}.
\newblock {\em Trans. Amer. Math. Soc.}, 323(2):691--701, 1991.

\bibitem{BB}
Jonathan~M. Borwein and Peter~B. Borwein.
\newblock {\em Pi and the {AGM}}.
\newblock Canadian Mathematical Society Series of Monographs and Advanced
  Texts, 4. John Wiley \& Sons, Inc., New York, 1998.
\newblock A study in analytic number theory and computational complexity,
  Reprint of the 1987 original, A Wiley-Interscience Publication.

\bibitem{WIN2}
Sarah Chisholm, Alyson Deines, Ling Long, Gabriele Nebe, and Holly Swisher.
\newblock $p-$adic analogues of ramanujan type formulas for $1/\pi$.
\newblock {\em Mathematics}, 1(1):9--30, 2013.

\bibitem{CC}
D.~V. Chudnovsky and G.~V. Chudnovsky.
\newblock Approximations and complex multiplication according to {R}amanujan.
\newblock In {\em Ramanujan revisited ({U}rbana-{C}hampaign, {I}ll., 1987)},
  pages 375--472. Academic Press, Boston, MA, 1988.

\bibitem{Cohen1}
Henri Cohen.
\newblock {\em Number theory. Vol. I. Tools and Diophantine equations}.
\newblock Graduate Texts in Mathematics, 239. Springer, New York,  2007.


\bibitem{WIN3a}
Alyson Deines, Jenny~G. Fuselier, Ling Long, Holly Swisher, and Fang~Ting Tu.
\newblock Generalized {L}egendre {C}urves and {Q}uaternionic {M}ultiplication.
\newblock \emph{J. Number Theory} 161, 175--203, 2016.

\bibitem{WIN3b}
Alyson Deines, Jenny~G. Fuselier, Ling Long, Holly Swisher, and Fang~Ting Tu.
\newblock {Hypergeometric series, truncated hypergeometric series, and Gaussian
  hypergeometric functions}.
\newblock {\em Directions in Number Theory: Proceedings of the 2014 WIN3}, 125--160,  2016.


\bibitem{drew}Andrew Sutherland, {\em Modular Polynomials},
\mbox{\url{https://math.mit.edu/~drew/ClassicalModPolys.html}}

\bibitem{Evans-Greene}
Ronald~J. Evans and John Greene.
\newblock Clausen's theorem and hypergeometric functions over finite fields.
\newblock {\em Finite Fields Appl.}, 15(1):97--109, 2009.

\bibitem{Evans-Greene2}
Ronald~J. Evans and John Greene.
\newblock Evaluations of hypergeometric functions over finite fields.
\newblock {\em Hiroshima Math. J.}, 39(2):217--235, 2009.

\bibitem{Evans-Greene3}
Ronald~J. Evans and John Greene.
\newblock A quadratic hypergeometric 2F1 transformation over finite fields.
\newblock  	\emph{Proc. Amer. Math. Soc.}, 145(3):1071--1076, 2017.

\bibitem{Evans81}
Ronald~J. Evans.
\newblock Identities for products of {G}auss sums over finite fields.
\newblock {\em Enseign. Math. (2)}, 27(3-4):197--209, 1982.

\bibitem{Evans83}
Ronald~J. Evans.
\newblock Character sum analogues of constant term identities for root systems.
\newblock {\em Israel J. Math.}, 46(3):189--196, 1983.

\bibitem{Evans86}
Ronald~J. Evans.
\newblock Hermite character sums.
\newblock {\em Pacific J. Math.}, 122(2):357--390, 1986.

\bibitem{Evans91}
Ronald~J. Evans.
\newblock Character sums over finite fields.
\newblock In {\em Finite fields, coding theory, and advances in communications
  and computing ({L}as {V}egas, {NV}, 1991)}, volume 141 of {\em Lecture Notes
  in Pure and Appl. Math.}, pages 57--73. Dekker, New York, 1993.



\bibitem{Erdelyi}
Arthur Erd\'elyi, Wilhelm Magnus, Fritz Oberhettinger, and Francesco~G.
  Tricomi.
\newblock {\em Higher transcendental functions. {V}ol. {I}}.
\newblock Robert E. Krieger Publishing Co., Inc., Melbourne, Fla., 1981.
\newblock Based on notes left by Harry Bateman, With a preface by Mina Rees,
  With a foreword by E. C. Watson, Reprint of the 1953 original.


\bibitem{Faltings}
Gerd Faltings.
\newblock Endlichkeitss\"{a}tze f\"{u}r abelsche Variet\"{a}ten
\"{u}ber Zahlk\"{o}rpern.
\newblock {\em Invent. Math.}, (73)  no. 3, 349--366, 1983.


\bibitem{FOP}
Sharon Frechette, Ken Ono, and Matthew Papanikolas.
\newblock Gaussian hypergeometric functions and traces of {H}ecke operators.
\newblock {\em Int. Math. Res. Not.}, (60):3233--3262, 2004.

\bibitem{FST}Sharon Frechette, Holly Swisher, and  Fang-Ting Tu. A cubic transformation formula for Appell-Lauricella hypergeometric functions over finite fields. \emph{Res. Number Theory} 4, no. 2, Paper No. 27, 27 pp. (2018).
\bibitem{Fuselier}
Jenny~G. Fuselier.
\newblock Hypergeometric functions over {$\mathbb F_p$} and relations to elliptic
  curves and modular forms.
\newblock {\em Proc. Amer. Math. Soc.}, 138(1):109--123, 2010.

\bibitem{Gessel-Stanton}
Ira Gessel and Dennis Stanton.
\newblock Strange evaluations of hypergeometric series.
\newblock {\em SIAM J. Math. Anal.}, 13(2):295--308, 1982.

\bibitem{Good}
Irving J. Good.
\newblock Generalizations to several variables of {L}agrange's expansion, with
  applications to stochastic processes.
\newblock {\em Proc. Cambridge Philos. Soc.}, 56:367--380, 1960.

\bibitem{Goursat}
{\'E}douard Goursat.
\newblock Sur l'\'equation diff\'erentielle lin\'eaire, qui admet pour
  int\'egrale la s\'erie hyperg\'eom\'etrique.
\newblock {\em Ann. Sci. \'Ecole Norm. Sup. (2)}, 10:3--142, 1881.

\bibitem{Greene}
John Greene.
\newblock Hypergeometric functions over finite fields.
\newblock {\em Trans. Amer. Math. Soc.}, 301(1):77--101, 1987.

\bibitem{Greene87}
John Greene.
\newblock Lagrange inversion over finite fields.
\newblock {\em Pacific J. Math.}, 130(2):313--325, 1987.

\bibitem{Greene93}
John Greene.
\newblock Hypergeometric functions over finite fields and representations of
  {${\rm SL}(2,q)$}.
\newblock {\em Rocky Mountain J. Math.}, 23(2):547--568, 1993.

\bibitem{Greene-Stanton}
John Greene and Dennis Stanton.
\newblock A character sum evaluation and {G}aussian hypergeometric series.
\newblock {\em J. Number Theory}, 23(1):136--148, 1986.

\bibitem{Helversen-Pasotto}
Anna P. Helversen-Pasotto.  L'identit\'e de Barnes pour les corps finis. (French) \emph{C. R. Acad. Sci. Paris S\'er}. A-B 286, no. 6, A297--A300, 1978.

\bibitem{IR}
Kenneth Ireland and Michael Rosen.
\newblock {\em A classical introduction to modern number theory}, volume~84 of
  {\em Graduate Texts in Mathematics}.
\newblock Springer-Verlag, New York, second edition, 1990.

\bibitem{Katz}
Nicholas~M. Katz.
\newblock {\em Exponential sums and differential equations}, volume 124 of {\em
  Annals of Mathematics Studies}.
\newblock Princeton University Press, Princeton, NJ, 1990.

\bibitem{Klein}
Felix Klein.
\newblock {\em Vorlesungen \"uber die hypergeometrische {F}unktion}, volume~39
  of {\em Grundlehren der Mathematischen Wissenschaften [Fundamental Principles
  of Mathematical Sciences]}.
\newblock Springer-Verlag, Berlin-New York, 1981.
\newblock Reprint of the 1933 original.


\bibitem{Koblitz}
Neal Koblitz.
\newblock The number of points on certain families of hypersurfaces over finite
  fields.
\newblock {\em Compositio Math.}, 48(1):3--23, 1983.

\bibitem{Koblitz-book-p-adic}Neal Koblitz.  \emph{$p$-adic numbers, $p$-adic analysis, and zeta-functions}. Second edition. Graduate Texts in Mathematics, 58. Springer-Verlag, New York, 1984. xii+150 pp.

\bibitem{Koike}
Masao Koike.
\newblock Hypergeometric series over finite fields and {A}p\'ery numbers.
\newblock {\em Hiroshima Math. J.}, 22(3):461--467, 1992.

\bibitem{KZ}
Maxim Kontsevich and Don Zagier.  {Periods}. \emph{Mathematics unlimited - 2001 and beyond}, 771--808, Springer, Berlin, 2001.

\bibitem{Kummer}
Ernst Kummer.
\newblock \"{U}ber die hypergeometrische Reihe (Fortsetzung).
\newblock {Journal f\"{u}r die reine und angewandte Mathematik.}, 15: 127--172, 1836.

\bibitem{Lagrange}
J.~L. Lagrange.
\newblock Nouvelle m\'{e}thode pour r\'{e}soudre des \'{e}quations
  litt\'{e}rales par le moyen de s\'{e}ries.
\newblock {\em M\'{e}m. Acad. Roy. des Sci. et Belles-Lettres de Berlin}, 24,
  1770.

\bibitem{Lang}
Serge Lang.
\newblock {\em Cyclotomic fields {I} and {II}}, volume 121 of {\em Graduate
  Texts in Mathematics}.
\newblock Springer-Verlag, New York, second edition, 1990.
\newblock With an appendix by Karl Rubin.

\bibitem{Lennon-count}
Catherine Lennon.
\newblock Gaussian hypergeometric evaluations of traces of {F}robenius for
  elliptic curves.
\newblock {\em Proc. Amer. Math. Soc.}, 139(6):1931--1938, 2011.

\bibitem{Lennon2}
Catherine Lennon. Trace formulas for Hecke operators, Gaussian hypergeometric functions, and the modularity of a threefold. \emph{J. Number Theory} 131(12):2320--2351, 2011.

\bibitem{Li-Barnes-II}
Wen-Ching Winnie  Li. Barnes' identities and representations of GL(2). II. Non-Archimedean local field case. \emph{J. Reine Angew.} Math. 345:69--92, 1983.

\bibitem{LLT2} Wen-Ching Winnie  Li, Ling Long, and  Fang-Ting Tu. A Whipple $_7F_6$ formula revisited. \emph{Matematica} 1, no. 2, 480--530, 2022.

\bibitem{L-SA} Wen-Ching Winnie  Li and Jorge Soto-Andrade.  Barnes' identities and representations of GL(2). I. Finite field case. \emph{J. Reine Angew. Math.} 344, 171--179, 1983.

\bibitem{Lin-Tu}
Yi-Hsuan Lin and Fang-Ting  Tu.  Twisted Kloosterman sums. \emph{J. Number Theory} 147, 666--690, 2015.

\bibitem{lmfdb:144.a3}
The {LMFDB Collaboration}.
\newblock {\itshape {The L-functions and Modular Forms Database, {\em Home page
  of the Elliptic Curve 144.a3}}}.
\newblock \mbox{\url{http://www.lmfdb.org/EllipticCurve/Q/144/a/3}}, 2013.
\newblock [Online; accessed 6 October 2015].

\bibitem{lmfdb:288.d3}
The {LMFDB Collaboration}.
\newblock {\itshape {The L-functions and Modular Forms Database, {\em Home page
  of the Elliptic Curve 288.d3}}}.
\newblock \mbox{\url{http://www.lmfdb.org/EllipticCurve/Q/288/d/3}}, 2013.
\newblock [Online; accessed 6 October 2015].

\bibitem{lmfdb:64.a3}
The {LMFDB Collaboration}.
\newblock {\itshape {The L-functions and Modular Forms Database, {\em Home page
  of the Elliptic Curve 64.a3}}}.
\newblock \mbox{\url{http://www.lmfdb.org/EllipticCurve/Q/64/a/3}}, 2013.
\newblock [Online; accessed 6 October 2015].

\bibitem{Long04} Ling Long. On Shioda-Inose structures of one-parameter families of K3 surfaces.
\emph{J. Number Theory} 109, no. 2, 299–318, 2004.







\bibitem{Long}
Ling Long.
\newblock Hypergeometric evaluation identities and supercongruences.
\newblock {\em Pacific J. Math.}, 249(2):405--418, 2011.

\bibitem{Long18}
Ling Long. Some numeric hypergeometric supercongruences. \emph{Vertex operator algebras, number theory and related topics}, 139–-156, Contemp. Math., 753, Amer. Math. Soc. 2020.

\bibitem{Long-Ramakrishna}
Ling Long and Ravi Ramakrishna.
\newblock Some supercongruences occurring in truncated hypergeometric series.
\newblock \emph{Adv. Math.} 290, 773--808, 2016.
%\newblock arXiv:1403.5232.

\bibitem{LTYZ}
Ling {Long}, Fang-Ting {Tu}, Noriko {Yui}, and Wadim {Zudilin}.
\newblock {Supercongruences for rigid hypergeometric Calabi--Yau threefolds}.
\newblock {\em Adv. in Math.}, no. 393 Article 108058, 2021. %updated 9-24-22

\bibitem{McCarthy}
Dermot McCarthy.
\newblock Transformations of well-poised hypergeometric functions over finite
  fields.
\newblock {\em Finite Fields Appl.}, 18(6):1133--1147, 2012.

\bibitem{McCarthy-p-adic}
Dermot McCarthy. Extending Gaussian hypergeometric series to the $p$-adic setting
\emph{International J. of Number Theory} 8(7):1581--1612, 2012.

\bibitem{McCarthy-Papanikolas}
Dermot McCarthy and Matthew~A.  Papanikolas. A finite field hypergeometric function associated to eigenvalues of a Siegel eigenform,
\emph{International J. of Number Theory} 11(8):2431-2450, 2015.


\bibitem{Milne}James S. Milne, {\em Abelian Varieties}
 \mbox{\url{http://www.jmilne.org/math/CourseNotes/AV.pdf}}

 \bibitem{Milne-cft}James S. Milne, {\em Class Field Theory}
 \mbox{\url{http://www.jmilne.org/math/CourseNotes/CFT.pdf}}

\bibitem{Ono}
Ken Ono.
\newblock Values of {G}aussian hypergeometric series.
\newblock {\em Trans. Amer. Math. Soc.}, 350(3):1205--1223, 1998.

\bibitem{Osburn-S}
Robert Osburn and Carsten Schneider.  Gaussian hypergeometric series and supercongruences. \emph{Math. Comp.} 78 no. 265, 275--292, 2009.

\bibitem{OZ}
Robert Osburn and Wadim Zudilin, On the (K.2) supercongruence of Van Hamme,
\emph{J. of Math. Analysis and Applications}, 433(1):706-711, 2016.

\bibitem{Petkova-Shiga}
Maria Petkova and Hironori Shiga. A new interpretation
of the Shimura curve with discriminant 6 in terms of
Picard modular forms.
\emph{Arch. Math. (Basel)}, 96(4):335?348, 2011.

\bibitem{Ramanujan}
Srinivasa Ramanujan.
\newblock Modular equations and approximations to {$\pi$} [{Q}uart. {J}.
  {M}ath. {\bf 45} (1914), 350--372].
\newblock In {\em Collected papers of {S}rinivasa {R}amanujan}, pages 23--39.
  AMS Chelsea Publ., Providence, RI, 2000.



\bibitem{RRW}
David P. Roberts, Fernando Rodriguez-Villegas. Hypergeometric motives. \emph{Notices Amer. Math. Soc.} 69, no. 6, 914–-929, 2022.


%\bibitem{RV}
%Fernando Rodriguez-Villegas.
%\newblock Hypergeometric motives.
%\newblock 2015.

%\bibitem{RV2}Fernando Rodriguez-Villegas. Hypergeometric families of Calabi-Yau manifolds. \emph{Calabi-Yau varieties and mirror symmetry} (Toronto, ON, 2001), 223--231, Fields Inst. Commun., 38, Amer. Math. Soc., Providence, RI, 2003.

\bibitem{Rouse}
Jeremy Rouse.
Hypergeometric functions and elliptic curves.
\emph{Ramanujan J.} 12, no. 2, 197--205, 2006.

\bibitem{Salerno} Adriana Salerno. Counting points over finite fields and hypergeometric functions. \emph{Funct. Approx. Comment. Math.} 49  no. 1, 137--157, 2013.

\bibitem{Schwarz}
Hermann~A. Schwarz.
\newblock Ueber diejenigen {F}\"alle, in welchen die {G}aussische
  hypergeometrische {R}eihe eine algebraische {F}unction ihres vierten
  {E}lementes darstellt.
\newblock {\em J. Reine Angew. Math.}, 75:292--335, 1873.

\bibitem{Serre1}
Jean-Pierre Serre.
\newblock {\em Repr\'esentations lin\'eaires des groupes finis}.
\newblock Hermann, Paris, revised edition, 1978.

\bibitem{Serre2}
Jean-Pierre Serre.
\newblock {\em Abelian {$l$}-adic representations and elliptic curves}.
\newblock Advanced Book Classics. Addison-Wesley Publishing Company, Advanced
  Book Program, Redwood City, CA, second edition, 1989.
\newblock With the collaboration of Willem Kuyk and John Labute.

\bibitem{Silverman}
Joseph~H. Silverman.
\newblock {\em The arithmetic of elliptic curves}, volume 106 of {\em Graduate
  Texts in Mathematics}.
\newblock Springer-Verlag, New York, 1986.

\bibitem{SilvermanII}
Joseph~H. Silverman.
\newblock {\em Advanced Topics in he arithmetic of elliptic curves}, volume 151 of {\em Graduate
  Texts in Mathematics}.
\newblock Springer-Verlag, New York, 1994.

\bibitem{Slater}
Lucy~Joan Slater.
\newblock {\em Generalized hypergeometric functions}.
\newblock Cambridge University Press, Cambridge, 1966.

\bibitem{S-B}
Jan Stienstra and Frits Beukers.
On the Picard-Fuchs equation and the formal Brauer group of certain elliptic K3-surfaces.
\emph{Math. Ann.} 271(2), 269--304, 1985.



\bibitem{Swisher}
Holly Swisher.
\newblock On the supercongruence conjectures of van {H}amme.
\newblock {\em Research in the Mathematical Sciences}, 2:18, 1--21, 2015.


\bibitem{Takeuchi-triangle}
Kisao Takeuchi.
\newblock Arithmetic triangle groups.
\newblock {\em J. Math. Soc. Japan}, 29(1):91--106, 1977.



\bibitem{Takeuchi}
Kisao Takeuchi. Commensurability classes of arithmetic triangle groups.
\emph{J. Fac. Sci. Univ. Tokyo Sect. IA Math.},
24(1):201--212, 1977.

\bibitem{Tu-Yang2}
Fang-Ting {Tu} and Yifan {Yang}.
\newblock {Evaluation of Certain Hypergeometric Functions over Finite Fields}.
\newblock {\em SIGMA 14 (2018), 050}, May 2018.



\bibitem{Tu-Yang1}
Fang-Ting Tu and Yifan Yang.
\newblock Algebraic transformations of hypergeometric functions and automorphic
  forms on {S}himura curves.
\newblock {\em Trans. Amer. Math. Soc.}, 365(12):6697--6729, 2013.

\bibitem{vanHamme}
Lucien van Hamme.
\newblock Some conjectures concerning partial sums of generalized
  hypergeometric series.
\newblock {\em Lecture Notes in Pure and Appl. Math}, 192:223--236, 1997.

\bibitem{Ve11}
M.~Valentina Vega.
\newblock Hypergeometric functions over finite fields and their relations to
  algebraic curves.
\newblock {\em Int. J. Number Theory}, 7(8):2171--2195, 2011.

\bibitem{Vidunas1}
Raimundas Vid{\=u}nas.
\newblock Transformations of some {G}auss hypergeometric functions.
\newblock {\em J. Comput. Appl. Math.}, 178(1-2):473--487, 2005.

\bibitem{Vidunas2}
Raimundas Vid{\=u}nas.
\newblock Algebraic transformations of {G}auss hypergeometric functions.
\newblock {\em Funkcial. Ekvac.}, 52(2):139--180, 2009.



\bibitem{Watkins}
M. Watkins,  Hypergeometric motives notes, Preprint \url{http://magma.maths.usyd.edu.au/~watkins/papers/known.pdf} (2017).

\bibitem{Weil1}

Andr{\'e} Weil. Numbers of solutions of equations in finite fields. Bull. Amer. Math. Soc. 55, (1949). 497--508.

\bibitem{Weil}

\newblock Andr{\'e} Weil. Jacobi sums as ``{G}r\"ossencharaktere''.
\newblock {\em Trans. Amer. Math. Soc.}, 73:487--495, 1952.



\bibitem{Whipple}
F.J.W. Whipple.
On well-poised series, generalised hypergeometric series having parameters in pairs, each pair with the same sum, \emph{Proc. London Math. Soc}. (2) 24, 247--263, 1926.

\bibitem{Wolfart}
J{\"u}rgen Wolfart.
\newblock Werte hypergeometrischer {F}unktionen.
\newblock {\em Invent. Math.}, 92(1):187--216, 1988.

\bibitem{Yamamoto}
Koichi Yamamoto.
\newblock On a conjecture of {H}asse concerning multiplicative relations of
  {G}aussian sums.
\newblock {\em J. Combinatorial Theory}, 1:476--489, 1966.

\bibitem{Yoshida}
Masaaki Yoshida.
\newblock {\em Fuchsian differential equations}.
\newblock Aspects of Mathematics, E11. Friedr. Vieweg \& Sohn, Braunschweig,
  1987.
\newblock With special emphasis on the Gauss-Schwarz theory.



\bibitem{Whittaker-Watson}
E.~T. Whittaker and G.~N. Watson.
\newblock {\em A course of modern analysis}.
\newblock Cambridge Mathematical Library. Cambridge University Press,
  Cambridge, 1996.
\newblock An introduction to the general theory of infinite processes and of
  analytic functions; with an account of the principal transcendental
  functions, Reprint of the fourth (1927) edition.

\bibitem{Yoshida-love}
Masaaki Yoshida.
\newblock {\em Hypergeometric functions, my love}.
\newblock Aspects of Mathematics, E32. Friedr. Vieweg \& Sohn, Braunschweig,
  1997.
\newblock Modular interpretations of configuration spaces.

\bibitem{Zudilin}
Wadim Zudilin, Ramanujan-type supercongruences, \emph{J. Number Theory}
129(8):1848--1857, 2009.

\end{thebibliography}
\end{document}